\documentclass{amsart}

\usepackage{amsfonts}
\usepackage{color}
\usepackage{verbatim}   
\usepackage{hyperref}   
\usepackage{amsmath}
\usepackage{amssymb}
\usepackage{amsthm}
\makeindex

\usepackage{diagrams}
\newdiagramgrid{hexgridA2}{1,1,1,1}{0.866,0.866,0.866,0.866}
\newdiagramgrid{hexgridG2}{1,1,1,1,1,1}{0.866,0.866,0.866,0.866,0.866,0.866,0.866,0.866}
\newdiagramgrid{sqgridB2}{1,1,1,1}{1,1,1,1}
\newdiagramgrid{sqgridBC2}{1,1,1,1,1,1,1,1}{1,1,1,1,1,1,1,1}
\newdiagramgrid{octagon}{1.414,1,1,1.414}{1.414,1,1,1.414}
\newdiagramgrid{octagonfat}{1.414,1.414,1.414,1.414}{1.414,1,1,1.414}
\newdiagramgrid{grid2F4}{1,1,0.414,1,1,0.414,1,1}{1,1,0.414,1,1,0.414,1,1}
\newarrow{DotLine}{.}{.}{.}{.}{.}

\newcommand{\Md}{\left.M^*\right.^2}
\newcommand{\Hom}{\mathrm{Hom}}

\newcommand{\orth}{{\rm orth}}

\numberwithin{equation}{section}

\newtheorem {Theorem}    {Theorem}[section]

\newtheorem {Fact}       [Theorem]    {Fact}
\newtheorem {Lemma}      [Theorem]    {Lemma}
\newtheorem {Corollary}  [Theorem]    {Corollary}
\newtheorem {Proposition}[Theorem]    {Proposition}
\newtheorem {Claim}      [Theorem]    {Claim}
\newtheorem* {ClaimN}    {Claim}
\newtheorem {Observation}[Theorem]    {Observation}
\newtheorem {Conjecture} [Theorem]    {Conjecture}

\theoremstyle{definition}
\newtheorem* {Definition} {Definition}
\newtheorem* {Remark} {Remark}

\theoremstyle{remark}

\newtheorem {Example}    [Theorem]    {Example}
\newtheorem {Reduction}    [Theorem]    {Reduction}

\newcounter{DM@bibnum}

\newcommand{\Edg}{{\mathcal E}}

\newcommand {\F} {{\mathbb F}}

\newcommand {\Q} {{\mathbb Q}}
\newcommand {\R} {{\mathbb R}}

\newcommand {\Z} {{\mathbb Z}}
\newcommand {\Id} {{\mathrm {Id}}}

\newcommand {\HH} {V}
\newcommand {\transp} {{\rm tr}}
\newcommand {\conj} {{\rm conj}}

\newcommand{\la}{\langle}
\newcommand{\ra}{\rangle}

\renewcommand{\Vert}{{\mathcal V}}

\DeclareMathOperator{\GL}{GL}
\DeclareMathOperator{\SL}{SL}
\DeclareMathOperator{\PSL}{PSL}
\DeclareMathOperator{\St}{St}
\DeclareMathOperator{\EL}{EL}

\DeclareMathOperator{\proj}{proj}
\DeclareMathOperator{\Kdim}{Kdim}
\DeclareMathOperator{\ad}{ad}

\DeclareMathOperator{\Aut}{Aut}

\DeclareMathOperator{\Ker}{Ker}

\DeclareMathOperator{\rk}{rk}

\DeclareMathOperator{\diag}{diag}

\DeclareMathOperator{\codist}{codist}
\DeclareMathOperator{\Asym}{Asym}
\DeclareMathOperator{\Sym}{Sym}
\DeclareMathOperator{\Cay}{Cay}

\DeclareMathOperator{\Alt}{Alt}

\DeclareMathOperator{\Mat}{Mat}

\renewcommand{\epsilon}{\varepsilon}

\newcommand{\M}{{\rm M}}

\newcommand{\eps}{\varepsilon}

\newcommand{\lam}{\lambda}            
                
\renewcommand{\phi}{\varphi}

\newcommand{\ebar}{\bar e}

\renewcommand{\hbar}{\bar h}               
\newcommand{\ibar}{\bar i}

\newcommand{\xbar}{\bar x}

               \newcommand{\Ihat}{\widehat {\mathstrut I}}

               \newcommand{\Mhat}{\widehat {\mathstrut M}}
               \newcommand{\Nhat}{\widehat {\mathstrut N}}

               \newcommand{\Rhat}{\widehat {\mathstrut R}}

               \newcommand{\Xhat}{\widehat {\mathstrut X}}

\newcommand{\grF}{{\mathfrak F}}

\newcommand{\dbC}{{\mathbb C}}

\newcommand{\dbE}{{\mathbb E}}
\newcommand{\dbF}{{\mathbb F}}
\newcommand{\dbG}{{\mathbb G}}

\newcommand{\dbN}{{\mathbb N}}

\newcommand{\dbR}{{\mathbb R}}

\newcommand{\dbT}{{\mathbb T}}

\newcommand{\dbZ}{{\mathbb Z}}

\newcommand{\Fq}{{\dbF_q}}

\newcommand{\FFF}{\grF}
\newcommand{\Gl}{\Gamma_l}
\newcommand{\Gs}{\Gamma_s}

\newcommand{\upbar}{\bar{\,\,\,}\!\!\bar{\,\,\,}}
\newcommand{\block}{\rm block}

\begin{document}


\title{Property $(T)$ for groups graded by root systems}

\author{Mikhail Ershov}
\address{University of Virginia}
\thanks{The first author was supported in part by the NSF grants DMS-0901703 and DMS-1201452
and the Sloan Research Fellowship grant BR 2011-105}
\email{ershov@virginia.edu}

\author{Andrei Jaikin-Zapirain}
\address{Departamento de Matem\'aticas Universidad Aut\'onoma de Madrid\\ and
Instituto de Ciencias Matem\'aticas  CSIC-UAM-UC3M-UCM}
\thanks{The second author is supported by   Spanish Ministry of Science
and Innovation, grant   MTM2011-28229-C02-01}
\email{andrei.jaikin@uam.es}

\author{Martin Kassabov}
\address{Cornell University and University of Southampton}
\thanks{The third author was partially funded by  grants
from National Science Foundation DMS~060024 and~0900932.}
\email{kassabov@math.cornell.edu}

\date {March 10, 2014}


\begin{abstract}
We introduce and study the class of groups graded by root systems.
We prove that if $\Phi$ is an irreducible classical root system of rank $\geq 2$
and $G$ is a group graded by $\Phi$, then under certain natural
conditions on the grading, the union of the root subgroups
is a Kazhdan subset of $G$. As the main application of this theorem
we prove that for any reduced irreducible classical root system $\Phi$ of
rank $\geq 2$ and a finitely generated commutative ring $R$ with $1$,
the Steinberg group $\St_{\Phi}(R)$ and the elementary Chevalley group
$\mathbb E_{\Phi}(R)$ have property $(T)$. We also show that there exists
a group with property $(T)$ which maps onto all finite simple groups of Lie type
and rank $\geq 2$, thereby providing a ``unified'' proof of expansion in these groups.
\end{abstract}

\maketitle
\tableofcontents


\section{Introduction}
\subsection{The main result}
In this paper by a \emph{ring} we will always mean an associative ring with $1$.
In a recent work of the first two authors~\cite{EJ} it was shown
that for any integer $n\geq 3$ and a finitely generated ring $R$,
the elementary linear group $\EL_n(R)$ and the Steinberg group $\St_n(R)$
have Kazhdan's property $(T)$ (in fact property $(T)$ for $\EL_n(R)$
is a consequence of property $(T)$ for $\St_n(R)$ since $\EL_n(R)$
is a quotient of $\St_n(R)$). In this paper we extend this result to
elementary Chevalley groups and Steinberg groups corresponding to other classical root systems of rank $\geq 2$
(see Theorem~\ref{thm:main_Steinberg} below).

We will use the term \emph{root system} in a very broad sense (see \S~4).
By a \emph{classical root system} we mean the root system of some semisimple
algebraic group (such root systems are often called \emph{crystallographic}).

If $\Phi$ is a reduced irreducible classical root system and $R$ a commutative ring,
denote by $\dbG_{\Phi}(R)$ the corresponding simply-connected Chevalley group over $R$
and by $\mathbb E_{\Phi}(R)$ the \emph{elementary subgroup} of $\dbG_{\Phi}(R)$, that is,
the subgroup generated by the root subgroups with respect to the standard torus.
For instance, if $\Phi=A_{n-1}$, then $\dbG_{\Phi}(R)=\SL_n(R)$ and $\mathbb E_{\Phi}(R)=\EL_n(R)$.
For brevity we will refer to $\mathbb E_{\Phi}(R)$ as an \emph{elementary Chevalley group}.
There is a natural epimorphism from the Steinberg group $\St_{\Phi}(R)$ onto $\dbE_{\Phi}(R)$.

\begin{Theorem}
\label{thm:main_Steinberg}
Let $\Phi$ be a reduced irreducible
classical root system of rank $\geq 2$. Let $R$ be a finitely
generated ring, which is commutative if $\Phi$ is not of type $A_n$.
Then the Steinberg group $\St_{\Phi}(R)$ and the elementary Chevalley group $\mathbb E_{\Phi}(R)$
have Kazhdan's property $(T)$.
\end{Theorem}
\begin{Remark} There are many cases when $\mathbb E_{\Phi}(R)=\dbG_{\Phi}(R)$.
For instance, this holds if $R=\Z[x_1,\ldots, x_k]$ or $R=F[x_1,\ldots, x_k]$,
where $F$ is a field, and $\Phi$ is of type $A_n$ (see~\cite{Su})
or $C_n$ (see~\cite{GMV}), with $n\geq 2$.
\end{Remark}
\vskip .16cm

Steinberg groups and elementary Chevalley groups over rings are typical examples of
\emph{groups graded by root systems} which are introduced and studied in this paper.
Our central result asserts that if $G$ is any group graded by a (finite) root
system $\Phi$ of rank $\geq 2$, the grading satisfies certain non-degeneracy
condition, and $\{X_{\alpha}\}_{\alpha\in\Phi}$ are the root subgroups, then
$\cup X_{\alpha}$ is a Kazhdan subset of $G$ (see Theorem~\ref{thm:main2} below).
Theorem~\ref{thm:main_Steinberg} follows primarily from this result and
relative property $(T)$ for the pair $(\St_2(R)\ltimes R^2, R^2)$ established
by Shalom when $R$ is commutative~\cite{Sh1} and by the third author for general $R$~\cite{Ka}; however,
additional considerations are needed in the case when $\Phi$ is not simply laced.
Before discussing the proofs of these results, we briefly comment on the previous
work on property $(T)$ for Chevalley and Steinberg groups and the proof
of the main theorem in~\cite{EJ}.

\subsection{Property $(T)$ for $\EL_n(R)$: summary of prior work}

By the 1967 foundational paper of Kazhdan~\cite{Kazh} and the subsequent work of Vaserstein~\cite{Va},
the Chevalley groups $\dbG_{\Phi}(\Z)=\mathbb E_{\Phi}(\Z)$ and $\dbG_{\Phi}(F[t])=\mathbb E_{\Phi}(F[t])$,
where $F$ is a finite field, have property $(T)$ for any reduced irreducible classical root system
of rank $\geq 2$. The question of whether the groups $\mathbb E_{\Phi}(R)$, with $\rk(\Phi)\geq 2$
(in particular, the groups $\EL_n(R)$, $n\geq 3$) have property $(T)$ for ``larger'' rings $R$
remained completely open for a long time.%
\footnote{The situation in rank $1$ is completely different.
It is easy to see that $\EL_2(R)$ does not
have $(T)$ whenever $R$ (possibly noncommutative) surjects onto $\Z$ or $F[t]$, with $F$ a finite field. A more delicate
argument shows that $\EL_2(R)$ does not have $(T)$ for any infinite commutative ring $R$ (with $1$).
}

In 2005, Kassabov and Nikolov~\cite{KN} showed that the
group $\EL_n(R)$, $n\geq 3$, has property $(\tau)$ (a certain weak form
of property $(T)$) for any finitely generated commutative ring $R$,
which gave an indication that these groups might also have property $(T)$.
This indication was partially confirmed by Shalom in 2006 who proved in~\cite{Sh2}
that the groups $\EL_n(R)$ have property $(T)$ whenever $R$ is commutative and
$n\geq \Kdim(R)+2$ (where $\Kdim(R)$ is the Krull dimension of $R$).
In 2007, Vaserstein~\cite{Va2} eliminated this restriction on the Krull
dimension by showing that $\EL_n(R)$, $n\geq 3$, has property $(T)$ for any finitely
generated commutative ring $R$. Finally, in~\cite{EJ} the result was extended
to arbitrary finitely generated (associative) rings, and the method of proof
was very different from the one used by Shalom and Vaserstein. To explain the idea
behind this method, we recall some standard terminology.

Let $G$ be a discrete group and $S$ a subset of $G$. Following the terminology
in~\cite{BHV}, we will say that $S$ is a \emph{Kazhdan subset} of $G$ if every
unitary representation of $G$ containing almost $S$-invariant vectors must contain a $G$-invariant vector. By definition, $G$ has property $(T)$ if it has a finite Kazhdan subset;
however, one can prove that $G$ has $(T)$ by finding an infinite Kazhdan
subset $K$ such that the pair $(G,K)$ has relative property $(T)$
(see~\S~2 for details).

If $G=\EL_n(R)$, where $n\geq 3$ and $R$ is a finitely generated ring,
the aforementioned results of Shalom and Kassabov yield relative property
$(T)$ for the pair $(G,X)$ where $X=\cup_{i\neq j}E_{ij}(R)$ is the union of
root subgroups. Thus, establishing property $(T)$ for $G$ is reduced
to showing that $X$ is a Kazhdan subset. An easy way to prove the latter
is to show that $G$ is boundedly generated by $X$ --- this is the so-called
bounded generation method of Shalom~\cite{Sh1}. However, $G$ is known to be boundedly
generated by $X$ only in a few cases, namely, when $R$ is a finite extension
of $\Z$ or $F[t]$, with $F$ a finite field.
In~\cite{EJ}, it was proved that $X$ is a Kazhdan subset of $G=\EL_n(R)$
for any ring $R$ using a different method, described in the next
subsection.

\subsection{Almost orthogonality, codistance and a spectral criterion from \cite{EJ}}

Suppose we are given a group $G$ and a finite collection of subgroups
$H_1,\ldots, H_k$ which generate $G$, and we want to know whether the union
of these subgroups $X=\cup H_i$ is a Kazhdan subset of $G$.
By definition, this will happen if and only if given a unitary
representation $V$ of $G$ without (nonzero) invariant vectors,
a unit vector $v\in V$ cannot be arbitrarily close to each of the subspaces
$V^{H_i}$ (where as usual $V^K$ denotes the subspace of $K$-invariant vectors).
In the simplest case $k=2$ the latter property is equivalent to asserting
that the angle between subspaces $V^{H_1}$ and $V^{H_2}$ must be bounded
away from $0$. For an arbitrary $k$, the closeness between the subspaces
$V_1,\ldots, V_k$ of a Hilbert space $V$ can be measured using the notion
of a codistance introduced in~\cite{EJ}. We postpone the formal
definition until~\S~2; here we just say that the codistance
between $\{V_i\}$, denoted by $\codist(V_1,\ldots, V_k)$ is a real number
in the interval $[1/k,1]$, and in the above setting
$\cup H_i$ is a Kazhdan subset of $G$ if and only if
$\sup \{\codist(V^{H_1},\ldots, V^{H_k})\}<1$ where $V$ runs over
all representations of $G$ with $V^{G}=\{1\}$.

Let $H$ and $K$ be subgroups of the same group. We define $\sphericalangle(H,K)$,
the angle between $H$ and $K$, to be the infimum of the set
$\{\sphericalangle(V^H,V^K)\}$ where $V$ runs over all representations of
$\la H,K\ra$ without invariant vectors; we shall also say that
$H$ and $K$ are $\eps$-orthogonal if $\cos\sphericalangle(H,K)\leq \eps$.
The idea of using such angles to prove property $(T)$ already appears in
1991 paper of Burger~\cite{Bu}  and is probably implicit in earlier works
on unitary representations. However, this idea has not been fully exploited
until the paper of Dymara and Januszkiewicz~\cite{DJ} which shows
that for a group $G$  generated by $k$ subgroups $H_1,\ldots, H_k$,
property $(T)$ can be established by controlling ``local information'',
the angles between $H_i$ and $H_j$. More precisely, in~\cite{DJ} it is proved
that if $H_i$ and $H_j$ are $\eps$-orthogonal for $i\neq j$
for sufficiently small $\eps$, then $\cup H_i$ is a Kazhdan subset of $G$
(so if in addition each $H_i$ is finite, then $G$ has property $(T)$).
This ``almost orthogonality method'' was generalized in~\cite{EJ}
using the notion of codistance. As a result, a new spectral criterion
for property $(T)$ was obtained, which is applicable to groups with a
given graph of groups decomposition, as defined below.

Let $G$ be a group and $\Gamma$ a finite graph. A decomposition of $G$
over $\Gamma$ is a choice of a vertex subgroup $G_{\nu}\subseteq G$ for each
vertex ${\nu}$ of $\Gamma$ and an edge subgroup $G_e\subseteq G$ for each
edge $e$ of $\Gamma$ such that
\begin{itemize}
\item[(i)] $G$ is generated by the vertex subgroups $\{G_{\nu}\}$
\item[(ii)] Each vertex group $G_{\nu}$ is generated by edge subgroups $\{G_e\}$,
with $e$ incident to ${\nu}$
\item[(iii)] If an edge $e$ connects $\nu$ and $\nu'$, then $G_e$ is contained in $G_{\nu}\cap G_{\nu'}$.
\end{itemize}
The spectral criterion~\cite[Theorem~5.1]{EJ} asserts that if a group
$G$ has a decomposition $(\{G_{\nu}\}, \{G_e\})$ over a graph $\Gamma$,
and for each vertex $\nu$ of $\Gamma$ the codistance between subgroups
$\{G_e\}$, with $e$ incident to $\nu$, is sufficiently small
with respect to the spectral gap of $\Gamma$, then the codistance
between the vertex subgroups $\{G_{\nu}\}$ is less than $1$ (and thus
the union of vertex subgroups is a Kazhdan subset of $G$).

\subsection{Groups graded by root systems and associated graphs of groups}

Property $(T)$ for the groups $\dbG_{\Phi}(R)$ and $\St_{\Phi}(R)$
will be proved using certain generalization of the spectral criterion
from~\cite{EJ}. First we shall describe the relevant graph decompositions
of those groups, for simplicity concentrating on the case $G=\EL_n(R)$.

For each $n\geq 2$ consider the following graph denoted by
$\Gamma(A_n)$. The vertices of $\Gamma(A_{n})$ are labeled
by the elements of the symmetric group $\Sym(n+1)$,
and two vertices $\sigma$ and $\sigma'$ are connected
if any only if they are not opposite to each other
in the Cayley graph of $\Sym(n+1)$ with respect to
the generating set $\{(12),(23),\ldots, (n, n+1)\}$.

Now let $R$ be ring and $n\geq 3$. The group $G=\EL_{n}(R)$ has a natural
decomposition over $\Gamma(A_{n-1})$ defined as follows. For each $\sigma\in \Sym(n)$
the vertex group $G_{\sigma}$ is defined to be the subgroup of $G$
generated by  $\{X_{ij} : \sigma(i)<\sigma(j)\}$ where $X_{ij}=\{e_{ij}(r): r\in R\}$.
Thus, vertex subgroups are precisely the maximal unipotent subgroups of $G$ normalized by the diagonal subgroup; in particular, the vertex subgroup corresponding to the identity permutation
is the upper-unitriangular subgroup of $\EL_n(R)$. If $e$ is the edge
connecting vertices $\sigma$ and $\sigma'$, we define the edge subgroup
$G_e$ to be the intersection $G_{\sigma}\cap G_{\sigma'}$ (note that this
intersection is non-trivial precisely when $\sigma$ and $\sigma'$ are connected
in $\Gamma(A_{n-1})$).

As already discussed in the last paragraph of~\S~1.2,
property $(T)$ for $G=\EL_n(R)$ is reduced to showing that
$\cup X_{ij}$ is a Kazhdan subset of $G$.
By the standard bounded generation argument~\cite{Sh1} it suffices to show that the larger
subset $\cup_{\sigma\in \Sym(n)} G_{\sigma}$ (the union of vertex subgroups
in the above decomposition) is Kazhdan, and one might attempt
to prove the latter by applying the spectral criterion from~\cite{EJ}
to the decomposition of $G$ over $\Gamma(A_{n-1})$ described above.
This almost works. More precisely, the attempted
application of this criterion yields a ``boundary case'',
where equality holds in the place of the desired inequality.
In order to resolve this problem, a slightly generalized version of the spectral
criterion must be used. The precise conditions entering this generalized
spectral criterion are rather technical (see~\S~3), but these
conditions are consequences of a very transparent property satisfied by $\EL_n(R)$,
namely the fact that $\EL_n(R)$ is strongly graded by a root system of type $A_{n-1}$
(which has rank $\geq 2$) as defined below.

Let $G$ be a group and $\Phi$ a classical root system. Suppose
that $G$ has a family of subgroups $\{X_{\alpha}\}_{\alpha\in\Phi}$
such that
\begin{equation}
\label{***}
[X_{\alpha},X_{\beta}]\subseteq\prod _{\gamma\in\Phi\cap (\Z_{>0} \alpha\oplus \Z_{>0} \beta)}X_{\gamma}
\end{equation}
for any $\alpha,\beta\in\Phi$ such that $\alpha\neq -\lam\beta$ with $\lam\in\R_{>0}$.
Then we will say that $G$ is graded by $\Phi$ and
$\{X_{\alpha}\}$ is a $\Phi$-grading of $G$.

Clearly, for any root system $\Phi$ the
Steinberg group $\St_{\Phi}(R)$ and the elementary Chevalley
group $\dbE_{\Phi}(R)$ are graded by $\Phi$ (with $\{X_{\alpha}\}$
being the root subgroups). On the other hand, any group $G$
graded by $\Phi$ has a natural graph decomposition. We already
discussed how to do this for $\Phi=A_n$ (in which case the underlying
graph is $\Gamma(A_n)$ defined above). For an arbitrary $\Phi$, the vertices
of the underlying graph $\Gamma(\Phi)$ are labeled by the elements
of $W_{\Phi}$, the Coxeter group of type $\Phi$, and given $w\in W_{\Phi}$,
the vertex subgroup $G_{w}$ is defined to be
$\la X_{\alpha}: w\alpha\in\Phi^+\ra$ where $\Phi^+$ is the set of
positive roots in $\Phi$ (with respect to some fixed choice of simple
roots). The edges of $\Gamma(\Phi)$ and the edge subgroups of $G$
are defined as in the case $\Phi=A_n$.

Once again, the above decomposition of $G$ over $\Gamma(\Phi)$
corresponds to the boundary case of the spectral criterion from~\cite{EJ},
and the generalized spectral criterion turns out to be applicable
under the addition assumption that the grading of $G$ by $\Phi$
is \emph{strong}. Informally a grading is strong if the inclusion in~(\ref{***})
is not too far from being an equality (see~\S~4 for precise definition).
For instance, if $\Phi$ is a simply-laced system, a sufficient condition for a
$\Phi$-grading $\{X_{\alpha}\}$ to be strong is that $[X_{\alpha}, X_{\beta}]=X_{\alpha+\beta}$
whenever $\alpha+\beta$ is a root.

We can now formulate the central result of this paper.

\begin{Theorem}
\label{thm:main2}
Let $\Phi$ be an irreducible classical root system of rank $\geq 2$, and let
$G$ be a group which admits a strong $\Phi$-grading
$\{X_{\alpha}\}$. Then $\cup X_{\alpha}$ is a Kazhdan subset of $G$.
\end{Theorem}

Theorem~\ref{thm:main2} in the case $\Phi=A_n$ was already established in \cite{EJ};
however, this was achieved by only considering the graph $\Gamma(A_2)$, called
the \emph{magic graph on six vertices} in \cite{EJ}. This was possible
thanks to an observation that every group strongly graded by $A_{n-1}$, $n\geq 3$.
must also be strongly graded by $A_2$; for simplicity we illustrate the latter
for $G=\EL_{n}(R)$. If $n=3k$, the $A_2$ grading follows from the well-known
isomorphism $\EL_{3k}(R)\cong \EL_3(\Mat_k(R))$, and for arbitrary $n$
one should think of $n\times n$ matrices as ``$3\times 3$ block-matrices with blocks of uneven size.'' This observation is a special case of the important concept of
a \emph{reduction of root systems} discussed in the next subsection
(see~\S~6 for full details).

The proof of Theorem~\ref{thm:main2} for arbitrary $\Phi$
follows the same general approach as in the case $\Phi=A_2$ done in~\cite{EJ},
although some arguments which are straightforward for $\Phi=A_2$ require
delicate considerations in the general case. Perhaps more importantly,
the proof presented in this paper provides a conceptual
explanation of certain parts of the argument from~\cite{EJ} and
shows that there was really nothing ``magic'' about the graph $\Gamma(A_2)$.

\subsection{Further examples}

So far we have discussed important, but very specific examples of groups
graded by root systems -- Chevalley and Steinberg groups.
We shall now describe two general methods of constructing new groups
graded by root systems. Thanks to Theorem~\ref{thm:main2} and suitable
results on relative property $(T)$, in this way we will also obtain new
examples of Kazhdan groups.

The first method is simply an adaptation of the construction of twisted
Chevalley groups to a slightly different setting. Suppose we are given a group $G$
together with a grading $\{X_{\alpha}\}$ by a root system $\Phi$ and
a finite group $Q$ of automorphisms of $G$ which permutes the root
subgroups between themselves. Under some additional conditions we can
use this data to construct a new group graded by a (different) root system.
Without going into details, we shall mention that the new group,
denoted by $\widehat{G^Q}$, surjects onto certain subgroup of $G^{Q}$,
the group of $Q$-fixed points of $G$, and the new root system often
coincides with the set of orbits under the induced action of $Q$ on the original root system $\Phi$.
As a special case of this construction, we can take $G=\St_{\Phi}(R)$
for some ring $R$ and let $Q$ be the cyclic subgroup generated by
an automorphism of $G$ of the form $d\sigma$ where $\sigma$ is a ring automorphism of
$\St_{\Phi}(R)$ and $d$ is a diagram automorphism of $\St_{\Phi}(R)$ having the same order as $\sigma$.

In this way we will obtain ``Steinberg covers'' of the usual twisted Chevalley
groups over commutative rings of type ${}^2 A_n$,  ${}^2 D_n, {}^3 D_4$ and ${}^2 E_6$.
The Steinberg covers for the groups of type ${}^2 A_n$ (which are unitary groups over rings with involution)
can also be defined over non-commutative rings; moreover, the construction
allows additional variations leading to a class of groups known as \emph{hyperbolic unitary Steinberg groups}
(see~\cite{HO}, \cite{Bak2}). Using this method one can also construct
interesting families of groups which do not seem to have direct counterparts
in the classical theory of algebraic groups; for instance, we will define Steinberg groups
of type ${}^2 F_4$ -- these correspond to certain groups constructed by Tits~\cite{Ti}
which, in turn, generalize twisted Chevalley groups of type ${}^2 F_4$.
We will show that among these Steinberg-type groups the ones graded by a root system of rank $\geq 2$
have property $(T)$ under some natural finiteness conditions on the data used to construct the twisted group.

The second method is based on the notion of a reduction of root systems defined below.
This method does not directly produce new groups
graded by root systems, but rather shows how given a group $G$ graded by
a root system $\Phi$, one can construct a new grading of $G$ by another
root system of smaller rank.

If $\Phi$ and $\Psi$ are two root systems,
a reduction of $\Phi$ to $\Psi$ is just a surjective map $\eta: \Phi\to\Psi\cup\{0\}$
which extends to a linear map between the real vector spaces spanned
by $\Phi$ and $\Psi$, respectively. Now if $G$ is a group with a $\Phi$-grading
$X_{\alpha}$, for each $\beta\in\Psi$ we set $Y_{\beta}=\la X_{\alpha}: \eta(\alpha)=\beta \ra$.
If the groups $\{Y_{\beta}\}_{\beta\in\Psi}$ happen to generate $G$
(which is automatic, for instance, if $\eta$ does not map any roots to $0$),
it is easy to see that $\{Y_{\beta}\}$ is a $\Psi$-grading of $G$. Furthermore,
if the original $\Phi$-grading was strong, then under some natural assumptions
on the reduction $\eta$ the new $\Psi$-grading will also be strong
(reductions with this property will be called $2$-good).
We will show that every classical
root system of rank $\geq 2$ has a $2$-good reduction to a classical root system
of rank $2$ (that is, $A_2$, $B_2$, $BC_2$ or $G_2$).
\footnote{The reduction of $A_n$ to $A_2$
was already implicitly used in the proof of property $(T)$ for $\St_{n+1}(R)$ in~\cite{EJ}.}
In this way we reduce the proof of Theorem~\ref{thm:main_Steinberg} to
Theorem~\ref{thm:main2} for classical root systems of rank $2$.
While Theorem~\ref{thm:main2} is not any easier to prove in this special case,
using such reductions we obtain much better Kazhdan constants for the groups
$\St_{\Phi}(R)$ and $\mathbb E_{\Phi}(R)$ than what we are able to deduce from direct
application of Theorem~\ref{thm:main2}.
\vskip .15cm

So far our discussion was limited to groups graded by classical root
systems, but our definition of $\Phi$-grading makes sense for
any finite subset $\Phi$. In this paper by a \emph{root system}
we mean any finite subset of $\R^n$ symmetric about the origin,
and Theorem~\ref{thm:main2} remains true for groups graded
by any root system satisfying a minor technical condition (these
will be called \emph{regular root systems}). There are plenty
of regular root systems, which are not classical, but there
is no easy recipe for constructing interesting groups graded by them.
What we know is that reductions can be used to obtain some exotic gradings on
familiar groups -- for instance, the groups $\St_n(R)$ and $\EL_n(R)$ naturally
graded by $A_{n-1}$ can also be strongly graded by certain two-dimensional
root system $I_n$ (see the end of~\S~\ref{sec:reductions}).
We believe that the analysis of this and other similar gradings
can be used to construct new interesting groups, but we leave such
an investigation for a later project.

\subsection{Application to expanders} The subject of expansion in finite simple
groups has seen a burst of activity over the past decade. One of the fundamental
results in this area, established in a combination of several papers~\cite{Ka2,KLN,BGT}, asserts
that all (non-abelian) finite simple groups form a family of expanders (formally
this means that the Cayley graphs of those groups with respect to certain generating
sets of uniformly bounded size form a family of expanders). Note that proofs
for different families of finite simple groups use different techniques and also
vary a lot in level of complexity -- for instance, the proof in the case of alternating groups~\cite{Ka2}
is very involved, while expansion for the groups $\SL_n(\dbF_p)$, $n\geq 3$
is merely a consequence of the fact that $\SL_n(\Z)$ has property $(T)$.

In this paper we address the question of which families $\mathcal F$
of finite simple groups can be realized as quotients of a single group
$G$ with property $(T)$. We will refer to a group which surjects onto every group
in a family $\mathcal F$ as a \emph{mother group} for $\mathcal F$. The main result of~\S~\ref{sec:mothergroup}
(see Theorem~\ref{thm:mothergroup}) asserts that the collection of all finite simple groups of Lie type and rank $\geq 2$ admits a mother group with property $(T)$ (thereby providing a unified proof of expansion for these groups).
It is known that this result cannot be extended to all finite simple groups (even those of Lie type) since
the family $\{\SL_2(\dbF_p)\}$ does not have a mother group with $(T)$; however,
it is still possible that all finite simple groups have a mother group with
property $(\tau)$ (which would be sufficient for expansion). To prove
Theorem~\ref{thm:mothergroup} we first divide all finite simple
groups of Lie type and rank $\geq 2$ into finitely many subfamilies. Then
for each subfamily $\mathcal F$ we construct a strong $\Psi$-grading for each
group $G\in\mathcal F$ by a suitable classical root system $\Psi$ (depending
only on $\mathcal F$). Finally we show that all groups in $\mathcal F$
are quotients of a (possibly twisted) Steinberg group associated to $\Psi$,
which has property~$(T)$ by one of the criteria established in~\S~\ref{sec:steinberg} or~\ref{sec:twisted}.
\vskip .15cm
{\bf A concluding remark.} To the best of our knowledge, the notion of a group graded by a root system 
(as defined in this paper) did not previously appear in the literature, but closely related classes of
groups were considered by several authors. These include groups with Steinberg relations studied
by Faulkner~\cite{Fa2} and groups graded by root systems in the sense of Shi~\cite{Shi}. The latter
class can be defined as groups which are graded by classical root systems in our sense and endowed with
a suitable action of the corresponding Weyl group. These groups are  investigated further
by Zhang~\cite{Zh2} where they are called root graded groups. A more general class of groups
which includes groups graded by Kac-Moody root systems has been studied in \cite{LN}.


\section{Preliminaries}
In this section we shall define the notions of property $(T)$, relative property~$(T)$,
Kazhdan constants and Kazhdan ratios, orthogonality constants,
angles and codistances between subspaces of Hilbert spaces,
and recall some basic facts about them.
We shall also state two new results on Kazhdan constants
for nilpotent groups and group extensions, which will be established
at the end of the paper (in~\S~\ref{sec:Kazhdanconstants}).

All groups in this paper will be assumed discrete, and we shall consider
their unitary representations on Hilbert spaces. By a \emph{subspace}
of a Hilbert space we shall mean a closed subspace unless indicated
otherwise.

\subsection{Property $(T)$}

\begin{Definition}
Let $G$ be a group and $S$ a subset of $G$.
\begin{itemize}
\item[(a)] Let $V$ be a unitary representation of $G$.
A nonzero vector $v\in V$ is called $(S,\eps)$-invariant if
$$\|sv-v\|\leq \epsilon\|v\| \mbox{ for any } s\in S.$$
\item[(b)] Let $V$ be a unitary representation of $G$ without nonzero
invariant vectors. The \emph{Kazhdan constant
$\kappa(G,S,V)$ } is the infimum of the set
$$
\{\epsilon> 0 : V \mbox{ contains an } (S,\epsilon)\mbox{-invariant vector}\}.
$$
\item[(c)] The \emph{Kazhdan constant $\kappa(G,S)$ }of $G$ \index{Kazhdan constant}
with
respect to $S$ is the infimum of the set $\{\kappa(G,S,V)\}$ where
$V$ runs over all unitary representations of $G$ without nonzero invariant vectors.
\item[(d)] $S$ is called a \emph{Kazhdan subset of $G$} if \index{Kazhdan subset}
$\kappa(G,S)>0$.
\item[(e)] A group $G$ has \emph{property $(T)$} if \index{property $(T)$}
$G$ has a finite Kazhdan subset, that is, if
$\kappa(G,S)>0$ for some finite subset $S$ of $G$.
\end{itemize}
\end{Definition}
\begin{Remark}
The Kazhdan constant $\kappa(G,S)$ may only be
nonzero if $S$ is a generating set for $G$ (see, e.g.,~\cite[Proposition 1.3.2]{BHV}). Thus, a
group with property $(T)$ is automatically finitely generated. Furthermore, if $G$ has property $(T)$, then $\kappa(G,S)>0$ for any finite generating set $S$ of $G$, but the Kazhdan constant $\kappa(G,S)$ depends on $S$.
\end{Remark}
\vskip .1cm
We note that if $S$ is a ``large'' subset of $G$, positivity of the Kazhdan constant $\kappa(G,S)$
does not tell much about the group $G$. In particular, the following holds (see, e.g.,~\cite[Proposition~1.1.5]{BHV}
or~\cite[Lemma~2.5]{Sh1} for a slightly weaker version):
\begin{Lemma}
\label{roottwo}
For any group $G$ we have $\kappa(G,G)\geq \sqrt{2}$.
\end{Lemma}

\subsection{Relative property $(T)$ and Kazhdan ratios}
Relative property $(T)$ has been originally defined for pairs $(G,H)$
where $H$ is a normal subgroup of $G$:

\begin{Definition} 
Let $G$ be a group and $H$ a normal subgroup of $G$.
The pair $(G,H)$ has \emph{relative property $(T)$} \index{relative property $(T)$}
if there exist a finite set $S$
and $\eps>0$ such that if $V$ is any unitary representation of $G$
with an $(S,\eps)$-invariant vector, then $V$ has a (nonzero) $H$-invariant vector.
The largest $\eps$ with this property (for a fixed set $S$) is called
the \emph{relative Kazhdan constant} \index{relative Kazhdan constant}
of $(G,H)$ with respect to $S$ and denoted by
$\kappa(G,H;S)$.
\end{Definition}
\begin{Remark} 
In the computation of relative Kazhdan constants it is enough
to consider cyclic representations of $G$. Indeed, if $V$ is any unitary representation
of $G$ which does not have $H$-invariant vectors, but has an $(S,\eps)$-invariant vector
$v$, then $W:=\overline{\mathrm{span}(Gv)}$, the cyclic subrepresentation generated by $v$,
has the same property. Of course, the same remark applies to the computation
of usual Kazhdan constants.
\end{Remark}

The generalization of the notion of relative property $(T)$ to pairs $(G,B)$,
where $B$ is an arbitrary subset of a group $G$, has been given by de Cornulier~\cite{Co}
and can be defined as follows (see also a remark in~\cite[\S~2]{EJ}):

\begin{Definition} 
Let $G$ be a group and $B$ a subset of $G$.
The pair $(G,B)$ has \emph{relative property $(T)$} \index{relative property $(T)$}
if for any
$\epsilon>0$ there are a finite subset $S$ of $G$  and $\mu>0$
such that if $V$ is any unitary representation of $G$ and $v\in V$ is
$(S,\mu)$-invariant, then $v$ is $(B,\eps)$-invariant.
\end{Definition}

In this more general setting it is not clear how to ``quantify''
the relative property $(T)$ using a single real number.
However, this is possible under the additional assumption that
the dependence of $\mu$ on $\eps$ in the above definition
may be expressed by a linear function. In this case we can
define the notion of a Kazhdan ratio:
\begin{Definition}
Let $G$ be a group and $B$ and $S$ subsets of $G$.
 The \emph{Kazhdan ratio $\kappa_r(G,B;S)$ } \index{Kazhdan ratio}
is the
largest $\delta\in\R$ with the following property: if $V$ is any
unitary representation of $G$ and $v\in V$ is
$(S,\delta\eps)$-invariant, then $v$ is
$(B,\eps)$-invariant.
\end{Definition}
Somewhat surprisingly, there is a simple relationship between
Kazhdan ratios and relative Kazhdan constants:
\begin{Observation}
\label{Kazhrat}
Let $G$ be a group, and let $B$ and $S$ be subsets of $G$. The following hold:
\begin{itemize}
\item[(i)] If $\kappa_r(G,B;S)>0$ and $S$ is finite, then $(G,B)$ has relative $(T)$.
\item[(ii)] If $B$ is a normal subgroup of $G$, then
$$ \sqrt{2}\kappa_r(G,B;S)\leq \kappa(G,B;S)\leq 2\kappa_r(G,B;S).$$
In particular, $(G,B)$ has relative $(T)$ if and only if $\kappa_r(G,B;S)>0$ for some finite set $S$.
\item[(iii)] $\kappa(G,S)\geq \kappa(G,B)\kappa_r(G,B;S)$
\end{itemize}
\end{Observation}
\begin{proof} (i) and (iii) follow immediately from the definition.
The first inequality in (ii) holds by Lemma~\ref{roottwo} applied to $B$
(here we just need $B$ to be a subgroup, not necessarily normal).
Finally, the second inequality in (ii) is a standard fact proved, for instance, 
in~\cite[Corollary~2.3]{Sh1}), but for completeness we reproduce the argument here.

Consider any unitary representation $V$ of $G$.
We have $V=V^B\oplus (V^B)^{\perp}$, where both $V^B$ and $(V^B)^{\perp}$
are $G$-invariant since $B$ is normal in $G$. Now take any nonzero $v\in V$,
and assume that $v$ is $(S,\delta)$-invariant for some $\delta$. Write
$v=v_{b}+v_{b}^{\perp}$ where $v_b\in V^B$ and $v_{b}^{\perp}\in (V^B)^{\perp}$.

For any $s\in S$ we have $\|sv-v\|\leq \delta \|v\|$. Since
both $V^B$ and $(V^B)^{\perp}$ are $G$-invariant,
$\|sv-v\|^2=\|sv_b-v_b\|^2+\|sv_b^{\perp}-v_b^{\perp}\|^2$, and thus
$$
\|sv_b^{\perp}-v_b^{\perp}\|\leq \delta \|v\|.
$$
On the other hand, $(V^B)^{\perp}$ has no $B$-invariant vectors, so there
exists $s\in S$ such that $\|sv_b^{\perp}-v_b^{\perp}\|\geq \mu\|v_b^{\perp}\|$
where $\mu=\kappa(G,B;S)$ (we can assume that $\mu>0$ since otherwise there is
nothing to prove). Combining the two inequalities, we get
$\|v_b^{\perp}\|\leq \frac{\delta}{\mu}\|v\|$.

Now take any $b\in B$. Since $b$ fixes $v_b$, we have
$$
\|bv-v\|=\|b v_{b}^{\perp}-v_{b}^{\perp}\|\leq 2\|v_{b}^{\perp}\|\leq \frac{2\delta}{\mu}\|v\|,
$$
so by definition of the Kazhdan ratio, $\kappa_r(G,B;S)\geq \frac{\mu}{2}$, as desired.
\end{proof}

\subsection{Using relative property $(T)$}

A typical way to prove that a group $G$ has property $(T)$
is to find a subset $K$ of such that
\begin{itemize}
\item[(a)] $K$ is a Kazhdan subset of $G$
\item[(b)] the pair $(G,K)$ has relative property $(T)$.
\end{itemize}

Clearly, (a) and (b) imply that $G$ has property $(T)$. Note
that (a) is easy to establish when $K$ is a large subset of $G$,
while (b) is easy to establish when $K$ is small, so to obtain
(a) and (b) simultaneously one typically needs to pick $K$
of intermediate size.

In all our examples, pairs with relative property $(T)$ will be produced with the aid
of the following fundamental result: if $R$ is any finitely generated ring,
then the pair $(\EL_2(R)\ltimes R^2, R^2)$ has relative property $(T)$.
This has been proved by Burger~\cite{Bu} for $R=\Z$, by Shalom~\cite{Sh1} for
commutative $R$ and by Kassabov~\cite{Ka} in general. In fact, we shall
use what formally is a stronger result, although its proof
is identical to the one given in~\cite{Ka}:

\begin{Theorem}
\label{KasShalom2}
Let $R\star R$ denote the free product of two
copies of the additive group of $R$, and consider the semi-direct
product $(R\star R)\ltimes R^2$, where the first copy of $R$
acts by upper-unitriangular matrices, that is, $r\in R$ acts
as left multiplication by the matrix
$\left(
\begin{array}{cc}
1& r\\ 0& 1
\end{array}
\right)$,
and the second copy of $R$
acts by lower-unitriangular matrices. Then the pair
$((R\star R)\ltimes R^2,R^2)$ has relative property $(T)$.
\end{Theorem}

In~\S~7 we shall state a higher-dimensional generalization of this theorem
along with an explicit bound for the relative Kazhdan constant
(see Theorem~\ref{relativeAn}). For completeness, we will
present proofs of both Theorem~\ref{KasShalom2} and \ref{relativeAn}
in Appendix~A.
\vskip .12cm

Given a group $G$, once we have found some subsets $B$ of $G$ for which $(G,B)$ has relative property $(T)$, it is easy
to produce many more subsets with the same property. First it is clear that
if $(G,B_i)$ has relative $(T)$ for some finite collection of subsets $B_1,\ldots, B_k$,
then $(G,\cup B_i)$ also has relative $(T)$. Indeed, suppose that given $\eps>0$
there exist finite subsets $S_1,\ldots,S_k$ of $G$ and $\mu_1,\ldots,\mu_k>0$
such that for each $1\leq i\leq k$, if a vector $v$ in some unitary representation of $G$
is $(S_i,\mu_i)$-invariant, then $v$ is also $(B_i,\epsilon)$-invariant.
If we set $S=\cup S_i$ and $\mu=\min\{\mu_i\}$, then any $v$ which is $(S,\mu)$-invariant
must also be $(B,\epsilon)$-invariant.

Here is a  more interesting result of this kind, on which the bounded generation method is based.

\begin{Lemma}[Bounded generation principle]
\label{BGP}
Let $G$ be a group, $S$ a subset of $G$ and  $B_1, \ldots, B_k$
a finite collection of subsets of $G$.  Let $B_1\ldots B_k$ be the set of all elements
of $G$ representable as $b_1\ldots b_k$ with $b_i\in B_i$.
\begin{itemize}
\item[(a)] Suppose that $(G,B_i)$ has relative $(T)$ for each $i$. Then $(G,B_1\ldots B_k)$ also has
relative $(T)$.
\item[(b)] Suppose in addition that $\kappa_r(G,B_i;S)>0$ for each $i$.
Then
$$
\kappa_r(G,B_1\ldots B_k;S)
\geq
\frac{1}{\sum\limits_{i=1}^k \frac{\textstyle 1}{\textstyle \kappa_r(G,B_i;S)}}
\geq
\frac{\min\{\kappa_r(G,B_i;S)\}_{i=1}^k}{k} > 0.
$$
\end{itemize}
\end{Lemma}
\begin{proof} Let $V$ be a unitary representation of $G$.
For any $b_1\in B_1,\ldots, b_k\in B_k$ and $v\in V$ we have
$$
\left\|\left(\prod_{i=1}^k b_i\right) v-v\right\|=
\left\|\sum_{i=1}^k \left(\prod_{j=1}^i b_j v-\prod_{j=1}^{i-1} b_j v\right)\right\|=
\left\|\sum_{i=1}^k \prod_{j=1}^{i-1} b_j (b_i v-v)\right\|\leq 
\sum_{i=1}^{k}\|b_i v-v\|.
$$
Thus, if $v$ is $(B_i,\eps_i)$-invariant for each $i$, then $v$ is
$(B_1\ldots B_k,\sum_{i=1}^k \eps_i)$-invariant. Both (a) and (b)
follow easily from this observation.
\end{proof}

For convenience, we introduce the following terminology:

\begin{Definition}
Let $G$ be a group and $B_1,\ldots, B_k$ a finite collection of subsets of $G$.
Let $H$ be another subset of $G$. We will say that $H$ \emph{lies in a bounded product of $B_1,\ldots, B_k$}
if there exists $N\in\dbN$ such that any element $h\in H$ can be written as $h=h_1\ldots h_N$
with $h_i\in \cup_{j=1}^k B_j$ for all $i$.
\end{Definition}

By Lemma~\ref{BGP} if a group $G$ has subsets $B_1,\ldots, B_k$ such that $(G,B_i)$ has relative $(T)$
for each $i$, then $(G,H)$ has relative $(T)$ for any subset $H$ which lies in a bounded product of $B_1,\ldots, B_k$.
This observation will be frequently used in the proof of property $(T)$ for groups graded by
non-simply laced root systems.

\subsection{Orthogonality constants, angles and codistances}

The notion of the orthogonality constant between two subspaces
of a Hilbert space was introduced and successfully applied in~\cite{DJ}
and also played a key role in~\cite{EJ}:

\begin{Definition}
Let $V_1$ and $V_2$ be two (closed) subspaces
of a Hilbert space $\HH$.
\begin{itemize}
\item[(i)] The \emph{orthogonality constant} $\orth(V_1,V_2)$
between $V_1$ and $V_2$ is defined by
$$
\orth(V_1,V_2)=\sup\{| \la v_1,v_2\ra |:\
||v_i||=1, v_i\in V_i \mbox{ for }i=1,2\}
$$
\item[(ii)] The quantity $\sphericalangle(V_1,V_2)=\arccos (\orth(V_1,V_2))$
will be called the \emph{angle} between $V_1$ and $V_2$.
Thus, $\sphericalangle(V_1,V_2)$ is the infimum of angles between
a nonzero vector from $V_1$ and a nonzero vector from $V_2$.
\end{itemize}
\end{Definition}

These quantities are only of interest when the subspaces
$V_1$ and $V_2$ intersect trivially. In general, it is more
adequate to consider the corresponding quantities after factoring
out the intersection. We call them
the \emph{reduced orthogonality constant and reduced angle}.

\begin{Definition}
Let $V_1$ and $V_2$ be two subspaces in a Hilbert space
$\HH$, and assume that neither of the subspaces $V_1$ and $V_2$
contains the other.
\begin{itemize}
\item[(i)] The \emph{reduced orthogonality constant} $\orth_{red}(V_1,V_2)$
between $V_1$ and $V_2$ is defined by
$$
\orth_{red}(V_1,V_2)=\sup\{ \la v_1,v_2\ra \mid
||v_i||=1, v_i\in V_i, v_i \perp V_1 \cap V_2 \mbox{ for }i=1,2\}.
$$
\item[(ii)] The quantity $\sphericalangle_{red}(V_1,V_2)=\arccos (\orth_{red}(V_1,V_2))$
will be called the \emph{reduced angle between $V_1$ and $V_2$}.
Thus, $\sphericalangle_{red}(V_1,V_2)$ is the infimum of the angles
between nonzero vectors $v_1$ and $v_2$, where $v_i \in
V_i$ and $v_i \perp V_1 \cap V_2$.
\end{itemize}
\end{Definition}
\begin{Remark} If $V_1$ and $V_2$ are two distinct planes in a
three-dimensional Euclidean space, then the reduced angle $\sphericalangle_{red}(V_1,V_2)$ coincides with the usual (geometric)
angle between $V_1$ and $V_2$.
\end{Remark}

Reduced angles play a key role in Kassabov's paper~\cite{Ka3}
(where they are called just `angles'), but in the present paper
the case of subspaces with trivial intersections will suffice.
In fact, in the subsequent discussion we shall operate with
orthogonality constants rather than angles.

The following simple result will be very important for our purposes.

\begin{Lemma}
\label{lem:perp} Let $V_1$ and $V_2$ be two subspaces of a Hilbert space
$V$. Then the reduced angle between the orthogonal complements
$V_1^\perp$ and $V_2^\perp$ is equal to the reduced angle between $V_1$
and $V_2$. Equivalently,
$$
\orth_{red}(V_1^{\perp},V_2^{\perp})=\orth_{red}(V_1,V_2).
$$
\end{Lemma}
\begin{proof} 
This result appears  as~\cite[Lemma 3.9]{Ka3}
as well as~\cite[Lemma~2.4]{EJ} (in a special case), but
it has apparently been known to functional analysts for
a long time (see~\cite{BGM} and references therein).
\end{proof}

The notion of codistance introduced in~\cite{EJ} generalizes
orthogonality constants to the case of more than two subspaces.

\begin{Definition} 
Let $V$ be a Hilbert space, and let $\{U_i\}_{i=1}^n$ be subspaces of $V$.
Consider the Hilbert space $V^n$ and
its subspaces $U_1\times U_2\times\ldots\times U_n$ and
$\diag(V)=\{(v,v,\ldots,v) : v\in V\}$.
The quantity
$$
\codist(\{U_i\})=
\left(\orth(U_1\times\ldots\times U_n,\diag(V))\right)^2
$$
will be called the
\emph{codistance} \index{codistance}
between the subspaces $\{U_i\}_{i=1}^n$. It is easy to see that
$$\
\codist(\{U_i\})=\sup\left\{\frac{\|u_1+\cdots+u_n\|^2}{n(\|u_1\|^2+\cdots+\|u_n\|^2)} : u_i\in U_i\right\}.
$$
\end{Definition}
For any collection of $n$ subspaces $\{U_i\}_{i=1}^n$ we have
$\frac{1}{n}\leq \codist(\{U_i\})\leq 1$, and $\codist(\{U_i\})=\frac{1}{n}$ if and only if
$\{U_i\}$ are pairwise orthogonal. In the case of two subspaces we have an
obvious relation between $\codist(U,W)$ and $\orth(U,W)$:
$$
\codist(U,W)=\frac{1+\orth(U,W)}{2}.
$$
Similarly one can define the \emph{reduced codistance},
but we shall not use this notion. The closely related notion
of (reduced) angle between several subspaces is investigated in~\cite{Ka3}.

We now define the orthogonality constants and codistances
for subgroups of a given group.

\begin{Definition}
\begin{itemize}
\item[(a)] Let $H$ and $K$ be subgroups of the same group
and let $G=\la H,K\ra$, the group generated by $H$ and $K$.
We define $\orth (H,K)$ to be the supremum of the quantities
$\orth (V^H,V^K)$ where $V$ runs over all unitary representations
of $G$ without nonzero invariant vectors.
\item[(b)] Let $\{H_i\}_{i=1}^n$ be subgroups of the same group,
and let $G=\la H_1,\ldots, H_n\ra$. The \emph{codistance } \index{codistance}
between $\{H_i\}$, denoted by $\codist(\{H_i\})$, is defined to
be the supremum of the quantities $\codist (V^{H_1},\ldots, V^{H_n})$,
where $V$ runs over all  unitary representations
of $G$ without nonzero invariant vectors.
\end{itemize}
\end{Definition}

The basic connection between codistance and property $(T)$,
already discussed in the introduction, is given by the following
lemma (see~\cite[Lemma~3.1]{EJ}):

\begin{Lemma}[\cite{EJ}]
\label{orthogT}
Let $G$ be a group and $H_1,H_2,\ldots, H_n$ subgroups of $G$ such that $G=\la H_1,\ldots, H_n\ra$.
Let $\rho=\codist(\{H_i\})$, and suppose that $\rho<1$. The following hold:
\begin{itemize}
\item[(a)] $\kappa(G,\bigcup H_i)\geq\sqrt{2(1-\rho)}. $
\item[(b)] Let $S_i$ be a generating set of $H_i$, and
let $\delta=\min\{\kappa(H_i,S_i)\}_{i=1}^n$. Then
$$
\kappa(G,\bigcup S_i)\geq\delta\,\sqrt{1-\rho}. 
$$
\item[(c)] Assume in addition that each pair $(G,H_i)$ has relative property $(T)$. Then $G$ has property $(T)$.
\end{itemize}
\end{Lemma}

\subsection{Kazhdan constants for nilpotent groups and group extensions}

We finish this section by formulating two theorems and one simple lemma
which provide estimates for Kazhdan constants. These results are new
(although they have been known before in some special cases~\cite{BHV,Ha,NPS}).
The two theorems will be established in~\S~\ref{sec:Kazhdanconstants} of this paper, while the lemma will be proved right away.

The first theorem concerns relative Kazhdan constants in central
extensions of groups:

\begin{Theorem}
\label{relKazhdan} Let $G$ be a group, $N$ a normal subgroup of
$G$ and   $Z$ a subgroup of $Z(G)\cap N$. Put $H=Z\cap [N,G]$. Assume that
$A$, $B$ and $C$ are subsets of $G$ satisfying the following conditions:
\begin{enumerate}
\item $A$ and $N$ generate  $G$,
\item $\kappa (G/Z,N/Z;B)\ge
\epsilon$,
\item $\kappa (G/H,Z/H;C)\ge \delta$.
\end{enumerate}
Then
$$
\kappa(G,N;A\cup B\cup C)\ge \frac {1}{\sqrt 3}\min\left\{  \frac{12\eps}{ 5\sqrt{72   \eps^2|A|+25|B|}}, \delta \right\}.
$$
\end{Theorem}

In a typical application of this theorem the following additional
conditions will be satisfied:
\begin{itemize}
\item[(a)] The group $G/N$ is finitely generated
\item[(b)] The group $Z/H$ is finite
\end{itemize}
In this case (3) holds automatically, and (1) holds for some finite set $A$.
Therefore, Theorem~\ref{relKazhdan} implies that under the additional
assumptions (a) and (b), relative property $(T)$ for the pair
$(G/Z,N/Z)$ implies relative property $(T)$ for the pair $(G,N)$.

\vskip .1cm
The second theorem that we shall use gives a bound for the codistance
between subgroups of a nilpotent group. It is not difficult to see
that if $G$ is an abelian group generated by subgroups $X_1,\ldots, X_k$,
then $\codist(X_1,\ldots, X_k)\leq 1-\frac{1}{k}$ (and this bound
is optimal). We shall prove a similar (likely far from optimal)
bound in the case of countable nilpotent groups:

\begin{Theorem}
\label{nilpotentcodistance}
Let $G$ be a countable  nilpotent group of class $c$ generated by
subgroups $X_1,\dots, X_k$.
Then
$$
\codist(X_1,\dots, X_k)\le 1-\frac{1}{4^{c-1}k}.
$$
\end{Theorem}
\vskip .1cm

We finish with a technical lemma which yields certain
supermultiplicativity property
involving Kazhdan ratios. It can probably be applied in a variety of situations, but in this paper
it will only be used to obtain a better Kazhdan constant for the
Steinberg groups of type $G_2$:

\begin{Lemma}
\label{auxg2} 
Let $G$ be a group, $H$ a subgroup of $G$ and $N$ a normal subgroup of $H$.
Suppose that there exists a subset $\Sigma$ of $G$ and real numbers $a,b>0$
such that
\begin{enumerate}
\item $\kappa_r(G,N;\Sigma)\ge \frac 1 a$

\item $\kappa(H/N, \overline{\Sigma\cap H})\ge \frac 1 b$ where $\overline{\Sigma\cap H}$ is the image
of $\Sigma\cap H$ in $H/N$.
\end{enumerate}
Then $\kappa_r(G,H;\Sigma)\ge \frac{\textstyle 1}{\textstyle \sqrt {2a^2+4b^2}}$.
\end{Lemma}
\begin{proof}
Fix $\eps>0$, and let $V$ be a unitary representation of $G$ and $v\in V$ such that 
$$
\|sv-v\|\le \eps\|v\| \textrm{\ for any\ } s\in\Sigma.
$$
Write $v=v_1+v_2$, where $v_1\in V^N$ and $v_2\in (V^N)^\perp$. Since $\kappa_r(G,N;\Sigma)\ge \frac 1 a$, we obtain that for every
$n\in N$ 
$$
\|nv_2-v_2\|= \|nv-v\|\le a\eps \|v\|.
$$
On the other hand, by Lemma~\ref{roottwo} there exists $n\in N$ such that $\|nv_2-v_2\|\geq \sqrt{2}\|v_2\|$.
Hence $\|v_2\|\le \frac{a\eps}{\sqrt 2}\|v\|$.

Since $N$ is normal in $H$, the subspaces $V^N$ and $(V^N)^\perp$ are $H$-invariant.
Hence for any $s\in \Sigma\cap H$ we have $\|sv_1-v_1\|\le \|sv-v\|\le \eps\|v\|$.

By Observation~\ref{Kazhrat}(ii) we have 
$$
\kappa_r(H/N,H/N;\overline{\Sigma\cap H})\geq 
\frac{\kappa(H/N,H/N;\overline{\Sigma\cap H})}{2}=
\frac{\kappa(H/N,\overline{\Sigma\cap H})}{2}\geq 
\frac{1}{2b}.
$$
Thus, considering $V^N$ as a representation of $H/N$, we obtain that
$\|hv_1-v_1\|\le 2b\eps\|v_1\|\le 2b\eps\|v\|$ for any $h\in H$. Hence for any $h\in H$,
$$
\|hv-v\|^2= \|hv_1-v_1\|^2+\|hv_2-v_2\|^2\le 4b^2\eps^2\|v\|^2+4\|v_2\|^2\le \eps^2(2a^2+4b^2)\|v\|^2.
\qedhere
$$
\end{proof}


\section{Generalized spectral criterion}
\label{sec:spec_criteria}

\subsection{Graphs and Laplacians}

Let $\Gamma$ be a finite  graph without loops. We will denote the set of
vertices of $\Gamma$ by $\Vert(\Gamma)$ and the set of edges by $\Edg(\Gamma)$.
For any edge $e=(x,y)\in \Edg(\Gamma)$, we let $\bar e=(y,x)$ be the inverse of $e$.
We assume that if $e\in \Edg(\Gamma)$, then also $\bar e \in \Edg(\Gamma)$.
If $e=(x,y)$, we let $e^-=x$  be the initial vertex of $e$ and by $e^+= y $
the terminal vertex of $e$.

Now assume that the graph $\Gamma$ is connected, and fix a Hilbert space $V$.
Let $\Omega^0(\Gamma,V)$ be the Hilbert space of functions $f:\Vert(\Gamma)\to V$ with the scalar product
$$
\la f,g \ra=\sum_{y\in \Vert(\Gamma)}\la f(y),g(y)\ra
$$
and let
$\Omega^1(\Gamma,V)$ be the Hilbert space of functions $f:\Edg(\Gamma)\to V$ with
the scalar product
$$
\la f,g \ra= \frac{1}{2}\sum_{e\in \Edg(\Gamma)}\la f(e),g(e)\ra.
$$
Define the linear operator
$$d:\Omega^0(\Gamma,V)\to \Omega^1(\Gamma,V) \mbox{ by } (df)(e)=f(e^+)-f(e^-).$$
We will refer to $d$ as the \emph{difference operator of $\Gamma$}.

\noindent
The adjoint operator $d^*:\Omega^1(\Gamma,V)\to \Omega^0(\Gamma,V)$ is
given by formula
$$(d^*f)(y)=\sum_{y=e^+}\frac{1}{2}\left(f(e)-f(\bar e))\right).$$
The symmetric operator $\Delta=d^*d:\Omega^0(\Gamma,V)\to \Omega^0(\Gamma,V)$ is
called the \emph{Laplacian of $\Gamma$} \index{Laplacian}
and is given by the formula
$$(\Delta f)(y)=\sum_{y=e^+}(f(y)-f(e^-))=\sum_{y=e^+} df(e).$$

The smallest positive eigenvalue of $\Delta$ is commonly denoted
by $\lam_1(\Delta)$ and called the \emph{spectral gap} of the graph $\Gamma$
(clearly, it is independent of the choice of $V$).

\subsection{Spectral criteria}

\begin{Definition}
Let $G$ be a group and $\Gamma$ a finite graph without loops. A \emph{
graph of groups decomposition (or just a decompoisition) of $G$ over $\Gamma$}
is a choice of a vertex subgroup
$G_{\nu}\subseteq G$ for every $\nu\in \Vert(\Gamma)$ and an edge subgroup
$G_e\subseteq G$ for every $e\in \Edg(\Gamma)$ such that
\begin{itemize}
\item[(a)] The vertex subgroups $\{G_{\nu} : \nu\in \Vert(\Gamma)\}$ generate $G$;
\item[(b)] $G_e=G_{\ebar} \mbox{ and } G_e\subseteq
G_{e^+}\cap G_{e^-}\mbox{ for every } e\in \Edg(\Gamma).$
\end{itemize}
We will say that the decomposition of $G$ over $\Gamma$ is \emph{regular}
if for each $\nu\in \Vert(\Gamma)$ the vertex group $G_{\nu}$ is generated
by the edge subgroups $\{G_e : e^+=\nu\}$
\end{Definition}

The following criterion is proved in~\cite{EJ}:
\begin{Theorem}
\label{th:spec_criteria}
Let $\Gamma$ be a connected $k$-regular graph and let $G$ be a group
with a given regular decomposition over $\Gamma$.
For each $\nu\in\Vert(\Gamma)$ let $p_{\nu}$ be
the codistance between the subgroups $\{G_e : e^+=\nu\}$ of $G_{\nu}$,
and let $p=\max\limits_{\nu} p_{\nu}$. Let $\Delta$ be the Laplacian of $\Gamma$,
and assume that $$p< \frac{\lambda_1(\Delta)}{2k}.$$
Then $\cup_{\nu\in\Vert(\Gamma)} G_{\nu}$ is a Kazhdan subset of $G$, and moreover
$$
\kappa(G,\cup G_{\nu}) \geq \sqrt{\frac{2(\lam_1(\Delta)-2pk)}{\lam_1(\Delta)(1-p)}}.
$$
\end{Theorem}

In~\cite{EJ} a slight modification of this criterion was applied to groups graded by root systems of type
$A_2$ with their canonical graph of groups decompositions (as described in the introduction).
In this case one has $p=\frac{\lambda_1(\Delta)}{2k}$, and thus
Theorem~\ref{th:spec_criteria} is not directly applicable; however this
problem was bypassed in \cite{EJ} using certain trick.
We shall now describe a generalization of
Theorem~\ref{th:spec_criteria} which essentially formalizes that trick
and allows us to handle the ``boundary'' case $p = \frac{\lambda_1(\Delta)}{2k}$.

\medskip
First, we shall use extra data -- in addition to a decomposition of the group $G$
over the graph $\Gamma$, we choose a normal subgroup $CG_{\nu}$ of $G_{\nu}$
for each vertex ${\nu}$ of $\Gamma$, called the \emph{core subgroup} \index{core subgroup}
of $G_{\nu}$.
We shall assume that if a representation of the vertex group $G_{\nu}$ does not have any
$CG_{\nu}$-invariant vectors, then
the codistance between the fixed subspaces of the edge groups is bounded
above by $\frac{\lambda_1(\Delta)}{2k}-\eps$ for some $\eps>0$ (independent
of the representation). In order for this extra assumption to be useful,
we need to know that there are sufficiently many representations of $G_{\nu}$
without $CG_{\nu}$-invariant vectors (for instance, if $CG_{\nu}=\{1\}$, there
are no such non-trivial representations), and thus we shall also
require that the core subgroups $CG_{\nu}$ are not too small.
\vskip .15cm

Let us now fix a group $G$, a regular decomposition of $G$ over a graph $\Gamma$, and
a normal subgroup $CG_{\nu}$ of $G_{\nu}$ for each $\nu\in \Vert(\Gamma)$.

Let $\HH$ be a unitary representation of $G$. Let
$\Omega^0(\Gamma,\HH)^{\{G_\nu\}}  $ denote the subspace of $\Omega^0(\Gamma,\HH)$
consisting of all function $g: V(\Gamma) \to \HH$ such that
$g(\nu) \in \HH^{G_\nu}$ for any $\nu\in \Vert(\Gamma)$.
Similarly we define the subspace $\Omega^1(\Gamma,\HH)^{\{G_e\}}$ of
$\Omega^1(\Gamma,\HH)$.

\vskip .15cm
For each vertex $\nu$ we consider the decomposition of $V$ into a direct sum
of three  subspaces:
$$
\HH=\HH^{G_{\nu}}\oplus ((\HH^{G_{\nu}})^{\perp}\cap \HH^{CG_{\nu}})\oplus (\HH^{CG_{\nu}})^{\perp}.
$$
Note that $(\HH^{G_{\nu}})^{\perp}\cap  \HH^{CG_{\nu}}$ (resp. $\HH^{CG_{\nu}})^{\perp}$)
is a representation of $G_{\nu}$ without $G_{\nu}$-invariant (resp. $CG^{\nu}$-invariant) vectors,
so we can apply codistance bounds from Theorem~\ref{th:gen_spec_criteria} (i) and (ii) below to those
representations of $G_{\nu}$.

Combining these decompositions over all vertices, we obtain the corresponding decomposition
of $\Omega^0(\Gamma,\HH)$ into a direct sum of three subspaces, and denote by $\rho_1,\rho_2,\rho_3$
the projection maps onto those subspaces.
\vskip .1cm

Explicitly, the projections $\rho_1,\rho_2,\rho_3$ are defined as follows:
For a function $g \in \Omega^0(\Gamma,\HH)$ and $\nu\in\Vert(\Gamma)$
we set
\begin{align*}
&\rho_1(g)(\nu)=\pi_{\HH^{G_{\nu}}}(g(\nu))&
&\rho_2(g)(\nu)=\pi_{(\HH^{G_{\nu}})^\perp \cap \HH^{CG_{\nu}}}(g(\nu))&\\
&\rho_3(g)(\nu)=\pi_{(\HH^{CG_{\nu}})^\perp}(g(\nu)),&
&&
\end{align*}
that is, the values of
$\rho_i(g)$ for $i=1,2$ and $3$ at the vertex $\nu$ are the projections of the vector $g(\nu)\in \HH$
onto the subspaces $\HH^{G_{\nu}}$, $(\HH^{G_{\nu}})^\perp \cap \HH^{CG_{\nu}}$ and $(\HH^{CG_{\nu}})^\perp$, respectively.

By construction $\rho_1$ is just the projection onto $\Omega^0(\Gamma,\HH)^{\{G_\nu\}}$, and
we have
$$
\|\rho_1(g)\|^2+\|\rho_2(g)\|^2+\|\rho_3(g)\|^2=\|g\|^2\mbox{ for every } g\in \Omega^0(\Gamma,\HH).
$$

Similarly, we define the projections $\rho_1,\rho_2,\rho_3$ on the space $\Omega^1(\Gamma,\HH)$:
for a function $g \in \Omega^1(\Gamma,\HH)$ and an edge $e\in \Edg(\Gamma)$ we set
\begin{align*}
&\rho_1(g)(e)=\pi_{\HH^{G_{e^+}}}(g(e))&
&\rho_2(g)(e)=\pi_{(\HH^{G_{e^+}})^\perp \cap \HH^{CG_{e^+}}}(g(e))&\\
&\rho_3(g)(e)=\pi_{(\HH^{CG_{e^+}})^\perp}(g(e)),&
&&
\end{align*}
Again we have
$$ 
\|\rho_1(g)\|^2+\|\rho_2(g)\|^2+\|\rho_3(g)\|^2=\|g\|^2\mbox{ for every } g\in \Omega^1(\Gamma,\HH).
$$

\begin{Claim}
\label{cl:invariance_under_rho}
The projections $\rho_i$, $i=1,2,3$, preserve the subspace
$\Omega^1(\Gamma,\HH)^{\{G_e\}}$.
\end{Claim}
\begin{proof}
The projection $\rho_1$ preserve the space $\Omega^1(\Gamma,\HH)^{\{G_e\}}$
because $G_{e^+}$ contains the group $G_{e}$. The other two projections
preserve this space because $CG_{e^+}$ is a normal subgroup of $G_{e^+}$.
\end{proof}

Now we are ready to state the desired generalization of Theorem~\ref{th:spec_criteria}.

\begin{Theorem}
\label{th:gen_spec_criteria}
Let $\Gamma$ be a connected $k$-regular graph. Let $G$ be a group
with a chosen regular decomposition over $\Gamma$, and
choose a normal subgroup $CG_{\nu}$ of $G_{\nu}$ for each $\nu\in\Vert(\Gamma)$.
Let $\bar p=\frac{\lambda_1(\Delta)}{2k}$, where $\Delta$ is the Laplacian of $\Gamma$.
Suppose that
\begin{itemize}
\item[(i)] For each vertex $\nu$ of $\Gamma$
the codistance between the subgroups
$\{G_e : e^+=\nu\}$ of $G_v$ is bounded above by $\bar p$.

\item[(ii)] There exists $\eps>0$ such that for any $\nu\in\Vert(\Gamma)$ and
any unitary representation $\HH$ of the vertex group $G_{\nu}$
without $CG_{\nu}$ invariant vectors, the codistance between the fixed subspaces of $G_e$,
with $e^+=\nu$, is bounded above by $\bar p (1-\eps)$;

\item[(iii)] There exist constants $A,B$ such that for any unitary representation $\HH$ of $G$ and for any function
$g\in \Omega^0(\Gamma,\HH)^{\{G_\nu\}}  $ one has
$$
\| dg \|^2 \leq A \|\rho_1(dg) \|^2 + B \|\rho_3(dg) \|^2.
$$
\end{itemize}
Then $\cup G_{\nu}$ is a Kazhdan subset of $G$ and
$$
\kappa(G,\cup G_{\nu}) \geq
\sqrt{\frac{4 \eps k}{ \eps \lambda_1(\Delta) A + \left(2k - \lambda_1(\Delta)\right) B}}> 0.
$$
\end{Theorem}
\begin{Remark}
\begin{itemize}
\item[(a)] Theorem~\ref{th:gen_spec_criteria} implies Theorem~\ref{th:spec_criteria} as
(ii) clearly holds with $\eps = 1 - p / \bar p >0$, and if we put $CG_{\nu}=G_{\nu}$
for each $\nu$, then the projection $\rho_2$ is trivial, and thus (iii) holds with $A=B=1$.

\item[(b)] If $\eps$ is sufficiently large, one can show that the conclusion of Theorem~\ref{th:gen_spec_criteria}
holds even if $\bar p$ is slightly larger than $\frac{\lambda_1(\Delta)}{2k}$, but we
are unaware of any interesting applications of this fact.

\item[(c)] The informal assumption that the core subgroups are not ``too small''
discussed above is ``hidden'' in the condition (iii).
\end{itemize}
\end{Remark}

\subsection{Proof of Theorem~\ref{th:gen_spec_criteria}}

The main step in the proof is to show that the image of $\Omega^0(\Gamma,\HH)^{\{G_\nu\}}  $ under the
the Laplacian $\Delta$ is sufficiently far from $(\Omega^0\left(\Gamma,\HH)^G\right)^\perp$:

\begin{Theorem}
\label{thm:laplacian_angle}
Let $A$ and $B$ be as in Theorem~\ref{th:gen_spec_criteria}. Then
for any $g \in \Omega^0(\Gamma,\HH)^{\{G_\nu\}}  $ we have
$$
\|\rho_1(\Delta g) \|^2\geq  \frac{\eps}{B(1-\bar p)+\eps A\bar p}\|\Delta g \|^2.
$$
\end{Theorem}

\begin{proof}
Let $g$ be an element of $\Omega^0(\Gamma,\HH)^{\{G_\nu\}}  $. This implies that
$dg \in \Omega^1(\Gamma,\HH)^{\{G_e\}}$ and therefore
$\rho_i(dg) \in \Omega^1(\Gamma,\HH)^{\{G_e\}}$ for $i=1,2,3$ by Claim~\ref{cl:invariance_under_rho}.
We have
$$
\rho_i(\Delta g)(\nu)=
\sum_{e^+=\nu} \rho_i(dg)(e).
$$
For $i=1$ we just use the triangle inequality:
$$
\|\rho_1(\Delta g)(\nu)\|^2  = \left\| \sum_{e^+=\nu} \rho_1(dg)(e) \right\|^2
\leq k \sum_{e^+=\nu} \left\|  \rho_1(dg)(e)  \right\|^2.
$$
Summing over all vertices we get
$$
\|\rho_1(\Delta g)\|^2 \leq k \sum_{\nu\in \Vert(\Gamma)} \sum_{e^+=\nu} \left\|  \rho_1(dg)(e) \right\|^2=
2k\, \| \rho_1(dg)\|^2.
$$

If $i=2$ and $i=3$ the vectors
$ \rho_i(dg)(e )$ are in $\HH^{G_e}$ and they
are orthogonal to the spaces $\HH^{G_{\nu}}$ and $\HH^{CG_{\nu}}$, respectively.
Since $(\HH^{G_{\nu}})^{\perp}$ (resp. $(\HH^{CG_{\nu}})^{\perp}$) is a representation
of $G_{\nu}$ without invariant (resp. $CG_{\nu}$-invariant) vectors,
we can use the bounds for codistances from (i) and (ii):
$$
\|\rho_2(\Delta g)(\nu)\|^2  =  \left\| \sum_{e^+=\nu} \rho_2(dg)(e ) \right\|^2
\leq k\bar p \sum_{e^+=\nu} \left\|  \rho_2(dg)(e ) \right\|^2,
$$
and
$$
\|\rho_3(\Delta g)(\nu)\|^2  = \left\| \sum_{e^+=\nu} \rho_3(dg)(e ) \right\|^2
\leq k \bar p (1-\eps)
\sum_{e^+=\nu} \left\|  \rho_3(dg)(e ) \right\|^2.
$$
Again, summing over all vertices $\nu$ yields
$$
\|\rho_2(\Delta g)\|^2 \leq  2k\bar p \,\| \rho_2(dg)\|^2
$$
$$
\|\rho_3(\Delta g)\|^2 \leq  2k\bar p\,\left(1-\eps\right) \| \rho_3(dg)\|^2.
$$
Thus we have
\begin{multline*}
\bar p\left (1-\frac {\eps A} B\right ) \|\rho_1(\Delta g)\|^2+ \|\rho_2(\Delta g)\|^2 + \|\rho_3(\Delta g)\|^2
\leq
\\
\leq
2k \bar p  \left(\left  (1-\frac {\eps A} B\right ) \|\rho_1(d g)\|^2+\|\rho_2(dg)\|^2 + \left(1-\eps\right) \| \rho_3(dg)\|^2 \right)
=
\\
=
2k\bar p \left( \|dg\|^2 - \frac{\eps A}B\|\rho_1(dg)\|^2  - \eps \|\rho_3(dg)\|^2 \right).
\end{multline*}
Combining this inequality with the norm inequality from (iii) and   the fact that by definition of $\Delta$ and $\lambda_1(\Delta)$ we have
$$
\|dg\|^2=\la\Delta g, g\ra \leq \frac{1}{ \lambda_1(\Delta) } \|\Delta g\|^2,
$$
 we get
\begin{multline*}
\left (\bar p\left (1-\frac {\eps A} B\right )-1\right) \|\rho_1(\Delta g)\|^2+ \|(\Delta g)\|^2
\leq
\\
\leq
2k \bar p
\left(1 - \frac{\eps }{B} \right)\|dg\|^2\leq \frac{2k \bar p}{\lambda_1(\Delta)}
\left(1 - \frac{\eps }{B} \right)\|\Delta g \|^2=\left(1 - \frac{\eps }{B} \right)\|\Delta g \|^2.
\end{multline*}
and so 
$$
\|\rho_1(\Delta g)\|^2 \ge
\frac{\eps }{B(1-\bar p)+\eps A\bar p}\|\Delta g \|^2.
\qedhere
$$
\end{proof}

\begin{proof}[Proof of Theorem~\ref{th:gen_spec_criteria}]
Let $\HH$ be a unitary representation of $G$ without invariant vectors.
Let $U$ denote the subspace of $\Omega^0(\Gamma,\HH)$ consisting of
all constant functions, let $W=\Omega^0(\Gamma,\HH)^{\{G_\nu\}}  $ and $V'=\overline{U+W}$.
Define $\widetilde \Delta: V' \to V'$ by
$\widetilde \Delta  = \proj_{V'} \circ \Delta_{| V'}$. Note that if $\iota:V'\to V$
is the inclusion map, then $\iota^*=\proj_{V'}$, considered as a map from $V$ to $V'$.
Hence
$$
\widetilde \Delta=\proj_{V'} \Delta_{| V'}=\iota^* \Delta\iota=
\iota^* d^*d\iota= (d\iota)^* d\iota.
$$
Thus, $\widetilde \Delta$ is self-adjoint and $\Ker \widetilde \Delta=\Ker(d\iota)=U$.
Hence the image of $\widetilde \Delta$ is dense in $U^{\perp V'}$; in fact, one can show
that the image of $\widetilde \Delta$ is closed (see~\cite[p.325]{EJ}), but we will not need
this fact.

Let $P=\frac{\textstyle \eps}{\textstyle B(1-\bar p)+\eps A\bar p}$. 
By Theorem~\ref{thm:laplacian_angle} we have
$\|\proj_{W}\Delta g\|^2\geq  P\|\Delta g\|^2$ for all
$g\in W$. In fact, the same inequality holds for all $g\in V'$
since $U=\Ker\,\Delta$. Notice that when $g\in V'$ we also have
$\|\proj_{W}\widetilde\Delta g\|^2=\|\proj_{W} \proj_{V'}\Delta g\|^2=
\|\proj_{W} \Delta g\|^2$ 
(since $W\subseteq V'$), and therefore
$$
\|\proj_{W}\widetilde\Delta g\|^2=\|\proj_{W} \Delta g\|^2\geq  P\|\Delta g\|^2\geq   P \|\widetilde \Delta g\|^2.
$$
Since the image of $\widetilde \Delta$ is dense in $U^{\perp V'}$, the
obtained inequality can be restated in terms of codistance:
$$
\codist(U^{\perp V'},W^{\perp V'})\leq 1-P.
$$
Note that $U\cap W=\{0\}$ (since $V$ has no $G$-invariant vectors
and $G$ is generated by the vertex subgroups $\{G_{\nu}\}$), so Lemma~\ref{lem:perp} implies
that $\codist(U,W)\leq 1-P.$ But by
definition of the codistance and the definition of subspaces $U$ and $W$ this implies
that the codistance between the vertex subgroups $\{G_{\nu}\}$ of $G$ is bounded
above by $1-P$. Finally, by Lemma~\ref{lem:perp} we have
$$
\kappa(G,\cup G_{\nu})\geq \sqrt{2(1-\codist(\{G_{\nu}\}))}\geq \sqrt{2P} =
\sqrt{\frac{4 \eps k}{ \eps \lambda_1(\Delta) A + \left(2k - \lambda_1(\Delta)\right) B}}> 0.
\qedhere
$$
\end{proof}


\section{Root systems}
The definition of a root system used in this paper is much less
restrictive than that of a classical root system. However,
most constructions associated with root systems we shall consider
are naturally motivated by the classical case.

\begin{Definition}
Let $E$ be a real vector space. A finite non-empty subset $\Phi$ of $E$
is called a \emph{root system in $E$} \index{root system}
if
\begin{itemize}
\item[(a)] $\Phi$ spans $E$;
\item[(b)] $\Phi$ does not contain $0$;
\item[(c)] $\Phi$ is closed under inversion, that is, if $\alpha \in
\Phi$ then $-\alpha \in \Phi$.
\end{itemize}
The dimension of $E$ is called the \emph{rank of $\Phi$}.
\end{Definition}
\begin{Remark} Sometimes we shall refer to the pair $(E,\Phi)$
as a root system.
\end{Remark}

\begin{Definition}
Let $\Phi$ be a root system in $E$.
\begin{itemize}
\item[(i)] $\Phi$ is called \emph{reduced} if any line in $E$ contains \index{root system!reduced}
at most two elements of $\Phi$;
\item[(ii)] $\Phi$ is called \emph{irreducible} if it cannot be represented \index{root system!irreducible}
as a disjoint union of two non-empty subsets, whose
$\R$-spans have trivial intersection.
\item[(iii)] A subset $\Psi$ of $\Phi$ is called a \emph{root subsystem of $\Phi$}
\index{root subsystem} 
if $\Psi=\Phi\cap \R\Psi$, where  $\R\Psi$ is the $\R$-span of $\Psi$.
\end{itemize}
\end{Definition}

The importance of the following definition will be explained later in this section.
\begin{Definition}
A root system will be called \emph{regular} \index{root system!regular} 
if any root is contained in an irreducible subsystem of rank 2.
\end{Definition}

\subsection{Classical root systems}

In this subsection we define classical root systems and state
some well-known facts about them. The reader is referred to~\cite[Ch.VI]{Bou} and~\cite[Ch.III]{Hu} for more details.

\begin{Definition}
A root system $\Phi$ in a space $E$ will be called \emph{classical} \index{root system!classical}
if $E$ can be given the structure of a Euclidean space with inner product $(\cdot,\cdot)$ such that
for any $\alpha,\beta\in \Phi$
\begin{itemize}
\item[(a)] $\frac{2(\alpha,\beta)}{(\beta,\beta)}\in\Z$;
\item[(b)] $\alpha-\frac{2(\alpha,\beta)}{(\beta,\beta)}\beta\in \Phi$.
\end{itemize}
Any inner product on $E$ satisfying (a) and (b) will be called \emph{admissible} \index{admissible inner product}.
\end{Definition}

\begin{Fact} 
\begin{itemize}
\item[(a)]
Every irreducible classical root system is isomorphic to one of the following:
$A_n$, $B_n (n\geq 2)$, $C_n (n\geq 3)$, $BC_n (n\geq 1)$, $D_n (n\geq 4)$,
$E_6$, $E_7$, $E_8$, $F_4$, $G_2$. The only non-reduced systems in this list
are those of type $BC_n$.

\item[(b)]Every irreducible classical root system of rank $\geq 2$ is regular.
\end{itemize}
\end{Fact}
\begin{proof} (a) is well known (see, e.g., \cite[VI.4.2, VI.4.14]{Bou}), and (b)
can be proved by straightforward case-by-case verification.
\end{proof}

If $\Phi$ is a classical irreducible root system in a space $E$,
then an admissible inner product $(\cdot,\cdot)$ on $E$ is uniquely
defined up to rescaling. In particular, we can compare \emph{lengths} of
different roots in $\Phi$ without specifying the Euclidean structure. Furthermore, the following hold:

\begin{itemize}
\item[(i)] If $\Phi=A_n$, $D_n$, $E_6$, $E_7$ or $E_8$, all roots in $\Phi$ have the same length;
\item[(ii)] If $\Phi=B_n$, $C_n$, $F_4$ or $G_2$, there are two different root lengths in $\Phi$;
\item[(iii)] If $\Phi=BC_n$, there are three different root lengths in $\Phi$.
\end{itemize}

As usual, in case (ii), roots of smaller length will be called \emph{short} \index{root!short}
and the remaining ones called \emph{long} \index{root!long}. 
In case (iii) roots of smallest length will be called
\emph{short}, roots of intermediate length called \emph{long} and roots of largest length
called \emph{double} \index{root!double}. 
The latter terminology is due to the fact that double roots in $BC_n$
are precisely roots of the form $2\alpha$ where $\alpha$ is also a root.
\vskip .1cm

\begin{Definition}
A subset $\Pi$ of a classical root system
$\Phi$ is called a \emph{base} (or a \emph{system of simple roots})
if every root in $\Phi$ is an integral linear combination of
elements of $\Pi$ with all coefficients positive or all
coefficients negative. Thus every base $\Pi$ of $\Phi$
determines a decomposition of $\Phi$ into two disjoint subsets
$\Phi^+(\Pi)$ and $\Phi^-(\Pi)= -\Phi^+(\Pi)$, called the sets
of \emph{positive (resp. negative) roots with respect to $\Pi$}.
\end{Definition}

\begin{figure}
\label{fig:rank2-root-systems}
\begin{picture}(350,300)(0,0)
\put(0,250){
\begin{picture}(0,0)(0,0)
\put(20,0){
\begin{diagram}[grid=hexgridA2,size=1.5em,abut,heads=littleblack]
& \alpha & & {} & \\
& & \luTo(1,2)\ruTo(1,2) && \\
{} &\lTo & \bullet & \rTo & \beta \\
& \ldTo(1,2)& & \rdTo(1,2)& \\
& {} & & {}  & \\
\end{diagram}
}
\put(0,0){$A_2$}
\end{picture}
}
\put(0,100){
\begin{picture}(0,0)(0,0)
\put(0,0){
\begin{diagram}[grid=hexgridG2,size=1.5em,abut,heads=littleblack]
& & & {} & & & \\
& & & \uTo& & & \\
\alpha & & {}\luTo[diagonalbase=(0,0)](1,2) & & {}\ruTo[diagonalbase=(0,0)](1,2) && {} \\
&\luTo(3,2)&&&&\ruTo(3,2)& \\
&{} &\lTo & \bullet & \rTo& \beta &\\
& &\ldTo(1,2)\ldTo(3,2) & & \rdTo(1,2)\rdTo(3,2)& & \\
{} & & {} & & {} && {} \\
& & & \dTo& & & \\
& & & {} & & & \\
\end{diagram}
}
\put(0,0){$G_2$}
\end{picture}
}
\put(200,250){
\begin{picture}(0,0)(0,0)
\put(20,0){
\begin{diagram}[grid=sqgridB2,size=1.5em,abut,heads=littleblack]
\alpha & & {} & & {} \\
& \luTo &\uTo& \ruTo& \\
{} &\lTo & \bullet & \rTo & \beta \\
& \ldTo& \dTo& \rdTo& \\
{}& & {} &  & {} \\
\end{diagram}
}
\put(0,0){$B_2$}
\end{picture}
}
\put(170,100){
\begin{picture}(0,0)(0,0)
\put(20,0){
\begin{diagram}[grid=sqgridBC2,size=1.5em,abut,heads=littleblack]
& & & & {} & & & &\\
& & & & \uTo& & & &\\
& & \alpha & & {} & & {} & & \\
& & & \luTo &\uTo & \ruTo & & & \\
{} &\lTo & {}& \lTo & \bullet & \rTo & \beta & \rTo& {} \\
& & & \ldTo & \dTo& \rdTo& & &\\
& & {} & & {}  &  & {} & & \\
& & & & \dTo & & & &\\
& & & & {} & & & &\\
\end{diagram}
}
\put(0,0){$BC_2$}
\end{picture}
}
\end{picture}
\caption{Classical irreducible root systems of rank 2.}
\end{figure}
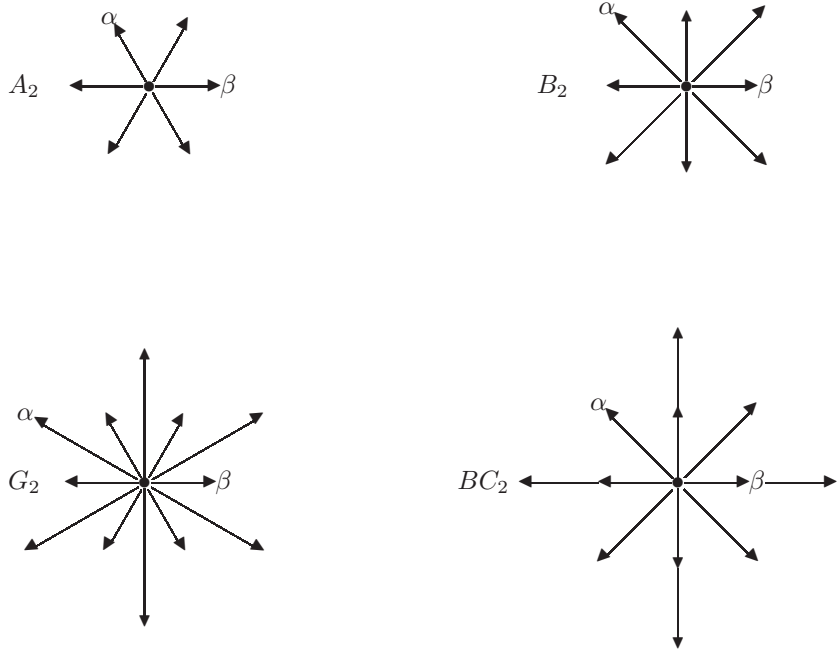

\begin{Example} Figure~\ref{fig:rank2-root-systems} illustrates each irreducible classical root system of rank 2 with a chosen base $\Pi=\{\alpha,\beta\}$.
\begin{itemize}
\item[] If $\Phi=A_2$, then $\Phi^+(\Pi)=\{\alpha,\beta,\alpha+\beta\}$
\item[] If $\Phi=B_2$, then $\Phi^+(\Pi)=\{\alpha,\beta,\alpha+\beta,\alpha+2\beta\}$
\item[] If $\Phi=BC_2$, then $\Phi^+(\Pi)=\{\alpha,\beta,\alpha+\beta,2\beta, \alpha+2\beta,2\alpha+2\beta\}$
\item[] If $\Phi=G_2$, then $\Phi^+(\Pi)=\{\alpha,\beta,\alpha+\beta,\alpha+2\beta, \alpha+3\beta,2\alpha+3\beta\}$
\end{itemize}
\end{Example}

Every classical root system $\Phi$ in a space $E$ has a base; in fact,
the number of (unordered) bases is equal to the order of
the Weyl group of $\Phi$. If $\Pi$ is a base of $\Phi$,
it must be a basis of $E$. Observe that if $f:E\to\R$ is any
functional which takes positive values on $\Pi$, then
$$\Phi^+(\Pi)=\{\alpha\in\Phi : f(\alpha)>0\}.$$

Conversely, if $f:E\to\R$ is any functional which does not vanish on
any of the roots in $\Phi$,
one can show that the set $\Phi_f=\{\alpha\in\Phi : f(\alpha)>0\}$ coincides
with $\Phi^+(\Pi)$ for some base $\Pi$. In fact, $\Pi$ can be
characterized as the elements $\alpha\in\Phi_f$ which are not
representable as $\beta+\gamma$, with $\beta,\gamma\in\Phi_f$.

\subsection{General root systems}
\label{susec:genral root systems}

In this subsection we extend the notions of a base and a set
of positive roots from classical to arbitrary root systems.
By the discussion at the end of the last subsection, if $\Phi$ is a classical root system,
the sets of positive roots with respect to different bases of $\Phi$
are precisely the \emph{Borel subsets} of $\Phi$ as defined below. \index{Borel subset}
The suitable generalization of the notion of a base, called the \emph{boundary of a Borel
set}, is less straightforward and will be given later. The terminology `Borel subset' will be explained in \S 4.3.

\begin{Remark}
If $\Phi$ is a reduced classical root system, the notions of
boundary and base for Borel subsets of $\Phi$ coincide. However, if $\Phi$ is not reduced,
the boundary of a Borel subset will be larger than its base.
\end{Remark}

\begin{Definition}
Let $\Phi$ be a root system in a space $E$. Let $\FFF=\FFF(\Phi)$ denote the set
of all linear functionals $f: E \to \R$ such that
\begin{enumerate}
\item $f(\alpha) \not=0$ for all $\alpha \in \Phi$;
\item $f(\alpha) \not= f(\beta)$ for any distinct
$\alpha,\beta \in \Phi$.
\end{enumerate}
For $f\in \FFF$, the set $\Phi_f=\{\alpha \in \Phi \,|\, f(\alpha)
> 0\}$ is called the \emph{Borel set of $f$}. The sets of this
form will be called \emph{Borel subsets of $\Phi$}. \index{Borel subset}
We will say that two elements $f, f' \in \FFF$  are equivalent and write $f\sim f'$ if $\Phi_f=\Phi_{f'}$.
\end{Definition}

\begin{Remark}
Note that condition (2) implies condition (1) (if $f(\alpha)=0$,
then $f(-\alpha)=f(\alpha)$), but we will not use this fact.
\end{Remark}

\begin{Remark}
Observe that for any $f\in \FFF$ we can order the elements in $\Phi_f$ as follows:
$$
\Phi_f =\{ \alpha_{f,1}, \alpha_{f,2}, \dots  \alpha_{f,k} \}
$$
where $k = |\Phi_f| = |\Phi|/2$ and
$$
f(\alpha_{f,1}) > f( \alpha_{f,2}) > \dots > f( \alpha_{f,k}) >0.
$$
If $f$ and $g$ are equivalent functionals, their Borel sets
$\Phi_f$ and $\Phi_g$ coincide, but the orderings on
$\Phi_f = \Phi_g$ induced by $f$ and $g$ may be different.
\end{Remark}
For instance, if $\Phi=A_2$ and $\{\alpha,\beta\}$ is a base of $\Phi$,
the functionals $f$ and $f'$ defined by
$f(\alpha) = f'(\alpha)=2$, $f(\beta) = 1$ and $f'(\beta)=3$
define the same Borel set consisting of the roots
$\alpha$, $\beta$ and $\alpha + \beta$, however the ordering induced by
$f$ and $f'$ are different
$$
f(\beta) <  f(\alpha) < f(\alpha+ \beta)
\quad \quad
f'(\alpha) < f'(\beta) < f'(\alpha+ \beta).
$$

\begin{Definition}
Let $\Phi$ be a root system. Two Borel sets $\Phi_f$ and $\Phi_g$
will be called
\begin{itemize}
\item \emph{opposite} if $\Phi_f\cap \Phi_g=\emptyset$ or, equivalently, $\Phi_g=\Phi_{-f}$;
\index{Borel subsets!opposite}
\item \emph{co-maximal} if an inclusion $\Phi_h \supset \Phi_f \cap \Phi_{g}$ implies that
$\Phi_h = \Phi_f$ or $\Phi_h=\Phi_g$; 
\index{Borel subsets!co-maximal}
\item \emph{co-minimal} if $\Phi_f$ and $\Phi_{-g}$ are co-maximal.
\index{Borel subsets!co-minimal}
\end{itemize}
\end{Definition}

\begin{Example}
Figure~\ref{fig:Borel_sets} shows the Borel sets in root systems of type $A_2$ and $B_2$.
Pairs of opposite Borel sets are connected with a dotted line and co-maximal ones are
connected with a solid line.
\end{Example}
\begin{figure}
\label{fig:Borel_sets}
\begin{picture}(350,170)(0,0)
\put(-50,70){
\begin{picture}(0,0)(0,0)
\put(0,0){
\begin{diagram}[grid=hexgridA2,size=3em,objectstyle=\scriptstyle,abut]
& \!\!\!\!\!\{\alpha,-\beta,-\alpha-\beta\} & \rLine & \{\alpha,-\beta,\alpha+\beta\}  \!\!\!\!\! & \\
\ldLine(1,2) & & \rdDotLine(2,4) \ldDotLine(2,4) & & \rdLine(1,2) \\
\{-\alpha,-\beta,-\alpha-\beta\} & & \rDotLine & & \{\alpha,\beta,\alpha+\beta\} \\
& \rdLine(1,2) & & \ldLine(1,2) & \\
&  \!\!\!\!\!\{-\alpha,\beta,-\alpha-\beta\} & \rLine& \{-\alpha,\beta,\alpha+\beta\}  \!\!\!\!\!  & \\
\end{diagram}
}
\put(0,-50){$A_2$}
\end{picture}
}
\put(150,70){
\begin{picture}(0,0)(0,0)
\put(0,0){
\begin{diagram}[grid=octagonfat,size=3em,objectstyle=\scriptstyle,abut]
& \!\!\!\!\! \!\!\!\!\!
\{\alpha,-\beta,-\alpha-\beta,-\alpha-2\beta\}& \rLine & \{\alpha,-\beta,\alpha+\beta,-\alpha-2\beta\}
\!\!\!\!\! \!\!\!\!\! &  \\
\{-\alpha,-\beta,-\alpha-\beta,-\alpha-2\beta\} \ldLine(1,1)& & \ldDotLine(2,4) \rdDotLine(2,4) & & \{\alpha,-\beta,\alpha+\beta,\alpha+2\beta\}\rdLine(1,1)  \\
\dLine&\rdDotLine(4,2)&&\ldDotLine(4,2)&\dLine \\
\{-\alpha,\beta,-\alpha-\beta,-\alpha-2\beta\} & & & & \{\alpha,\beta,\alpha+\beta,\alpha+2\beta\}  \\
& \!\!\!\!\! \!\!\!\!\! \{-\alpha,\beta,-\alpha-\beta,\alpha+2\beta\}\rdLine(1,1) &\rLine & \{-\alpha,\beta,\alpha+\beta,\alpha+2\beta\} \!\!\!\!\! \!\!\!\!\!  \ldLine(1,1)&  \\
\end{diagram}
}
\put(0,-60){$B_2$}
\end{picture}
}
\end{picture}
\caption{Borel Sets in root systems of type $A_2$ and $B_2$.}
\end{figure}
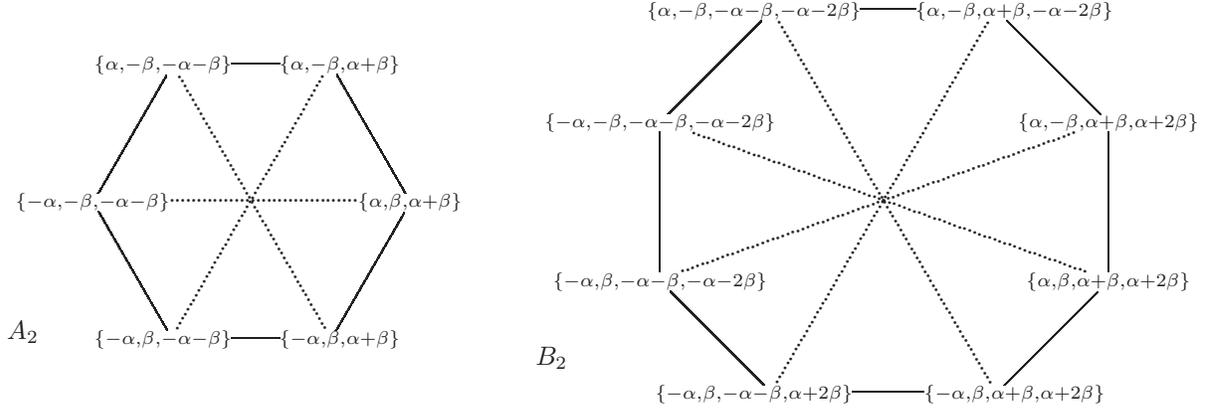

\begin{Lemma}
\label{cominimal} 
Let $\Phi$ be a root system in a space $E$,
and let $\Phi_f$ and $\Phi_g$ be distinct Borel sets. The following are equivalent:
\begin{itemize}
\item[(i)] $\Phi_f\cap \Phi_{-g}$ spans one-dimensional subspace
\item[(ii)] $\Phi_f$ and $\Phi_g$ are co-maximal
\item[(iii)] If $h\in \FFF$ is such that $\Phi_h\supset\Phi_f\cap \Phi_{g}$, then
$\Phi_h\cap \Phi_g=\Phi_f\cap \Phi_{g}$ or $\Phi_h\cap \Phi_f=\Phi_f\cap \Phi_{g}$.
\end{itemize}
\end{Lemma}
\begin{proof}
\begin{list}{}{\labelwidth=0pt \itemindent=4ex \leftmargin=0pt}
\item[(i)$\Rightarrow$ (ii)]
Since $\Phi_f\cap \Phi_{-g}$ spans one-dimensional subspace,
there exists a vector $v\in E$ such that
$$
\Phi_f=(\Phi_f\cap
\Phi_{g})\cup (\Phi_f\cap \Phi_{-g})\subset (\Phi_f\cap
\Phi_{g})\cup \R_{>0}v
$$
and 
$$
\Phi_{g}=(\Phi_f\cap
\Phi_{g})\cup (\Phi_{-f}\cap \Phi_{g})\subset (\Phi_f\cap
\Phi_{g})\cup \R_{<0}v.
$$ 
Thus, if $\Phi_h$ contains $\Phi_f\cap
\Phi_{g}$, then $\Phi_h=\Phi_f$ in the case $h(v)>0$ or
$\Phi_h=\Phi_{g}$ in the case $h(v)<0$. Hence $\Phi_f$ and $\Phi_{g}$
are co-maximal.

\item[(ii)$\Rightarrow$ (iii)] is obvious.

\item[(iii)$\Rightarrow$ (i)] 
Let $U$ be the subspace spanned by $\Phi_f\cap \Phi_{-g}$,
and suppose that $\dim U>1$. Since
any sufficiently small perturbation of $f$ does not change
its equivalence class, we may assume
that the the restrictions of $f$ and $g$ to $U$ are linearly independent.

For any $x\in [0,1]$ consider $h_x=xf+(1-x)g$. Note that
$h_0=g$ is negative on $\Phi_f\cap \Phi_{-g}$ and $h_1=f$ is positive on
$\Phi_f\cap \Phi_{-g}$. Since the restrictions of $f$ and $g$ to
$U$ are linearly independent, by continuity there exists $x\in (0,1)$ such that $h=h_x$
is positive on some but not all roots from $\Phi_f\cap \Phi_{-g}$.
Thus there exist $\alpha,\beta\in \Phi$ such that $f(\alpha)>0$,
$g(\alpha)<0$, $h(\alpha)>0$ and $f(\beta)>0$, $g(\beta)<0$, $h(\beta)<0$,
so $\Phi_f\cap \Phi_g\neq \Phi_f\cap \Phi_h$
and  $\Phi_f\cap \Phi_g\neq \Phi_h\cap \Phi_g$. On the the hand,
since $h=xf+(1-x)g$, it is clear that $\Phi_f\cap \Phi_g$
is contained in $\Phi_h$. This contradicts (iii).
\end{list}
\end{proof}

Part (a) of the next definition generalizes the notion of a base of a root system.

\begin{Definition}
\begin{itemize}
\item[(a)] \emph{The boundary of a Borel set $\Phi_f$} \index{Borel subset!boundary}
is the set 
$$
\partial \Phi_f =\bigcup_g  ( \Phi_f \setminus \Phi_{g})=\bigcup_g
(\Phi_f \cap \Phi_{-g}),\mbox{ where $\Phi_{g}$ and $\Phi_{f}$ are
co-maximal.}
$$ 
Equivalently,
$$
\partial \Phi_f =\Phi_f \cap (\bigcup_g \Phi_{g}),
\mbox{ where $\Phi_{g}$ and $\Phi_{f}$ are co-minimal.}
$$
\item[(b)] 
\emph{The core of a Borel set $\Phi_f$} is the set
\index{Borel subset!core} 
$$
C_f = \Phi_f \setminus \partial \Phi_f=\bigcap_g( \Phi_f \cap \Phi_{g}),
\mbox{ where $\Phi_{g}$ and $\Phi_{f}$ are co-maximal.}
$$
\end{itemize}
\end{Definition}

If $\Phi$ is a classical reduced system and $\Phi_f$ is a Borel subset of $\Phi$,
it is easy to see that the boundary of $\Phi_f$ is precisely the base
$\Pi$ for which $\Phi^+(\Pi)=\Phi_f$.
However for non-reduced systems this is not the case and the boundary also contains
all roots which are positive multiples of the roots in the base.

\begin{Example}
In each of the following examples we consider a classical
rank $2$ root system $\Phi$, its base $\Pi=\{\alpha,\beta\}$
and the Borel set $\Phi^+(\Pi)$.
\begin{itemize}
\item[1.] If $\Phi=A_2$, the core of the Borel set $\{\alpha,\beta,\alpha+\beta\}$
is $\{\alpha+\beta\}$ and the boundary is $\{\alpha,\beta\}$.
\item[2.]
If $\Phi=B_2$, the core of the Borel set  $\{\alpha,\beta,\alpha+\beta,\alpha+2\beta\}$ is $\{\alpha+\beta,\alpha+2\beta\}$
and the boundary is again $\{\alpha,\beta\}$.
\item[3.]
If $\Phi=BC_2$, the core of the Borel set  $\{\alpha,\beta,2\beta,\alpha+\beta,2\alpha+2\beta,\alpha+2\beta\}$ is $\{\alpha+\beta,2\alpha+2\beta,\alpha+2\beta\}$, while
the boundary is $\{\alpha,\beta,2\beta\}$.
\end{itemize}
\end{Example}

\begin{Lemma}
\label{corebase}
Let $\Phi$ be a root system, $f\in \FFF=\FFF(\Phi)$, and let $\alpha,\beta\in \Phi_f$
be linearly independent. Then any root in $\Phi$ of the form $a\alpha+b\beta$, with $a,b>0$,
lies in $C_f$.
\end{Lemma}
\begin{proof}
Assume the contrary, in which case $a\alpha+b\beta\in \partial \Phi_f$. Thus, there
exists $g\in \FFF(\Phi)$ such that $\Phi_f$ and $\Phi_g$ are co-minimal
and $a\alpha+b\beta\in \Phi_f\cap\Phi_g$. But then $g(\beta)>0$ or $g(\alpha)>0$,
so $\Phi_f\cap\Phi_g$ contains linearly independent roots $a\alpha+b\beta$
and $\alpha$ or $\beta$. This contradicts Lemma~\ref{cominimal}.
\end{proof}

\begin{Lemma}
\label{coresubsystem}
Let $\Phi$ be a root system,
$\Psi$ a subsystem, $f\in \FFF=\FFF(\Phi)$ and $f_0$ the restriction of $f$ to $\R\Psi$.
Then the core $C_{f_0}$ of $\Psi_{f_0}$ is a subset of the core $C_f$ of $\Phi_f$.
\end{Lemma}
\begin{proof}
 Let $v\in C_{f_0}$. We have to show that for any $g\in \FFF$ such that
$\Phi_f$ and $\Phi_g$ are co-maximal, $g(v)>0$. Assume the
contrary, that is, $g(v)<0$. Then by Corollary~\ref{cominimal},
$\Phi_f\cap \Phi_{-g}\subset \R v$. Let  $g_0$ be
the restriction of $g$ on $\Psi$. Then  $\emptyset\neq\Psi_{f_0}\cap
\Psi_{-g_0}\subseteq\Phi_f\cap \Phi_{-g}\subset \R v$. Again by
Corollary~\ref{cominimal},  $\Psi_{f_0}$ and $\Psi_{-g_0}$ are co-minimal,
so $\Psi_{f_0}$ and $\Psi_{g_0}$ are co-maximal. Since $v\in C_{f_0}$,
we have $g(v)=g_0(v)>0$, a contradiction.
\end{proof}

\begin{Lemma}
\label{coresubsystem2}
Every root in an irreducible rank $2$ system $\Phi$ is contained in the core
of some Borel set.
\end{Lemma}
\begin{proof}
Let $E$ be the vector space spanned by $\Phi$, and take any $\alpha\in \Phi$.
Since $\Phi$ is irreducible, there exist $\beta,\gamma\in\Phi$ such
that $\alpha,\beta$ and $\gamma$ are pairwise linearly independent.
Replacing $\beta$ by $-\beta$ and $\gamma$ by $-\gamma$ if necessary,
we can assume that $\alpha=b\beta+c\gamma$ with $b,c>0$. If we now
take any $f\in\FFF(\Phi)$ such that $f(\beta)>0$ and $f(\gamma)>0$,
then $\alpha\in C_f$ by Lemma~\ref{corebase}.
\end{proof}
\begin{Corollary}
\label{rootincore}
Every root in a regular root system is contained in the core of
some Borel set.
\end{Corollary}
\begin{proof}
This follows from Lemmas~\ref{coresubsystem} and \ref{coresubsystem2}
and the fact that if $\Psi$ is a subsystem of $\Phi$, then
any element of $\FFF(\Psi)$ is the restriction of some element of
$\FFF(\Phi)$ to $\R\Psi$.
\end{proof}

\subsection{Weyl graphs}
To each root system $\Phi$ we shall associate two graphs
$\Gamma_l(\Phi)$ and $\Gamma_s(\Phi)$, called the
\emph{large Weyl graph} and the \emph{small Weyl graph}, respectively.
\index{Weyl graph!large}
\index{Weyl graph!small}
Both Weyl graphs $\Gl=\Gl(\Phi)$ and $\Gs=\Gs(\Phi)$
will have the same vertex set: 
$$
\Vert(\Gl)=\Vert(\Gs)=\FFF(\Phi)/\sim.
$$
Thus vertices of either graph are naturally labeled by Borel subsets of $\Phi$:
to each vertex $f\in \Vert(\Gl)=\Vert(\Gs)$ we associate the Borel set $\Phi_f$.

\begin{itemize}
\item 
Two vertices $f$ and $g$ are connected in the large Weyl graph
$\Gl$ if and only if their Borel sets are not opposite;
\item 
Two vertices $f$ and $g$ are connected in the small Weyl graph $\Gs$ if
and only if there exists functionals $f'$ and $g'$ such that
$\Phi_f \subset C_{f'} \cup \Phi_g$ and $\Phi_g \subset C_{g'}
\cup \Phi_f$.
\end{itemize}
To each (oriented) edge $e$ in $\Edg(\Gl)$ or $\Edg(\Gs)$ we
associate the set $\Phi_e =  \Phi_{e^+} \cap \Phi_{e^-}$.
Note that $\Phi_e$ is always non-empty by construction.

\begin{Remark}
If $\Phi$ is an irreducible classical root system, both Weyl graphs of $\Phi$
are Cayley graphs of $W=W(\Phi)$, the Weyl group of $\Phi$, but
with respect to different generating sets. The large Weyl
graph $\Gl(\Phi)$ is the Cayley graph with respect to the
set $W\setminus \{\alpha_{long}(\Phi)\}$ where $\alpha_{long}(\Phi)$
is the longest element of $W$ relative to the (standard)
Coxeter generating set $S_{\Phi}$.

The generating set corresponding to the small Weyl graph $\Gs(\Phi)$
is harder to describe. At this point we will just mention
that it always contains the Coxeter generating set $S_{\Phi}$, but it is equal to $S_{\Phi}$
only for systems of type $A_2$.
\end{Remark}

\begin{Example}
Figure~\ref{fig:Weyl_graphs} shows the Weyl graphs in the root systems of type $A_2$ and $B_2$.
The edges of the small Weyl graph are denoted by solid lines and ones in the
large Weyl graph are either by solid or by dotted lines.
\end{Example}
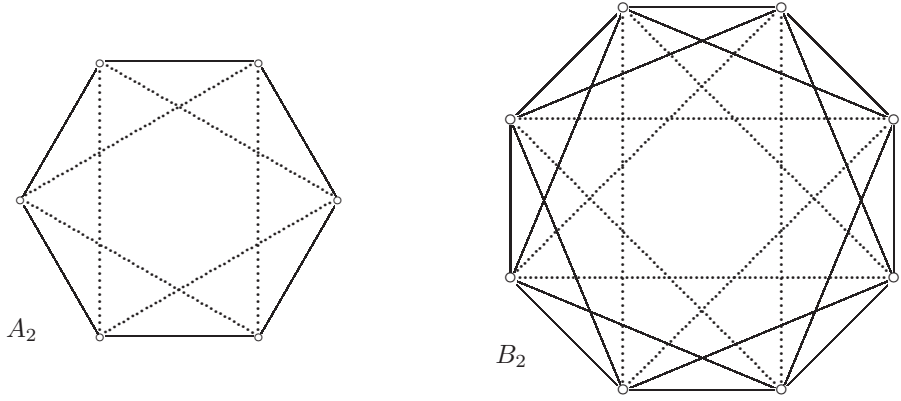
\begin{figure}
\label{fig:Weyl_graphs}
\begin{picture}(350,170)(0,0)
\put(0,70){
\begin{picture}(0,0)(0,0)
\put(0,0){
\begin{diagram}[grid=hexgridA2,size=3em,objectstyle=\scriptstyle,abut]
& \circ & \rLine & \circ & \\
\ldLine(1,2) & & \rdDotLine(3,2)\ldDotLine(3,2) & & \rdLine(1,2) \\
\circ & & & & \circ \\
& \rdLine(1,2) \rdDotLine(3,2) \dDotLine & & \ldLine(1,2) \ldDotLine(3,2)  \dDotLine& \\
& \circ & \rLine& \circ  & \\
\end{diagram}
}
\put(0,-50){$A_2$}
\end{picture}
}
\put(185,70){
\begin{picture}(0,0)(0,0)
\put(0,0){
\begin{diagram}[grid=octagon,size=3em,abut]
& \circ & \rLine & \circ &  \\
\circ \ldLine(1,1)\ldLine(1,3)& & \ldLine(3,1) \rdLine(3,1)  \rDotLine \rdDotLine(3,3) \ldDotLine(3,3) && \circ \rdLine(1,1) \rdLine(1,3) \\
\dLine& \rdLine(1,3)\dDotLine \rdDotLine(3,3) && \ldLine(1,3) \dDotLine \ldDotLine(3,3)&\dLine \\
\circ & & \rDotLine & & \circ  \\
& \circ \rdLine(1,1) \rdLine(3,1) &\rLine & \circ \ldLine(1,1) \ldLine(3,1) &  \\
\end{diagram}
}
\put(0,-60){$B_2$}
\end{picture}
}
\end{picture}
\caption{Weyl graphs corresponding to root systems of type $A_2$ and $B_2$.}
\end{figure}

The structure of the large Weyl graph is very transparent.
\begin{Lemma} 
\label{spectral_gap}
Let $\Phi$ be a root system.
\begin{itemize}
\item[(a)] The large Weyl graph  $\Gl=\Gl(\Phi)$ is a regular graph with $N$ vertices and
degree $N-2$, where $N$ is the number of distinct Borel sets of $\Phi$.

\item[(b)] The eigenvalues of the adjacency matrix of $\Gl$ are $N-2$
with multiplicity $1$, $0$ with multiplicity $N/2$ and $-2$ with
multiplicity $N/2-1$. Therefore the spectral gap of the Laplacian
of $\Gl$ is equal to the degree of $\Gl$.
\end{itemize}
\end{Lemma}
\begin{proof} (a) is clear.
(b) A constant function is an eigenvector with eigenvalue $N-2$, any
``antisymmetric'' function (one with $F(x) = - F(\bar x)$ where
$x$ and $\bar x$ are opposite vertices) has eigenvalue $0$, and the space
of antisymmetric functions has dimension $N/2$. Finally, any
``symmetric'' function with sum $0$ is an eigenfunction with
eigenvalue $-2$,  and the space of such functions has dimension $N/2-1$.
\end{proof}

The role played by the small Weyl graph in this paper
will be discussed at the end of this section. The key
property we shall use is the following lemma:

\begin{Lemma}
\label{coreconnected}
Let $\Phi$ be a regular root system. Then the graph $\Gs(\Phi)$ is
connected.
\end{Lemma}
\begin{proof}
Let $f,g\in \FFF$ be two functionals such that $\Phi_f$ and
$\Phi_{g}$ are distinct. We prove that there exists a path in
$\Gs$ from $f$ to $g$ by downward induction of $|\Phi_f\cap
\Phi_{g}|$. If $\Phi_f$ and $\Phi_g$ are
co-maximal, then $f$ and $g$ are connected (by an edge) in the
small Weyl graph $\Gs$ by Lemma~\ref{cominimal} and
Corollary~\ref{rootincore}. If $\Phi_f$ and $\Phi_g$ are not co-maximal,
then by Lemma~\ref{cominimal} there exists $h$ such that
$\Phi_f\cap \Phi_g$ is properly contained in $\Phi_h\cap \Phi_f$
and $\Phi_h\cap \Phi_g$. By induction, there are paths that
connects $h$ with both $f$ and $g$. Hence $f$ and $g$ are
connected by a path in $\Gs$.
\end{proof}

\begin{Corollary}
Both large and small Weyl graphs of any irreducible classical root system
of rank $\geq 2$ are connected.
\end{Corollary}

We have computed the diameter of the small Weyl graph for some root systems,
and in all these examples the diameter is at most $3$.
We believe that this is true in general.
\begin{Conjecture}
\label{boundeddiameter} 
If $\Phi$ is an irreducible classical root system
of rank $\geq 2$, then the diameter of $\Gs$ is at most $3$.
\end{Conjecture}

\subsection{Groups graded by root systems}

\begin{Definition}
Let $\Phi$ be a root system and $G$ a group. A \emph{$\Phi$-grading of $G$}\index{grading}
(or just \emph{grading of $G$}) is a collection of subgroups
$\{X_\alpha\}_{\alpha\in\Phi}$ of $G$, called \emph{root subgroups}\index{root subgroup}
 such that
\begin{itemize}
\item[(i)] $G$ is generated by $\cup X_{\alpha}$;
\item[(ii)] For any $\alpha, \beta\in \Phi$, with $\alpha\not\in\R_{<0}\beta$,
we have
$$
[X_\alpha,X_\beta] \subseteq \la X_\gamma \mid \gamma = a \alpha +
b \beta \in\Phi, \ a,b \geq 1 \ra
$$
\end{itemize}
If $\{X_\alpha\}_{\alpha\in\Phi}$ is a collection of subgroups
satisfying (ii) but not necessarily (i), we will simply say
that $\{X_\alpha\}_{\alpha\in\Phi}$ is a \emph{$\Phi$-grading}
(without specifying the group).
\end{Definition}

Each grading of a group $G$ by a root system $\Phi$
determines canonical graph of groups decompositions
of $G$ over the large and small Weyl graphs of $\Phi$.
The vertex and edge subgroups in these decompositions
are defined as follows.

\begin{Definition}
Let $\Phi$ be a root system, $G$ a group and
$\{X_{\alpha}\}_{\alpha\in\Phi}$ a $\Phi$-grading of $G$.
For each $f\in \Vert(\Gl) = \Vert(\Gs)$ we set
$$
G_f = \la X_\alpha \mid \alpha \in \Phi_f \ra,
$$
and for each $e\in \Edg(\Gl)\supset \Edg(\Gs)$
we set $$ G_e = \la X_\alpha \mid \alpha \in \Phi_e \ra.
$$
We will call $G_f$ the \emph{Borel subgroup of $G$ corresponding to $f$}.\index{Borel subgroup}
\end{Definition}
\begin{Remark}
We warn the reader that our use of the term `Borel subgroup' is potentially
misleading. Assume that $\Phi$ is classical, irreducible and reduced.
Let $F$ be a field and $G=\dbE_{\Phi}(F)=\dbG_{\Phi}(F)$ the corresponding
simply-connected Chevalley group over $F$. Let $\{X_{\alpha}\}_{\alpha\in\Phi}$ be the root
subgroups (relative to the standard torus $H$), so that $\{X_{\alpha}\}$ is a $\Phi$-grading
of $G$. Then Borel subgroups of $G$ in our sense are smaller than Borel
subgroups in the sense of Lie theory. In fact, Borel subgroups in our sense
are precisely the unipotent radicals of those Borel subgroups in the sense of Lie
theory which contain $H$. Equivalently, our Borel subgroups are
maximal unipotent subgroups of $G$ normalized by $H$.
\end{Remark}

\begin{Example}
If $G$ is a group graded by a root system of type $A_2$,
Figure~\ref{fig:A2_graph_groups} shows the canonical decomposition
of $G$ over the large Weyl graph of $A_2$, called the
``magic graph'' in~\cite{EJ}.
\end{Example}

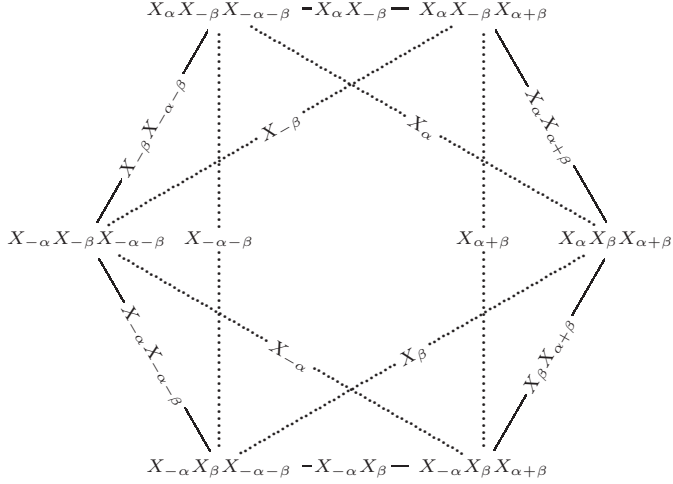
\begin{figure}
\label{fig:A2_graph_groups}
\begin{picture}(300,250)(0,0)
\put(0,130){
\begin{diagram}[grid=hexgridA2,size=5em,objectstyle=\scriptstyle,scriptlabels,PS,nohug]
      & X_{\alpha}X_{-\beta}X_{-\alpha-\beta}
                & \rLine~{X_{\alpha} X_{-\beta}}
                         & X_{\alpha}X_{-\beta}X_{\alpha+\beta}
                                     &  \\
\ldLine(1,2)~{X_{-\beta}X_{-\alpha-\beta}}
     &          & \rdDotLine(3,2)~{X_{\alpha}} \ldDotLine(3,2)~{X_{-\beta}}
                         &           &  \rdLine(1,2)~{X_{\alpha}X_{\alpha+\beta}} \\
X_{-\alpha}X_{-\beta}X_{-\alpha-\beta}
                                       \rdLine[diagonalbase=(0,0)](1,2)~{X_{-\alpha}X_{-\alpha-\beta}} \rdDotLine[diagonalbase=(0,0)](3,2)~{X_{-\alpha}}
     & \dDotLine~{X_{-\alpha-\beta}}
                &        & \dDotLine~{X_{\alpha+\beta}}
                                     & X_{\alpha}X_{\beta}X_{\alpha+\beta}
                                       \ldLine[nohug,diagonalbase=(0,0)](1,2)~{X_{\beta}X_{\alpha+\beta}} \ldDotLine[diagonalbase=(0,0)](3,2)~{X_{\beta}}  \\
     &          &        &           & \\
     & X_{-\alpha}X_{\beta}X_{-\alpha-\beta}
                & \rLine~{X_{-\alpha} X_{\beta}}
                         & X_{-\alpha}X_{\beta}X_{\alpha+\beta}
                                     & \\
\end{diagram}
}
\end{picture}
\caption{Weyl graph of groups for a groups graded by a root system of type $A_2$.}
\end{figure}

\begin{Definition}
Let $\{X_{\alpha}\}$ be a $\Phi$-grading of a group $G$.
For each $f\in\FFF(\Phi)$, the \emph{core subgroup} $G_{C_f}$ of $G_f$ \index{core subgroup}
is the subgroup generated by the root subgroups in the core, that is,
$$G_{C_f}=\la X_{\alpha} \mid \alpha \in C_f\ra.$$
\end{Definition}

\begin{Lemma}
In the above notations, for each $f\in \FFF$ the core subgroup
$G_{C_f}$ is a normal subgroup of $G_f$.
\end{Lemma}
\begin{proof}
This is an immediate consequence of Lemma~\ref{corebase}.
\end{proof}

\begin{Definition}
Let $\Phi$ be a root system
and $\{X_\alpha\}_{\alpha\in \Phi}$ a $\Phi$-grading.
\begin{itemize}
\item[(i)] Take any Borel subset $\Phi_f$ (with $f\in \FFF(\Phi)$) 
and any root $\gamma\in C_f$. We will say that the grading $\{X_\alpha\}$ is
\emph{strong at the pair $(\gamma,\Phi_f)$} if
$$
X_\gamma \subseteq \la X_\beta \mid \beta \in \Phi_f\mbox{ and }
\beta\not\in\R\gamma\ra.
$$
\item[(ii)] We will say that the grading $\{X_\alpha\}$ is
\emph{strong} if $\{X_\alpha\}$ is strong at every pair \index{grading!strong}
$(\gamma,\Phi_f)$ (with $\gamma\in C_f$).
\item[(iii)] Given an integer $k$, we will say that the grading 
$\{X_\alpha\}_{\alpha\in \Phi}$ is 
\emph{$k$-strong} if for any irreducible subsystem $\Psi$ of rank $k$ of $\Phi$ the grading
$\{X_\alpha\}_{\alpha\in \Psi}$ is strong. \index{grading!$k$-strong}
\end{itemize}
\end{Definition}
\begin{Remark}
In~\S~\ref{sec:steinberg}~and~\S~\ref{sec:twisted} we will need to verify that the natural gradings of certain Steinberg groups
and twisted Steinberg groups are strong. With the exception of \S~\ref{sec:moregroups},
all examples we will consider satisfy the following property:
\vskip .12cm
For any two functionals $f,f'\in \FFF(\Phi)$ there exists
an automorphism $w\in \Aut(G)$ which permutes the root
subgroups $\{X_{\alpha}\}_{\alpha\in\Phi}$ between themselves,
and the induced action of $w$ on $\Phi$ sends the Borel
subset $\Phi_f$ to the Borel subset $\Phi_{f'}$.
\vskip .12cm
In the presence of this property, in order to prove that
the grading $\{X_{\alpha}\}$ is strong it suffices to
check that it is strong at $(\gamma,\Phi_f)$ where $f\in\FFF(\Phi)$
is a fixed functional and $\gamma$ runs over $C_f$.
In each of our examples, we shall use a functional $f$
such that $\Phi_f$ is the set of positive roots
(with respect to a fixed system of simple roots).
To simplify the terminology, we shall say that
the grading is \emph{strong at $\gamma$} if it \index{grading!strong}
is strong at $(\gamma,f)$ for the $f$ that we fixed.
\end{Remark}

\begin{Example}
Let $\Phi$ be a root system of type $A_2$ and let $\{X_{\gamma}\}$.
A sufficient condition for the grading to be strong is that
$[X_\alpha,X_\beta] = X_{\alpha+\beta}$ for any pairs of roots $\alpha$ and $\beta$
such that $\alpha+\beta$ is also a root.
This condition is also necessary under the additional
assumption that every element in a Borel subgroup can be expressed uniquely as
a product of elements in the 3 root subgroups (put in some fixed order).
\end{Example}

\begin{Example} 
Let $\Phi$ be a root system of type $B_2$ and let $\{X_{\gamma}\}$
be a $\Phi$-grading. Assume that there exists an abelian group $R$ such that
each of the root subgroups $\{X_{\gamma}\}$ is isomorphic to $R$; thus we can
denote the elements of $X_{\gamma}$ by $\{x_{\gamma}(r) : r\in R\}$
so that $x_{\gamma}(r+s)=x_{\gamma}(r)x_{\gamma}(s)$.

Now let $\{\alpha,\beta\}$ be a base of $\Phi$, with $\alpha$ a long root.
Let $f$ be any functional such that $\partial \Phi_f=\{\alpha,\beta,\alpha+\beta,\alpha+2\beta\}$
(in which case $C_f=\{\alpha+\beta,\alpha+2\beta\}$). By definition of grading
there exist functions $p,q: R\times R\to R$ such that
$$
[x_{\alpha}(r),x_{\beta}(s)]=x_{\alpha+\beta}(p(r,s))x_{\alpha+2\beta}(q(r,s)) \mbox{ for all } r,s\in S
$$
Then the grading is strong at the pair
$(\alpha+\beta,f)$ (resp. $(\alpha+2\beta,f)$) whenever the image of $p$ (resp. $q$)
generates $R$ as a group.
\end{Example}

If $\Phi$ is a non-reduced root system, it is sometimes useful to slightly
modify a given $\Phi$-grading using a simple operation called fattening:

\begin{Definition}
Let $\Phi$ be a non-reduced root system and $\{X_{\alpha}\}_{\alpha\in \Phi}$
a $\Phi$-grading of some group $G$. For each $\alpha\in \Phi$ we set
$\widetilde X_{\alpha}=\la X_{a\alpha} : a\geq 1\ra$. We will say that
$\{\widetilde  X_{\alpha}\}_{\alpha\in \Phi}$ is the fattening of the grading $\{X_{\alpha}\}_{\alpha\in \Phi}$.
\end{Definition}
It is easy to see that the fattening $\{\widetilde  X_{\alpha}\}$ is also $\Phi$-grading. Moreover,
$\{\widetilde  X_{\alpha}\}$ is strong whenever $\{X_{\alpha}\}$ is strong.

\subsection{A few words about the small Weyl graph}
We end this section explaining how the small Weyl graph
and the notion of the core of a Borel set will be used in this paper.
Unlike the large Weyl graph, which plays a central role
in the proof of Theorem~\ref{th:relativeT_wrt_ root subgroup},
the small Weyl graph is just a convenient technical tool.

As discussed in~\S~\ref{sec:spec_criteria}, given a group $G$ graded by a regular
root system $\Phi$, Theorem~\ref{th:relativeT_wrt_ root subgroup}
for $G$ will be proved by applying the generalized spectral criterion
(Theorem~\ref{th:gen_spec_criteria}) to the canonical decomposition of $G$ over the large Weyl graph $\Gl$.
The small Weyl graph $\Gs$ will be used to verify hypothesis (iii)
in Theorem~\ref{th:gen_spec_criteria}.

In fact, for many root systems we could use a different definition
of the core of a Borel subset (leading to a different small Weyl graph and different
core subgroups) which would still work for applications in~\S~\ref{sec:T_for groups}.
We could not make the cores any larger than we did (otherwise hypothesis (ii)
in Theorem~\ref{th:gen_spec_criteria} would not hold), but we could often
make then smaller -- the only properties we need is that
the small Weyl graph is connected (Lemma~\ref{coreconnected})
and the core subgroups are normal in the ambient vertex groups.

For instance, if $\Phi$ is a simply-laced classical root system,
we could let $C_f$ consist of just one root, namely, the root of maximal height in
the Borel set $\Phi_f$ (this definition coincides with ours for $\Phi=A_2$).
In this case the small Weyl graph would become the Cayley graph of the Weyl group $W(\Phi)$
with respect to the standard Coxeter generating set. If $\Phi$ is a non-simply-laced
classical root system, this seemingly more natural definition of the core
does not work, although we could still make the core smaller except when
$\Phi=B_2$ or $BC_2$.


\section{Property $(T)$ for groups graded by root systems}
\label{sec:T_for groups}

In this section we prove Theorem~\ref{thm:main2},
in fact a slightly generalized version of it dealing
with groups graded by regular (not necessarily classical)
root systems.

\begin{Theorem}
\label{th:relativeT_wrt_ root subgroup}
Let $\Phi$ be a regular root system, and let
$G$ be a group which admits a strong $\Phi$-grading $\{X_{\alpha}\}$.
Then $\cup X_{\alpha}$ is a Kazhdan subset of $G$, and moreover
the Kazhdan constant $\kappa(G,\cup X_\alpha)$ is bounded below
by a constant $\kappa_\Phi$ which depends only on the root system $\Phi$.
\end{Theorem}

Theorem~\ref{th:relativeT_wrt_ root subgroup} will be established by
applying the generalized spectral criterion from Theorem~\ref{th:gen_spec_criteria}
to the canonical decomposition of $G$ over the large Weyl graph of $\Phi$.
Thus, we need to show that the hypotheses (i)-(iii) in Theorem~\ref{th:gen_spec_criteria}
are satisfied in this setting.

\subsection{Estimates of codistances in nilpotent groups}

We start by proving an upper bound on codistances between certain families of subgroups in nilpotent groups
(Lemma~\ref{nilpotent1} below). While this result is quite technical, once it is established, verification
of conditions (i) and (ii) in the proof of Theorem~\ref{th:relativeT_wrt_ root subgroup} will be rather
straightforward.

\begin{Lemma}
\label{nilpotent1} 
Let $N$ be a nilpotent group, and let $\{X_i\}_{i=1}^n$ be a finite family
of subgroups of $N$ such that for each $1\leq i\leq n$, the product set $N_i=\prod_{j=i}^n X_j$ is a normal subgroup of $N$,
$N_1=N$ and $[N_i,N]\subseteq N_{i+1}$ for each $i$.

Suppose now that we are given another family $\{K_j\}_{j=1}^m$
of subgroups of $N$ and an integer $l$ such that for each $i$,
the inclusion $X_i\subseteq K_j$ holds for at least $l$ distinct indices $j$.
The following hold:
\begin{itemize}
\item[(a)] $\codist(K_1,\ldots, K_m)\leq\frac{m-l}{m}$
\item[(b)] For each $1\leq i\leq n$ let $H_i$ be the subgroup generated by $\{\cup K_j: X_i\not\subseteq K_j\}$,
and let $C$ be a normal subgroup of $N$ which is contained in the intersection $\bigcap_{i=1}^n H_i$.
Then for any representation $V$ of $N$ without $C$-invariant vectors
we have 
$$
\codist(V^{K_1},\ldots, V^{K_m})\leq\frac{m-l}{m}\cdot (1-\delta),
$$
where $\delta=\frac{8}{(m-2)4^{c}}$ and $c$ is the nilpotency class of $N$.
\end{itemize}
\end{Lemma}

\begin{proof}
Let $\HH$ be a unitary representation of $N$ without invariant vectors. For each $1\leq i\leq n$
let $\HH_i=\HH^{N_{i}}$ and $\HH_i^\perp$ the orthogonal complement of $\HH_i$ in $\HH$.
Since $N_i$ is normal in $N$, both $\HH_i$ and $\HH_i^\perp$
are $N$-invariant. Finally, let $\HH_{(i)}=\HH_{i-1} \cap  \HH_{i}^{\perp}$.
Since $\HH^{N}=\{0\}$ by assumption, we clearly have the decomposition
$$
\HH = \bigoplus_i \HH_{(i)}.
$$

Let $\pi_i: \HH \to \HH_{(i)}$ be the orthogonal projection. Thus for any $v\in \HH$
we have $v = \sum\limits_{i=1}^{n} \pi_i(v)$.
Let $\Omega_i = \{ 1\leq j\leq m : X_i\subseteq K_j\}$, and note that by assumption
$|\Omega_i|\geq l$ for each $i$.

\begin{Claim}
\label{many_zero}
Let $v \in \HH^{K_j}$ for some $j$.
Then $\pi_i(v)\in \HH^{K_j}$ for all $i$, and moreover $\pi_i(v) =0$ if $j\in\Omega_i$.
\end{Claim}
\begin{proof}
Since each of the groups $N_i$ is normalized by $K_j$
and $v$ is $K_j$-invariant, its projection $\pi_i(v)$ is
also $K_j$-invariant, which proves the first assertion.
On the other hand, by construction $\HH_{(i)}$ has no
$X_i$-invariant vectors. Hence if $j\in\Omega_i$, then
$\HH_{(i)}$ has no $K_j$-invariant vectors, and thus $\pi_i(v)=0$.
\end{proof}


We are now ready to prove both assertions of Lemma~\ref{nilpotent1}.

(a) By definition of codistance, we need to show that
given any vectors $v_j\in V^{K_j}$ for $1\leq j\leq m$, we have
$$
\left\|\sum_{j=1}^m v_j\right\|^2 \leq (m-l)\sum \|v_j\|^2.
$$
Using the decomposition of $\HH$ as the direct sum of $\HH_{(i)}$ we obtain
$$
\left\| \sum_{j=1}^m v_j \right\|^2 =
\left\| \sum_i \sum_{j=1}^m \pi_i(v_j) \right\|^2 =
\sum_i \left\| \sum_{j=1}^m \pi_i(v_j) \right\|^2
$$
Using Claim~\ref{many_zero} and the fact that $|\Omega_i|\geq l$, we get
\begin{multline}
\label{comp:codistance1}
\sum_i \left\| \sum_{j=1}^m \pi_i(v_j) \right\|^2=
\sum_i \left\| \sum_{j\not\in \Omega_i} \pi_i(v_j) \right\|^2
\leq
\sum_i (m-l) \sum_{j\not\in\Omega_i} \| \pi_i(v_j) \|^2 \\
=(m-l) \sum_{j=1}^m \sum_i\| \pi_i(v_j) \|^2=
(m-l) \sum_{j=1}^m \| v_j \|^2.
\end{multline}

(b) Suppose now that $V$ has no $C$-invariant vectors, and fix $i$ with $1\leq i\leq n$.
By assumption, $C\subseteq H_i$, so $V$ has no $H_i$-invariant vectors as well.
Recall that $H_i=\la K_j: j\not\in \Omega_i \ra$, and of course, $H_i$ is nilpotent
of class $\leq c$. Therefore, by Theorem~\ref{nilpotentcodistance} we have
$\codist(\{K_j : j\not\in\Omega_i\})\leq 1-\delta$ where $\delta$ is as in the
statement of Lemma~\ref{nilpotent1}(b). Equivalently, given vectors
$v_j\in K_j$ for $j\not\in\Omega_i$ we have
$$
\left\| \sum_{j\not\in \Omega_i} \pi_i(v_j) \right\|^2
\leq (1-\delta)(m-l) \sum_{j\not\in\Omega_i} \| \pi_i(v_j) \|^2.
$$
The result of (b) now follows by combining this bound with the same calculation
as in~\eqref{comp:codistance1}.
\end{proof}

\subsection{Estimates of codistances in Borel subgroups}

For the rest of this section we fix a regular root system $\Phi$
and a group $G$ with a strong $\Phi$-grading $\{X_{\alpha}\}$.
Let $\Gamma_{l}=\Gamma_l(\Phi)$ be the large Weyl graph of $\Phi$.
For each vertex of $\Gamma_l$ we fix a
functional $f\in\FFF$ representing that vertex. The vertex
itself will also be denoted by $f$, and the associated vertex
subgroup will be denoted by $G_f$.

In this subsection the vertex $f$ of $\Gamma_l$ will be fixed, and
let $E_f$ denote the set of all edges $e\in \Edg(\Gamma_l)$
with $e^+=f$. We shall use Lemma~\ref{nilpotent1} to obtain the following
bound on codistances between edge subgroups of $G_f$:

\begin{Proposition}
\label{claim56} The following hold:
\begin{itemize}
\item[(a)] $\codist(\{G_e: e\in E_f\})\leq \frac{1}{2}$.
\item[(b)] Let $CG_f$ be the core subgroup of $G_f$ and $\HH$
a unitary representation of $G_f$ without $CG_f$-invariant vectors.
Then 
$$
\codist(\{\HH^{G_e}: e\in E_f\})\leq \frac{1-\eps_{\Phi}}{2},
$$
where $\eps_{\Phi}=\frac{8}{(bor(\Phi)-2)4^{|\Phi|/2}}$
and $bor(\Phi)$ is the number of Borel subsets in $\Phi$.
\end{itemize}
\end{Proposition}
\begin{proof}
 Let $\{\alpha_{f,i}\}_{i=1}^{|\Phi|/2}$ be a (fixed)
ordering of the roots in the Borel set $\Phi_f$ induced by $f$,
that is, we assume that
$$
f(\alpha_{f,1})<f(\alpha_{f,2})<\ldots <f(\alpha_{f,|\Phi/2|}).
$$
We shall apply Lemma~\ref{nilpotent1} by letting
$N=G_f$, $C=CG_f$, $n=|\Phi|$, $m=|E_f|$,
$X_i=X_{\alpha_{f,i}}$ for $1\leq i\leq n$ and
$K_1,\ldots, K_m$ be the edge subgroups $\{G_e: e\in E_f\}$
listed in an arbitrary order.

The required conditions on the subgroups $N_i$ introduced in
the statement of Lemma~\ref{nilpotent1} hold by our ordering
of roots.

For each $1\leq i\leq n$ set $$E_{f,i}=\{e\in E_f: \alpha_{f,i}\not\in \Phi_e\}$$
\begin{Claim}
\label{eft} 
For any $i$ we have $|E_{f,i}| = |E_f|/2$.
\end{Claim}
\begin{proof}
The neighbors of the vertex $f$ in $\Gl$ can be grouped in pairs consisting
of two opposite Borel sets. For any pair of opposite Borel subsets
the root  $\alpha_{f,i}$ lies in exactly one of them,
which yields the claim.
\end{proof}
Note that $X_i\subseteq G_e$ if and only if $e\not\in E_{f,i}$, so
by Claim~\ref{eft}, we can take $l=m/2=|E_f|/2$ in the statement
of Lemma~\ref{nilpotent1}. Thus, Proposition~\ref{claim56}(a)
follows from Lemma~\ref{nilpotent1}(a).

To deduce Proposition~\ref{claim56}(b) from Lemma~\ref{nilpotent1}(b) we only
need to check that for each $1\leq i\leq n$, the core subgroup $CG_f$ is contained
in the group $G_{f,i}$ defined by
$$
G_{f,i}=\la\cup G_e : e\in E_{f,i}\ra.
$$
This is established in Claim~\ref{contain_core} below.

\begin{Claim}
\label{trivial}
The set
$$
\bigcup\{ \Phi_e : \ e\in E_{f,i}\}
$$
contains all roots from $\Phi_f$, which are not multiples of $\alpha_{f,i}$.
\end{Claim}
\begin{proof}
Let $\beta\in\Phi_f$, and assume that $\beta$ is not a multiple of
$\alpha_{f,i}$. Then there exists another functional $f'\in \FFF$ such that
$\beta\in\Phi_{f'}$ but $\alpha_{f,i}\not\in \Phi_{f'}$.
Hence $\Phi_f$ and $\Phi_{f'}$ are connected by an edge $e\in E_f$
(because $\Phi_f \neq \Phi_{f'}$ and $\Phi_f \cap \Phi_{f'}\neq \emptyset$).
Then $\beta\in \Phi_e$ but $\alpha_{f,t}\not\in\Phi_e$. Thus by
definition $e\in E_{f,i}$ and $\beta$ lies in the set defined above.
\end{proof}

\begin{Claim}
\label{contain_core}
The core subgroup $CG_f$ is contained in $G_{f,i}$ for each $i$.
\end{Claim}
\begin{proof}
By Claim~\ref{trivial} the group $G_{f,i}$ contains
the root subgroup $X_\beta$ for each $\beta\in \Phi_f$ which is
not a multiple of $\alpha_{f,i}$.

If the root $\alpha_{f,i}$ lies on the boundary of $\Phi_f$, then
the set $\Phi_f\setminus {\R\alpha_{f,i}}$ contains the core $C\Phi_f$,
and thus $G_{f,i}$ contains the core subgroup $CG_f$.
If the root $\alpha_{f,t}$ lies in the core $C\Phi_f$,
the inclusion $CG_f\subseteq G_{f,i}$ follows from the assumption
that the grading is strong.

This finishes the proof of Claim~\ref{contain_core} and thus also the proof
of Proposition~\ref{claim56}.
\end{proof}
\renewcommand{\qedsymbol}{}
\end{proof}

\subsection{Norm estimates}
In this subsection we establish the ``norm inequality'' (Corollary~\ref{claim57a})
which is needed to verify hypothesis (iii) in Theorem~\ref{th:gen_spec_criteria}.
This inequality will be proved by considering both the small and the large Weyl graphs.
We note that this is the only part of the paper where the small Weyl graph is used.
In this subsection we assume that $\HH$ is a \emph{representation of the whole group $G$}.

Recall that the large Weyl graph $\Gl$ and the small Weyl graph $\Gs$ have the same
sets of vertices. Also recall that $\Omega^0(\Gl,\HH)=\Omega^0(\Gs,\HH)$
is the set of all functions from $\Vert(\Gl)=\Vert(\Gs)$ to $\HH$
and $\Omega^0(\Gl,\HH)^{\{G_\nu\}}$ the set of all functions
$g\in \Omega^0(\Gl,\HH)$ such that $g(f)\in V^{G_f}$ for each vertex $f$.
Denote by $d_l$ and $d_s$ the difference operators of $\Gl$ and $\Gs$, respectively.
\begin{Lemma}
\label{lm:projectionbound}
Let $g \in \Omega^0(\Gl,\HH)^{\{G_\nu\}} = \Omega^0(\Gs,\HH)^{\{G_\nu\}}$.
If $\Phi$ is a regular root system, then
\begin{itemize}
\item[(a)]
$\|d_s g \|^2 \leq \|d_l g\|^2$
\item[(b)]
$\|d_l g \|^2 \leq C_\Phi \|d_s g\|^2$,
\end{itemize}
where the constant $C_\Phi$ depends only on the root system.
\end{Lemma}
\begin{proof}
(a) is clear since $\Gs$ is a subgraph of $\Gl$
and (b) holds since $\Gs$ is connected.
Indeed, for each edge $e$ in $\Edg(\Gl)$ we can find a
path in $\Gs$ connecting the endpoints and write
$g(e^+) - g(e^-) = \sum_{i} (g(e_i^+) - g(e_i^-)$.
Use the triangle inequality we get
$$
\|g(e^+) - g(e^-)\|^2 \leq k \sum_{i} \|g(e_i^+) - g(e_i^-)\|^2,
$$
where $k$ is the length of the path.
\end{proof}

\begin{Example}
The constant $C_\Phi$ can be easily computed for ``small'' root systems, e.g.,
one can take $C_{A_2} = 5$, $C_{B_2}=C_{BC_2}=3$ and $C_{G_2}=2$.
It is unclear how the constant $C_{\Phi}$ depends on the rank of the root system.
\end{Example}

For the rest of this subsection, for a subgroup $H$ of $G$
we denote by
$$
\pi_H: V\to V^H\qquad\mbox{ and }\qquad \pi_{H^{\perp}}: V\to (V^{H})^{\perp}
$$
the projections onto $V^H$ and its orthogonal complement $(V^{H})^{\perp}$, respectively.

For an edge $e$ of $\Gl$ we let
$GR_e=\la X_\alpha : \alpha\in \Phi_{e^+}\setminus \Phi_{e} \ra$
\begin{Claim}
\label{cl:invariance}
For any edge $e$ of $\Gl$ and any $v\in V^{G_e}$ we have
$\pi_{G_{e^+}}(v)=\pi_{GR_e}(v)$.
\end{Claim}
\begin{proof}
Let $k=|\Phi_e|$ and let $\{\beta_j\}_{1\leq j\leq k}$
be the roots in $\Phi_e$
ordered so that $$f(\beta_1)>f(\beta_2)>\ldots>f(\beta_k).$$
For $1\leq i\leq k$ let $H_i$ be the subgroup generated by
$GR_e$ and $\{X_{\beta_j}\}_{1\leq j\leq i}$.
By construction $H_0=GR_e$,  $H_k=G_{e^+}$ and
each $H_i$ is normalized by $X_{\beta_{i+1}}$.

By assumption, the vector $v$ is $X_{\beta_{i+1}}$-invariant
for each $0\leq i\leq k-1$. Hence $\pi_{H_i}(v)$ is also $X_{\beta_{i+1}}$-invariant,
so $\pi_{H_i}(v)=\pi_{H_{i+1}}(v)$.
Combining these equalities for all $i$ we get
$
\pi_{G_{e^+} }(v)=\pi_{H_k}(v)=
\pi_{H_0}(v)=\pi_{GR_e}(v).
$
\end{proof}

Now recall that $\Omega^1(\Gl,\HH)$ (resp. $\Omega^1(\Gs,\HH)$) is the space
of all functions from $\Edg(\Gl)$ (resp. $\Edg(\Gs)$) to $\HH$. Notice that
these two spaces are different unlike the spaces $\Omega^0(\Gl,\HH)=\Omega^0(\Gs,\HH)$.

As in \S~\ref{sec:spec_criteria} we have projections $\rho_1, \rho_2, \rho_3$
defined on each of those spaces
with $ \|\rho_1(g)\|^2+\|\rho_2(g)\|^2+\|\rho_3(g)\|^2=\|g\|^2\mbox{ for all } g.$
Note that in our new notations, for any $g\in \Omega^1(\Gl,\HH)$ (resp. $g\in \Omega^1(\Gs,\HH)$) and
$e\in \Edg(\Gl)$ (resp. $e\in \Edg(\Gs)$) we have
\begin{align*}
&\rho_1(g)(e)=\pi_{G_{e^+}}(g(e))&
&\rho_3(g)(e)=\pi_{CG_{e^+}^{\perp}} (g(e)).&
\end{align*}

The following is the main result of this subsection:
\begin{Theorem}
\label{th:projectionbound2}
For any $g\in \Omega^0(\Gl,\HH)^{\{G_\nu\}}$ we have
$$
\|d_s g \|^2 \leq \| \rho_1(d_s g) \|^2 + D_\Phi\| \rho_3(d_l g) \|^2,
$$
where the constant $D_\Phi$ depends only on the root system.
\end{Theorem}
\begin{Remark}
Notice that the second term on the right hand side involves the differential of the large Borel graph,
while the other two terms involve the differential of the small Borel graph.
\end{Remark}
\begin{proof}
Let $e$ be an edge in the small Borel graph $\Gs$.
Since $g(e^+)=\pi_{G_{e^+}}(g(e^+))=\pi_{GR_e}(g(e^+))$ and $g(e^-)\in \HH^{G_{e}}$,
by Claim~\ref{cl:invariance} we have
\begin{multline*}
\|g(e^+) - g(e^-) \|^2 =
\| \pi_{GR_e}\big(g(e^+) - g(e^-) \big) \|^2 +
   \| \pi_{{GR_e}^\perp} \big(g(e^+) - g(e^-) \big) \|^2=\\
   \| \pi_{G_{e^+}}\big(g(e^+) - g(e^-) \big) \|^2 +
   \| \pi_{{GR_e}^\perp} (g(e^-)) \|^2.
\end{multline*}
By the definition of the small Borel graph, we can find another
vertex $f'$ such that
$$
\Phi_{e^+}\setminus \Phi_{e^-}=\Phi_{e^+}\setminus \Phi_{e} \subset C_{f'}.
$$
This implies that $GR_e \subseteq CG_{f'}\subset G_{f'}$, and so
$ \pi_{{GR_e}^\perp} (g(f'))=0$.
Therefore
$$
\|g(e^+) - g(e^-) \|^2 =
   \| \pi_{G_{e^+} }\big(g(e^+) - g(e^-) \big) \|^2 +
   \| \pi_{{GR_e}^\perp} \big(g(f') - g(e^-) \big) \|^2.
$$
Since clearly $f'\neq e^-$, there is an edge $e'$ in $\Gl$ connecting
$f'$  and $e^-$, with $(e')^+=f'$ and $(e')^-=e^-$.
The inclusion $GR_e \subseteq CG_{f'}$ implies that
\begin{multline*}
\| \pi_{{GR_e}^\perp} \big(g(f') - g(e^-) \big) \|^2
\leq \| \pi_{{CG_{f'}}^\perp} \big(g(f') - g(e^-) \big) \|^2= \\
\| \rho_3( d_lg)(e')\|^2\le \| \rho_3( d_lg)\|^2.
\end{multline*}

Summing over all edges $e$ of $\Gs$ we get
\begin{multline*}
\| d_s g\|^2 = \| \rho_1(d_s g) \|^2 + \sum_{e\in \Edg(\Gs)}
\| \pi_{{GR_e}^\perp} \big(g(f') - g(e^-) \big) \|^2 \leq \\
\| \rho_1(d_s g) \|^2 +\frac{\left  |\Edg(\Gs)\right| }2 \| \rho_3(d_l g) \|^2.
\qedhere
\end{multline*}
\end{proof}
\begin{Example}
Carefully doing the above estimates in the case of some ``small'' root systems gives that
once can take $D_\Phi=1$ if $\Phi$ is of type $A_2$, $B_2$, $BC_2$ or $G_2$.
\end{Example}

Combining Lemma~\ref{lm:projectionbound}(b), Theorem~\ref{th:projectionbound2}
and the obvious inequality $\|\rho_1(d_s g) \|\leq \|\rho_1(d_l g) \|$,
we obtain the desired inequality, which verifies hypothesis (iii)
of Theorem~\ref{th:gen_spec_criteria} in our setting.
Its statement only involves the large Weyl graph:
\begin{Corollary}
\label{claim57a}
Let $g$ be a function in $\Omega^0(\Gl,\HH)^{\{G_\nu\}}$.
Then
$$
\|d_l g \|^2 \leq
A_\Phi \|\rho_1(d_l g) \|^2 + B_\Phi\| \rho_3(d_l g) \|^2,
$$
where $A_\Phi$ and $B_\Phi$  are constants which
depend on the root system $\Phi$.
\end{Corollary}
\begin{Example}
In the case $\Phi=A_2$, the above estimates show that one can take
$A_{A_2}=5$ and $B_{A_2}=5$.
These bounds are not optimal ---  in~\cite{EJ} it is shown  that
one can use  $A_{A_2}=3$ and $B_{A_2}=5$.
\end{Example}

\subsection{Proof of Theorem~\ref{th:relativeT_wrt_ root subgroup}}
\label{subsec:proof of T}

\begin{proof}[Proof of Theorem~\ref{th:relativeT_wrt_ root subgroup}]
As explained at the beginning of this section, we shall apply
Theorem~\ref{th:gen_spec_criteria} to the canonical decomposition
of $G$ over the large Weyl graph $\Gamma=\Gl\Phi$. Let us check
that inequalities (i)-(iii) are satisfied.

By Lemma~\ref{spectral_gap} the spectral gap of the Laplacian
$\lam_1(\Delta)$ is equal to the degree of $\Gamma$. Hence
in the notations of Theorem~\ref{th:gen_spec_criteria} we have $\bar p=1/2$.
Thus, (i) and (ii) hold by Proposition~\ref{claim56}.
Finally, (iii) holds by Corollary~\ref{claim57a}.

Since all parameters in these inequalities depend only on $\Phi$, Theorem~\ref{th:gen_spec_criteria}
yields that $\kappa(G,\cup G_{f})\geq \mathcal K_{\Phi}>0$, with
$\mathcal K_{\Phi}$ depending only on $\Phi$.
Finally, to obtain the same conclusion with $G_{f}$'s replaced by root subgroups $X_{\alpha}$'s, we only
need to observe that each $G_{f}$ lies in a bounded product
of root subgroups, where the number of factors does not exceed $|\Phi|/2$.
\end{proof}


\section{Reductions of  root systems}
\label{sec:reductions}

\subsection{Reductions}

\begin{Definition}
Let $\Phi$ be a root system in a space $V=\R\Phi$. A \emph{reduction of $\Phi$} \index{reduction}
is a surjective linear map $\eta: V\to V'$ where $V'$ is another
nonzero real vector space. The set $\Phi'=\eta(\Phi)\setminus \{0\}$
is called the \emph{induced root system}. We will also say that
$\eta$ is a \emph{reduction of $\Phi$ to $\Phi'$} and symbolically
write $\eta:\Phi\to\Phi'$.
\end{Definition}

\begin{Lemma}  Let $\Phi$ be a root system, $\eta$
a reduction of $\Phi$, and $\Phi'$ the induced root system. Let
$\{X_{\alpha}\}_{\alpha\in \Phi}$ be a $\Phi$-grading.
For any $\alpha'\in \Phi'$ put
$$
Y_{\alpha'} = \la X_{\alpha} \mid \eta(\alpha) = \alpha' \ra.
$$
Then  $\{Y_{\alpha'}\}_{\alpha' \in \Phi'}$ is a $\Phi'$-grading, which will
be called the \emph{coarsened} grading. \index{grading!coarsened}
\end{Lemma}
\begin{proof} This is a direct consequence of the following fact: if $A=\la S_1\ra $ and
$B=\la S_2\ra$ are two subgroups of the same group, then $[A,B]$ is contained in
the subgroup generated by all possible commutators in $S_1\cup
S_2$ of length at least 2 with at least one entry from  $S_1$ and
$S_2$.
\end{proof}
A reduction $\eta:\Phi\to\Phi'$ enables us to replace a grading of a given group $G$
by the ``large'' root system $\Phi$ by the coarsened grading by the ``small''
root system $\Phi'$ which may be easier to analyze. Note that the root subgroups
of the coarsened grading need not generate $G$ since we ``lose'' root subgroups
of the initial grading which lie in $\ker \eta$. Likewise, since different roots
of $\Phi$ may map to the same root of $\Phi'$, the coarsened grading need not
be strong even if the initial grading is strong.

\vskip .1cm
Since we are mostly interested
in strong gradings, we would like to have a natural sufficient condition on $\eta$
and the initial $\Phi$-grading which ensures that the coarsened $\Phi'$-grading is strong.
If the only assumption on the initial $\Phi$-grading $\{X_{\alpha}\}$ is that it is strong,
we would limit ourselves to reductions with trivial kernel (which are not interesting).
However, if we assume that $\{X_{\alpha}\}$ is $k$-strong for some $k < \rk(\Phi)$,
we can let $\eta$ be any $k$-good reduction, as defined below, which is much less
restrictive.

\begin{Definition}
Let $k\geq 2$ be an integer. A reduction $\eta$ of $\Phi$ to $\Phi'$
is called \emph{$k$-good} if \index{reduction!$k$-good}
\begin{itemize}

\item[(a)] for any $\gamma \in \ker \eta\cap \Phi$,
there exists an irreducible regular subsystem $\Psi$ of $\Phi$
of rank $k$ such that $\gamma \in \Psi$ and
$\ker \eta \cap \Psi\subseteq \R \gamma$;

\item[(b)] for any $f\in \FFF(\Phi')$, $\gamma'\in C_f$ and $\gamma\in \Phi$
with $\eta(\gamma)=\gamma'$,
there exists an irreducible subsystem $\Psi$ of $\Phi$ of rank $k$
and $g\in \FFF(\Psi)$ such that $\gamma \in C_g$,
$\eta(\Psi_g)\subseteq \Phi'_f$ and $\Psi\cap \eta^{-1}(\R\gamma')\subseteq \R\gamma$.
\end{itemize}
\end{Definition}

\begin{Lemma}\label{strongreduction}
Let $\Phi$ be a root system, let $\eta$ be a $k$-good reduction of $\Phi$,
and $\Phi'=\eta(\Phi)\setminus\{0\}$ the induced root system.
Let $\{X_{\alpha}\}_{\alpha\in \Phi}$ be
a $k$-strong grading of a group $G$.
Then the coarsened grading $\{Y_{\alpha'}\}_{\alpha'\in\Phi'}$ is a strong grading of $G$.
\end{Lemma}
\begin{proof}
First let us show that $G$ is generated by $\{Y_{\alpha'}\}$.
Since the subgroups $\{X_{\alpha}\}_{\alpha\in \Phi}$
generate $G$, it is enough to show that every $X_{\gamma}$
lies in the subgroup generated by $\{Y_{\alpha'}\}$. This is clear if
$\eta(\gamma)\ne 0$. Assume now that $\eta (\gamma)=0$.
Since the reduction $\eta$ is $k$-good, $\Phi$
has an irreducible regular subsystem $\Psi$ of rank $k$ such
that $\gamma \in \Psi$ and $\ker \eta \,\cap\, \Psi\subseteq \R \gamma$.
Since $\Psi$ is regular, $\gamma\in C_{f}$ for some $f\in \FFF(\Psi)$
by Corollary~\ref{rootincore}.
Since the grading $\{X_{\alpha}\}_{\alpha\in \Phi}$
is $k$-strong, $X_{\gamma}$ lies in the group generated by
$\{X_{\alpha}\}_{\alpha\in \Psi_f\setminus \R\gamma}\subseteq
\{X_{\alpha}\}_{\alpha\in \Phi\setminus \ker\eta}$, and so
$X_{\gamma}$ lies in the subgroup generated by $\{Y_{\alpha'}\}$.

The fact that $\{Y_{\alpha'}\}$ is a strong grading of $G$ follows
directly from part (b) of the definition of a $k$-good reduction and
the assumption that the grading $\{X_{\alpha}\}_{\alpha\in \Phi}$ is $k$-strong.
\end{proof}

\begin{Corollary}
\label{identityreduction}
Let $\Phi$ be a system such that any root lies in an irreducible
subsystem of rank $k$. If a $\Phi$-grading of a group $G$ is
$k$-strong, then it is also strong.
\end{Corollary}
\begin{proof}
The identity map $\R\Phi\to\R\Phi$ is clearly a reduction. It is $k$-good if and
only if any root of $  \Phi$ lies in an irreducible subsystem of
rank $k$.
\end{proof}

\subsection{Examples of good reductions}
\label{goodreductions}
In this subsection we present several examples of good reductions
that will be used in this paper. In particular, we will establish the
following result:
\begin{Proposition}
\label{prop:classicalreduction}
Every irreducible classical root system of rank $>2$ admits a
$2$-good reduction to an irreducible classical root system of rank $2$.
\end{Proposition}

The following elementary fact can be proved by routine
case-by-case verification.
\begin{Claim}
\label{automatic}
Let $\Phi$ be a classical irreducible root system of rank $\geq 2$.
For $\alpha\in \Phi$ let $N_{\Phi}(\alpha)$ be the set of all $\beta\in\Phi$
such that
$\mathrm{span}\{\alpha,\beta\} \cap \Phi$ is an irreducible rank $2$ system.
Then for any $\alpha\in\Phi$, the set $\{\alpha\}\cup N_{\Phi}(\alpha)$
spans $\R\Phi$.
\end{Claim}

Now let $\eta:\Phi\to\Phi'$ be a reduction of root systems,
where $\Phi$ is classical irreducible of rank $\geq 2$.
Claim~\ref{automatic} implies that $\eta$ always
satisfies condition (a) in the definition of a $2$-good reduction.
Indeed, $\ker\eta\neq\dbR\Phi$ (since $\Phi'\neq\emptyset$), so
given $\gamma\in\ker\eta\cap\Phi$, by Claim~\ref{automatic} there exists 
$\beta\in N_{\Phi}(\gamma)\setminus \ker\eta$. Then clearly
the subsystem $\Psi=(\dbR \gamma+\dbR\beta)\cap\Phi$ has the required property.

In order to speed up verification of condition (b) in the examples below
we shall use symmetries of the ``large'' root system $\Phi$
which project to symmetries of the ``small'' root system $\Phi'$
under $\eta$. Formally, we shall use the following observation:
\vskip .1cm
Suppose that a group $Q$ acts linearly on $\dbR\Phi$ preserving 
$\Phi$ and the subspace $\ker\eta$. Thus we have the induced action
of $Q$ on $\Phi'$ given by
$$
q\, \eta(\alpha) = \eta(q\,\alpha) \mbox{ for all }\alpha\in \Phi, q\in Q.
$$
Then to prove that $\eta$ satisfies condition (b) in the definition of a
$k$-good reduction it suffices to check that condition for one representatives from each
$Q$-orbit in
$\{(f,\gamma): f\in \FFF(\Phi')/\sim,\eta(\gamma)\in C_f'\}$.

In the following examples $\{e_1,\ldots,e_n\}$ is the standard
orthonormal basis of $\R^n$.
\begin{Reduction}
\label{reductionAn}
The system $A_n$ (or rather its canonical realization)
is defined to be the set of vectors of $\R^{n+1}$ of
length $\sqrt 2$ with integer coordinates that sum to $0$ (note that $A_n$
spans a proper subspace of $\R^{n+1}$). Thus
$$
A_n=\{e_i-e_j : 1\leq i,j\leq n+1, i\neq j\}.
$$
Choose non-empty pairwise disjoint subsets $I_1, I_2, I_3$ such that
$I_1\sqcup I_2\sqcup I_3=\{1,\ldots, n-1\}$. 
Then the map $\eta:\R^{n+1}\to\R^3$ defined by
$$
\eta(x_1,\ldots,x_{n+1})=
\left(\sum_{i\in I_1} x_i,\,\sum_{i\in I_2} x_i,\, \sum_{i\in I_3} x_i\right),
$$
is a reduction of $A_n$ to $A_2$. Let us show that it is $2$-good
directly from definition (recall that we only need to check condition (b)).

Without loss of generality we can assume that $\Phi'_f$ is the Borel
subset with the base $\{(1,-1,0), (0,1,-1)\}$. The only root in $C_f$
is $\gamma'=(1,0,-1)$, and any root $\gamma$ satisfying $\eta(\gamma)=\gamma'$
has the form $\gamma=e_{i_1}-e_{i_3}$ for some $i_1\in I_1$ and $i_3\in I_3$.
Now choose any $i_2\in I_2$, let $\Psi=\{\pm (e_{i_1}-e_{i_2}),\pm (e_{i_2}-e_{i_3}),\pm(e_{i_1}-e_{i_3})\}$,
and let $g\in\mathcal F(\Psi)$ be such that $\Psi_g=\{e_{i_1}-e_{i_2},e_{i_2}-e_{i_3},e_{i_1}-e_{i_3}\}$.
Then condition (b) is clearly satisfied.
\end{Reduction}

\begin{Reduction}
\label{reductionBn}
The system $B_n$ consists of all integer vectors in $\R^n$ of
length $1$ or $\sqrt 2$. Thus
$$
B_n=\{\pm e_i\pm e_j : 1\leq i<j\leq n\}\cup\{\pm e_i : 1\leq i\leq n\}.
$$
A natural reduction of $B_n$ to $B_{2}$ is given
by the map $\eta \colon \R^{n}\to \R^{2}$ defined by
$$
\eta(x_1,\ldots,  x_n)=(x_1,x_2).
$$
Let us show that this reduction is $2$-good. Let $Q$ be the dihedral group
of order $8$, acting naturally on the first two coordinates of $\R^n$.
This action preserves $\ker\eta$, and the induced $Q$-action on
$\{(f,\gamma'): f\in \FFF(B_2)/\!\sim,\,\gamma'\in C_f\}$
has two orbits. The following table shows how to verify condition (b)
from the definition of a good reduction for one representative in
each orbit (using the notations from that definition). We do not specify the functionals $f$
and $g$ themselves; instead we list the bases  of
the corresponding Borel sets $\Phi'_f$ and $\Psi_g$.

\vskip .2cm

{\begin{tabular}{|l|l|l|l|} \hline
 $\gamma'$ & $\gamma$ & $\mbox{ base of }\Phi'_f$ & $\mbox{ base of }\Psi_g$   \\
\hline
 (1,0)& $e_1+xe_i\, (i\ge 3)$   & (1,-1), (0,1) & $e_1-e_2,e_2+xe_i$\\
\hline
 (1,1)& $e_1+e_2$  & (1,-1),(0,1)&
 $e_1-e_2,e_2 $\\
\hline
\end{tabular}}
\end{Reduction}

\begin{Reduction}
\label{reductionDn}
The system $D_n$ consists of all integer vectors in $\R^n$ of length
$\sqrt 2$. Thus
$$
D_n=\{\pm e_i\pm e_j : 1\leq i<j\leq n\}.
$$
A natural reduction of $D_n (n\ge 3)$ to $B_{2}$ is given
by the map $\eta\colon \R^{n}\to \R^{2}$ defined by
$$\eta(x_1,\ldots, x_n)=(x_1,x_{2}).$$ This reduction is
$2$-good -- the proof is similar to the case of $B_n$.
\end{Reduction}

\begin{Reduction}
\label{reductionCn}
The system $C_n$ consists of all integer vectors in $\R^n$ of
length $\sqrt 2$ together with all vectors of the form $2e$, where
$e$ is an integer vector of length 1. Thus
$$
C_n=\{\pm e_i\pm e_j : 1\leq i<j\leq n\}\cup\{\pm 2e_i : 1\leq i\leq n\}.
$$
A natural reduction of $C_n (n\ge 3)$ to $BC_2$ is given by
the map $\eta \colon \R^{n}\to
\R^{2}$ defined by
$$
\eta(x_1,\ldots,  x_n)=(x_1,x_2).
$$
Let us show that this reduction is $2$-good. We use the same
action of the dihedral group of order $8$ as in the example $B_n\to B_2$,
but this time there are three orbits in
$\{(f,\gamma'): f\in \FFF(BC_2)/\!\sim,\,\gamma'\in C_f\}$. The following
table covers all the cases.

\vskip .2cm

{\begin{tabular}{|l|l|l|l|} \hline
 $\gamma'$ & $\gamma$ & $\mbox{ base of }\Phi'_f$ & $\mbox{ base of }\Psi_g$   \\
\hline
 (1,0)& $e_1+ xe_i \,(i\ge 3)$   & (1,-1), (0,1) & $e_1-e_2,e_2+ xe_i$\\
\hline
 (1,1)& $e_1+e_2$  & (1,-1),(0,1)&
 $e_1-e_2,2e_2$\\
\hline
 (2,0)& $2e_1   $   & (1,-1), (0,1) & $e_1-e_2,2e_2 $\\
\hline
\end{tabular}}
\end{Reduction}

\begin{Reduction}
The root system $C_n$ also admits a natural reduction to $B_2$.

Fix $1\le i<n$. The map
$\eta_i \colon \R^n\to \R^2$ given by
$$
\eta_i(x_1,\dots,x_n)=(x_1+\ldots+x_i,x_{i+1}+\ldots+x_n)
$$
is a reduction of $C_n$ to $C_2$. Composing $\eta_i$ with some
isomorphism $C_2\to B_2$, we obtain an explicit reduction of $C_n$ to $B_2$.

For instance, in the case $i=n-1$ we obtain
the following reduction $\eta$ from $C_n$ to $B_2$:
$$
\eta(x_1,\dots,x_n)=\left(\frac{x_1+\ldots+x_{n-1}-x_n}{2},\,\frac{x_1+\ldots+x_{n-1}+x_n}{2}\right).
$$

Let us show that it is 2-good. This time we take $Q=\Z/2\Z \times \Z/2\Z$, acting on $C_n$
via the maps $\phi_{\eps_1,\eps_2}$, with $\eps_{1},\eps_{2}=\pm 1$, defined by
$$
\phi_{\eps_1,\eps_2}(x_i)=\eps_1 x_i \mbox{ for }1\le i\le n-1, \quad \phi_{\eps_1,\eps_2}(x_n)=\eps_2 x_n.
$$
There are four $Q$-orbits in
$\{(f,\gamma'): f\in \FFF(B_2)/\!\sim,\,\gamma'\in C_f\}$.
All the cases are described in the following table:

\vskip .2cm

{\begin{tabular}{|l|l|l|l|} \hline
 $\gamma'$ & $\gamma$ & $\mbox{ base of }\Phi'_f$ & $\mbox{ base of }\Psi_g$   \\
\hline
 (1,0)& $e_i-e_n \,(i\le n-1)$   & (1,-1), (0,1) & $-2e_n,e_i+e_n$\\
\hline
 (1,1)& $2e_i \,(i\le n-1)$  & (1,-1),(0, 1)& $-2e_n,e_i+e_n$\\
 \hline
 (1,1)& $ e_i+e_j \,(1\le i<j\le n-1)$  & (1,-1),(0, 1)& $e_j-e_n,e_i+e_n$\\
\hline (0,1) & $e_i+e_n \,(i\le n-1) $ & ( 1,1),(-1,0)& $2e_i,
-e_i+e_n$\\
\hline (-1,1) & $2e_n  $ & ( 1,1),(-1,0)& $2e_i,
-e_i+e_n$\\
\hline
\end{tabular}}

\end{Reduction}

\begin{Reduction}
The system $BC_n$ is the union of $B_n$ and $C_n$ (in their standard
realizations). Thus
$$
BC_n=\{\pm e_i\pm e_j : 1\leq i<j\leq n\}\cup\{\pm e_i, \pm 2e_i : 1\leq i\leq n\}.
$$
Once again, the map $\eta: \R^n\to\R^2$ given by 
$$
\eta(x_1,\ldots,x_n)=(x_1,x_2)
$$
is a reduction of $BC_n$ to $BC_2$.
To show that this reduction is 2-good we use the same action
of the dihedral group of order 8 as in the reductions $B_n\to B_2$ and $C_n\to BC_2$.
There are three $Q$-orbits in $\{(f,\gamma'): f\in \FFF(BC_2)/\!\sim,\, \gamma'\in C_f\}$,
whose representatives are listed in the following table.

\vskip .2cm

{\begin{tabular}{|l|l|l|l|} \hline
 $\gamma'$ & $\gamma$ & $\mbox{ base of }\Phi'_f$ & $\mbox{ base of }\Psi_g$   \\
\hline
 (1,0)& $e_1+xe_i\, (i\ge 3)$   & (1,-1), (0,1) & $e_1-e_2,e_2+xe_i$\\
\hline
 (1,1)& $e_1+e_2$  & (1,-1),(0,1)&
 $e_1-e_2,e_2$\\
\hline
 (2,0)& $2e_1   $   & (1,-1), (0,1) & $e_1-e_2,e_2 $\\
\hline
\end{tabular}}
\end{Reduction}

\begin{Reduction}
\label{reductionF4}
The system $G_2$ consists of 12 vectors of lengths $\sqrt 2$ and
$\sqrt 6$ of $\R^3$ with integer coordinates that sum to $0$.
Thus
$$
G_2=\{e_i-e_j : 1\leq i,j\leq 3, i\neq j\}\cup
\{\pm(2e_i-e_j-e_k): 1\leq i,j,k \leq 3, i\neq j\neq k\neq i\}
$$
The system $F_4$ consists of vectors $v$ of length 1 or $\sqrt 2$ in $\R^4$
such that the coordinates of $2v$ are all integers and are either all
even or all odd.
Thus
$$
F_4=\{\pm e_i: 1\leq i\leq 4\}\cup\left\{\frac{1}{2}(\pm e_1\pm e_2\pm e_3\pm e_4)\right\}
\cup\{\pm e_i\pm e_j: 1\leq i\leq j\leq 4\}.
$$
A reduction of $F_4$ to $G_2$ is given by the map
$\eta:\R^4\to \R^3$ defined by
$$
\eta(x_1,x_2,x_3,x_4)=(x_1-x_2, x_2-x_3,x_3-x_1).
$$
Let us show that this reduction is $2$-good. This time we use
an action of $S_3 \times \Z/2\Z$ where  $S_3$ permutes the first three
coordinates, and the non-trivial element of $\Z/2\Z$ sends
$(x_1,x_2,x_3,x_4)$ to $(-x_1,-x_2,-x_3,x_4)$. This reduces all
the calculations to the following cases.

\vskip .2cm

{\begin{tabular}{|l|l|l|l|} \hline
 $\gamma'$ & $\gamma$ & $\mbox{ base of } \Phi'_f$ & $\mbox{ base of }\Psi_g$   \\
\hline
 (0,1,-1)& $(1,1,0,0)$    & (1,-1,0), (-1,2,-1) &
  $(1,0,1,0), (0,1,-1,0)$\\

 \hline
 (0,1,-1)& $(0,0,-1,x)$    & (1,-1,0), (-1,2,-1) &
 $(0,-1,0,x), (0,1,-1,0)$\\
 \hline
 (0,1,-1)& $(\frac 12,\frac 12,-\frac 12,x)$   & (1,-1,0), (-1,2,-1) &
 $(\frac 12,-\frac 12,\frac 12,x), (0,1,-1,0)$\\
 \hline
 (1,0,-1)& $(0,-1,-1,0)$    & (1,-1,0), (-1,2,-1) &
 $(0,-1,0,0), (0,1,-1,0)$\\
 \hline
 (1,0,-1)& $(1,0,0,x)$    & (1,-1,0), (-1,2,-1) &
 $(1,0,1,0), (0, 0,-1,x)$\\
 \hline
 (1,0,-1)& $(\frac 12,-\frac 12,-\frac 12,x)$&   (1, -1,0), (-1,2,-1) &
 $(\frac 12,-\frac 12,\frac 12,x), (0, 0,-1,0)$\\
 \hline
 (1,1,-2)& $(1,0,-1,0)$&   (1, -1,0), (-1,2,-1) &
 $(1,-1,0,0), (0, 1,-1,0)$\\
 \hline
 (2,-1,-1)& $(1,-1,0,0)$&   (1, -1,0), (-1,2,-1) &
 $(0,-1,0,1), (1, 0,0,-1)$\\
 \hline
\end{tabular}}

\end{Reduction}

\begin{Reduction}
\label{reductionE}
The root system  $E_8$ consists of the vectors of length $\sqrt 2$
in $\Z^8$ and $(\Z+\frac{1}{2})^8$ such that the sum of all
coordinates is an even number. The system $E_7$ is   the
intersection of $E_8$ with the hyperplane of vectors orthogonal
to $(0,0,0,0,0,0,1,-1)$ in $E_8$ and the system  $E_6$ is the
intersection of $E_7$ with the hyperplane of vectors orthogonal
to $(0,0,0,0,1,-1,0,0)$. The map $\eta\colon \R^8\to \R^3$
$$
\eta(x_1,x_2,x_3,x_4,x_5,x_6,x_7,x_8)=(x_1-x_2, x_2-x_3,-x_3-x_1).
$$
is a reduction of $E_8$ to $G_2$, and the restriction
of $\eta$ to $\R^7$ (resp. $\R^6$) is a reduction of
$E_7$ (resp. $E_6$) to $G_2$.

Each of these reductions is 2-good, and the proof is similar to the case $F_4\to G_2$.
\end{Reduction}
\begin{Reduction} 
Let $n\ge 3$. Then the map $\eta:\R^n\to \R^3$
defined by $$\eta(x_1,\ldots,x_n)=(x_1,x_2,x_3)$$ is a 3-good
reduction of $BC_n$ to $BC_3$. The proof is analogous to the previous
examples.
\end{Reduction}

\begin{Reduction}  
Let $n\ge 3$ be a natural number and let  $a_k=(\cos \frac{2\pi k}n, \sin \frac{2\pi k}n)\in \R^2$. Define the root system $I_n\subset \R^2$ by
$$
I_n=\{a_k-a_l:\ 1\le l\ne k\le n\}.
$$ 
It is easy to see that $I_3=A_2$, $I_4=C_2$ and $I_6=G_2$. For any $n$, if we normalize all roots in $I_n$ (that is, replace, each $v\in I_n$ by $\frac{v}{\|v\|}$), we obtain the 2-dimensional root system whose elements connect
the origin with vertices of a regular $2n$-gon. This root system arises in the classification of finite Coxeter groups and is sometimes denoted by $I_2(n)$.

The map $\eta\colon \R^{n+1}\to \R^2$ defined by
$$
\eta(x_1,\cdots, x_{n+1})=\sum_{i=1}^{n+1} x_ia_i
$$ 
is a reduction of $A_n$ to $I_{n+1}$.  We shall prove that it is 2-good.
\vskip .2cm

First, we need to describe the boundary of Borel subsets of $I_n$.
If $ f\in \FFF(I_n)$ then, since $f$ is different from 0 on $I_n$, we can find a permutation
$i_1,\ldots, i_n$ of $\{1,\ldots,n\}$ such that
\begin{equation}
\label{eq:Inorder}
f(a_{i_1})>f(a_{i_2})>\cdots >f(a_{i_n}),
\end{equation}
so that the Borel set $I_f$ is equal to $\{a_{i_s}-a_{i_t} : s<t\}$.
We claim that the root $a_{i_s}-a_{i_t}$ lies in $\partial I_f$, 
the boundary of $I_f$, if and only if $t=s+1$.

The forward direction is clear since if $t\geq s+2$, then
$a_{i_s}-a_{i_t}=(a_{i_s}-a_{i_{s+1}})+(a_{i_{s+1}}-a_{i_t})$, whence
$a_{i_s}-a_{i_t}\in C_f$ by Lemma~\ref{corebase}.

To prove the converse, first observe that $\{i_2,i_3\}=\{i_1-1,i_1+1\}$ and for $ 1\le k\le n/2$ we have that $i_{2k}=i_1+k(i_2-i_1)$ and $i_{2k-1}=i_1+(k-1)(i_3-i_1)$. It is then easy to see (algebraically or geometrically)
that 
\begin{equation}
\label{eq:Inorder2}
a_{i_{2k-1}}-a_{i_{2k}}\in \R(a_1-a_2) \mbox{ and } a_{i_{2k}}-a_{i_{2k+1}}\in \R(a_2-a_3).
\end{equation}
Combined with what we already showed, this implies that $\partial I_f\subseteq (\R(a_1-a_2)\cup \R(a_2-a_3))\cap I_f$.
On the other hand, it is clear that the boundary of any Borel set in any root system of rank $2$ is a union of
two half-lines. Therefore, $\partial I_f= (\R(a_1-a_2)\cup \R(a_2-a_3))\cap I_f$, and from \eqref{eq:Inorder2}
we deduce that $a_{i_s}-a_{i_{s+1}}\in \partial I_f$ for all $1\leq s\leq n-1$.
\vskip .12cm
Now let $\gamma^\prime\in C_f$ and take any $\gamma\in A_{n+1}$ with $\eta(\gamma)=\gamma^\prime$. 
If $\gamma=e_{i_t}-e_{i_s}$, then $\gamma^\prime=a_{i_t}-a_{i_s}$, and since $\gamma^\prime\in C_f $, 
by the above argument $t-s\ge 2$. Hence we can take $\Psi_g$ to be  $\{e_{i_t}-e_{i_{t+1}}, e_{i_{t+1}}-e_{i_s}, \gamma\}$. This proves that the reduction is 2-good.
\end{Reduction}


\section{Steinberg groups over commutative rings}
\label{sec:steinberg}

In this section we prove property $(T)$ for Steinberg groups of rank $\geq 2$
over finitely generated commutative rings and estimate asymptotic
behavior of Kazhdan constants.

\subsection{Graded covers}
 Let $\Gamma$ be a finite graph and $G$ a group with a chosen decomposition
$(\{G_{\nu}\}_{\nu\in\Vert(\Gamma)}, \{G_{e}\}_{e\in\Edg(\Gamma)})$ over $\Gamma$.
If $H$ is another group with a decomposition
$(\{H_{\nu}, \{H_{e}\}\})$ over $\Gamma$,
we will say that the decomposition $(\{H_{\nu}\}, \{H_{e}\})$
is isomorphic to $(\{G_{\nu}\}, \{G_{e}\})$
if there are isomorphisms $\iota_{\nu}:G_{\nu}\cong H_{\nu}$
for each $\nu\in\Vert(\Gamma)$ and  $\iota_{e}:G_{e}\cong H_{e}$
for each $e\in\Edg(\Gamma)$ such that ${\iota_{e^+}}_{| G_e}=\iota_e$
and $\iota_{\bar e}=\iota_e$.

Among all groups which admit a decomposition over the graph $\Gamma$ isomorphic
to $(\{G_{\nu}\}, \{G_{e}\})$
there is the ``largest'' one, which surjects onto any other group with this property.
This group will be called the \emph{cover of $G$ corresponding to $(\{G_{\nu}\}_{\nu\in\Vert(\Gamma)}, \{G_{e}\}_{e\in\Edg(\Gamma)})$} and can be defined as the free product
of the vertex subgroups $\{G_\nu\}_{\nu\in \Vert(\Gamma)}$ amalgamated along
the edge subgroups $\{G_e\}_{e\in \Edg(\Gamma)}$.
We will be particularly interested in the special case of this construction
dealing with decompositions associated to gradings by root systems.

\begin{Definition}
Let $G$ be a group, $\Phi$ a root system and
$\{X_{\alpha}\}_{\alpha\in\Phi}$ a $\Phi$-grading of $G$. Let $\Gamma_l\Phi$
be the large Weyl graph of $\Phi$, and consider the canonical decomposition
of $G$ over $\Gamma_l\Phi$. The cover of $G$ corresponding to this
decomposition will be called the \emph{graded cover of $G$ with respect to the grading $\{X_{\alpha}\}$}.
\index{graded cover}
\end{Definition}
Graded covers may be also
defined using generators and relations. Assume that $G_f=\la\cup
_{\alpha\in \Phi_f} X_\alpha |\ R_f\ra$ for each $f\in
\mathcal{F}(\Phi)$. Then the graded cover of $G$ with respect to
$\{X_{\alpha}\}_{\alpha\in\Phi}$ is isomorphic to
$$
\la \cup _{\alpha\in \Phi} X_\alpha | \ \cup_{f\in \mathcal{F}(\Phi)} R_f\ra.
$$ 
Observe that if $\pi:G\to G'$ is an epimorphism,
and $\{X_{\alpha}\}$ is a $\Phi$-grading of $G$, then $\{\pi(X_{\alpha})\}$
is a $\Phi$-grading of $G'$. If in addition $\pi$ is injective on all the Borel
subgroups of $G$, then the graded covers of $G$ and $G'$ coincide.

Here is a simple observation about automorphisms of graded covers recorded
here for later use.

\begin{Definition}
If $G$ is a group and $\{X_{\alpha}\}_{\alpha\in \Phi}$
is a grading of $G$, an automorphism $\pi\in \Aut(G)$ will be called
\emph{graded } (with respect to $\{X_{\alpha}\}$) if $\pi$ permutes \index{graded automorphism}
the root subgroups $\{X_{\alpha}\}$ between themselves, and the induced
action of $\pi$ on $\Phi$ sends Borel sets to Borel sets.
The group of all graded automorphisms
of $G$ will be denoted by $\Aut_{gr}(G)$.
\end{Definition}

\begin{Lemma}
\label{aut_gradedcover}
Let $G$ be a group, $\{X_{\alpha}\}$
a grading of $G$ and $\widetilde G$ the graded cover of $G$
with respect to $\{X_{\alpha}\}$. Then each graded automorphism of $G$
naturally lifts to a graded automorphism of $\widetilde G$, and the
obtained map $\Aut_{gr}(G)\to \Aut_{gr}(\widetilde G)$ is a monomorphism.
\end{Lemma}

\subsection{Steinberg groups over commutative rings}
In this subsection we sketch the definition  of Steinberg
groups  and show that the natural grading of these groups is
strong. Our description of Steinberg groups  follows
Steinberg's lecture notes on Chevalley groups~\cite{St} and
Carter's book~\cite{Ca}. 

We begin by recalling a few basic facts about simple complex
Lie algebras (see \cite{Hu} for more details).
Let $\Phi$ be a  reduced irreducible classical root system of rank $l$
and
$\mathcal{L}$ a simple complex Lie algebra corresponding to
$\Phi$. Let $\mathcal{H}$ be a Cartan subalgebra of $\mathcal{L}$.
Then $\mathcal{H}$ is abelian, $\dim \mathcal{H}=l$,
and we have the following
decomposition of $\mathcal{L}$:
$$
\mathcal{L}=\mathcal{H}\oplus (\oplus_{\alpha\in \Phi}
\mathcal{L}_\alpha),
$$
where $\mathcal{L}_\alpha=\{l\in \mathcal{L}:\ [h,l]_L=\alpha(h)l
\textrm{\ for all \ } h\in \mathcal{H}\}$ (as usual we consider
$\Phi$ as a subset of $\mathcal{H}^*$).  Moreover, each
$\mathcal{L}_\alpha$ is one dimensional. 

For any  $\alpha,\beta\in\Phi$ we put $\la \beta,\alpha \ra=\frac{2(\beta,\alpha)}{(\alpha,\alpha)}$,
where $(\cdot,\cdot)$ is an admissible scalar product on $\Phi$ (since $(\cdot,\cdot)$ is unique
up to scalar multiples, the pairing $\la \cdot,\cdot \ra$ is well defined).
For any $\alpha\in \Phi$ let $h_{\alpha}\in \mathcal H$ be the unique element
such that
\begin{equation}
\beta(h_{\alpha})=\la\beta,\alpha\ra.
\label{Chevrel1}
\end{equation}
Note that if $\{\alpha_1,\ldots,\alpha_l\}$ is a base of $\Phi$, then
$h_{\alpha_1},\ldots,h_{\alpha_l}$ is a basis of $\mathcal H$.

\begin{Proposition}
\label{Chevbasis}
There exist nonzero elements $x_{\alpha}\in \mathcal{L}_\alpha$ for $\alpha\in\Phi$
such that
\begin{eqnarray}
\label{Chevrel2}
 {[} x_\alpha, x_{-\alpha}]_L & = & h_\alpha, \\
\label{liebracket}
[x_\alpha, x_\beta ]_L &=&\left \{
\begin{array}{cl}
 \pm
(r+1)x_{\alpha+\beta} & \mbox{ if }\,\,\alpha+\beta\in \Phi\\
0& \mbox{ if }\,\,\alpha+\beta\not \in \Phi\end{array}\right .  ,\end{eqnarray}
where 
$r=\max \{s\in \Z:\ \beta-s\alpha\in \Phi\}$.
\end{Proposition}

Any basis $\{h_{\alpha_i},x_\alpha: 1\le i\le l, \alpha\in \Phi\}$ of
$\mathcal{L}$ with this property (for some choice of Cartan subalgebra) is called a \emph{Chevalley basis}.
It is unique up to sign changes and automorphisms of $\mathcal{L}$. We will also
need the following characterization of simple Lie algebras in terms of Chevalley bases:
\begin{Lemma}
\label{simple_recognition}
Let the root system $\Phi$ and pairing $\la\cdot,\cdot\ra$ be as above, and
let $\{\alpha_1,\ldots, \alpha_l\}$ be a base of $\Phi$.
Let $L$ be a finite-dimensional complex Lie algebra
with basis $\{x_{\alpha}\}_{\alpha\in \Phi}\sqcup\{h_{i}\}_{i=1}^l$
such that
\begin{itemize}
\item[(i)] $[x_{\alpha_i},x_{-\alpha_i}]=h_{i}$ for $1\leq i\leq l$;
\item[(ii)] $[h_{i},x_{\alpha}]=\la\alpha,\alpha_i\ra x_{\alpha}$ for $1\leq i\leq l$ and $\alpha\in \Phi$;
\item[(iii)] $[h_i,h_j]=0$ for $1\leq i,j\leq l$;
\item[(iv)] elements $\{x_{\alpha}\}$ satisfy \eqref{liebracket}.
\end{itemize}
Then $L$ is a simple Lie algebra with root system $\Phi$, that is, $L\cong \mathcal L$.
\end{Lemma}

From now on we fix a Chevalley basis $\mathcal{B}$ of $\mathcal L$. Denote by $\mathcal{L}_{\Z}$ the
subset of $\mathcal{L}$ consisting of all linear combinations of the elements
of $\mathcal{B}$ with integer coefficients. By definition of a
Chevalley basis, $\mathcal{L}_{\Z}$ is a Lie subring of $\mathcal L$.
For any commutative ring $R$ put $\mathcal{L}_R=R\otimes_{\Z}
\mathcal{L}_{\Z}$. The Lie bracket of $\mathcal{L}_{\Z}$
naturally extends to a Lie bracket of $\mathcal{L}_R$.

\begin{Proposition}
\label{chevalleyaut} 
Let $S=\Q[t,s]$ and $T=\Z[t,s]$ be
the polynomial rings in two
variables over $\Q$ and $\Z$, respectively. Then for every
$\alpha,\beta \in \Phi$, with $\alpha\neq\beta$, we have
\begin{itemize} \item[(i)]   the derivation $ad (tx_\alpha)=t\ad(x_\alpha)$ of
$\mathcal{L}_S$ is nilpotent; 
\item[(ii)] $x_\alpha(t)=\exp (t\ad
(x_\alpha))=\sum_{i=0}^\infty \frac{t^i}{i!} \ad(x_\alpha) ^i$ is
a well defined automorphism of $\mathcal{L}_S$ which preserves
$\mathcal{L}_T$,
 \item[(iii)] there exist $c_{ij}(\alpha,\beta)\in \Z$ such that the
following equality  holds:
\begin{equation}
\label{form3}
[x_\alpha(t), x_{\beta}(s)]=\prod_{i,j}x_{i\alpha+j\beta}(
c_{ij}(\alpha,\beta)t^{i}s^{j}),
\end{equation}
where the product on the right is taken over all roots
$i\alpha+j\beta\in \Phi$, with $i,j\in {\mathbb{N}}$, arranged in some
fixed order. Moreover, the constants $c_{ij}(\alpha,\beta)$
depend only on the set $\{(i,j)\in \dbZ\times \dbZ : i\alpha+j\beta\in \Phi\}$
and the chosen order.
\item[(iv)] If $\alpha+\beta\not \in\Phi$, then the product in \eqref{form3} is empty
(and thus $[x_\alpha(t), x_{\beta}(s)]=1$). If $\alpha+\beta\in\Phi$,
then $c_{11}(\alpha,\beta)=\pm (r+1)$ where $r$ is given by \eqref{liebracket}.
In particular, $c_{11}(\alpha,\beta)$ does not depend on the chosen order.
\end{itemize}
\end{Proposition}
\begin{Remark}
Up to sign, the constants $c_{ij}(\alpha,\beta)$ are independent
of the choice of the Chevalley basis (once the order in \eqref{form3} has been fixed).
\end{Remark}
\vskip .14cm

Let $A_\alpha(t)$ be the matrix representing $x_\alpha(t)$ with
respect to $\mathcal{B}$ (note that $\mathcal{L}_S$ is a free
$S$-module). By Proposition~\ref{chevalleyaut}(ii), the entries
of $A_\alpha(t)$ are in $T$. 

Now let $R$ be a commutative ring $R$ and let
$r\in R$. Let $x_\alpha(r)$ be the automorphism of
$\mathcal{L}_R$ represented by $A_\alpha(r)$ (the matrix obtained
from $A_\alpha(t)$ by replacing $t$ by $r$) with respect to
$\mathcal{B}$. We denote by $X_\alpha=X_\alpha(R)$ the set
$\{x_\alpha(r):\ r\in R\}$. This is a subgroup of
$\Aut(\mathcal{L}_R)$ isomorphic to $(R,+)$. By Proposition~\ref{chevalleyaut}(iii), 
$\{X_\alpha\}_{\alpha\in \Phi}$ is a $\Phi$-grading.

\begin{Definition}
Let $\Phi$ be a reduced irreducible classical root system
and $R$ a commutative ring.
\begin{itemize}
\item[(a)] The subgroup of $\Aut(\mathcal{L}_R)$ generated by
$\cup_\alpha X_\alpha$ is called the
\emph{adjoint elementary Chevalley group over $R$ corresponding to $\Phi$} and will be denoted by $\mathbb E_{\Phi}^{\rm ad}(R)$.

\item[(b)] The \emph{Steinberg group} $\St_\Phi(R)$ is the graded cover of \index{Steinberg group}
$\mathbb E_{\Phi}^{\rm ad}(R)$ with respect to the grading $\{X_\alpha\}_{\alpha\in\Phi}$.
\end{itemize}
\end{Definition}
\begin{Remark} Elementary Chevalley groups of simply-connected type (and other non-adjoint types)
can be constructed in a similar way, except that the adjoint representation of
the Lie algebra $\mathcal L_{R}$ should be replaced by a different representation.
The graded cover for each such group is isomorphic to $\St_\Phi(R)$.
\end{Remark}

The Steinberg group $\St_\Phi(R)$ can also be defined
as the group generated by the elements 
$\{x_\alpha(r):\ \alpha \in \Phi, r\in R\}$ 
subject to the following relations for every 
$\alpha\ne -\beta \in \Phi$ and $t,u\in R$:
\begin{eqnarray*}
x_\alpha(t)x_\alpha(u)
&= &x_\alpha(t+u) \\
 \label{conmutator}
[x_\alpha(t), x_{\beta}(u)]&=&\prod_{i,j\in \mathbb{N}, i\alpha+j\beta\in\Phi}
x_{i\alpha+j\beta}(c_{ij}(\alpha,\beta)t^{i}u^{j}),
\end{eqnarray*}
where the constants $c_{ij}(\alpha,\beta)$ come from \eqref{form3}.

Note that while the second definition of Steinberg groups has an advantage of being explicit,
the first one shows that the isomorphism class of $\St_\Phi(R)$ does not depend on the choice of
Chevalley basis.

\begin{Remark}  Note that according to our definition the Steinberg group
$\St_{A_1}(R)$ is the free product of two copies of $(R,+)$. This
definition does not coincide with the usual definition in the
literature, but it is convenient for the purposes of this paper.
\end{Remark}

\begin{Remark}
If we do not assume that $R$ is commutative, then
$\mathcal{L}_R$ does not have a natural structure of a Lie
algebra over $R$, and so the above construction of $\St_\Phi(R)$ is
not valid. However in the case $\Phi=A_n$, we can still
define the Steinberg group as the graded cover of $\EL_{n+1}(R)$.
When $\Phi\ne A_n$ we are not aware of any natural way to define the Steinberg
group $\St_\Phi(R)$ when $R$ is noncommutative.

In the special case $\Phi=A_2$, the Steinberg group $\St_{A_2}(R)$
can even be defined for any \emph{alternative} ring $R$ (see~\cite[Appendix]{Fa}).
\end{Remark}
\vskip .16cm

The following proposition will be used
frequently in the rest of the paper. It shows that some ``natural"
subgroups of Steinberg groups are quotients of Steinberg groups.

\begin{Definition}
Let $\Phi$ be a root system. A subset $\Psi$ of $\Phi$ is called a
{\bf weak subsystem} if $\Phi\cap ( \sum_{\gamma\in
\Psi}\Z\gamma)=\Psi$.
\end{Definition}

\begin{Proposition} 
\label{subgroups} 
Let $\Phi$ be a reduced irreducible classical root
system and $\Psi$ an irreducible weak subsystem. Then $\Psi$ is
classical and the subgroup $H$ of $\St_\Phi(R)$ generated by
$\{X_\gamma:\ \gamma \in \Psi\}$ is a quotient of $\St_{\Psi}(R)$.
\end{Proposition}
\begin{proof}
Note that $\Psi$ is a root system. Moreover, $\Psi$ is classical,
since if $(\cdot,\cdot)$ is an admissible scalar product on $\Phi$, then
$(\cdot,\cdot)$ restricted to $\Psi$ is an admissible scalar product on
$\Psi$.

Let $\mathcal{L}$ be a simple complex Lie algebra
corresponding to $\Phi$, choose a Cartan subalgebra $\mathcal H$ of
$\mathcal{L}$, and define $\{h_{\alpha}\}_{\alpha\in\Phi}$ by \eqref{Chevrel1}. 
Let $\{h_{\alpha_i},x_\alpha: 1\le i\le l,
\alpha\in \Phi\}$ a Chevalley basis of $\mathcal{L}$ (relative to $\mathcal H$).
We claim that the Lie subalgebra $\mathcal{L}^\Psi$ generated by $\{
x_\alpha: \ \alpha\in \Psi\}$ is a simple complex Lie algebra
corresponding to the root system $\Psi$, and moreover,
if $\{\beta_1,\ldots,\beta_m\}$ is a base of $\Psi$, then
$\{x_\alpha: \alpha\in \Psi\}\sqcup\{h_{\beta_i}\}$ is a Chevalley basis of $\mathcal{L}^\Psi$.

By Lemma~\ref{simple_recognition} (to be applied to $\Psi$) and definition of Chevalley basis,
to prove both statements it suffices to check that
\begin{itemize}
\item[(i)] the pairing $\la\cdot,\cdot\ra_{\Psi}$ on $\Psi$ is obtained from
the pairing $\la\cdot,\cdot\ra_{\Phi}$ on $\Phi$ by restriction;
\item[(ii)] If $\alpha,\beta\in \Psi$, then the value of $r$ in
relation~\eqref{liebracket} does not change if $\Phi$ is replaced by $\Psi$.
\end{itemize}

Assertion (i) is clear since the scalar product on $\Psi$ is obtained from the scalar product
on $\Phi$ by restriction. Assertion (ii) holds
since $\Psi$ is a weak subsystem and the value of $r$ in \eqref{liebracket}
depends only on the structure of the $\Z$-lattice
generated by $\alpha$ and $\beta$.

Thus, in view of Proposition~\ref{chevalleyaut}(iii),
the values of the coefficients $c_{ij}(\alpha,\beta)$, with $\alpha,\beta\in \Psi$,
do not depend on whether we consider $\alpha,\beta$  as roots of $\Phi$ or $\Psi$.
It follows that the defining relations of $\St_{\Psi}(R)$ hold in $H$, so
$H$ is a quotient of $\St_{\Psi}(R)$.
\end{proof}

\begin{Remark} 
In most cases a weak subsystem of a classical
root system is also a subsystem, but not always. For instance, the long
roots of $G_2$ form a weak subsystem of type $A_2$, but they
do not form a subsystem. Note that the short roots of $G_2$ do not
even form a weak subsystem.
\end{Remark}
\vskip .12cm
Next we explicitly describe some relations in the Steinberg groups
corresponding to  root systems of rank 2.

\begin{Proposition}
\label{constants} 
There exists a Chevalley basis such that
\begin{enumerate}
\item[(A)] 
if $\Phi=A_2=\{\pm \alpha,\pm \beta,\pm
(\alpha+\beta)$, then
$$ 
[x_\alpha(t), x_{\beta}(t)]
=x_{\alpha+\beta}(tu),\ \ [x_{-\alpha}(t), x_{\alpha+\beta}(u)]
=x_{\beta}(tu).
$$
\item [(B)] 
if $\Phi=B_2=\{\pm \alpha,\pm \beta,\pm
(\alpha+\beta),\pm (\alpha+2\beta)$, then
$$
[x_\alpha(t), x_{\beta}(u)]=x_{\alpha+\beta}(tu)
x_{\alpha+2\beta}(tu^2),\ \ [x_{-\alpha}(t), x_{\alpha
+\beta}(u)]=x_{\beta}(tu) x_{\alpha+2\beta}(-tu^2),
$$
$$
[x_{\alpha+\beta}(t),x_\beta(u)]=x_{\alpha+2\beta}(2tu).
$$
\item [(G)] 
if $\Phi=G_2=\{\pm \alpha,\pm
\beta,\pm(\alpha+\beta),\pm (\alpha+2\beta),\pm
(\alpha+3\beta),\pm (2\alpha+3\beta)\}$, then
$$
[x_\alpha(t),x_\beta(u)]=x_{\alpha+\beta}(tu)x_{\alpha+2\beta}(tu^2)x_{\alpha+3\beta}(tu^3)x_{2\alpha+3\beta}(t^2u^3),
$$
$$
[x_\alpha(t),x_{\alpha+3\beta}(u)]=x_{2\alpha+3\beta}(tu),
$$
$$
[x_{\alpha+\beta}(t),x_\beta(u)]=x_{\alpha+2\beta}(2tu)x_{\alpha+3\beta}(3tu^2)x_{2\alpha+3\beta}(3t^2u).
$$
\end{enumerate}
\end{Proposition}
\begin{proof} In each of those cases $\{\alpha,\beta\}$ is a base of $\Phi$. The above relations which
involve only positive roots with respect to this base hold by~\cite[Prop 33.3,33.4,33.5]{Hu2},
for a suitable choice of Chevalley basis. Moreover, the Chevalley basis for type $B_2$ constructed
in~\cite[Prop 33.4]{Hu2} satisfies the additional conditions
\begin{align*}
&w_{\alpha}x_{\alpha}(t)w_{\alpha}^{-1}=x_{-\alpha}(-t),&
&w_{\alpha}x_{\beta}(t)w_{\alpha}^{-1}=x_{\alpha+\beta}(t),&\\
&w_{\alpha}x_{\alpha+\beta}(t)w_{\alpha}^{-1}=x_{\beta}(-t),&
&w_{\alpha}x_{\alpha+2\beta}(t)w_{\alpha}^{-1}=x_{\alpha+2\beta}(t),&
\end{align*}
where $w_{\alpha}=x_{\alpha}(1)x_{-\alpha}(-1)x_{\alpha}(1)$ is the Weyl group element corresponding to $\alpha$.
Conjugating the relation $[x_\alpha(t), x_{\beta}(u)]=x_{\alpha+\beta}(tu)
x_{\alpha+2\beta}(tu^2)$ by $w_{\alpha}$ and using the above conditions, we conclude that
$[x_{-\alpha}(t), x_{\alpha +\beta}(u)]=x_{\beta}(tu) x_{\alpha+2\beta}(-tu^2)$. 
The desired relation for type $A_2$
involving $-\alpha$ can be obtained similarly.

We warn the reader that notations in~\cite{Hu2} are different from ours, with the roles of $\alpha$ and $\beta$ switched
for types $B_2$ and $G_2$.
\end{proof}
\begin{Proposition}
\label{strongclassical}
Let $\Phi$ be a reduced irreducible classical root system of rank
$l\ge 2$ and $R$ a commutative ring, let $G=\St_{\Phi}(R)$
and $\{X_{\alpha} : \alpha\in\Phi\}$  the root subgroups of $G$.
Then $\{X_{\alpha} : \alpha\in\Phi\}$ is a $k$-strong grading of
$G$ for any $2\le k\le l$, and in particular, it is strong.
\end{Proposition}
\begin{proof}
 Let $\Psi$ be an irreducible subsystem of $\Phi$ of rank $\geq 2$.
Then by Proposition \ref{subgroups}, $\Psi$ is classical, and the
subgroup $H$  generated by $\{X_{\alpha} : \alpha\in\Psi\}$  is a
quotient of the Steinberg group $\St_{\Psi}(R)$.
Thus, the grading
$\{X_{\alpha} : \alpha\in\Psi\}$ of $H$ is strong if the natural
$\Psi$-grading of $\St_{\Psi}(R)$ is strong. Since $\Psi$
is regular, using Corollary~\ref{identityreduction} with $k=2$,
we deduce that it is enough to prove
Proposition~\ref{strongclassical} when $l=k=2$.
In this case the result easily follows from
Proposition~\ref{constants}. We illustrate this for $\Phi=G_2$.

Consider a functional $f$ and let  $\{\alpha, \beta\}$ be a base
on which $f$ takes positive values, with $\alpha$ a long root.
Then the core $C_f$ is equal to 
$\{ \alpha+\beta , \alpha+2\beta , \alpha+3\beta , 2\alpha+3\beta \}$. 
Since each of the maps $(t,u)\mapsto tu$, $(t,u)\mapsto tu^2$,
$(t,u)\mapsto tu^3$ from $R\times R$ to $R$ is clearly
surjective, the first two relations in Proposition~\ref{constants}(G)
imply that
$$
\begin{array}{lll} X_{\alpha+\beta}&\subseteq&
[X_\alpha,X_\beta]X_{\alpha+2\beta}X_{\alpha+3\beta}X_{2\alpha+3\beta},\\ &&\\
X_{\alpha+2\beta}&\subseteq&
[X_\alpha,X_\beta]X_{\alpha+ \beta}X_{\alpha+3\beta}X_{2\alpha+3\beta},\\ &&\\
X_{\alpha+3\beta}&\subseteq&
[X_\alpha,X_\beta]X_{\alpha+ \beta}X_{\alpha+2\beta}X_{2\alpha+3\beta},\\ &&\\
X_{2\alpha+3\beta}&=& [X_\alpha,X_{\alpha+3\beta}].\end{array}
$$
Hence the grading $\{X_{\alpha}\}_{\alpha\in G_2}$ is strong.
\end{proof}

\subsection{Standard sets of generators of Steinberg groups.}
Let $R$ be a commutative ring generated by $T=\{t_0=1,t_1,\ldots,
t_d\}$. We denote by $T^*$ the set
$$
\left\{t_{i_1}\cdots t_{i_k}: 0\le i_1<\ldots <i_k\le d \right\}.
$$
In the following proposition we describe a
set of generators of $\St_\Phi(R)$ that we will call
\emph{standard}.
\begin{Proposition}
\label{standgenset}
Let $\Phi$ be a reduced irreducible classical
root system of rank at least 2 and $R$ a commutative ring
generated by $T=\{t_0=1,t_1,\ldots, t_d\}$. Let
$\Sigma=\Sigma_\Phi(T)$ be the following set:
\begin{enumerate}
\item  if $\Phi=A_n$, $B_n (n\ge 3)$, $D_n$, $E_6$, $E_7$, $E_8$, $F_4$,
$$
\Sigma=\{x_\alpha(t):\ \alpha\in \Phi,\ t\in  T\},
$$

\item if $\Phi=B_2$, $C_n$, 
$$
\Sigma=\left \{\begin{array}{cc}
x_\alpha(t),\ t\in
T&\alpha\in \Phi \textrm{\ is a short root}\\
x_\alpha(t), \ t\in T^*&\ \alpha\in \Phi \textrm{\ is a long
root}\end{array}\right \},
$$

\item if $\Phi=G_2$, 
$$
\Sigma= \left \{\begin{array}{cc}
x_\alpha(t),\ t\in
T&\alpha\in \Phi \textrm{\ is a long root}\\
x_\alpha(t), \ t\in  T^*&\ \alpha\in \Phi \textrm{\ is a short
root}\end{array}\right \}.
$$
\end{enumerate}
Then $\Sigma$ generates $\St_\Phi(R)$.
\end{Proposition}
\begin{proof} 
First we consider the case $\Phi=A_n$, $B_n (n\ge 3)$, $D_n$, $E_6$, $E_7$, $E_8$, $F_4$. 
We prove  by induction on $k$ that for any $\gamma\in \Phi$ and any monomial $m$ in variables from $T$ of degree $k$,
the element $x_\gamma(m)$ lies in $\la \Sigma\ra$,
the subgroup generated by $\Sigma$. This statement clearly implies the proposition.

The base of induction is clear. Assume that the statement is true
for monomials of degree $\leq k$. Let $m$ be a monomial of
degree $k+1$.

If $\Phi=A_n$, $D_n$, $E_6$, $E_7$, $E_8$, $F_4$, then we can find a
subsystem $\Psi$ of $\Phi$ isomorphic to $A_2$ which contains
$\gamma$. We write $\gamma=\gamma_1+\gamma_2$, where
$\gamma_1,\gamma_2\in \Psi$, and $m=m_1m_2$, where $m_1$, $m_2$
are monomials of degree $\leq k$. Then
$x_\gamma(m)=[x_{\gamma_1}(m_1),x_{\gamma_2}(m_2)]^{\pm 1}$ and we
can apply the inductive hypothesis.

If $\Phi=B_n, n\ge 3$, then any long root lies in an irreducible
subsystem isomorphic to $A_2$, whence the statement holds when
$\gamma$ is a long root. Assume $\gamma$ is a short root. Then
there are a long root $\alpha$ and a short root $\beta$ such that
$\gamma=\alpha+\beta$. Note that $\alpha$ and $\beta$ generate a
subsystem of type $B_2$. Without loss of generality we may assume
that the relations of Proposition~\ref{constants} hold.  Then  we
obtain that
$$
x_\gamma(m)=x_{\alpha+\beta}(m)=[x_\alpha(m),x_\beta(1)]x_{\alpha+2\beta}(-m).
$$
Since the roots $\alpha$ and $\alpha+2\beta$ are long,
by induction $x_{\gamma}(m)$ lies $\la \Sigma\ra$.

In the case $\Phi=B_2$ the proposition is an easy consequence of the
following lemma (we do not need the second part
of this lemma now; it will be used later).

\begin{Lemma}
\label{generatorsB2} 
Let $\{\alpha,\beta\}$ be a base of $B_2$ with $\alpha$ a long root.
Consider the semidirect product $St_{A_1}(R) \ltimes N$, where 
$N=\la X_{\beta},X_{\alpha+\beta},X_{\alpha+2\beta}\ra\subset St_{B_2}(R)$
and the action of $St_{A_1}(R)$ on $N$ comes from the conjugation
action of $\la X_{\alpha},X_{-\alpha}\ra\subset St_{B_2}(R)$ on $N$.
Let
\begin{align*}&S_1=\{x_\alpha(t),x_{-\alpha}(t):\
t\in  T^*\cup T^2\};&& S_2=\{
x_\beta(t),x_{\alpha+\beta}(t):\ t\in  T\} \mbox{ and }&\\
& S_3=\{ x_{\alpha+2\beta}(t) :\ t\in  T^* \}.&&&
\end{align*}
Let $G$ be the subgroup of $St_{A_1}(R) \ltimes N$ generated by the set
$S=S_1\cup S_2\cup S_3$.
Then the following
hold:
\begin{enumerate}
\item $G$ contains $N$;
\item $X_{\alpha+2\beta}/([N,G]\cap X_{\alpha+2\beta})$ is of
exponent 2 and generated by $S_3$.
\end{enumerate}
\end{Lemma}
\begin{proof}
Without loss of generality we may assume that  the relations
from Proposition~\ref{constants} hold in $\St_{B_2}(R)$.

We prove   by induction on $k$ that for any 
$\gamma\in \{\alpha+\beta,\alpha+2\beta,\beta\}$ and any monomial $m$ in $T$ 
of degree $k$, the element $x_\gamma(m)$ lies
in $\la S \ra$. This clearly implies the first
statement.

The base of induction is clear. Assume that the statement holds for
all monomials of degree $\leq k$. Let $m$ be a monomial of
degree $k+1$.

\emph{Case 1: $\gamma=\alpha+2\beta$}. If $m\in T^*$,  then
$x_{\alpha+2\beta}(m)\in S_3$. If $m\not \in T^*$, we can write $m=m_1m_2^2$ with
$m_1\in T^*$ and $m_2\neq 1$, and we obtain that
$$
x_{\alpha+2\beta}(m_1m_2^2)=[x_\alpha(m_1),x_\beta(m_2)]x_{\alpha+\beta}(-m_1m_2).
$$
Thus, by induction, $x_{\alpha+2\beta}(m)\in G$.

\emph{Case 2: $\gamma=\alpha+\beta$}. If $m\in T^*$, then
$$
x_{\alpha+\beta}(m)= [x_\alpha(m), x_\beta(1)]x_{\alpha+2\beta}(-m)
$$ 
and we are done. If $m\not\in T^*$, we write
$m=t^2 m_1 $, where $t\in T \setminus \{1\}$. Then
we have
\begin{equation}
\label{eq:exercise}
x_{\alpha+\beta}(m)=[x_\alpha(t^2),
x_\beta(m_1)][x_{\alpha}(1),x_{\beta}(-tm_1)] x_{\alpha+\beta}(tm_1)\in G.
\end{equation}

\emph{Case 3: $\gamma=\beta$}. This case is analogous to Case~2, but this
time we use the relation
$x_{\beta}(vu)=[x_{-\alpha}(v), x_{\alpha+\beta}(u)] x_{\alpha+2\beta}(vu^2)$ for $u,v\in R$.
\vskip .14cm
The proof of the second part is an easy exercise based on
Proposition~\ref{constants} and equality \eqref{eq:exercise}.
\end{proof}

We now go back to the proof of Proposition~\ref{standgenset}.
In the case $\Phi=B_2$ the result follows from
Lemma~\ref{generatorsB2}(1) since for any
$t\in T$ and long root $\gamma\in B_2$, the element
$x_\gamma(t^2)$ can be expressed as a product of elements from $\Sigma$.
For instance, in the case $\gamma=\alpha+2\beta$ we have
\begin{equation}
\label{t2}
x_{\alpha+2\beta}(t^2)=[x_\alpha(1),x_\beta(t)]x_{\alpha+\beta}(-t).
\end{equation}

If $\Phi=C_n$, then any root lies in a subsystem of type $B_2$, and
so the result follows from the previous case.

Finally, consider the case $\Phi=G_2$. The long roots of $G_2$ form a weak subsystem of
type $A_2$, so Proposition~\ref{standgenset} for type $A_2$ (which we already established)
and Proposition \ref{subgroups} imply that long root subgroups of $\St_{G_2}(R)$
lie in $\la \Sigma\ra$. It remains to show that $x_{\gamma}(m)\in \la \Sigma\ra$
for any short root $\gamma$ and any monic monomial $m$ in $T$, which we will do by induction on the
degree of $m$. By symmetry, it suffices to prove that $x_{\alpha+2\beta}(m)\in \la\Sigma\ra$.

If $m$ is square-free, then $m\in T^*$, so $x_{\gamma}(m)\in \Sigma$
by definition; otherwise $m=tu^2$ for some monomials $t$ and $u$, with $u\neq 1$. By
Proposition~\ref{constants}(C) we have
$$
x_{\alpha+2\beta}(tu^2)=x_{2\alpha+3\beta}(-tu^3)x_{\alpha+3\beta}(-t^2u^3)[x_{\alpha+\beta}(u),x_{-\alpha}(t)]x_{\beta}(tu).
$$
All factors on the right-hand side lie in $\la\Sigma\ra$, namely
$x_{\alpha+\beta}(u),x_{\beta}(tu)\in \la\Sigma\ra$ by the induction hypothesis and
$x_{-\alpha}(t),x_{2\alpha+3\beta}(-tu^3),x_{\alpha+3\beta}(-t^2u^3)\in\la\Sigma\ra$ since the roots $-\alpha$,
$\alpha+3\beta$ and $2\alpha+3\beta$ are long. Thus, $x_{\alpha+2\beta}(m)\in \la\Sigma\ra$, as desired.
\end{proof}

\subsection{Property $(T)$ for Steinberg groups}
In this subsection we establish property $(T)$ for Steinberg
groups (of rank $\geq 2$) over finitely generated rings and
obtain asymptotic lower bounds for the Kazhdan constants.
It will be convenient to use the following notation.
\vskip .15cm

If $\kappa$ is some quantity depending on $\dbN$-valued
parameters $n_1,\ldots, n_r$ (and possibly some other parameters)
and $f:\dbN^r\to\R_{>0}$ is a function, we will write
$$
\kappa \succcurlyeq f(n_1,\ldots, n_r)
$$ 
if there exists
an absolute constant $C>0$ such that $\kappa \geq C f(n_1,\ldots, n_r)$
for all $n_1,\ldots, n_r\in\dbN$.

\vskip .15cm

Let $\Phi$ be a reduced irreducible classical root system of rank $\geq 2$ and
$R$ a finitely generated ring (which is commutative if $\Phi$ is not of type $A$).
By Proposition~\ref{strongclassical} the standard grading $\{X_{\alpha}\}$
of $\St_{\Phi}(R)$ is strong, so to prove property $(T)$ it suffices
to check relative property $(T)$ for each of the pairs $(\St_{\Phi}(R),X_{\alpha})$.
However, in order to obtain a good bound for the Kazhdan
constant of $\St_{\Phi}(R)$ with respect to a finite generating set
of the form $\Sigma_{\Phi}(T)$ (as defined in Proposition~\ref{standgenset}),
we need to proceed slightly differently.

We shall use a good reduction of $\Phi$ to a root system of rank $2$
(of type $A_2$, $B_2$, $BC_2$ or $G_2$) described in~\S~\ref{goodreductions}.
Proposition~\ref{strongclassical} implies that the coarsened grading
$\{Y_{\beta}\}$ of $\St_{\Phi}(R)$ is also strong, so
Theorem~\ref{th:relativeT_wrt_ root subgroup} can be applied to this
grading, this time yielding a much better bound for the Kazhdan constant.

To complete the proof of property $(T)$ for $\St_{\Phi}(R)$ we still
need to establish relative property $(T)$, this time for the pairs
$(\St_{\Phi}(R),Y_{\beta})$. Qualitatively, this is not any harder
than proving relative $(T)$ for $(\St_{\Phi}(R),X_{\alpha})$; however,
we also need to explicitly estimate the corresponding Kazhdan ratios
(which will affect the eventual bound for the Kazhdan constant
of $\St_{\Phi}(R)$ with respect to a finite generating set).
\vskip .12cm

\noindent {\bf Terminology:} For brevity, in the sequel instead of saying \emph{relative property $(T)$ for the pair $(G,H)$}
we will often say \emph{relative property $(T)$ for $H$} if $G$ is clear from the context.
\index{relative property $(T)$}
\vskip .12cm

Our main tool for proving relative property $(T)$ is the following result of Kassabov~\cite{Ka}
which generalizes Theorem~\ref{KasShalom2} and will be proved in Appendix~A in a slightly
extended form (see Theorem~\ref{thm:relT_St_p}).

\begin{Theorem}[Kassabov]
\label{relativeAn} 
Let $n\ge 2$ and $R$ a ring generated by
$T=\{1=t_0,t_1,\ldots,t_d\}$. Let $\{\alpha_1,\cdots,\alpha_n\}$
be a system of simple roots of $A_n$. Consider the semidirect
product $G=\St_{A_{n-1}}(R)\ltimes N$, where $N\cong R^n$ is the
subgroup of $\St_{A_n}(R)$ generated by
$X_{\alpha_1},X_{\alpha_1+\alpha_2},\ldots, X_{\alpha_1+\cdots+\alpha_n}$ 
and the action of $\St_{A_{n-1}}(R)$ on
$N$ comes from the action of the subgroup 
$\langle X_\gamma:\ \gamma\in \R\alpha_2+\ldots +\R\alpha_n\rangle$ of $\St_{A_n}(R)$
on $N$ (we will refer to this action as the standard action of
$\St_{A_{n-1}}(R)$ on $R^n$). Let 
$S=\Sigma_{A_{n-1}}(T)\cup (\Sigma_{A_n}(T)\cap N)$ and let $G'=\la S,N\ra$. Then
$$
\kappa_r (G', N; S)\succcurlyeq\frac{1}{ \sqrt { d+n} }.
$$
\end{Theorem}
\begin{Remark} Note that if $n\geq 3$ we  have $G'=G$.
\end{Remark}
\vskip .12cm

Combining Theorems~\ref{relativeAn} and~\ref{relKazhdan}, we obtain
the following result which immediately implies relative property $(T)$
for the root subgroups of $\St_{B_2}(R)$.

\begin{Corollary}
\label{relativeB2}
Let $\{\alpha,\beta\}$ be a base of $B_2$, with $\alpha$ a long root,
and define the semidirect product $\St_{A_1}(R)\ltimes N$ and the set
$S$ as in Lemma~\ref{generatorsB2}. Then 
$$\kappa (\St_{A_1}(R)\ltimes N,N; S)\succcurlyeq\frac{1}{2^{d/2}},
\quad \mbox{ and therefore }\quad
\kappa_r(\St_{B_2}(R),\cup_{\gamma\in B_2} X_{\gamma};\Sigma)\succcurlyeq\frac{1}{2^{d/2}},$$
where $\Sigma$ is the standard generating set of $\St_{B_2}(R)$.
\end{Corollary}
\begin{proof} 
Let 
$A=\{x_\alpha(t),x_{-\alpha}(t) :\ t \in  T^*\cup T^2\}$, 
$B= \{x_\alpha(t ),x_{-\alpha}(t ), x_\beta(t),x_{\alpha+\beta}(t ):\ t\in  T\}$ and let 
$C=\{ x_{\alpha+2\beta}(t) :\ t \in T^*,\}$. 
so that $S=A\cup B\cup C$. Let $G$ be the subgroup of $ \St_{A_1}(R)\ltimes N$
generated by $S$.

By Lemma~\ref{generatorsB2}(1),
$N$ is a subgroup of $G$. Let $Z=X_{\alpha+2\beta}$.
The relations of $\St_{B_2}(R)$ described
in Propositions~\ref{constants} imply that $(\St_{A_1}(R)\ltimes N)/Z\cong \St_{A_1}(R)\ltimes (N/Z)$
is isomorphic to $\St_{A_1}(R)\ltimes R^2$ (with the standard action of $\St_{A_1}(R)$ on $R^2$),
and the image of $G/Z$ under this isomorphism contains the subgroup denoted by $G'$
in Theorem~\ref{relativeAn}. Hence $\kappa(G/Z,N/Z;B)\succcurlyeq \frac 1{\sqrt d}$.

Now let $H=Z\cap [N,G]$. By Lemma~\ref{generatorsB2}(2), $Z/H$ is an
elementary abelian 2-group generated by $C$, whence 
$\kappa (G/H,Z/H;C)\succcurlyeq \frac{1}{\sqrt {|C|}}\succcurlyeq \frac{1}{2^{d/2}}$. 
Since $Z\subseteq Z(G)\cap N$ and
$AN$ generates $G/N$, by Theorem~\ref{relKazhdan} we have
$\kappa (G ,N;S) \succcurlyeq\frac {1}{ 2^{d/2} }$. Since $G$ is a subgroup of
$\St_{A_1}(R)\ltimes N$, we deduce that $\kappa (\St_{A_1}(R)\ltimes N,N; S)\succcurlyeq\frac{1}{2^{d/2}}$.
\vskip .12cm

To prove the second assertion, note that for any root $\gamma\in B_2$ there is a homomorphism
$\phi:\St_{A_1}(R)\ltimes N\to \St_{B_2}(R)$ such that $\phi(N)\supset X_{\gamma}$
and $\phi(S)\subseteq \Sigma\cup \Sigma'$, where $\Sigma'=\{X_{\delta}(t^2): t\in T, \delta\mbox { is a long root}\}.$
Thus, $\kappa_r(\St_{B_2}(R), X_{\gamma};\Sigma\cup \Sigma')\succcurlyeq\frac{1}{2^{d/2}}$. Finally,
by \eqref{t2}, the same asymptotic inequality holds with $\Sigma\cup \Sigma'$ replaced by $\Sigma$.
\end{proof}

Before turning to the case-by-case verification of relative property $(T)$ (which is the main
part of the proof of Theorem~\ref{main} below), we briefly summarize how
this will be done for different root systems.  Let $\Phi$ be a reduced irreducible
classical root system with $rk(\Phi)\geq 2$.
\begin{itemize}
\item[(1)] 
If $\Phi$ is simply-laced, relative property $(T)$ for the root subgroups of $\St_{\Phi}(R)$
will follow almost immediately from Theorem~\ref{relativeAn} (with the aid of Proposition~\ref{subgroups}).
\item[(2)] 
If $\Phi$ is non-simply-laced and $\Phi\neq B_2$, relative property $(T)$ for some of the root subgroups
(either short ones or long ones) will again follow from Theorem~\ref{relativeAn}. To prove relative property $(T)$
for the remaining root subgroups, we will show that each of them is contained in a bounded product of a finite set
and root subgroups for which relative $(T)$ has already been established.
\item[(3)] 
Finally, the most difficult case $\Phi=B_2$ has almost been established in Corollary~\ref{relativeB2}.
\end{itemize}
Explicit bounds for the Kazhdan constants and Kazhdan ratios will follow from
Observation~\ref{Kazhrat} and Lemma~\ref{BGP}; these results will be often used
without further mention.

We are now ready to prove the main result of this section.
\begin{Theorem}
\label{main} 
Let $\Phi$ be a reduced irreducible classical root
system of rank at least 2 and $R$ a ring (which is commutative if $\Phi$ is not of type $A$)
generated by a finite set $T=\{1=t_0,t_1,\ldots,t_d\}$. Let
$\Sigma$ be the corresponding standard generating set of $\St_{\Phi}(R)$
(as defined in Proposition~\ref{standgenset}). Then
$\St_{\Phi}(R)$ has property $(T)$ and
$$
\kappa(\St_\Phi(R),\Sigma)\succcurlyeq \mathcal K(\Phi,d)
$$
where
$$
\mathcal K(\Phi,d)=\left\{
\begin{array}{cl}
\frac{1}{\sqrt{n+d}} & \mbox{ if  }\,\,\Phi=A_n, B_n (n\ge 3), D_n\\
\frac{1}{\sqrt{d}} & \mbox{ if }\,\,\Phi=E_6,E_7, E_8, F_4\\
\frac{1}{2^{d/2}} & \mbox{ if }\,\,\Phi=B_2, G_2 \\
\frac{1}{\sqrt{n +2^{d}   }} & \mbox{ if }\,\,\Phi=C_n
\end{array}
\right.
$$
\end{Theorem}
\begin{Remark} 
As we mentioned earlier, the Steinberg group $\St_{A_2}(R)$
can be defined for any alternative ring $R$ with $1$. Recently Zhang~\cite{Zh}
proved that all such groups, with $R$ finitely generated, have property $(T)$ as well.
\end{Remark}
\begin{proof}
Let $G=\St_{\Phi}(R)$.
We will show using case-by-case analysis that there exists a root system $\Phi'$
of type $A_2,B_2,BC_2$ or $G_2$ and a strong $\Phi'$-grading $\{Y_{\beta}\}$ of $G$ such that
$$
\kappa_r(G,\cup Y_{\beta}; \Sigma)\succcurlyeq \mathcal K(\Phi,d).
$$
This will imply the assertion of the theorem since
$\kappa(G,\Sigma)\geq \kappa(G,\cup Y_{\beta})\kappa_r(G,\cup Y_{\beta}; \Sigma)$
and $\kappa(G,\cup Y_{\beta})\geq C$ for some absolute constant $C>0$
by Theorem~\ref{th:relativeT_wrt_ root subgroup}.

\vskip .15cm

\emph{Case $\Phi=A_n$.} 
This case is covered by~\cite{EJ};
however, we include the argument for completeness.

We use the reduction of $A_n$ to $A_2$
given by the map 
$$
(x_1,\dots,x_{n+1}) \mapsto (x_1,x_2,\sum_{i\geq 3} x_i).
$$ 
The coarsened grading is strong by
Reduction~\ref{reductionAn}, Lemma~\ref{strongreduction} and
Proposition~\ref{strongclassical}. The new root subgroups are
$$
Y_{(-1,1,0)} = X_{e_2-e_1},
\quad
Y_{(1,0,-1)} =\prod_{i=3}^{n+1} X_{e_1-e_i},
\quad
Y_{(-1,0,1)} =\prod_{i=3}^{n+1} X_{e_i-e_1},
$$
$$
Y_{(1,-1,0)} = X_{e_1-e_2},
\quad
Y_{(0,1,-1)} =\prod_{i=3}^{n+1} X_{e_2-e_i},
\quad\mbox{and}\quad
Y_{(0,-1,1)} =\prod_{i=3}^{n+1} X_{e_i-e_2}.
$$

We shall prove that 
$\kappa_r (G,Y_{(1,0,-1)};\Sigma)\succcurlyeq \frac{1}{\sqrt{n+d}}$; 
the other cases are similar. The roots
$\{e_i-e_j : 2\leq i\neq j\leq n+1\}$ form a subsystem of type
$A_{n-1}$. Thus, if $H$ is the subgroup generated by 
$\{X_{e_i-e_j} : 2\leq i\neq j\leq n+1\}$ and
$$
E=\prod_{i=2}^{n+1} X_{e_1- e_i}\supset Y_{(1,0,-1)},
$$
then by Proposition~\ref{subgroups} there is a natural epimorphism
$\St_{A_{n-1}}(R)\ltimes R^n\to H\ltimes E$.
Hence using Theorem~\ref{relativeAn} we have
$$
\kappa_r(\St_{A_n}(R),  Y_{(1,0,-1)};\Sigma)\ge
\kappa_r(H\ltimes E,  E;\Sigma)\ge
\kappa_r(\St_{A_{n-1}}(R)\ltimes R^n,  R^n;\Sigma)
\succcurlyeq\frac{1}{\sqrt{n+d}}.
$$

\medskip
\emph{Case $\Phi=B_n$, $n\geq 3$.} 
We use the reduction of $B_n$ to
$B_2$ given by the map 
$$
(x_1,\dots,x_{n+1}) \mapsto (x_1,x_2).
$$
The coarsened grading is strong by Reduction~\ref{reductionBn},
Lemma~\ref{strongreduction} and Proposition~\ref{strongclassical}.
The new root subgroups are
$$
Y_{(\pm 1, 0)} =X_{\pm e_1} \prod_{i=3}^n X_{\pm e_1-e_i} \prod_{i=3}^n X_{\pm e_1+e_i},
\quad
Y_{(\pm 1, \pm 1)} = X_{\pm e_1 \pm e_2},
$$
$$
Y_{(0, \pm 1)} =X_{\pm e_2} \prod_{i=3}^n X_{\pm e_2-e_i} \prod X_{\pm e_2+e_i}.
$$

We shall prove that 
$\kappa_r(G,Y_{(1,0)};\Sigma)\succcurlyeq \frac{1}{\sqrt{n+d}}$; 
the other cases are similar. Since
$\{e_i-e_j : 1\leq i\neq j\leq n\}$ form a subsystem of type
$A_{n-1}$, arguing as in the previous case, we obtain that
$\kappa_r(\St_{B_n}(R),\prod_{i=3}^n X_{e_1- e_i};\Sigma)\succcurlyeq \frac{1}{\sqrt{n+d}}$. 
The same argument applies to
$\kappa_r(\St_{B_n}(R),\prod_{i=3}^n X_{e_1 + e_i};\Sigma)$. It
remains to show that 
$\kappa_r(\St_{B_n}(R),X_{e_1};\Sigma)\succcurlyeq\frac{1}{\sqrt{d}}$.

By the same argument as above, 
$\kappa_r(G,X_{\gamma};\Sigma)\succcurlyeq\frac{1}{\sqrt{d}}$
for any long root $\gamma$.
From Proposition~\ref{constants} it follows that 
$$
X_{e_1}\subseteq [x_{e_2}(1), X_{e_1-e_2}]X_{e_1+e_2}\subseteq
\Sigma X_{e_1-e_2} \Sigma X_{e_1-e_2}X_{e_1+e_2}.
$$
Therefore 
$\kappa_r(G,X_{e_1};\Sigma)\succcurlyeq\frac{1}{\sqrt{d}}$ by Lemma~\ref{BGP}.

\medskip
\emph{Case $\Phi=D_n$, $n\geq 3$.} 
We use the reduction of $D_n$ to $B_2$ given by the map 
$$
(x_1,\dots,x_{n}) \mapsto (x_1,x_2).
$$ 
The coarsened
grading is strong by Reduction~\ref{reductionDn}, Lemma~\ref{strongreduction} and
Proposition~\ref{strongclassical}. The
new root subgroups are
$$
Y_{(\pm 1, 0)} =\prod_{i=3}^n X_{\pm e_1-e_i} \prod_{i=3}^n X_{\pm e_1+e_i},
\quad
Y_{(\pm 1, \pm 1)} = X_{\pm e_1 \pm e_2},
$$
$$
Y_{(0,\pm 1)} = \prod_{i=3}^n X_{\pm e_2-e_i} \prod_{i=3}^n X_{\pm e_2+e_i}.
$$
The proof of relative
property~$(T)$ is the same as in the case of $B_n$.

\medskip
\emph{Case $\Phi=F_4$, $E_6$, $E_7$ or $E_8$.} 
In this case we
do not have to do any reduction to a root system of bounded rank,
since the rank is already bounded. The grading is strong by
Proposition~\ref{strongclassical}. However, if one wants to
obtain explicit estimates for the Kazhdan constants, one can use
Reductions~\ref{reductionF4} and~\ref{reductionE}
(again the coarsened gradings are strong from Lemma~\ref{strongreduction} 
and Proposition~\ref{strongclassical}).

In order to prove that 
$\kappa_r (\St_\Phi(R), X_\gamma;\Sigma)\succcurlyeq\frac 1 {\sqrt d}$
for any root $\gamma\in \Phi$, we simply observe that
$\gamma$ lies in a subsystem of type $A_2$, and
the same argument as in the case $A_2$ can be applied.

\medskip
\emph{Case $\Phi=B_2$.} 
The grading is strong by Proposition~\ref{strongclassical}, and the inequality 
$\kappa_r ( \St_{B_2}(R), \cup_\gamma X_\gamma;\Sigma)\succcurlyeq \frac 1 {2^{d/2}}$,
holds by Corollary~\ref{relativeB2}.

\medskip
\emph{Case $\Phi=C_n$, $n\geq 3$.} 
We use the reduction of $C_n$ to
$BC_2$ given by the map 
$$
(x_1,\dots,x_{n}) \mapsto (x_1,x_2).
$$
The coarsened grading is strong by Reduction~\ref{reductionCn},
Lemma~\ref{strongreduction} and Proposition~\ref{strongclassical}.
The new root subgroups are
$$
Y_{(\pm 1, 0)} =  \la \prod_{i=3}^{n} X_{\pm e_1-e_i},
\quad
\prod_{i=3}^{n} X_{\pm e_1+e_i}\ra,
\quad
Y_{(\pm 1, \pm 1)} = X_{\pm e_1 \pm e_2},
\quad
Y_{(\pm 2, 0)} =X_{\pm 2e_1 },
$$
$$
Y_{(0,\pm 1)} =
\la \prod_{i=3}^{n}  X_{\pm e_2-e_i}, \prod_{i=3}^{n} X_{\pm e_2+e_i}\ra
\quad \mbox{and } \quad
Y_{(0,\pm 2)} =X_{\pm 2e_2 }.
$$
Note that $Y_{(1, 0)}\neq AB$ for  $A=\prod_{i=3}^{n} X_{e_1-e_i}$ and $B=\prod_{i=3}^{n} X_{e_1+e_i}$;
however, it is easy to see that $Y_{(1,0)}=ABABAB$. Similar factorization exists for other short root subgroups.

Arguing as in the case $\Phi=B_n, n\ge 3$, we conclude that
$\kappa_r(\St_{\Phi}(R), Y_\gamma;\Sigma)\succcurlyeq\frac 1{\sqrt{n+d}}$ 
when $\gamma$ is a short or a long root.
If $\gamma$ is a double root, then $\gamma$ lies in a weak subsystem of type $B_2$.
Thus, from Corollary~\ref{relativeB2} and Proposition~\ref{subgroups} we obtain that 
$\kappa_r(\St_{\Phi}(R), Y_\gamma;\Sigma)\succcurlyeq\frac{1}{2^{d/2}}$. 
Hence $\kappa_r(\St_{\Phi}(R), \cup Y_\gamma;\Sigma)\succcurlyeq \frac 1{\sqrt {n+2^d  }}$.

\medskip
\emph{Case $\Phi=G_2$.} 
The  grading is strong by   Proposition~\ref{strongclassical}.
If $\gamma$ is a long root, then
$\kappa_r(\St_{G_2}(R), X_\gamma;\Sigma)\succcurlyeq \frac {1}{\sqrt d}$, 
because the long roots form a weak subsystem of
type $A_2$.

Now, we will show that if $\gamma$ is a short root, then
$\kappa_r(\St_{G_2}(R), X_\gamma;\Sigma)\succcurlyeq \frac{1}{2^{d/2}}$. 
Without loss of generality we may assume that the
relations from Proposition~\ref{constants} for
$\St_{G_2}(R)$ hold   and $\gamma=\alpha+2\beta$.

Calculating
$[x_\alpha(r),x_\beta(2)][x_\alpha(2r),x_\beta(1)]^{-1}$ 
we see that 
$$
X_\gamma(2R)=\{x_\gamma(2r):\ r\in R\}\subseteq
X_{\alpha} X_\alpha^{x_\beta(2)}X_\alpha X_\alpha^{
x_\beta(1)}X_{\alpha+3\beta}X_{2\alpha+3\beta},
$$
so $X_{\gamma}(2R)$
lies inside a bounded product of long root subgroups and $2$ fixed elements of $\Sigma$.

Similarly, calculating
$[x_\alpha(r),x_\beta(t)][x_\alpha(tr),x_\beta(1)]^{-1}$ for any
$1\ne t\in T$ we obtain that
$X_\gamma((t^2-t)R)=\{x_\gamma((t^2-t)r):\ r\in R\}$ lies inside a
bounded product of long root subgroups and $2$ fixed elements of
$\Sigma$.

Let $I=2R+\sum_{t\in T}(t^2-t)R$. Since $|T|=d+1$, using our previous observations
and Lemma~\ref{BGP}(b), we conclude that
$$
\kappa_r(\St_{G_2}(R), X_\gamma(I);\Sigma)\succcurlyeq \frac{1}{d\sqrt d }.
$$

The group $X_\gamma(R)/X_\gamma(I)$ is an elementary abelian
2-group generated by $S=\{x_\gamma(t):\ t\in T^*\}$. Since
$|S|=2^{d+1}$, we have
$$
\kappa (  X_\gamma(R)/ X_\gamma(I),S)\succcurlyeq \frac{1}{2^{d/2} }.
$$
By Lemma~\ref{auxg2}, these two inequalities imply  that $\kappa_r(\St_{G_2}(R),
X_\gamma(R); \Sigma)\succcurlyeq \frac 1{2^{d/2}}$.

\end{proof}


\newcommand \zetta {\omega}
\newcommand \epps {\eps}

\section{Twisted Steinberg groups}
\label{sec:twisted}

\subsection{Constructing twisted groups}

\newcommand \Tw {{Q}}

In this subsection we introduce a general method for constructing
new groups graded by root systems from old ones using the
machinery of twists. The method generalizes the construction of
twisted Chevalley groups~\cite{St}.

Let $\Phi$ be  a root system, $G$ a group and $\{X_{\alpha}\}_{\alpha\in\Phi}$
a $\Phi$-grading of $G$.  Let $\Tw\subset \Aut(G)$ be a
group of automorphisms of $G$ such that
\begin{itemize}
\item[(i)] Each element of $\Tw$ is a graded automorphism (as defined in \S~7.1),
so that there is an induced action of $\Tw$ on $\Phi$.
\item[(ii)] $\Tw$ acts linearly on $\Phi$, that is,
the action of $\Tw$ on $\Phi$ extends to an $\R$-linear action
of $\Tw$ on the real vector space spanned by $\Phi$.
\end{itemize}
\begin{Remark} In all our applications $\Tw$ will be a finite group
(in fact, usually a cyclic group).
\end{Remark}

Let $V=\R\Phi$ be the $\R$-vector space spanned by $\Phi$. Suppose we
are given another $\R$-vector space $W$ and a reduction
$\eta:V\to W$ such that
\begin{equation}
\label{eq:invariant}
\eta(q\alpha)=\eta(\alpha) \mbox{ for any } \alpha\in\Phi \mbox{ and } q\in \Tw.
\end{equation}
Let $\Psi=\eta(\Phi)\setminus\{0\}$ be the induced root system and
let $\{Y_\beta\}_{\beta\in \Psi}$ be the coarsened grading, that is,
$Y_{\beta}=\la X_{\alpha}: \eta(\alpha)=\beta\ra$.   Finally,
let $Z_{\beta}=Y_{\beta} ^{\Tw}$ be the
set of $\Tw$-fixed points in $Y_{\beta}$, and assume that the
following additional condition holds:
\begin{itemize}
\item[(iii)] $\{Z_\alpha\}_{\alpha\in \Psi}$ is a
$\Psi$-grading.
\end{itemize}
Then we define the {\bf twisted
group} $\widehat{G^\Tw}$ to be the graded cover of $\la Z_\alpha:\
\alpha\in \Psi\ra$ with respect to the $\Psi$-grading
$\{Z_\alpha\}_{\alpha\in \Psi}$.

Here is a slightly technical but easy-to-use criterion which
ensures that condition (iii) holds.
This criterion will be applicable in all of our examples.
\begin{Proposition}
\label{zgrading} Let $\Psi'$ be the set of all roots in $\Psi$
which are not representable as $a\gamma$ for some other $\gamma\in \Psi$
and $a>1$ (in particular, $\Psi'=\Psi$ if $\Psi$ is reduced).
Assume that
\begin{itemize}
\item[(a)] For any $\gamma,\delta\in \Psi$ such that $\delta=a\gamma$ with $a\geq 1$,
we have $Y_{\delta}\subseteq Y_{\gamma}$.
\item[(b)] For any Borel subset $B$ of $\Psi$,
any element of $\la Y_{\gamma}\ra_{\gamma\in B}$ can be {\bf uniquely} written as
$\prod_{i=1}^k y_{\gamma_i}$ where $\gamma_1,\ldots,\gamma_k$
are the roots in $B\cap \Psi'$ taken in some fixed order
and $y_{\gamma_i}\in Y_{\gamma_i}$ for all $i$.
\end{itemize}
Then $\{Z_\alpha\}_{\alpha\in \Psi}$ is a $\Psi$-grading.
\end{Proposition}
\begin{proof} Let $\alpha,\beta\in\Psi$ such that $\beta\not\in \R_{<0}\alpha$,
in which case there exists a Borel subset $B$ containing both $\alpha$ and $\beta$.
Let $\Omega=\{\gamma\in \Psi : \gamma=a\alpha+b\beta \mbox{ with }\alpha,\beta\geq 1\}$,
and let $\gamma_1,\ldots,\gamma_k$ be the roots in $B\cap \Psi'$ (as in condition (b)).
Let $I$ be the set of all $i\in\{1,\ldots, k\}$ such that $\R_{\geq 1}\gamma_i\cap \Omega\neq \emptyset$.
For each $i\in I$ let $\delta_i=a_i\gamma_i$ such that $\delta_i\in\Omega$
and $a_i\in\R_{\geq 1}$ is smallest possible.

Now take any $z\in Z_{\alpha}$ and $w\in Z_{\beta}$. Since $\{Y_{\gamma}\}_{\gamma\in \Psi}$
is a $\Psi$-grading, using condition~(a) it is easy to show that
$[z,w]=\prod_{i\in I} y_{\delta_i}$
where $y_{\delta_i}\in Y_{\delta_i}$ for each $i$. Since $z$ and $w$ are fixed by $Q$, for any $q\in Q$
we have $[z,w]=\prod_{i\in I} q(y_{\delta_i})$.

Thus, we have obtained two factorizations for $[z,w]$, and since $Y_{\delta_i}\subseteq Y_{\gamma_i}$,
they both satisfy the requirement in (b). Therefore, (b) implies that $q(y_{\delta_i})=y_{\delta_i}$
for each $i$. Hence $y_{\delta_i}\in Y_{\delta_i}^Q=Z_{\delta_i}$ for each $i\in I$, and so $[z,w]\in \la Z_{\gamma} : \gamma\in\Omega\ra$.
\end{proof}

If $Q$ is finite, one always has a natural choice for the pair $(W,\eta)$
satisfying \eqref{eq:invariant}. Indeed, by condition (ii)
the action of $\Tw$ on $\Phi$ extends to a linear action of $\Tw$ on $V$.
Then we can take $W=V^{\Tw}$, the subspace of $\Tw$-invariant vectors and
$\eta:V\to W$ the natural projection, that is,
\begin{equation}
\label{eq:invariant2}
\eta(v)=\frac{1}{|\Tw|}\sum_{q\in \Tw}qv.
\end{equation}
In fact, in all our examples the pair $(W,\eta)$ will be of this
form up to isomorphism, but it will be more convenient to
define $W$ and $\eta$ first and then check \eqref{eq:invariant}
rather than realize $W$ as the subspace of $\Tw$-invariant vectors in $V$.

\begin{Remark} If $\Tw$ and $\Tw'$ are conjugate in the group $\Aut_{gr}(G)$
of graded automorphisms of $G$, then the corresponding twisted groups
$\widehat{G^{\Tw}}$ and $\widehat{G^{\Tw'}}$ are easily seen to be isomorphic.
\end{Remark}
\vskip .1cm

We will discuss in detail six families of twisted groups.
The first five families will all be of the following form, while the
construction of the sixth family will involve minor modifications. Let
$\Phi$ be classical, reduced and irreducible of rank $\geq 2$.
We take $G=\St_\Phi(R)$ for some ring $R$, which is commutative if $\Phi$ is not of type $A$.
The acting group $Q$ will be a finite (usually cyclic) subgroup of
$\Aut(G)$ whose elements are compositions of diagram, ring
and diagonal automorphisms (as defined in \S~8.2).

The latter restriction on $Q$ implies that it naturally acts on the corresponding
adjoint elementary Chevalley group $\dbE_{\Phi}^{\rm ad}(R)$.
Let $\dbE_{\Phi}^{\rm ad}(R)^Q$ be the subgroup of $Q$-fixed points of $\dbE_{\Phi}^{\rm ad}(R)$,
and let $\widetilde{\dbE_{\Phi}^{\rm ad}(R)^Q}$ be the group generated by the intersections
of $\dbE_{\Phi}^{\rm ad}(R)^Q$ with the root subgroups of the $\Psi$-grading of $\dbE_{\Phi}^{\rm ad}(R)$.

Thus, if $\widehat{\St_\Phi(R)^Q}$ is the twisted group
obtained via the above procedure, there is a natural epimorphism
from $\widehat{\St_\Phi(R)^Q}$ onto $\widetilde{\dbE_{\Phi}^{\rm ad}(R)^Q}$.
We will refer to  $\widehat{\St_\Phi(R)^Q}$ as a \emph{twisted Steinberg group}, \index{Steinberg group!twisted}
and to $\widetilde{\dbE_{\Phi}^{\rm ad}(R)^Q}$ as a \emph{twisted Chevalley group}, and
we will also say that $\widehat{\St_\Phi(R)^Q}$ is a \emph{Steinberg cover}\index{Steinberg cover}
of $\dbE_{\Phi}^{\rm ad}(R)^Q$.

We note that the term `twisted Chevalley group' usually has a more
restricted meaning  -- instead of all possible finite groups of automorphisms $Q$
as above, one only considers those which are used in (canonical) realizations of
finite simple groups of twisted Lie type. In this section we will
mostly deal with the Steinberg covers for these types of twisted Chevalley groups.
The obtained Steinberg groups are summarized below and will be
studied in Examples~1-5.

\begin{itemize}
\item[1.] Groups $\St^{\zetta}_{C_{n}}(R,*)$ where $R$ is a ring,
$*$ is an involution on $R$ and $\zetta$ is a central unit of $R$
satisfying $\zetta^*=\zetta^{-1}$. These groups are Steinberg covers
for \emph{hyperbolic unitary groups} (see \cite{HO}).
The special case $\zetta=1$ corresponds to twisted Chevalley groups of type ${}^{2}A_{2n-1}$
(unitary groups in even dimension).

\item[2.] Groups $\St_{BC_{n}}(R,*)$ where $R$ is a ring and
$*$ is an involution on $R$. These groups are Steinberg covers for twisted
Chevalley groups of type ${}^{2}A_{2n}$ (unitary groups in odd dimension).

\item[3.] Groups $\St_{B_{n}}(R,\sigma)$ where $R$ is a commutative ring
and $\sigma$ is an involution on $R$. These groups are Steinberg covers for
twisted Chevalley  groups of type ${}^{2}D_{n}$.

\item[4.] Groups $\St_{G_{2}}(R,\sigma)$ where $R$ is a commutative ring
and $\sigma$ an automorphism of $R$ of order $3$. These groups
are Steinberg covers for twisted Chevalley groups of type ${}^{3}D_{4}$.

\item[5.] Groups $\St_{F_{4}}(R,\sigma)$ where $R$ is a commutative ring
and $\sigma$ an involution of $R$. These groups are Steinberg covers
for twisted Chevalley groups of type ${}^{2}E_{6}$.
\end{itemize}
\begin{Remark} Our notations for the twisted Steinberg groups are chosen
in such a way that the subscript indicates the root system
by which this twisted Steinberg group is naturally graded.
\end{Remark}

In Example~3 we shall define a more general family of twisted groups
(which will include the groups $\St_{B_{n}}(R,\sigma)$ as a special case)
using an observation that the (classical) Steinberg group $\St_{D_n}(R)$ arises
as the twisted Steinberg group $\St^{-1}_{C_{n}}(R,*)$ (from Example~1)
in the special case when $*$ is trivial (and $R$ is commutative).
\vskip .1cm

In the last example (Example~6) we construct certain groups $\St_{{}^2 F_{4}}(R,*)$,
where $R$ is a commutative ring of characteristic $2$ and $*:R\to R$ is an injective homomorphism such
that $(r^*)^*=r^2$. These groups are graded by a root system in $\R^2$ with $24$ roots
and can be defined as graded covers of certain ``algebraic-like'' groups
constructed by Tits~\cite{Ti}. We note that the standard definition of
twisted Chevalley groups of type ${}^2 F_{4}$ is only valid when $R$ is a perfect field,
in which case they coincide with Tits' groups.
Unlike Examples~1-5, in the construction of the groups $\St_{{}^2 F_{4}}(R,*)$,
the initial group $G$ will not be the entire Steinberg group $\St_{F_{4}}(R)$,
but certain subgroup of it. The general twisting procedure will also be slightly modified
in this example, as we will have to apply the \emph{fattening} operation (defined in \S~4.4) \index{fattening}
to the coarsened grading $\{Y_{\beta}\}$.

\subsection{Graded automorphisms of $\mathbb E_{\Phi}^{\rm ad}(R)$ and $\St_{\Phi}(R)$}

In this section we describe some natural families of graded
automorphisms of (non-twisted) adjoint elementary Chevalley groups and Steinberg groups.
Each automorphism will be defined via its action on the root subgroups,
and we will need to justify that it can be extended to the entire group.

In the case of elementary Chevalley groups and Steinberg groups over commutative rings
the following observation will provide the justification:

\begin{Observation}
\label{aut_induced}
Let $\Phi$ be a reduced irreducible classical
root system, $R$ a commutative ring and $\mathcal L_R$ the $R$-Lie algebra
of type $\Phi$, as defined in \S~7. Let $f\in \Aut(\mathcal L_R)$ be an automorphism
which permutes the root subspaces of $\mathcal L_R$. Then
$\mathbb E_{\Phi}^{\rm ad}(R)$, considered as a subgroup of $\Aut(\mathcal L_R)$,
is normalized by $f$, and moreover, the conjugation by $f$
permutes the root subgroups $\{X_{\alpha}(R)\}$ of $\,\mathbb E_{\Phi}^{\rm ad}(R)$.
Thus $f$ naturally induces a graded automorphism of $\mathbb E_{\Phi}^{\rm ad}(R)$
and hence also induces a graded automorphism of $\St_{\Phi}(R)$ by
Lemma~\ref{aut_gradedcover}.
\end{Observation}

If $G=\St_{A_m}(R)$, with $R$ noncommutative,
the existence of the automorphism of $G$ with a given action on the root subgroups
is easy to establish using the standard presentation of $\St_{A_m}(R)$
recalled below.

As usual, we realize $A_m$ as the subset $\{e_i-e_j:\ 1\le i\ne j\le m+1\}$  of $\R^{m+1}$.
The group $\St_{A_m}(R)$  has generators $\{x_{e_i-e_j}(r) : 1\leq i\neq j\leq m+1, r\in R\}$ and relations
$$x_{e_i-e_j}(r+s)=x_{e_i-e_j}(r)x_{e_i-e_j}(s)\mbox{ and }\qquad\,\,\,$$
$$[x_{e_i-e_j}(r),x_{e_k-e_l}(s)]=\left \{ \begin{array}{cl}
x_{e_i-e_l}(rs) & \textrm{if $j=k$ and $i\ne l$}\\
x_{e_k-e_j}(-sr) & \textrm{   if $j\ne k$ and $i=l$}\\    0
&\textrm{ if  $j\ne k$ and $i\ne l$}\end{array} \right . $$

\vskip .2cm
In each of the following examples we fix a ring $R$ and a reduced irreducible
classical root system $\Phi$, and $G$ will denote one of the groups
$\mathbb E_{\Phi}^{\rm ad}(R)$ or $\St_{\Phi}(R)$, unless additional restrictions are imposed.

{\bf Type I: \it ring automorphisms.} Let $\sigma$ be an
automorphism of the ring $R$. Then we can define the automorphism
$\phi_\sigma$ of $G$ by
$$\phi_\sigma(x_\alpha(r))=x_\alpha(\sigma(r)),\ \alpha\in \Phi,\ r\in
R.$$
If $R$ is commutative, $\phi_\sigma$ is well defined since it is
induced (as in Observation~\ref{aut_induced}) by the automorphism of $\mathcal{L}_R$
which sends $r\otimes l$ (where $l\in
\mathcal{L}_\Z$ and $r\in R$) to $\sigma (r)\otimes l$.

If $R$ is arbitrary and $G=\St_{A_n}(R)$, the automorphism $\phi_\sigma$ is well defined
since it clearly respects the defining relations of $G$.

{\bf Type II: \it diagonal automorphisms.} Let $Z(R)^{\times}$ be the group
of invertible elements of $Z(R)$, let $\Z\Phi$ denote the $\Z$-span of $\Phi$,
and let $\mu:\Z\Phi\to Z(R)^{\times}$ be a homomorphism. Then we can define the automorphism $\chi_\mu$ of
$G$ given by
$$\chi_\mu(x_\alpha(r))=x_\alpha(  \mu(\alpha)r ),\ \alpha\in \Phi,\ r\in
R.$$
If $R$ is commutative, $\chi_\mu$ is well defined since it is
induced by the automorphism of $\mathcal{L}_R$
which fixes $h_\alpha$  and sends $ x_\alpha$
to $\mu(\alpha)x_\alpha$ for any $\alpha\in \Phi$.

As for type $I$, if $R$ is arbitrary and $G=\St_{A_n}(R)$, the automorphism $\chi_{\mu}$ is well defined
since it respects the defining relations.

{\bf Type III: \it root system automorphisms (commutative case).} In this example
we assume that $R$ is commutative.
Let $V=\R\Phi$ be the $\R$-span of $\Phi$, and let $\pi$ be an
automorphism of $V$ which stabilizes $\Phi$ (equivalently, we can
start with an automorphism of $\Phi$ and uniquely extend it to an
automorphism of $V$). Then there are constants
$\gamma_\alpha=\pm 1 (\alpha\in \Phi)$ such that the map
$$\lambda_\pi(x_\alpha(r))=x_{\pi(\alpha)}(\gamma_\alpha r),\ \alpha\in \Phi,\
r\in R, $$  can be extended to an automorphism of $G$.

The existence of an automorphism of $\mathcal{L}_R$ which induces
$\lambda_\pi$ is a consequence of the Isomorphism
Theorem for simple Lie algebras (\cite[Theorem 3.5.2]{Ca}, see also
\cite[Proposition 12.2.3]{Ca}).

Note that there is no canonical choice for the constants $\gamma_\alpha$
(except when $\Phi=A_n$), so in our notation $\lambda_\pi$ is only unique
up to a diagonal automorphism (which acts as multiplication by $\pm 1$
on each root subgroup).

{\bf Type III': \it mixed automorphisms of $\St_{A_m}(R)$.}
In this example we assume that $G=\St_{A_m}(R)$ and $R$ is arbitrary.
Let $V$ be the $\R$-span of $A_m$.
It is well known that every automorphism of $A_m$ has the form
$$a(\pi,\delta): e_i-e_j\mapsto (-1)^\delta (e_{\pi(i)}- e_{\pi(j)})$$
for some permutation $\pi\in \Sigma_{m+1}$ and $\delta=0,1$.
In particular, $\Aut(A_m)$ has order $2(m+1)!$ (if $m\geq 2$), and
it is easy to see that the automorphisms with $\delta=0$ are precisely
the elements of the Weyl group of $A_m$.

If $R$ is commutative, we have already associated an automorphism of $G$
of type~III to each element of $\Aut(A_m)$. The type III automorphism of $G$
corresponding to $a(\pi,0)\in \Aut(A_m)$ can be defined even if $R$ is not
commutative. It will be denoted by $\lambda^+_\pi$ and is given by
$$\lambda^+_\pi(x_{e_i-e_j}(r))=x_{e_{\pi(i)}-e_{\pi(j)}}( r),\ \alpha\in \Phi,\
r\in R.$$

Similarly, if $R$ is commutative, we will denote by $\lambda^-_{\pi}$
the type III automorphism of $G$  corresponding to $a(\pi,1)\in \Aut(A_m)$.
It is given by
$$\lambda^-_{\pi}(x_{e_i-e_j}(r))=x_{e_{\pi(j)}-e_{\pi(i)}}(- r),\ \alpha\in \Phi,\
r\in R. $$
The formula for $\lambda^-_{\pi}$ will not define an automorphism of $G$
if $R$ is noncommutative. However, if we are given an anti-automorphism $*$ of $R$,
for each $\pi\in \Sigma_{m+1}$ we can define an automorphism $\lambda^-_{\pi,*}$ of $G$
by setting
$$\lambda^-_{\pi,*}(x_{e_i-e_j}(r))=x_{e_{\pi(j)}-e_{\pi(i)}}(- r^*),\ \alpha\in \Phi,\
r\in R. $$
These automorphisms will be called \emph{mixed}.

Note that if $R$ is commutative, then $\lambda^-_{\pi,*}$ is just the composition
of $\lambda^{-}_{\pi}$ and the ring automorphism $\phi_*$.

\vskip .2cm
The collection of twisted groups that can be constructed using these
four types of automorphisms and their compositions is clearly too
large for case-by-case analysis and is beyond the scope of this paper.
We shall concentrate on automorphisms which yield natural analogues
of twisted Chevalley groups listed at the end of \S~8.1.

Among all root system automorphisms of particular importance are \emph{diagram
automorphisms} -- the ones induced by an automorphism of the Dynkin diagram
of $\Phi$. For instance, in the case $\Phi=A_m$,
there is unique (non-trivial) diagram automorphism (for a given choice of simple roots) --
in the above notations it is the automorphism $\lam^{-}_{\pi}$ where $\pi\in\Sigma_{m+1}$ is given by
$\pi(i)=m+2-i$. Each of the twisted Chevalley groups of type ${}^k {\Phi}$,
where $\Phi=A_n, D_n$ or $E_6$ and $k=2$, or $\Phi=D_4$ and $k=3$, is obtained
from $\dbE_{\Phi}^{\rm ad}(R)$ using the twisting by the composition
of a diagram automorphism and a ring automorphism of the same order $k$.

\subsection{Unitary Steinberg groups over non-commutative rings with involution}

In this subsection we shall define (twisted) Steinberg groups
corresponding to (quasi-split) unitary groups and, in the case of even dimension,
their generalizations, called hyperbolic unitary Steinberg groups.
We shall establish property $(T)$ for most of those groups. To simplify the exposition,
we will not provide explicit estimates for the Kazhdan constants, although in most
cases reasonably good estimates can be obtained by adapting the arguments from \S~7.

Throughout this subsection we fix a ring $R$, and let
$*:R\to R$ be an involution, that is, an anti-automorphism of order $\geq 2$.

As we already stated, in the classical setting unitary groups are obtained from Chevalley groups of type $A_m$
via twisting  by the order $2$ automorphism  $$Dyn_*=\lam^{-}_{\pi,*}\mbox{ where } \pi \mbox{ is the permutation }
i\mapsto m+2-i.$$ In even dimension (that is, if $m$ is odd), there is an interesting generalization of this construction,
where instead of $Dyn_*$ one uses the composition of $Dyn_*$ with a suitable diagonal automorphism of order $2$.

To each $\zetta\in Z(R)^{\times}$ we can associate a  homomorphism
$T_{\zetta}:\Z A_m\to Z(R)^{\times}$ given by
$$T_{\zetta}(e_i-e_j)=\left \{\begin{array} {cl} 1 & \textrm{if  $i,j\le
(m+1)/2$ or $i,j> (m+1)/2$}\\ \zetta &  \textrm{if $i \le
(m+1)/2<j$}\\ \zetta^{-1} & \textrm{if $j\le  (m+1)/2
<i$}\end{array}\right .$$
Note that the homomorphism $T_{1}:\Z A_m \to Z(R)^{\times}$ is the trivial homomorphism.
\vskip .1cm
For each such $\zetta$ we define the automorphism $q_{\zetta}$ of $\St_{A_m}(R)$
given by
\begin{equation}
\label{qzetta}
q_{\zetta}=Dyn_* \chi_{T_{\zetta}}.
\end{equation}
(recall that $\chi_{T_{\zetta}}$ is a diagonal automorphism, defined in \S~8.2).

Now let $$U(R)=\{r\in R^{\times}: rr^*=1\}\quad\mbox{ and }\quad U(Z(R))=U(R)\cap Z(R).$$ It is easy to see
that if $m$ is odd and $\zetta\in U(Z(R))$, then $q_{\zetta}$ has order $2$.

The groups obtained from Chevalley groups of type $A_m$ via twisting by $q_{\zetta}$ (with $m$ odd
and $\zetta\in U(Z(R))$) are called \emph{hyperbolic unitary groups}. These groups have been originally defined by Bak~\cite{Bak} and are discussed in detail in the book by Hahn and O'Meara~\cite{HO} (see also \cite{Bak2}).

\begin{Remark} It is easy to show that if $\chi$ is any diagonal automorphism of
the Chevalley group $\SL_{m+1}(R)$ such that the composition $Dyn_* \chi$ has
order $2$, then $\chi$ is graded conjugate (in fact, conjugate by a diagonal automorphism)
to $q_{\zetta}$ for some $\zetta\in U(Z(R))$ if $m$ is odd, and graded conjugate
to $Dyn_*$ if $m$ is even; see also Observation~\ref{symisoasym} below.
This yields a simple characterization of hyperbolic unitary groups among all twisted Chevalley groups.
\end{Remark}

\vskip .15cm
Before turning to Example~1, we introduce some additional terminology  from \cite{HO}
(we note that our notations are different from \cite{HO}).

\begin{Definition}
Let $\zetta\in U(Z(R))$.
Put $$\Sym_{-\zetta}(R)=\{r\in R :\ r^*\zetta=-r\}\quad \mbox{ and }
\quad \Sym_{-\zetta}^{\min}(R)=\{r-r^*\zetta:\ r\in R\}.$$
A  \emph{form parameter} \index{form parameter}
of the triple $(R,*,\omega)$ is a subgroup $I$ of
$(R,+)$ such that
\begin{itemize}
\item[(i)] $\Sym_{-\zetta}^{\min}(R)\subseteq I\subseteq \Sym_{-\zetta}(R)$
\item[(ii)] For any $u\in I$ and $s\in R$ we have $s^*us\in I$.
\end{itemize}
\end{Definition}
The following simplified terminology will be used in the case $\zetta=\pm 1$.
\begin{itemize}
\item The set $\Sym_{1}(R)$ will be denoted by $\Sym(R)$,
and its elements will be called symmetric.
\item The set $\Sym_{-1}(R)$ will be denoted by $\Asym(R)$, and its
elements will be called antisymmetric
\end{itemize}
For a subset $A$ of $\Sym_{-\zetta}(R)$ we let $\la A\ra_{-\zetta}$ be
the form parameter generated by $A$, that is,
$$
\la A\ra_{-\zetta}=\{x\in \Sym_{-\zetta}(R):\ x=\sum_{i=1}^k s_i a_i
s_i^* +(r -r^*\zetta ) \mbox{ with } a_i\in A, s_i, r\in R\}.
$$

\begin{Remark}
If $A=\{a_1,\ldots, a_m\}$ is finite, then any element $x\in \la
A\ra_{-\zetta}$ has an expansion in the form
$$ x=\sum_{i=1}^m s_i a_i s_i^* +(r -r^* \zetta),$$ that is,  with only one term $sas^*$ for each $a\in A$.
This is because $uau^*+vav^*=(u+ v)a(u+v)^*+(r -r^*\zetta)$ for $r=
 -uav^*$.
\end{Remark}

\vskip .2cm
{\bf Example 1: \it Hyperbolic unitary Steinberg groups.}
Let $\Phi=A_{2n-1}$ and $G=\St_{\Phi}(R)$.
Fix $\zetta\in U(Z(R))$, let $q=q_{\zetta}\in \Aut(G)$ and $Q=\la q\ra$.

The twisted group $\widehat{G^Q}$ constructed in this example
will be denoted by $\St^\zetta_{C_n}(R,*)$. This group
is graded by the root system $C_n$ and corresponds to the group of
transformations preserving the sesquilinear form
$$f(u,v)=\sum_{i=1}^{ n}  u_{i}v_{\bar i}^* +\zetta  u_{\bar i}v_i^* \mbox{ on } R^{2n}\mbox{ where }\ibar=2n+1-i.$$
\begin{Remark} This form is $\zetta$-hermitian, that is,
$f(v,u)=\zetta(f(u,v)^*)$. For more information on groups fixing this form see \cite[Chapter 5.3]{HO}.
\end{Remark}

We shall use the standard realization for both $A_{2n-1}$ and $C_n$,
and to avoid confusion we shall denote the roots of $A_{2n-1}$
by $e_i-e_j$, with $1\leq i\neq j\leq 2n$, and the roots of
$C_n$ by $\pm \epps_i\pm \epps_j$ and $\pm 2 \epps_i$,
with $1\leq i\neq j\leq n$.
\vskip .1cm
The action of $q$ on the root subgroups of $G$ is given by
$$q(x_{e_i-e_j}(r))=
\left\{
\begin{array}{ll}
x_{e_{\bar j}-e_{\bar i}}(-r^*)& \mbox{ if } i,j\leq n\mbox{ or }i,j>n\\
x_{e_{\bar j}-e_{\bar i}}(-\zetta r^*)& \mbox{ if } i\leq n<j\\
x_{e_{\bar j}-e_{\bar i}}(-\zetta^* r^*)& \mbox{ if } j\leq n<i\\
\end{array}
\right.$$
Define $\eta: \bigoplus_{i=1}^{2n}\R e_i\to \bigoplus_{i=1}^{n}\R \epps_i$ by
$\eta(e_i)=\epps_i$ if $i\leq n$ and $\eta(e_i)=-\epps_{\bar i}$ if $i>n$.
It is straightforward to check that $\eta$ is $q$-invariant.
Then $$\eta(e_i-e_j)=
\left\{
\begin{array}{ll}
\epps_i-\epps_j& \mbox{ if } i,j\leq n\\
\epps_i+\epps_{\bar j}& \mbox{ if } i\leq n,\,\, j>n\\
-\epps_{\bar i}-\epps_j& \mbox{ if } i>n,\,\, j\leq n\\
\epps_{\bar j}-\epps_{\bar i} & \mbox{ if } i,j> n\\
\end{array}
\right.$$
so the root system $\Psi=\eta(\Phi)\setminus\{ 0\}$ is indeed of type $C_n$ (with standard realization).

Let $\{Y_{\gamma}\}_{\gamma\in\Psi}$ denote the coarsened $\Psi$-grading of $G$.
If $\gamma\in\Psi$ is a short root, the corresponding root
subgroup $Y_{\gamma}$ consists of elements $\{y_{\gamma}(r,s) :
r,s\in R\}$ where $$\begin{array}{lll} y_{\epps_i-\epps_j}(r,s) & =
& x_{e_i-e_j}(r) x_{e_{\bar j}-e_{\bar i}}(s)\\ &&\\
 y_{\pm(\epps_i+\epps_j)}(r,s)&=&x_{\pm(e_i-e_{\bar j})}(r)
x_{\pm(e_{j}-e_{\bar i})}(s)\end{array}$$

If $\gamma\in\Psi$ is a long root, the corresponding root subgroup
$Y_{\gamma}$ consists of elements $\{y_{\gamma}(r) : r\in R\}$
where
$$y_{2\epps_i}(r)=x_{e_i-e_{\bar i}}(r).$$

Computing $q$-invariants and letting $Z_{\gamma}=Y_{\gamma}^q$,
we get
$$\begin{array}{lll}
 Z_{\epps_i-\epps_j}& =&\{z_{\epps_i-\epps_j}(r)=x_{e_i-e_j}(r)
x_{e_{\bar j}-e_{\bar i}}(-r^*):\ r\in R\}\ \mbox{ for }i<j \\ &
&\\Z_{\epps_i-\epps_j}& =&\{z_{\epps_i-\epps_j}(r)=x_{e_i-e_j}(-r^*)
x_{e_{\bar j}-e_{\bar i}}(r):\ r\in R\}\ \mbox{ for } i>j\\ & &\\
 Z_{ \epps_i+\epps_j }& =& \{z_{ \epps_i+\epps_j }(r)=x_{ e_i-e_{\bar
j} }(r) x_{ e_{j}-e_{\bar i} }(-\zetta r^*): \ r \in R\} \ \mbox{ for } i<j\\
&&\\Z_{-\epps_i-\epps_j}& =& \{z_{-\epps_i-\epps_j}(r)=x_{-e_i+e_{\bar
j} }(-r^*) x_{-e_{j}+e_{\bar i} }(\zetta^* r): \ r \in R\} \ \mbox{ for } i<j\\ &&\\
Z_{   2\epps_i}&=&\{z_{  2\epps_i}(r)=x_{ e_i-e_{\bar i} }( r) :\
r\in \Sym_{-\zetta}(R)\} \\ &&\\
Z_{   -2\epps_i}&=&\{z_{  -2\epps_i}(r)=x_{ -e_i+e_{\bar i} }( -r^*)
:\ r\in \Sym_{-\zetta}(R)\} \end{array}$$
Note that
\begin{itemize}
\item[] $Z_{\gamma}\cong (R,+)$ if $\gamma$ is a short root and
\item[] $Z_{\gamma}\cong (\Sym_{-\zetta}(R),+)$ if $\gamma$ is a long root
\end{itemize}
\begin{Remark} When $\gamma$ is short, there is no
``canonical'' isomorphism between $Z_{\gamma}$ and $(R,+)$, so a choice
needs to be made in the definition of $z_{\gamma}(r)$.
\end{Remark}
\vskip .1cm

It is easy to see that the hypothesis of Proposition~\ref{zgrading}
holds in this example. Hence $\{Z_{\gamma}\}_{\gamma\in \Psi}$
is a $\Psi$-grading, and we can form the graded cover $\widehat{G^{Q}}$.

Thus, by definition $\widehat{ G^{Q}}=\la Z\, |\, E \ra$
where $Z=\sqcup_{\gamma\in\Psi} Z_{\gamma}$ and $E$ is the set
of commutation relations (inside $G$) expressing the elements of
$[Z_{\gamma},Z_{\delta}]$ in terms of $\{Z_{a\gamma+b\delta}: a,b\geq 1\}$
(where $\delta\not\in\R_{<0}\gamma$). These relations are obtained by straightforward calculation.

Below we list the non-trivial commutation relations between the positive root subgroups
(omitting the relations where the commutator is equal to $1$).

\begin{align*}
&(E1)& &[z_{\epps_i-\epps_j}(r), z_{\epps_j-\epps_k}(s)]=z_{\epps_i-\epps_k}(rs)\  \mbox{ for } i<j<k&\\
&(E2)& &[z_{\epps_i-\epps_j}(r), z_{\epps_i+\epps_j}(s)]=z_{2\epps_i}(   sr^*-\zetta rs^*)  \ \mbox{ for } i<j&\\
&(E3)& &[z_{2\epps_{j}}(r),
z_{\epps_i-\epps_j}(s)]=z_{\epps_i+\epps_j}(-sr)
z_{2\epps_{i}}( srs^*) \  \mbox{ for }i<j &\\
&(E4)& &[z_{ \epps_i-\epps_{j}}(r), z_{\epps_j+\epps_k}(s)]=\left \{
\begin{array}{ll}
z_{\epps_i+\epps_k}(rs) &
   \mbox{ for } i<j<k \\
  z_{\epps_i+\epps_k}( sr^*)
 & \mbox{ for } k<i<j   \\
 z_{\epps_i+\epps_k}(  -\zetta rs^* )
&  \mbox{ for } i<k<j  \end{array}\right .&\\
\end{align*}
The remaining relations (involving negative root subgroups) are analogous.
We list just those relations which will be explicitly used later in the paper.
  \begin{align*}
&(E5)& &[z_{-2\epps_i}(r),z_{\epps_{i}+\epps_j}(s)]=\left
\{\begin{array}{ll}
z_{\epps_j-\epps_i}(-r^*s) z_{2\epps_j}(s^*rs)
 &  \mbox{ for } i<j\\
z_{\epps_j-\epps_i}( sr^*) z_{2\epps_j}(srs^*)
 &  \mbox{ for } i>j
\end{array}\right . &\\
&(E6)& &[z_{-2\epps_i}(r), z_{\epps_i-\epps_j}(s)]=\left\{
\begin{array}{ll}
z_{-\epps_i-\epps_j}(rs) z_{-2\epps_j}(s^* rs)  & \mbox{ for } i<j\\
z_{-\epps_i-\epps_j}( { -sr}) z_{-2\epps_j}(srs^*) & \mbox{ for } i>j
\end{array}\right . &\\
&(E7)& &[z_{\epps_i-\epps_j}(r), z_{\epps_i+\epps_j}(s)]=z_{2\epps_i}
( \zetta s^* r-r^* s)  \ \mbox{ for } i>j&\\
&(E8)& & [z_{\epps_i-\epps_j}(r), z_{-\epps_i-\epps_j}(s)]=
\left\{
\begin{array}{ll}
z_{-2\epps_j}( \zetta s^* r-r^*s)  & \mbox{ for } i<j\\
z_{-2\epps_j}( sr^*-\zetta r s^*)  & \mbox{ for } i>j
\end{array}\right . &\\
&(E9)& &[z_{2\epps_{j}}(r),
z_{\epps_i-\epps_j}(s)]=z_{\epps_i+\epps_j}(rs)
z_{2\epps_{i}}( s^*rs) \  \mbox{ for }i>j &
\end{align*}

The group $\widehat{G^{Q}}$ we just constructed will be denoted by $\St_{C_n}^\zetta (R,*)$.

\vskip .2cm
\noindent {\bf Variations of $\St_{C_n}^\zetta (R,*)$ involving form parameters.}
The defining relations show that $\St_{C_n}^\zetta (R,*)$
admits a natural family of subgroups  also graded by $C_n$,
obtained by decreasing long root subgroups.

Let $J$ be a form parameter of $(R,*,\zetta)$. Given $\gamma\in C_n$,
let
$$Z_{J,\gamma}=\left\{
\begin{array}{ll}
Z_{\gamma}& \mbox{ if } \gamma \mbox{ is a short root}\\
\{z_{\gamma}(r): r\in J\} & \mbox{ if } \gamma \mbox{ is a long root}.
\end{array}
\right.
$$
The defining relations of $\St_{C_n}^\zetta (R,*)$
imply that $\{Z_{J,\gamma}\}_{\gamma\in C_n}$ is a grading.
Define $\overline{\St}_{C_n}^\zetta(R,*,J)$ to be the subgroup
of $\St_{C_n}^\zetta (R,*)$ generated by $Z_J:=\cup Z_{J,\gamma}$, and let $\St^\zetta_{C_n}(R,*,J)$ be the
graded cover of $\overline{\St}_{C_n}^\zetta(R,*,J)$. It is not hard to show
that $\overline{\St}_{C_n}^\zetta(R,*,J)$ has the presentation
$\la Z_J | E_J\ra$ where $E_J\subseteq E$ is set of those commutation relations of $\overline{\St}_{C_n}^\zetta(R,*)$
which only involve generators from $Z_J$.
\vskip .15cm

Here are two important observations. The first one is that non-twisted Steinberg groups of type $C_n$ and $D_n$
are special cases of the groups $\{\St^\zetta_{C_n}(R,*,J)\}$. The second observation
describes some natural isomorphisms between these groups.

\begin{Observation}
\label{ex1degenerate}
Assume that the ring $R$ is commutative, so that the identity map $id:R\to R$ is an involution.
The following hold:
\begin{itemize}
\item[(1)] The group
$\St_{C_n}^{-1}(R,id)$ coincides with $\St_{C_n}(R)$, the usual
(non-twisted) Steinberg group of type $C_n$.
\item[(2)] $J=\{0\}$ is a possible form parameter of $(R,id,1)$, and
the group $\St_{C_n}^{1}(R,id,\{0\})$ coincides with $\St_{D_n}(R)$,
the usual Steinberg group of type $D_n$. This happens because
the long root subgroups in the $C_n$-grading on $\St_{C_n}^{1}(R,id,\{0\})$
are trivial, and we can ``remove'' those roots to obtain a $D_n$-grading.
\end{itemize}
\end{Observation}

\begin{Observation}
\label{symisoasym}
Let $\zetta\in U(Z(R))$, and let $\zetta'= \zetta \mu^{-1}\mu^*$ for some invertible element $\mu\in Z(R)$. Then
the automorphisms $q_\zetta$
and $q_{\zetta'}$ are graded-conjugate and so $\St_{C_n}^{\zetta}(R,*)$ and
$\St_{C_n}^{\zetta'}(R,*)$ are isomorphic. In particular,
$\St_{C_n}^{\zetta}(R,*)\cong \St_{C_n}^{-\zetta}(R,*)$ whenever $Z(R)$ contains an invertible
antisymmetric element.
\end{Observation}
\begin{Remark}
 An explicit isomorphism is constructed as follows. If $Z_\gamma=\{z_\gamma(r)\}$ are the root subgroups of
$\St^{\zetta}_{C_n}(R)$ and $Z'_\gamma=\{z'_\gamma(r)\}$ are the root subgroups of $\St_{C_n}^{\zetta'}(R)$,
then the map $\phi$ defined on root subgroups   as
$$\phi(z_\gamma(r))=\left \{
\begin{array}{ccc} z'_\gamma(r) & &\textrm{ if\ } \gamma=\eps_i-\eps_j\\
 z'_\gamma(\mu^* r) & &\textrm{ if\ } \gamma=\eps_i+\eps_j \\
  z'_\gamma(\mu^{-1} r) & &\textrm{ if\ } \gamma=-\eps_i-\eps_j \end{array}
 \right.$$
is an isomorphism.
\end{Remark}

We now turn to the proof of property $(T)$ for hyperbolic unitary Steinberg groups.

\begin{Lemma}
\label{grading:ex1} Let $R$ be a ring with involution $*$, let
$\zetta\in U(Z(R))$, and let $J$ be a form parameter of $(R,*,\zetta)$.
\begin{itemize}
\item[(a)] If $n\geq 3$, the $C_n$-grading on $\St_{C_n}^\zetta
(R,*,J)$ is strong.
\item[(b)] Assume that the left ideal of $R$
generated by $J$ equals $R$. Then the  $C_n$-grading on
$\St_{C_n}^\zetta (R,*,J)$ is $2$-strong (in particular, the grading
is strong if $n=2$).
\end{itemize}
\end{Lemma}
\begin{proof}(a) By definition, we need to check that the grading is strong
at $(\gamma,B)$ for every Borel subset $B$ and $\gamma\in C(B)$, the core of $B$,
and by symmetry it suffices to consider the case when $B$ is the standard Borel.
If $\gamma\in C(B)$ is a long root, the grading is strong at $(\gamma,B)$
by relations (E3) with $s=1$. If $n\geq 3$ and $\gamma\in C(B)$ is a short root,
the grading is strong at $(\gamma,B)$ by relations (E1) or (E4).

(b) If $n=2$, the grading is strong at short root subgroups by
relations (E3). The same argument shows that the grading is $2$-strong
for any $n\geq 2$.
\end{proof}

\begin{Proposition} \label{prop:ex1}
Let $R$ be a finitely generated ring with involution $*$,
$\zetta\in U(Z(R))$ and $J$  a  form parameter of $(R,*,\zetta)$. Assume that $J$ is finitely generated as
a form parameter. The following hold:
\begin{itemize}
\item[(a)] The group $H=\St_{C_n}^\zetta(R,*,J)$ has property $(T)$
for any $n\geq 3$.
\item[(b)] Assume in addition that $\zetta=-1$, $1_R\in J$
(so, in particular,  the left ideal of $R$ generated by $J$ equals $R$),
and $R$ is a finitely generated right module over its subring generated by a  finite set of  elements from $J$.
Then the group $\St_{C_2}^{-1}(R,*,J)$ has property $(T)$.
\end{itemize}
\end{Proposition}
\begin{proof}
Lemma~\ref{grading:ex1} ensures that the $C_n$-grading is strong,
so we only need to check relative property $(T)$ for root subgroups.

(a) Relations (E1) ensure that any short root subgroup
$Z_{\gamma}$ can be put inside a group which is a quotient of
$\St_{A_2}(R)=\St_3(R)$ and hence the pair $(H,Z_{\gamma})$ has
relative property $(T)$. To prove relative $(T)$ for long root
subgroups we realize each of them as a subset of a bounded product of short
root subgroups and some finite set. Without loss of generality,
we will establish the required factorization for the long root
$\gamma=2\epps_1$.

Let $T$ be a finite set which generates $J$ as a form parameter of $(R,*,\zetta)$.
By the remark following the definition of a form parameter,  any $r\in J$
can be written as $r=\sum_{t\in T} s_t t s_t^* + (u-\zetta u^*)$ for
some $s_t,u\in R$. Relations (E2) and (E3) yield the following
identity:
$$z_{2\epps_1}(r)=\prod_{t\in T} [z_{2\epps_2}(t), z_{\epps_1-\epps_2}(s_t)]z_{\epps_1+\epps_2}(\sum_{t\in T}s_t t)
[z_{\epps_1-\epps_2}(1),z_{\epps_1+\epps_2}(u)]
$$ It
follows that
$$
Z_{2\epps_1}\subseteq \prod_{t\in T}
Z_{\epps_1-\epps_2}^{z_{2\epps_2}(t)} Z_{\epps_1-\epps_2}
Z_{\epps_1+\epps_2}Z_{\epps_1-\epps_2} Z_{\epps_1+\epps_2}Z_{\epps_1-\epps_2} Z_{\epps_1+\epps_2}.
$$
The set $\{z_{2\epps_2}( t):\ t\in T\}$ of conjugating elements is
finite, so we obtained the desired factorization. \vskip .2cm (b)
Relative property $(T)$ in this case will be established in
Proposition~\ref{involution_gen} in \S~8.6.
\end{proof}
\vskip .2cm

{\bf Example 2: \it Unitary Steinberg groups in odd dimension.}
Let $\Phi=A_{2n}$ and $G=\St_{\Phi}(R)$. Let $q=Dyn_*\in \Aut(G)$
and $Q=\la q\ra$.

The twisted group $\widehat{G^Q}$ constructed in this example
will be denoted by $\St_{BC_n}(R,*)$ and graded by the root system $BC_n$.
It corresponds to the group of transformations preserving the Hermitian form
$$
f(u,v)=u_{n+1} v_{n+1}^*+\sum_{i=1}^n (u_{i}v_{\bar i}^* +u_{\bar i}v_{i}^*)\mbox{ on }R^{2n+1}
\mbox{ where } \bar i=2n+2-i.
$$
The action of $q=Dyn_*$ on the root subgroups of $G$ is given by
$$q: x_{e_i-e_j}(r)\mapsto x_{e_{\bar j}-e_{\bar i}}(- r^*).$$
Define $\eta: \bigoplus_{i=1}^{2n+1}\R e_i\to \bigoplus_{i=1}^{n}\R \epps_i$ by
$\eta(e_i)=\epps_i$ if $i\leq n$ and $\eta(e_i)=-\epps_{\bar i}$ if $i\geq n+2$
and $\eta(e_{n+1})=0$.
Similarly to Example 1, we check that $\eta$ is $q$-invariant and
the root system $\Psi=\eta(\Phi)\setminus\{0\}$ is indeed of type $BC_n$.

If $\gamma\in\Psi$ is a long root, the corresponding root subgroup
$Y_{\gamma}$ consists of elements $\{y_{\gamma}(r,s) : r,s\in R\}$
where
\begin{itemize}
\item[] $y_{\epps_i-\epps_j}(r,s)=x_{e_i-e_j}(r) x_{e_{\bar
j}-e_{\bar i}}(s)$ \item[]
$y_{\pm(\epps_i+\epps_j)}(r,s)=x_{\pm(e_i-e_{\bar j})}(r)
x_{\pm(e_{j}-e_{\bar i})}(s)$.
\end{itemize}
If $\gamma\in\Psi$ is a short root, the root subgroup
$Y_{\gamma}$ consists of elements $\{y_{\gamma}((r,s, t)) : r
,s,t\in R\}$ where
$$y_{\pm \epps_i}( r,s,t) =x_{\pm(e_i-e_{n+1})}(r) x_{\pm(e_{n+1}-e_{\bar i})}(s)x_{\pm(e_i-e_{\bar i})}(t).$$
Note that the groups $Y_{\pm \epps_i}$ are not abelian, and
multiplication in them is determined by $$\begin{array}{lll}y_{
\epps_i}( r_1,s_1,t_1) y_{  \epps_i}((r_2,s_2,t_2))& =& y_{
\epps_i}( r_1+r_2,s_1+s_2, t_1 + t_2 - r_2s_1 )\\ &&\\ y_{
-\epps_i}( r_1,s_1,t_1) y_{ - \epps_i}( r_2,s_2,t_2 ) &=& y_{
-\epps_i}( r_1+r_2,s_1+s_2, t_1 + t_2 +  s_1r_2 )\end{array}$$
Finally, the double root subgroup $Y_{\pm 2\epps_i}$ is the subgroup of $Y_{\pm \epps_i}$
consisting of all elements of the form $y_{\pm \epps_i}((0,0,t))$ where
$t \in R$.

Now, calculating $Z_\alpha=Y_\alpha^{\la q\ra}$ we obtain that
$$\begin{array}{lll} Z_{\epps_i-\epps_j}&=&\{z_{\epps_i-\epps_j}(r)=x_{e_i-e_j}(r)
x_{e_{\bar j}-e_{\bar i}}(-r^*):\ r\in R\}\ \mbox{ for } i<j,\\ &&\\
Z_{\epps_i-\epps_j}&=&\{z_{\epps_i-\epps_j}(r)=x_{e_i-e_j}(-r^*)
x_{e_{\bar j}-e_{\bar i}}(r):\ r\in R\}\ \mbox{ for } i>j,\\ &&\\
Z_{\epps_i+\epps_j }&=&\{z_{ \epps_i+\epps_j }(r)=x_{ e_i-e_{\bar j}
}(r) x_{ e_{j}-e_{\bar i} }(-r^*): \ r\in R\}\ \mbox{ for } i<j,\\ &&\\
Z_{-(\epps_i+\epps_j) }&=&\{z_{ -(\epps_i+\epps_j) }(r)=x_{ -e_i+e_{\bar j}
}(-r^*) x_{ -e_{j}+e_{\bar i} }(r): \ r\in R\}\ \mbox{ for } i<j,\\
&&\\Z_{\epps_i}&=&\{z_{\epps_i}(r,t) =
x_{e_i-e_{n+1}}(r) x_{e_{n+1}-e_{\bar i}}(-r^*)
x_{e_i-e_{\bar i}}(t):\ r,t\in R,\
rr^*=t+t^*\},\\&&
\\ Z_{-\epps_i}&=&\{z_{-\epps_i}(r,t)
=x_{-e_i+e_{n+1}}(-r^*) x_{-e_{n+1}+e_{\bar i}}(r)
 x_{-e_i+e_{\bar
i}}(-t^*):\ r,t\in R,\ rr^*=t+t^*\},
 \\ &&\\ Z_{\pm 2\epps_i}&=&\{
z_{\pm \epps_i}(0,t)\in Z_{\pm \epps_i}\}.
\end{array}$$
Clearly,
\begin{itemize}
\item[] $Z_{\gamma}\cong (R,+)$ if $\gamma$ is a long root and
\item[] $Z_{\gamma}\cong (\Asym(R),+)$ if $\gamma$ is a double root.
\end{itemize}
Define $$P(R,*)=\{(r,t)\ : r,t\in R \mbox{ and } t+t^*=rr^*\}.$$
and introduce the group structure on $P(R,*)$ by setting
$$(r_1,t_1)(r_2,t_2)=(r_1+r_2,t_1+t_2+r_2 r_1^*).$$
Then for any short root $\gamma\in\Psi$, the root
subgroup $Z_{\gamma}$ is canonically isomorphic to $P(R,*)$
via the map $(r,t)\mapsto z_{\gamma}(r,t)$. Thus, the subgroup
$Z_{\gamma}$ is usually not abelian, and there is a natural injection
$Z_{\gamma}/Z_{2\gamma} \to (R,+)$ which need not be an
isomorphism.

Applying  Proposition~\ref{zgrading} we obtain that $\{Z_{\gamma}\}_{\gamma\in\Psi}$ is a grading.
The corresponding graded cover $\widehat{G^Q}$ will be denoted by $\St_{BC_n}(R,*)$.

Below we list the non-trivial commutation relations between the positive root subgroups
(again the the remaining relations are similar).
 \begin{align*}
&(E1)&& [z_{\epps_i-\epps_j}(r), z_{\epps_j-\epps_k}(s)]=z_{\epps_i-\epps_k}(rs) \ \mbox{ for } i<j<k&\\
&(E2)&& [z_{\epps_i-\epps_j}(r), z_{\epps_i+\epps_j}(s)]=z_{\epps_i}( 0,sr^* - rs^* )\ \mbox{ for } i<j&\\
&(E3)&& [z_{\epps_j}( r,t) , z_{\epps_i-\epps_j}(s)]=z_{\epps_i}( -sr,sts^*) z_{\epps_i+\epps_j}(-st)\ \mbox{ for }i<j&\\
&(E4)&& [z_{\epps_{i}}((r,t)),
z_{\epps_j}( s,q) ]=z_{\epps_i+\epps_j}(-rs^*)\ \mbox{ for }i<j&\\
&(E5)& &[z_{ \epps_i-\epps_{j}}(r), z_{\epps_j+\epps_k}(s)]=\left \{
\begin{array}{ll}
z_{\epps_i+\epps_k}(rs) &
\mbox{ for }   i<j<k \\
  z_{\epps_i+\epps_k}(sr^*)
 & \mbox{ for } k<i<j  \\
 z_{\epps_i+\epps_k}( -rs^* )
& \mbox{ for } i<k<j  \end{array}\right .&
\end{align*}

As in Example~1 we can construct a family of generalizations of $\St_{BC_n}(R,*)$,
this time by decreasing the short root subgroups.
Let $I$ be a left ideal of $R$. Define
$$P(R,I,*)=\{(r,t)\in P(R,*) : r\in I\}= \{(r,t)\ : r\in I, t\in R \mbox{ and } t+t^*=rr^*\}.$$

For each $\gamma\in BC_n$ we put
$$Z_{I,\gamma}=\left\{
\begin{array}{ll}
Z_{\gamma}& \mbox{ if } \gamma \mbox{ is a long or a double root}\\
\{z_\gamma(r,t):\ (r,t)\in P(R,I,*)\} & \mbox{ if } \gamma \mbox{ is a short root}.
\end{array}
\right.$$

We define $\St_{BC_n}(R,*,I)$ to be the graded cover of the subgroup
of $\St_{BC_n}(R,*)$ generated by $\cup_{\gamma\in BC_n} Z_{I,\gamma}$.
\begin{Observation} The group $\St_{BC_n}(R,*,\{0\})$ is isomorphic to
$\St_{C_n}^{1}(R,*)$.
\end{Observation}

\begin{Proposition}
\label{propT_ex2} Let $R$ be a finitely generated ring with involution $*$.
Assume that  $\{r\in I:\ \exists t\in R, rr^*=t+t^*\}$
is finitely generated as a left ideal and $\Asym(R)$ is finitely
generated as a form parameter of $(R,*,1)$.
The following hold:
\begin{itemize}
\item[(a)]
The group
$\St_{BC_n}(R,*,I)$ has property $(T)$ for any $n\geq 3$.
\item[(b)]  Assume in addition that there exists an invertible antisymmetric
element $\mu\in Z(R)$, and $R$ is a finitely generated module over a ring generated  by a finite set of  symmetric elements.
Then the group $\St_{BC_2}(R,*,I)$ has property $(T)$.
\end{itemize}
\end{Proposition}
\begin{proof} We shall prove (a) and (b) simultaneously.
The fact that the grading is strong in both cases is verified as in Lemma~\ref{grading:ex1},
so we only need to check relative property~$(T)$. Observe that the
set $\Psi$ of long and double roots in $BC_n$ is a weak subsystem
of type $C_n$, so the corresponding root subgroups generate a quotient of $\St_{C_n}^1(R,*)$
(this is proved similarly to Proposition~\ref{subgroups}).
Hence relative property~$(T)$ for long and double root
subgroups follows directly from Proposition~\ref{prop:ex1}(a) if $n\ge 3$. Note that the existence of
an invertible antisymmetric element $\mu\in Z(R)$ implies that $\St^{1}_{C_2}(R,*)\cong \St^{-1}_{C_2}(R,*)$ by
Observation~\ref{symisoasym}.
Thus, relative property~$(T)$ in the case  $n=2$ follows from Proposition~\ref{prop:ex1}(b).

Finally, we claim that every short root subgroup lies in a bounded
product of fixed conjugates of long and double root subgroups.
This follows easily from relations (E3) and the fact that $\{r\in
I:\ \exists t\in R, rr^*=t+t^*\}$ is finitely generated as a left
ideal.
\end{proof}

The result of Proposition~\ref{propT_ex2} is not entirely satisfactory since its hypotheses
may be hard to verify in specific examples. Things becomes much easier under the additional
assumption that $R$ contains a central element $a$ such that $a+a^*=1$:

\begin{Lemma}
\label{lemma:BCn_centralelement}
Assume that there exists $a\in Z(R)$ such that $a+a^*=1$. The following hold:
\begin{itemize}
\item[(1)] $\Sym(R)=\Sym^{\min}(R)$ and $\Asym(R)=\Asym^{\min}(R)$.
\item[(2)] $P(R,*,I)=\{(r,rar^*+t): r\in I, t\in \Asym(R)\}$.
In particular, the set $\{r\in I:\ \exists t\in R, rr^*=t+t^*\}$ is equal to $I$.
\end{itemize}
\end{Lemma}
\begin{proof} (1) For any $x\in \Sym(R)$ we have $x=xa+xa^*=xa+x^*a^*=
xa+(xa)^*\in \Sym^{\min}(R)$ where the last equality holds since $a$ is central.
The equality $\Asym(R)=\Asym^{\min}(R)$ is proved similarly.

(2) By direct computation, any element of the form $(r,rar^*+t)$ with $r\in I, t\in \Asym(R)$,
lies in $P(R,*,I)$. Conversely, given $(r,u)\in P(R,*,I)$, we can write $u=rar^*+t$
for some $t$, and then we must have $t+t^*=0$.
\end{proof}

Thanks to this lemma, we obtain the following variation of Proposition~\ref{propT_ex2}:
\begin{Proposition}
\label{propT_ex2new} Let $R$ be a finitely generated ring with involution $*$.
Assume that there exists $a\in Z(R)$ such that $a+a^*=1$, and let
$I$ be a finitely generated left ideal of $R$.
The following hold:
\begin{itemize}
\item[(a)]
The group
$\St_{BC_n}(R,*,I)$ has property $(T)$ for any $n\geq 3$.
\item[(b)] Assume in addition that there exists an invertible antisymmetric  element
$\mu\in Z(R)$, and $R$ is a finitely generated module over a ring generated  by a finite set of  symmetric elements.
Then $\St_{BC_2}(R,*,I)$ has property $(T)$.
\end{itemize}
\end{Proposition}

\subsection {Twisted groups of types ${}^2 D_n$  and ${}^{2,2}A_{2n-1}$}
Recall that the next family on our agenda were the Steinberg covers
of the twisted Chevalley groups of type ${}^2 D_{n}$ ($n\geq 4$).
These groups can be constructed using our general twisting procedure by
taking $G=\St_{D_n}(R)$, where $R$ is a commutative ring endowed with
involution $\sigma$ and $Q\subseteq \Aut(G)$ the subgroup of order $2$
generated by $Dyn_{\sigma}$, the composition of the ring automorphism
$\phi_{\sigma}$ and the Dynkin involution of $D_n$. However, we shall
present a more general construction, making use of Observation~\ref{ex1degenerate}(2).

Recall that the Steinberg group $\St_{D_n}(R)$, for $R$ commutative,
was realized as the group $\St_{C_n}^{1}(R,*,\{0\})$ where $*$ is the trivial
involution. It turns out that if we start with any ring $R$ (not necessarily
commutative) endowed with an involution $*$ and an automorphism $\sigma$
of order $\leq 2$ which commutes with $*$, then the analogous twisting
on $\St_{C_n}^{1}(R,*)$ can be constructed.
\vskip .2cm

{\bf Example 3: \it Steinberg groups $\St^{1}_{BC_{n}}(R,*,\sigma)$.}
Let $R$ be a ring endowed with an involution $*$ and
an automorphism $\sigma$ of order $\leq 2$ which commutes with $*$.
The fixed subring of $\sigma$ will be denoted by $R^{\sigma}$.
In this example we will construct the group
$\St_{BC_{n}}^{1}(R,*,\sigma)$ graded by the root system $BC_{n}$.

Let $\Phi=C_{n+1}$ and $G=\St^{1}_{C_{n+1}}(R,*)$, the group constructed
in Example~1 with $\zetta=1$. Denote the roots of $\Phi$
by $\pm \epps_i\pm\epps_j$ and $\pm 2\epps_i$, and let $\{Z_\gamma\}_{\gamma\in \Phi}$
be the grading of $G$ constructed in Example~1.

Let $\rho$ be the automorphism of $\oplus_{i=1}^{n+1} \R\epps_i$ given by
\vskip .1cm
\centerline{ $\rho(\eps_i)=\eps_i$ for $1\leq i\leq n$ and $\rho(\eps_{n+1})=-\eps_{n+1}$.}
\vskip .1cm
\noindent Clearly $\rho$ stabilizes $\Phi$.
We claim that there exists an automorphism $q=q_{\sigma}\in \Aut(G)$ of order $2$ such that
\begin{equation}
\label{defofq}
q(z_{\gamma}(r))=z_{\rho(\gamma)}(\pm\sigma(r))\mbox{ for all }\gamma\in\Phi, r\in R
\end{equation}
(for some choice of signs). Unlike Examples~1 and 2, we cannot prove the existence of such
$q$ by referring to general results from \S~8.2. One (rather tedious) way to prove this
is first to define $q$ as an automorphism of the free product $\star_{\gamma\in\Phi}Z_{\gamma}$
(using \eqref{defofq}), and then show that for a suitable choice of signs in \eqref{defofq},
$q$ respects the defining relations of $G$ established in Example~1 and hence induces an automorphism of $G$.
However, we will also give a conceptual argument for the existence of $q$ at the end of this example.

Now let $\eta:\oplus_{i=1}^{n+1} \R\epps_i\to \oplus_{i=1}^{n} \R\alpha_i$
be the reduction given by $\eta(\epps_i)=\alpha_i$ for $i\leq n$
and $\eta(\epps_{n+1})=0$. It is clear that $\eta$ is $q$-invariant
and the induced root system $\Psi=\eta(\Phi)\setminus \{0\}=
\{\pm \alpha_i\pm \alpha_j\}\cup\{\pm \alpha_i\}\cup\{\pm 2\alpha_i\}$ is of type $BC_{n}$.

Let $\{W_{\alpha}\}_{\alpha\in\Phi}$ be the $q$-invariants of the $\Phi$-grading of $G$.
By Proposition~\ref{zgrading} $\{W_{\alpha}\}$ is a grading, and thus
we can form the graded cover $\widehat{G^{\la q \ra}}$ which will be denoted
by $\St_{BC_{n}}^{1}(R,*,\sigma)$.

An easy calculation shows that
$$\begin{array}{lll}
 W_{\pm \alpha_i\pm \alpha_j }& =& \{w_{\pm \alpha_i\pm \alpha_j }(r)=z_{\pm \epps_i \pm \epps_j }(r) : \ r \in R^\sigma\}  \\ &&\\
W_{  \pm \alpha_i}&=&\{w_{\pm \alpha_i}(r,t)=z_{ \pm(
\epps_i-\epps_{n+1}) }(r)z_{\pm(\epps_i+\epps_{n+1}) }(\sigma(r))z_{ \pm
2\epps_i}(t) :\ t\in \Asym(R ), \ t-r\sigma(r^*)\in R^\sigma\}\ \\&&\\
W_{\pm 2\alpha_i}&=&\{w_{\pm 2\alpha_i}(t)=z_{\pm 2\epps_i}(t) :\
t\in
\Asym(R^\sigma)\}=\{w_{\pm \alpha_i}(0,t) :\
t\in \Asym(R^\sigma)\}\\ &&
\end{array}$$
Thus, $W_{\gamma}\cong (R^{\sigma},+)$ if $\gamma$ is a long root, and
$W_{\gamma}\cong \Asym(R^{\sigma},+)$ if $\gamma$ is a double root. Let
$$Q(R,*,\sigma)=\{(r,t): t\in \Asym(R)\mbox{ and } t-r\sigma(r^*)\in R^\sigma\},$$
and define the group structure on $Q(R,*,\sigma)$ by setting
$$(r_1,t_1)\cdot (r_2,t_2)=(r_1+r_2,t_1+t_2+r_2\sigma(r_1)^*-\sigma(r_1)r_2^*).$$
It is straightforward to check that if $\gamma$ is a short root, $W_{\gamma}$
is isomorphic to $Q(R,*,\sigma)$ via the map $(r,t)\mapsto w_{\gamma}(r,t)$.

The commutation relations between the positive root
subgroups of the grading $\{W_{\alpha}\}$ are as follows.
\begin{align*}
&(E1)&& [w_{\alpha_i-\alpha_j}(r), w_{\alpha_j-\alpha_k}(s)]=w_{\alpha_i-\alpha_k}(rs)\ \mbox{ for }i<j<k&\\
&(E2)&& [w_{\alpha_i-\alpha_j}(r), w_{\alpha_i+\alpha_j}(s)]=w_{\alpha_i}(0,sr^* - rs^*)\ \mbox{ for }i<j&\\
&(E3)&& [w_{\alpha_j}(r,t), w_{\alpha_i-\alpha_j}(s)]=w_{\alpha_i}(-sr,sts^*)w_{\alpha_i+\alpha_j}(s(r\sigma(r^*)-t) )\
\mbox{ for } i<j&\\
&(E4)&& [w_{\alpha_{i}}(r,t),
w_{\alpha_j}(s,q)]=w_{\alpha_i+\alpha_j}(-r\sigma(s^*)-\sigma(r)s^*)\
\mbox{ for } i<j&\\
&(E5)& &[w_{ \alpha_i-\alpha_{j}}(r),
w_{\alpha_j+\alpha_k}(s)]=\left \{
\begin{array}{ll}
w_{\alpha_i+\alpha_k}(rs) &
 \mbox{ for }  i<j<k \\
  w_{\alpha_i+\alpha_k}(sr^*)
 &\mbox{ for } k<i<j  \\
 w_{\alpha_i+\alpha_k}( -rs^* )
&\mbox{ for } i<k<j\ \end{array}\right .&
\end{align*}
\noindent {\bf Variations of the groups $\St_{BC_{n}}(R,*,\sigma)$.}
Let $I\subseteq R$ be a left $R^\sigma$-submodule and $J\subseteq \Asym(R)$ a form parameter of $(R,*,1)$.
Then we can define the group $\St^{1}_{BC_{n}}(R,*,\sigma, I, J)$ by decreasing the short and double root subgroups.
Define $$Q(R,*,\sigma,I,J)=\{(r,t)\in Q(R,*,\sigma): r\in I, t\in J\}.$$
For a root $\alpha\in BC_{n}$ we put

 $$ W_{I,J,\alpha}=\left\{
\begin{array}{ll}
W_{\alpha}& \mbox{ if } \alpha \mbox{ is a long root}\\
\{w_\alpha(r,t): (r,t)\in Q(R,*,\sigma,I,J) \} & \mbox{ if } \alpha \mbox{ is a short root} \\
\{w_\alpha(t)\in W_\alpha:\ t\in J\cap R^\sigma\} & \mbox{ if } \alpha \mbox{ is a double root}\\
\end{array}
\right.
$$
It is straightforward to check that $\{W_{I,J,\alpha}\}_{\alpha\in BC_n}$ is a grading.
We define $$\St_{BC_{n}}^1(R,*,\sigma,I,J)$$ to be the graded cover of the subgroup
of $\St_{BC_{n}}^1(R,*,\sigma)$ generated by $\cup_{\alpha\in BC_{n}} W_{I,J,\alpha}$.
\vskip .12cm

Now assume that $R$ is commutative, so that the identity map $\mathbf{id}$
is an involution. Then $J=\{0\}$ is a valid form parameter of $(R,\mathrm{id},1)$, and
the double root subgroups of $\St^{1}_{BC_{n}}(R,\mathrm{id},\sigma, I,\{0\})$ are
trivial. Hence we obtain a group graded by a system of type
$B_{n}$. This group will be denoted by $\St_{B_{n}}(R,\sigma,I)$.

We let $\St_{B_{n}}(R,\sigma)=\St_{B_{n}}(R,\sigma,R)=\St^{1}_{BC_{n}}(R,\mathrm{id},\sigma, R,\{0\})$
This is the Steinberg  cover for the twisted Chevalley group of type
$^2 D_{n+1}$ over $R$, which we discussed at the beginning of this example.
\vskip .12cm

We now state a sufficient condition for the groups
$\St^{1}_{BC_{n}}(R,*,\sigma, I,J)$ to have property $(T)$.

\begin{Proposition}
\label{propT_BCn1}
Assume that
\begin{itemize}
\item[(i)] $R^{\sigma}$ is finitely generated as a ring
\item[(ii)] $\{r\in I:\ \exists\, t\in J,
t-r\sigma(r^*)\in  R^\sigma\}$ is finitely generated as an
$R^\sigma$-module
\item[(iii)] $J\cap R^\sigma$ is finitely generated as a form parameter
of $(R^{\sigma},*,1)$.
\end{itemize}
Then the group $\St^1_{BC_{n}}(R,*,\sigma,I,J)$
has property $(T)$ for any $n\geq 3$.
\end{Proposition}
\begin{proof} The proof is analogous to that of Proposition~\ref{propT_ex2}(a).
\end{proof}

If $R$ is commutative, involution $*$ is trivial and $J=\{0\}$, the set defined in (ii) above
coincides with $I$ (since $r\sigma(r)$ always lies in $R^{\sigma}$), and condition (iii)
is of course vacuous. Thus, as a special case of Proposition~\ref{propT_BCn1}, we have
the following:

\begin{Proposition}
\label{propT_Bn}
Assume that
\begin{itemize}
\item[(i)] $R^{\sigma}$ is finitely generated as a ring
\item[(ii)] $I$ is finitely generated as an $R^\sigma$-module
\end{itemize}
Then the group $\St_{B_{n}}(R,\sigma,I)$ has property $(T)$ for any $n\geq 3$.
\end{Proposition}
\vskip .2cm

As in Example~2, the hypotheses necessary to prove property $(T)$ can be simplified
in the presence of a nice element. This time we wish to assume that $R$ contains
a (not necessarily central) element $a$ such that $a+\sigma(a^*)=1$.

\begin{Lemma}
\label{lemma:BCn1_niceelement} Let $R,*,\sigma, I$ and $J$ be as above, and
suppose that there exists $a\in R$ is such that $a+\sigma(a)^*=1$. Then
$$Q(R,*,\sigma, I,J)=\{(r,ra\sigma(r^*)-(ra\sigma(r^*))^*+t): r\in I, t\in J\}.$$
In particular, the set $\{r\in I:\ \exists\, t\in J, t-r\sigma(r^*)\in  R^\sigma\}$
(appearing in condition (ii) of Proposition~\ref{propT_BCn1}) is equal to $I$.
\end{Lemma}
\begin{proof} The proof of this lemma is analogous to that of Lemma~\ref{lemma:BCn_centralelement}. \end{proof}
\vskip .2cm

\noindent {\bf Another definition of the groups $\St_{BC_{n}}^1(R,*,\sigma)$.}
There is a less intuitive, but in some sense more convenient, way
to construct the groups $\St^{1}_{BC_{n}}(R,*,\sigma)$. The construction
we described uses the twist by $q_{\sigma}$ on the group $\St^{1}_{C_{n+1}}(R,*)$
which, in turn, was itself constructed using the twist by $Dyn_*$ on
$\St_{A_{2n+1}}(R)$. It is easy to see that $\St^{1}_{BC_{n}}(R,*,\sigma)$
can also be obtained directly from $\St_{A_{2n+1}}(R)$ as follows.

Let $\pi'$ be the permutation $(n+1,n+2)$ and
$\tau$ the automorphism of $\St_{A_{2n+1}}(R)$ defined by
$$\tau(x_{e_i-e_j}(r))=x_{e_{\pi'(i)}-e_{\pi'(j)}}(\sigma(r)).$$
Note that $\tau$ commutes with $Dyn_*$, and let $Q$ be the group
generated by $\tau$ and $Dyn_*$ (so that $Q\cong \Z/2\Z\times \Z/2\Z$).
Then $\St^{1}_{BC_{n}}(R,*,\sigma)$ can be obtained from
$\St_{A_{2n+1}}(R)$  using the twist by $Q$. One advantage of this approach
is that the existence of the automorphism $q_{\sigma}$ defined above
follows automatically, without case-by-case verification.
\vskip .1cm
{\bf Summary of Examples~1-3.} For the reader's convenience
below we list all the twisted Steinberg groups constructed in Examples~1-3,
including the key special cases and relations between them. In all examples,
$n\geq 2$ is an integer, $R$ is a ring and $*$ is an involution on $R$.
\vskip .1cm

\noindent 1. The groups $\St^{\omega}_{C_n}(R,*,J)$ where $\omega$ is an element
of $U(Z(R))$ and $J$ is a form parameter of $(R,*,\omega)$.

\noindent \emph{Special cases:}
\begin{itemize}
\item[(i)] $\St^{\omega}_{C_n}(R,*)=\St^{\omega}_{C_n}(R,*,R)$;
\item[(ii)] $\St_{C_n}(R)=\St^{-1}_{C_n}(R,id)$ where $R$ is commutative;
\item[(iii)] $\St_{D_n}(R)=\St^{1}_{C_n}(R,id,\{0\})$ where $R$ is commutative.
\end{itemize}
\vskip .1cm

\noindent 2. The groups $\St_{BC_n}(R,*,I)$ where $I$ is a left ideal of $R$.

\noindent \emph{Special cases:}
\begin{itemize}
\item[(i)] $\St_{C_n}^1(R,*)=\St_{BC_n}(R,*,\{0\})$.
\end{itemize}

\vskip .1cm

\noindent 3. The groups $\St^{1}_{BC_n}(R,*,\sigma,I,J)$ where $\sigma$ is an automorphism
of order $\leq 2$ commuting with $*$, $\,I\subseteq R$ is a left $R^{\sigma}$-submodule
and $J\subseteq \Asym(R)$ is a form parameter of $(R,*,1)$.

\noindent \emph{Special cases:}
\begin{itemize}
\item[(i)] $\St_{B_n}(R,\sigma, I)=\St^{1}_{BC_n}(R,id,\sigma, I, \{0\})$ where $R$ is commutative;
\item[(ii)] $\St_{B_n}(R,\sigma)=\St_{B_n}(R,\sigma, R)$;
\item[(iii)] $\St_{B_{n}}(R)=\St_{B_n}(R,id)$.
\end{itemize}

\subsection{Further twisted examples}
In this subsection we prove property $(T)$ for twisted Steinberg
groups of type   $^3 D_4$ and $^2 E_6$. In both examples $R$ is a
commutative ring and $\sigma: R\to R$ is a finite order automorphism.

{\bf Example 4:  \it Steinberg groups of type $^3 D_4$.}
The group in this example will be
denoted by $\St_{G_2}(R,\sigma)$ and is graded by the root system
$G_{2}$. It is the Steinberg  cover for the twisted Chevalley
group of type $^3 D_4$ over $R$.

We use the standard realization of $D_n$ in $\R^n$: $D_n=\{\pm
e_i\pm e_j: \ 1\le i\ne j\le n\}$. For a suitable choice of Chevalley basis,
the commutation relations in $\St_{D_n}(R)$ are as follows:
\begin{align*}
&[x_{e_i-e_j}(r), x_{e_j-e_k}(s)]=x_{e_i-e_k}(rs)\\
&[x_{e_i-e_j}(r), x_{e_j+e_k}(s)]=\left \{
\begin{array}{ll}
x_{e_i+e_k}(rs) &
   \textrm{if $i,j<k$ or $i,j>k$} \\
 x_{e_i+e_k}( -rs )
& \textrm{if $j>k>i$ or $i>k>j$} \end{array}\right .\\
&[x_{e_i-e_j}(r), x_{-e_i-e_k}(s)]=\left \{
\begin{array}{ll}
x_{-e_j-e_k}(-rs) &
   \textrm{if $i,j<k$ or $i,j>k$} \\
 x_{-e_j-e_k}( rs )
& \textrm{if $j>k>i$ or $i>k>j$} \end{array}\right .\\
&[x_{e_k+e_i}(r), x_{-e_j-e_k}(s)]=\left \{
\begin{array}{ll}
x_{ e_i-e_j}( rs) &
      \textrm{if $i,j<k$ or $i,j>k$} \\
 x_{e_i-e_j}( -rs )
&\textrm{if $j>k>i$ or $i>k>j$}\end{array}\right .
\end{align*}

We realize $G_2$ as the set of vectors $\pm (\epps_i-\epps_j)$ and $\pm (2\epps_i-\epps_j-\epps_k)$
where $i,j,k\in \{1,2,3\}$ are distinct. We let
$$\alpha= 2\epps_2-\epps_1-\epps_3 \quad \mbox{ and }\quad \beta=\epps_1-\epps_2$$
and take $\{\alpha,\beta\}$ as our system of simple roots.

Let $\Phi=D_{4}$ (with standard realization) and $G=\St_{\Phi}(R)$.
Let $\sigma: R\to R$ be an automorphism satisfying $\sigma^3=\mathrm{id}$ and $\pi$ the
isometry of $\R^4$ represented by the following matrix with
respect to the basis $\{e_1,e_2,e_3,e_4\}$: $$\frac 12\left
(\begin{array}{cccc} 1& 1&1&1\\ 1& 1&-1&-1\\ 1& -1&1&-1\\ -1&
1&1&-1\end{array}\right). $$ Then it is clear that $\pi$ stabilizes
$D_4$. Let $q=\phi_\sigma \lambda_\pi\in \Aut(G)$ (as defined in \S~8.2).
With a suitable choice of signs in the definition of $\lambda_\pi$,
we can assume that $q$   maps $x_{e_2-e_3}(r)$
to $x_{e_2-e_3}(\sigma(r))$, $x_{e_1-e_2}(r)$ to $x_{e_3-e_4}(\sigma(r))$, 
$x_{e_3-e_4}(r)$ to $ x_{e_3+e_4}(\sigma(r))$ and $x_{e_3+e_4}(r)$ to 
$ x_{e_1-e_2}(\sigma(r))$. Then it easy to see that $q$ is an automorphism of order 3.

Define $\eta: \bigoplus_{i=1}^{4}\R e_i\to \bigoplus_{i=1}^{3}\R \epps_i$ by
$\eta(e_1)=\epps_1-\epps_3$, $\eta(e_2)=\epps_2-\epps_3$,
$\eta(e_3)=\epps_1-\epps_2$ and $\eta(e_4)=0$.
It is easy to see that the root system $\Psi=\eta(\Phi)$ is of type $G_2$.
Furthermore,
\begin{align*}
&\eta^{-1}(\beta)=\{e_1-e_2,e_3-e_4,e_3+e_4\}& & \eta^{-1}(\alpha)= \{e_2-e_3\},&\\
&\eta^{-1}(\alpha+\beta)=\{e_1-e_3,e_2-e_4,e_2+e_4\}& &\eta^{-1}( \alpha+3\beta)=\{e_1+e_3\},&\\
&\eta^{-1}( \alpha+2\beta)=\{e_1-e_4,e_1+e_4,e_2+e_3\}& &\eta^{-1}(2\alpha+3\beta)=\{e_1+e_2\}&.
\end{align*}
If $\gamma\in\Psi$ is a short root, the corresponding root
subgroup $Y_{\gamma}$ consists of elements
$$\{y_\gamma(r,s,t)=x_{\gamma_1}(r)x_{\gamma_2}(s)x_{\gamma_3}(t):\ \eta^{-1}(\gamma)=\{\gamma_1,\gamma_2,\gamma_3\}, \ r,s,t\in R\}.$$
If $\gamma\in\Psi$ is a long root, then
$$Y_{\gamma}=\{y_\gamma(r)=x_{\gamma_1}(r):\ \eta(\gamma_1)=\gamma,\ r\in
R\}.$$

If $\gamma\in\Psi$ is a short root, the corresponding root
subgroup $Z_{\gamma}=Y_\gamma^q$ is isomorphic to $(R,+)$, and if
$\gamma\in\Psi$ is a long root, then $Z_{\gamma}\cong
(R^{\sigma},+)$. Positive root subgroups can be explicitly
described as follows (we define $z_\gamma(r)$ in such a way that the relations in the non-twisted case coincide with the relations (G) from Proposition~\ref{constants}):
\begin{gather*}
Z_\beta=\{z_{\beta}(r)=x_{e_1-e_2}(r) x_{e_3-e_4}(\sigma(r))x_{e_3+e_4}(\sigma^2(r)):\ r\in R\},\\
Z_{\alpha+\beta}=\{z_{\alpha+\beta}(r)=x_{e_1-e_3}(-r) x_{e_2-e_4}(\sigma(r))x_{e_2+e_4}(\sigma^2(r)):\ r\in R\},\\
Z_{\alpha+2\beta}=\{z_{\alpha+2\beta}(r)=x_{e_1-e_4}(-r) x_{e_2+e_3}(- \sigma(r))x_{e_1+e_4}(-\sigma^2(r)):\ r\in R\},\\
Z_\alpha=\{z_{\alpha}(r)=x_{e_2-e_3}(r): r\in R^{\sigma}\} ,\quad
Z_{\alpha+3\beta}=\{z_{\alpha+3\beta}(r)=x_{e_1+e_3}(r):\ r\in
R^{\sigma}\},\\
Z_{2\alpha+3\beta}=\{z_{2\alpha+3\beta}(r)=x_{e_1+e_2}(r):\ r\in
R^{\sigma}\}.
\end{gather*}

Below we list the commutation relations between positive root subgroups which will
be used in the sequel:
\begin{align*}
&(E1)& &[z_{\alpha}(t),z_{\beta}(u)]=z_{\alpha+\beta}(tu)\cdot z_{\alpha+2\beta}(tu\sigma(u))\cdot
z_{\alpha+3\beta}(tu\sigma(u)\sigma^2(u))\cdot z_{2\alpha+3\beta}(t^2 u\sigma(u)\sigma^2(u))&\\
&(E2)&
&[z_{\alpha}(t),z_{\alpha+3\beta}(u)]=z_{2\alpha+3\beta}(tu)&
\\
&(E3)&
&[z_{\alpha+\beta}(t),z_{\beta}(u)]=z_{\alpha+2\beta}(t\sigma(u)+u\sigma(t))\cdot
z_{ \alpha+3\beta}(u\sigma(u)\sigma^2(t)+u\sigma(t)\sigma^2(u)+t\sigma(u)\sigma^2(u))\\&&&z_{2\alpha+3\beta}(t\sigma(t)\sigma^2(u)+t\sigma(u)\sigma^2(t)+u\sigma(t)\sigma^2(t))&
\end{align*}

\begin{Proposition}
\label{prop:3D4}
The group $G=\St_{G_2}(R,\sigma)$ has property $(T)$ provided
\begin{itemize}
\item[(i)] $R^{\sigma}$ is finitely generated as a ring
\item[(ii)]$R$ is a finitely generated module over $R^{\sigma}$
\end{itemize}
\end{Proposition}
\begin{proof} As usual, we need to check two things
\begin{itemize}
\item[(a)] The $\Psi$-grading of $G$ is strong at each root subgroup
\item[(b)] The pair $(G,Z_{\gamma})$ has relative $(T)$  for each $\gamma\in \Psi$
\end{itemize}

Relations (E2) imply condition (a) for the root subgroup $Z_{2\alpha+3\beta}$. Condition~(a) for the root subgroups $Z_{\alpha+3\beta}$ and $Z_{\alpha+\beta}$
follows from relations (E1) as we can take  $u=1$ and let $t$ be an arbitrary element of $R^\sigma$ in the case of $Z_{\alpha+3\beta}$ and take  $t=1$ and  let $u$ be an arbitrary element of $R$ in the case of $Z_{\alpha+\beta}$. In the non-twisted case there is no problem with $Z_{\alpha+2\beta}$ either as we can take
$u=1$ and arbitrary $t\in R$ in (E1) (in general we cannot do this as $t$ must come from $R^{\sigma}$).

To check the required property for the subgroup $Z_{\alpha+2\beta}$ in the general (twisted) case
we need to show that elements of the form $u\sigma(u)t$ with $u\in R, t\in R^{\sigma}$ span $R$.
Indeed, denote the span of those elements by $M$. Then $M$ contains all elements of $R^{\sigma}$,
in particular all elements of the form $u+\sigma(u)+\sigma^2(u)$. It also contains
all elements of the form $(u+1)(\sigma(u+1))-u\sigma(u)-1=u+\sigma(u)$. Since
$u=u+\sigma(u)+\sigma^2(u)-(\sigma(u)+\sigma(\sigma(u)))$, we are done with (a).
\vskip .1cm

We now prove (b). The subgroup of $G$ generated by long root subgroups is isomorphic to a quotient of $\St_3(R^{\sigma})$ (and $R^{\sigma}$ is finitely generated), so condition  (b) for long root subgroups
holds by Theorem~\ref{main}.  It remains to check (b) for short root subgroups.
We shall show that any short root subgroup lies in
a bounded product of long root subgroups and finite sets.
By symmetry, it is enough to establish this property for $Z_{\alpha+2\beta}$.
For any set $S$ we put $Z_{\alpha+2\beta}(S)=\{z_{\alpha+2\beta}(s):\ s\in S\}$.

Put $A=\{u\sigma(u):\ u\in R\}$. In the proof of (a) we showed that $A$ generates $R$ as an $R^{\sigma}$-module.
Thus by our assumption there is a finite subset $U\subseteq A$ which generates $R$ as an $R^{\sigma}$-module.
Let $S$ be a finite generating set of $R^{\sigma}$.

Now fix $s\in S$ and $u\in U$, and let $t\in R^{\sigma}$ be arbitrary. Similarly to the case of non-twisted
$G_2$, if we calculate the quantity $[z_{\alpha}(t), z_{\beta}(su)][z_{\alpha}(ts), z_{\beta}(u)]^{-1}$
using relations (E1), we obtain that $\{z_{\alpha+2\beta}(t(s^2-s)u\sigma(u)): t\in R^{\sigma}\}$
lies in a bounded product of long root subgroups and fixed elements of short root subgroups.
Similarly, this property holds for the set $\{z_{\alpha+2\beta}(2t u\sigma(u)): t\in R^{\sigma}\}$ and
hence also for the set $Z_{\alpha+2\beta}(IU)$ where $I=2 R^{\sigma} +\sum_{s\in S} (s^2-s)R^{\sigma}$
and $IU=\{\sum_{u\in U} r_u u : r_u\in I\}$.

As we have already seen (in the case of non-twisted $G_2$), $I$ is a finite index ideal of $R^{\sigma}$
whence $IU$ has finite index in $R^{\sigma}U=R$. Hence $Z_{\alpha+2\beta}$ can be written as
a product of $Z_{\alpha+2\beta}(IU)$ and some finite set. This finishes the proof of (b).
\end{proof}

{\bf Example 5: \it Steinberg groups of type $^2 E_6$/twisted Steinberg group
of type $F_4$.} The group in this example will be denoted
by $\St_{F_4}(R,\sigma)$ and is graded by the root system $F_{4}$. We will only sketch the details of the construction.

Let $\Phi=E_{6}$ and $G=\St_{\Phi}(R)$. Let $\{\alpha_1,\ldots, \alpha_6\}$
be a system of simple roots of $\Phi$ ordered as shown below:

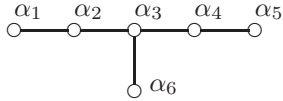
\centerline{
\setlength{\unitlength}{0.4mm}
\begin{picture}(100,40)(40,20)
\put(50,45){\circle{4}}
\put(70,45){\circle{4}}
\put(90,45){\circle{4}}
\put(110,45){\circle{4}}
\put(130,45){\circle{4}}
\put(90,25){\circle{4}}
\put(52,45){\line(1,0){16}}
\put(72,45){\line(1,0){16}}
\put(92,45){\line(1,0){16}}
\put(112,45){\line(1,0){16}}
\put(90,27){\line(0,1){16}}
\put(50,50){$\alpha_1$}
\put(70,50){$\alpha_2$}
\put(90,50){$\alpha_3$}
\put(110,50){$\alpha_4$}
\put(130,50){$\alpha_5$}
\put(95,25){$\alpha_6$}
\end{picture}
}

Let $\sigma: R\to R$ be an automorphism of order $\leq 2$, let
$\pi$ be the automorphism of $\Phi$ given by $\pi(\alpha_i)=\alpha_{6-i}$ for $i=1,2,4,5$
and $\pi(\alpha_i)=\alpha_{i}$ for $i=3,6$, and let $q=\lam_{\pi}\phi_{\sigma}\in \Aut(G)$.
With a suitable choice of signs in the definition of $\lam_{\pi}$,
we can assume that $q$ is an automorphism of $G$ of order $2$ and is given by
$$q(x_{\pm \alpha_i}(r))=\left\{
\begin{array}{ll}
x_{\pm \alpha_{6-i}}(-\sigma(r)) & \mbox{ for } i=1,2,4,5\\
x_{\pm \alpha_{i}}(-\sigma(r)) & \mbox{ for } i=3,6.\\
\end{array}
\right.
$$

Let $V$ be the $\R$-span of $\Phi$, and consider the induced action
of $q$ on $V$ (so that $q(\alpha_i)=\pi(\alpha_i)$ for each $i$).
Let  $W=V^{q}$ be the subspace of $q$-invariants and
define $\eta:V\to W$ by \eqref{eq:invariant2}, that is,
$$\eta(v)=\frac{v+qv}{2}.$$
It is easy to see that $\Psi=\eta(\Phi)$ is a root
system of type $F_4$ with base $\beta_1,\beta_2,\beta_3,\beta_4$
where $\beta_1=\frac{\alpha_1+\alpha_5}{2},
\beta_1=\frac{\alpha_2+\alpha_4}{2}, \beta_3=\alpha_3\mbox{ and }
\beta_4=\alpha_6$:

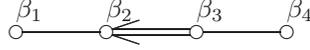
\centerline{
\setlength{\unitlength}{0.4mm}
\begin{picture}(80,20)(40,20)
\put(40,25){\circle{4}}
\put(70,25){\circle{4}}
\put(100,25){\circle{4}}
\put(130,25){\circle{4}}
\put(42,25){\line(1,0){26}}
\put(72,24){\line(1,0){26}}
\put(72,26){\line(1,0){26}}
\put(72,26){\line(3,1){9}}
\put(72,24){\line(3,-1){9}}
\put(102,25){\line(1,0){26}}
\put(40,30){$\beta_1$}
\put(70,30){$\beta_2$}
\put(100,30){$\beta_3$}
\put(130,30){$\beta_4$}
\end{picture}
}

As in Example 4, if $\gamma\in \Psi$ is a short root,
the root subgroup $Z_{\gamma}$ is isomorphic to $(R,+)$, and
if $\gamma\in \Psi$ is a long root, then $Z_{\gamma}\cong (R^{\sigma},+)$.

\begin{Proposition}
\label{prop:2E6} Assume that both $R$ and $R^{\sigma}$ are finitely generated
rings. Then the group $\St_{F_4}(R,\sigma)$ has property $(T)$.
\end{Proposition}
\begin{proof}
The proof is identical to the case of classical (non-twisted) $F_4$.
\end{proof}

\subsection{Proof of relative property $(T)$ for type $C_2$}
In this subsection we prove relative property $(T)$ for the pairs
$(G,Z_{\gamma})$ where $G$ is a twisted Steinberg group of the form
$\St_{C_2}^{-1}(R,*,J)$ for a suitable triple $(R,*,J)$
and $Z_{\gamma}$ is one of its root subgroups.
The main ingredient in the proof is Theorem~\ref{relKazhdan}.

\begin{Proposition}
\label{involution_gen} Let $R$ be a ring with involution $*$
and $J$ a form parameter of $(R,*,-1)$ containing $1_R$. Assume that
\begin{enumerate}
\item There is a finite subset $T=\{t_1,\ldots, t_d\}$ of $J$ and $a_1,\ldots,a_l\in R$ such that
$R=\sum_{i=1}^l a_iR_0$, where $R_0$ is the ring generated by $T$.
\item  $J$ is generated as a form parameter by a finite set $U=\{u_1,\ldots,u_D\}$.
\end{enumerate}
Let
$$W_1=\{w\in R : w\mbox { is a monomial in $T$ of degree }
\leq d\}\cup\{1\}$$  and $$W_2=\{w+ w^*\in R : w\mbox { is a monomial in $T$
of degree } \leq 2d\}.$$
Let
$$W_{short}=\bigcup\limits_{i=1}^l \{a_i w : w \in W_1\}$$ and
$$W_{long}=W_2\cup\bigcup\limits_{i=1}^l\{a_i w a_i^* : w\in W_2\}\cup U\cup T\cup \{t^2 : t\in T\}\cup\{1\}.$$
For a short root $\gamma\in C_2$ set $S_{\gamma}=\{z_{\gamma}(r): r\in W_{short}\cup W_{short}^*\}$, and
for a long root $\gamma\in C_2$ set $S_{\gamma}=\{z_{\gamma}(r): r\in W_{long}\}$, and
let $S=\cup_{\gamma\in C_2}S_{\gamma}$. Then
for every $\gamma\in C_2$  we have
$$\kappa_r(\St_{C_2}^{-1} (R,*,J); Z_{\gamma},S)>0.$$
\end{Proposition}

\begin{proof}
In the notations of Example~1 with $\omega=-1$, let $\alpha=2\eps_2$
and $\beta=\eps_1-\eps_2$. Then $\{\alpha,\beta\}$ is a base for $C_2$, 
and as established in Example~1, we have the following relations:

\begin{align}
\label{rel1}
[z_{\beta}(r),z_{\alpha+\beta}(s)]&=z_{\alpha+2\beta}(rs^*+sr^*)\\
\label{rel2}
[z_{\alpha}(r),z_{\beta}(s)]&=z_{\alpha+\beta}(-sr)z_{\alpha+2\beta}(srs^*)\\
\label{rel3}
[z_{-\alpha}(r),z_{\alpha+\beta}(s)]&=z_{\beta}(sr)z_{\alpha+2\beta}(srs^*).
\end{align}

Let $N=\langle Z_{\alpha+\beta}, Z_\beta,
Z_{\alpha+2\beta}\rangle$,  $S^+=\cup \{S_{\gamma}\}_{\gamma \in \pm\alpha,\alpha+\beta,\beta,\alpha+2\beta}$,
$G=\langle S^+\rangle$, $Z=Z_{\alpha+2\beta}$ and $H=Z\cap [N,G]$.
We claim that Proposition~\ref{involution_gen} follows from Lemma~\ref{condc2} below and Theorem~\ref{relKazhdan}.

\begin{Lemma} \label{condc2}The following hold:
\begin{itemize}
\item[(a)] $N$ is contained in $G$;
\item[(b)] $Z/H$ is  a group of exponent $2$ generated by
(the image of) $S_{\alpha+2\beta}$.
\end{itemize}
\end{Lemma}

Indeed, let $E$ be the subgroup of the Steinberg group $\St_2(R_0)$ generated by
$\{x_{12}(r), x_{21}(r) : r\in W_{long}\}$, and consider the semi-direct product
$E\ltimes (\oplus_{i=1}^lR_0^2)$ where $E$ acts on each copy of $R_0^2$ by right
multiplication. Since $W_{long}$ contains $T\cup \{1\}$ and $T$
generates $R_0$ as a ring, the pair $(E\ltimes (\oplus_{i=1}^lR_0^2), \oplus_{i=1}^lR_0^2)$
has relative $(T)$ by Theorem~\ref{thm:relT} and a remark after it. 

Relations \eqref{rel2} and \eqref{rel3} and Lemma~\ref{condc2}(a) imply that $(G/Z,N/Z)$
as a pair is a quotient of $(E\ltimes (\oplus_{i=1}^lR_0^2), \oplus_{i=1}^lR_0^2)$,
that is, there exists an epimorphism $\pi:E\ltimes( \oplus_{i=1}^lR_0^2) \to G/Z$ such that
$\pi(\oplus_{i=1}^lR_0^2)=N/Z$. 
Therefore, the pair $(G/Z,N/Z)$ has relative $(T)$ as well.

This result and Lemma~\ref{condc2}(b) imply that the hypotheses of Theorem~\ref{relKazhdan} hold
if we put $A=B=S^+$, $C=S_{\alpha+2\beta}$, $\,\eps=\kappa(G/Z,N/Z)$ and $\delta=\frac{1}{\sqrt{|C|}}$.
Applying this theorem we get that
$\kappa(G,N;S^+)>0$, so in particular $\kappa_r(\St_{C_2}^{-1} (R,*,J), Z_{\gamma};S)>0$
for $\gamma\in\{\beta,\alpha+\beta,\alpha+2\beta\}$. 

To prove Theorem~\ref{involution_gen} for the remaining root subgroups it is sufficient to
know that for any root $\gamma\in C_2$ there is a graded automorphism $\phi$ of
$\St_{C_2}^{-1} (R,*,J)$
which sends $Z_{\gamma}$ to $Z_{\beta}$ or $Z_{\alpha+2\beta}$ and leaves the set $S\cup S^{-1}$ invariant
(of course, replacing $S$ by $S\cup S^{-1}$ does not affect Kazhdan ratio). The existence
of such automorphism $\phi$ easily follows from the definition of the group
$\St_{C_2}^{-1} (R,*,J)$. In fact, we can choose $\phi$ such that
\begin{itemize}
\item[(i)] for any short root $\delta$ we have $\phi(z_{\delta}(r))=z_{\phi(\delta)}(r)$ for all $r\in R$
or $\phi(z_{\delta}(r))=z_{\phi(\delta)}(-r^*)$ for all $r\in R$ and
\item[(ii)] for any long root $\delta$ we have $\phi(z_{\delta}(r))=z_{\phi(\delta)}(r)$ for all $r\in J$
or $\phi(z_{\delta}(r))=z_{\phi(\delta)}(-r)$ for all $r\in J$. 
\end{itemize}
This completes the proof of Theorem~\ref{involution_gen} modulo Lemma~\ref{condc2}.
\vskip .15cm

Before proving Lemma~\ref{condc2}, we establish another auxiliary result, from which
Lemma~\ref{condc2} will follow quite easily.

\begin{Lemma}
\label{lemma:aux}
For any $r\in R$ the following hold:
\begin{itemize}
\item[(i)] $z_{\alpha+\beta}(r)\in G$ and $z_{\beta}(r)\in G$
\item[(ii)] $z_{\alpha+\beta}(r)\in [N,G]\la S_{\alpha+2\beta}\ra$
\end{itemize}
\end{Lemma}
\begin{proof}
Note that it suffices to prove both statements when $r$ is
of the form $r=a_i w$, where $w$ is a monomial in $T^+$.
Let us prove that both (i) and (ii) hold for such $r$ by induction on
$m=length(w)$.

If $m\leq d$, then $z_{\alpha+\beta}(r), z_\beta{(r)}\in   S^+\subset G$ by definition of $S^+$.
Also by \eqref{rel1} we have $z_{\alpha+\beta}(r)=z_{\alpha+2\beta}(rr^*)[z_{\alpha}(1),z_{\beta}(r)]^{-1}$,
so both (i) and (ii) hold.

Now fix $m>d$, and assume that for any monomial $w^\prime\in T$  of length less than $m$
both (i) and (ii) hold for $r=a_iw'$.

\begin{Claim} 
\label{claim_tails}
Let $q$ be some tail of $w$ with $2\leq length(q)\leq d+1$ so that
$w=pq$ for some $p$.
Then (i) and (ii) hold for $r=a_i pq=a_i w$ if and only if (i) and (ii) hold for $r=a_i pq^*$.
\end{Claim}
\begin{Remark} Note that if $q=t_{i_1}\ldots t_{i_s}$, then $q^*=t_{i_s}\ldots t_{i_1}$
is the monomial obtained from $q$ by reversing the order of letters.
\end{Remark}
\begin{proof}
Consider the element $v=p(q+q^*)$. Then
\begin{equation}
[z_{\alpha}(q+q^*),z_{\beta}(a_i p)]=z_{\alpha+\beta}(-a_i v) z_{\alpha+2\beta}(a_i p(q+q^*)p^*a_i^*).
\label{eq1:C2rel}
\end{equation}
Notice that $p(q+q^*)p^*=u+u^*$ for $u=pqp^*$. Furthermore,
$length(u)\leq 2m-2$, so we can write $u=w_1 w_2^*$ where $w_1$ and $w_2$ are
monomials of length $<m$ . Then
\begin{multline*}
z_{\alpha+2\beta}(a_i(u+u^*)a_i^*)=z_{\alpha+2\beta}(a_i w_1 (a_i w_2)^*+a_i w_2 (a_i w_1)^*)\\
=[z_{\beta}(a_i w_1),z_{\alpha+\beta}(a_i w_2)]\in
G\cap [N,G]
\quad \mbox{ by induction. }
\end{multline*}
Since $z_{\alpha}(q+q^*)\in S^+$ and
$z_{\beta}(a_i p)\in G$ by induction, from \eqref{eq1:C2rel} we get
$$z_{\alpha+\beta}(a_i v)= 
z_{\alpha+2\beta}(a_i(u+u^*)a_i^*)[z_{\alpha}(q+q^*),z_{\beta}(a_i p)]^{-1}\in G\cap [N,G].$$
Since $z_{\alpha+\beta}(a_i v)=z_{\alpha+\beta}(a_i pq)z_{\alpha+\beta}(a_i pq^*)$, we conclude that
$z_{\alpha+\beta}(a_i pq)\in G \iff z_{\alpha+\beta}(a_i pq^*)\in
G$ and $z_{\alpha+\beta}(a_i pq)\in [N,G]\la S_{\alpha+2\beta}\ra\iff
z_{\alpha+\beta}(a_i pq^*)\in [N,G]\la S_{\alpha+2\beta}\ra$.
A similar argument shows that  $z_{ \beta}(a_i pq)\in G \iff z_{ \beta}(a_i  pq^* )\in G$.
\end{proof}
By Claim~\ref{claim_tails}, in order to prove that $z_{\alpha+\beta}(a_i w), z_{\beta}(a_i w)\in G$ we are
allowed to replace $w$ by another word obtained by reversing some
tail of $w$ of length $\leq d+1$, and this operation can be
applied several times. The corresponding permutations clearly
generate the full symmetric group on $d+1$ letters, and since $T$
has $d$ elements, we can assume that $w$ has a repeated letter at the end: $w=pt^2$ where $t\in T$.
But then we have
$$[z_{\alpha}(t^2),z_{\beta}(a_i p)]=z_{\alpha+\beta}(-a_i w) z_{\alpha+2\beta}(a_i pt^2p^*a_i^*)$$ and
$$[z_{\alpha}(1),z_{\beta}(a_i pt)]=z_{\alpha+\beta}(-a_i pt) z_{\alpha+2\beta}(a_i pt^2p^*a_i^*),$$
whence
$$z_{\alpha+\beta}(a_i w)=z_{\alpha+\beta}(a_i pt)[z_{\alpha}(1),z_{\beta}(a_i pt)] [z_{\alpha}(t^2),z_{\beta}(a_i p)]^{-1}.$$
Since  $z_{\beta}(a_i pt),z_{\beta}(a_i p)\in G$
and $z_{\alpha+\beta}(a_i pt)\in G\cap [N,G]\la S_{\alpha+2\beta}\ra$ by induction,
and $z_{\beta}(a_i pt),z_{\beta}(a_i p)\in N$ and $z_{\alpha}(1), z_{\alpha}(t^2)\in G$
by definition, we conclude that $z_{\alpha+\beta}(a_i w)\in G\cap [N,G]\la S_{\alpha+2\beta}\ra$.
A similar argument shows that $z_{\beta}(a_i w)\in G$.
\end{proof}

\begin{proof}[Proof of Lemma \ref{condc2}]
By Lemma~\ref{lemma:aux}, $G$ contains the root subgroups $Z_{\beta}$ and
$Z_{\alpha+\beta}$. Hence, by relations \eqref{rel1}, $G$ contains all
elements of the form $z_{\alpha+2\beta}(r+r^*)$. Since by definition
$G$ also contains all elements of the form $z_{\alpha+2\beta}(u)$ with $u\in U$
and $U$ generates $J$ as a form parameter, we conclude that $G$ contains
$Z=Z_{\alpha+2\beta}$. This completes the proof of Lemma~\ref{condc2}(a).

From relations \eqref{rel2} we have $Z\subseteq [N,G]Z_{\alpha+\beta}$,
which together with Lemma~\ref{lemma:aux}(ii) shows
that $Z\subseteq [N,G]\la S_{\alpha+2\beta}\ra$. This proves
the second assertion of Lemma~\ref{condc2}(b).
Finally, by \eqref{rel1}, for every $r\in J$ we have $z_{\alpha+2\beta}(2r)=[z_{\beta}(r),z_{\alpha+\beta}(1)]\in [N,G]$,
which proves the first assertion of Lemma~\ref{condc2}(b).
\end{proof}
\end{proof}

 \subsection{Twisted groups of type $^2F_4$}
 Let $R$ be a commutative ring of characteristic 2 and $*:R\to R$ an injective homomorphism such that $(r^*)^*=r^2$ for any $r\in R$.
 We will use a standard realization of the root system $F_4$ inside $\R^4$:
 $$F_4=\{\pm e_i,\pm e_i\pm e_j, \frac 12(\pm e_1\pm e_2\pm e_3\pm e_4): 1\le i\ne j\le 4\}.$$
Let $\overline{F_4}$ by the root system, obtained from $F_4$ by normalizing all the roots:
 $$\overline{F_4}=\{\bar v=\frac{v}{|v|}:\ v\in F_4\}=\{\pm  e_i,\frac 1{\sqrt 2}(\pm   e_i\pm e_j), \frac 1{2}(\pm e_1\pm e_2\pm e_3\pm e_4): 1\le i\ne j\le 4\}.$$
Let $G=\St_{F_4}(R)$ and let $\{X_\gamma\}_{\gamma \in F_4}$ denote the standard grading of $G$.
For each $\bar \gamma\in \overline{F_4}$ define the subgroup $\widetilde X_{\bar \gamma}$ by
 $$\widetilde X_{\bar \gamma} =\{\widetilde x_{\bar \gamma}(r)\ :r\in R\}
\mbox{ where }\widetilde x_{\bar \gamma}(r)
=\left \{
\begin{array}{cl} x_\gamma(r) & \textrm{if $\gamma$
 is a short root}\\x_\gamma(r^*) & \textrm{if $\gamma$ is a long root}
\end{array}\right.$$
Since $R$ has a characteristic 2, it is easy to show that
 $\{\widetilde X_{\bar \gamma}\}_{\bar \gamma \in \overline{F_4}}$ is an $\overline{F_4}$-grading,
 and the elements of its root subgroups satisfy the following commutation relations:
for any $\alpha,\beta\in \overline{F_4}$ we have

$$[\widetilde x_\alpha(r),\widetilde x_\beta(s)]=\left \{ \begin{array}{cl} 1 & \textrm{if the angle between $\alpha$ and $\beta$ is $\frac{\pi}4$, $\frac{\pi}3$, or $\frac{\pi}2$}\\
\widetilde x_{\alpha+\beta}(rs) & \textrm{if the angle between $\alpha$ and $\beta$ is $\frac{2\pi}3$}\\
\widetilde x_{\sqrt 2\alpha+\beta}(r^*s) \widetilde x_{\alpha+\sqrt 2\beta}(rs^*)& \textrm{if the angle between $\alpha$ and $\beta$ is $\frac{3\pi}4$}\end{array}\right..$$
For instance, consider the case when the angle between $\alpha$ and $\beta$ is $3\pi/4$ and $\alpha$ is a long root.
Then $\{\sqrt{2}\alpha,\beta\}\subset F_4$ is a base for a subsystem of type $B_2$, and therefore by Proposition~\ref{constants}(B)
we have
$$[x_{\sqrt{2}\alpha}(r),x_{\beta}(s)]=x_{\sqrt{2}\alpha+\beta}(rs)x_{\sqrt{2}\alpha+2\beta}(rs^2)\mbox{ for all }r,s\in R.$$
(Note that the choice of Chevalley basis does not affect the relations since $R$ has characteristic 2).
Hence for any $r,s\in R$ we have
\begin{multline*}
[\widetilde x_\alpha(r),\widetilde x_\beta(s)]=[x_{\sqrt{2}\alpha}(r^*),x_\beta(s)]=
x_{\sqrt{2}\alpha+\beta}(r^* s)x_{\sqrt{2}\alpha+2\beta}(r^* s^2)\\
=x_{\sqrt{2}\alpha+\beta}(r^* s)x_{\sqrt{2}(\alpha+\sqrt{2}\beta)}((rs^*)^*)=
\widetilde x_{\sqrt{2}\alpha+\beta}(r^* s)\widetilde x_{\alpha+\sqrt{2}\beta}(rs^*).
\end{multline*}
We denote by $\widetilde G$ the graded cover of the group generated by  $\{\widetilde X_\alpha\}_{\alpha\in \overline{F_4}}$.
Since the commutation relation between elements of two root subgroups is determined entirely by the angle between
the corresponding roots, we can construct a graded automorphism of $\widetilde G$  from any isometry of the root system
$\overline {F_4}$ as follows. Let $\rho$ be an isometry of $\R^4$ which preserves $\overline {F_4}$. Then we can define an automorphism of $\widetilde G$, denoted by the same symbol $\rho$:
 $$\rho(\widetilde x _\alpha(r))=\widetilde x_{\rho(\alpha)}(r).$$
 Let $q$ be the isometry represented by the following matrix with respect to the basis $\{e_1, e_2, e_3, e_4\}$:
 $$\frac{1}{\sqrt 2}\left ( \begin{array}{cccc}
 1&1 & 0 & 0\\ 1&-1 & 0 & 0\\0&0 & 1&1\\ 0&0& 1& -1\end{array}\right ), $$ and $\tau$ be the isometry represented by the following matrix with respect to the basis $\{e_1, e_2, e_3, e_4\}$:
 $$\frac{1}{\sqrt 2}\left ( \begin{array}{cccc}
 1&0 & -1 & 0\\ 0&1& 0 &-1\\1&0 & 1&0\\ 0&1&0& 1\end{array}\right ). $$  Then $q$ has order 2 and $\tau$ has order 8. Moreover they commute.

 {\bf Example 6: \it Twisted groups of type $^2F_4$.} The group in this example is denoted by $\St_{^2F_4}(R)$ and is obtained from $\widetilde G$ using the twist by the automorphism $q$.
 Define $$\eta: \bigoplus_{i=1}^{4}\R e_i\to   \bigoplus_{i=1}^{2}\R \epps_i$$ by
$\eta(e_1)=(1+\sqrt 2)\epps_1$, $\eta(e_2)=\epps_1$,
$\eta(e_3)=(1+\sqrt 2)\epps_2$ and $\eta(e_4)=\epps_2$.
It is easy to see that the root system $\Psi=\eta(\overline{F_4})$ is as in the diagram below.

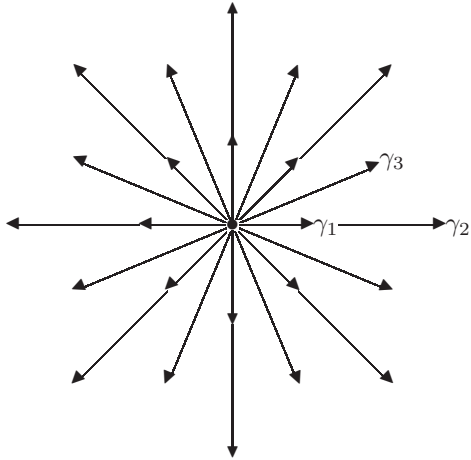
\begin{diagram}[grid=grid2F4,size=2.5em,abut,heads=littleblack]
& & & & {} & & & & \\
& {} &  & {} & \uTo & {} &   & {} & \\
& & \luTo(2,2)&  & \luTo(1,3)\ruTo(1,3) {}&  &\ruTo(2,2) & & \\
& {} &  & {} &\uTo & {} &   & {}\gamma_3 & \\
{}  &\lTo & {}\luTo(3,1) & \lTo & \bullet  \luTo(1,1)\ruTo(1,1) & \rTo & \ruTo(3,1)\gamma_1 &\rTo & \gamma_2 \\
& {} &  & {}\ldTo(3,1) \ldTo(1,3) \ldTo(1,1)& \dTo & {} \rdTo(1,1 )\rdTo(3,1) \rdTo(1,3)&  & {} & \\
& & \ldTo(2,2)&   & {} &  & \rdTo(2,2)&  & \\
& {}  &  & {}  &\dTo & {} &   & {} & \\
& & & & {}  & & & & \\
\end{diagram}

We see that there are three types of roots. We shall call them short, long and double by analogy with $BC_2$
even though this time the double roots are $(\sqrt{2}+1)$ times longer than the short ones.

The short roots are
 $$\{\pm \epps_i,\frac 1{\sqrt 2}(\pm \epps_i \pm \epps_j):\ 1\le i\ne j\le 2\},$$
the double roots are $$\{\pm (\sqrt 2+1)\epps_i,\frac {\sqrt 2+1}{\sqrt 2}(\pm \epps_i \pm \epps_j):\ 1\le i\ne j\le 2\}$$ and the long roots are $$\{\pm \frac 1{\sqrt 2}((1+\sqrt 2)\epps_i\pm\epps_j):\ 1\le i\ne j\le 2\}.$$
Since $\tau$ commutes with $q$, it permutes the roots of $\Psi$. It is easy to see that
this action has 3 orbits corresponding to short, double and long roots. In fact, $\tau$
acts on $\dbR\Psi$ simply as the counterclockwise rotation by $\frac{\pi}{4}$ (with respect to the basis
$\eps_1,\eps_2$). 
\vskip .2cm
In this example we denote by $\{Y_{\alpha}\}_{\alpha\in\Psi}$ not the coarsened grading on $\widetilde G$,
but its fattening (see \S~4.5), where the short root subgroups are the ones being fattened.

If $\alpha\in\Psi$ is a long or double root, the corresponding root
subgroup $Y_{\alpha}$ consists of elements
$$\{y_\alpha(r,s)=\widetilde x_{\alpha_1}(r)\widetilde x_{\alpha_2}(s) :\ \eta^{-1}(\alpha)=\{\alpha_1,\alpha_2\}, \ r,s,t\in R\}.$$
The chosen order on the set $\{\alpha_1,\alpha_2\}$ is not important since the root subgroups  
$\widetilde X_{\alpha_1}$ and $\widetilde X_{\alpha_2}$ commute.

If $\alpha\in\Psi$ is a short root, then $Y_{\alpha}$ consists of elements
$$ y_\alpha(r,s,t,u)=\widetilde x_{\alpha_1}(r)\widetilde x_{\alpha_2}(s)\widetilde x_{\alpha_3}(t)
\widetilde x_{\alpha_4}(u), $$ where $\eta^{-1}(\alpha)=\{\alpha_1,\alpha_2\}$, $\eta^{-1}((\sqrt{2}+1)\alpha)=
\{\alpha_3,\alpha_4\}$ and $r,s,t,u\in R$. 

Here parameterization does depend on how elements of  $\eta^{-1}(\alpha)$ are ordered, so we shall specify the order as follows.
By the above discussion there exists unique $0\leq i\leq 7$ such that $\alpha=\tau^i \epps_1$.
Then we shall put $\alpha_1=\tau^i (\frac{e_1-e_2}{\sqrt 2})$ and $\alpha_2=\tau^i (e_2)$.
\vskip .15cm

Next we describe the subgroups $Z_\alpha=Y_\alpha^{\la q\ra}$. Since $\tau$ acts on $\Psi$
and commutes with $q$, it also permutes the subgroups $\{Z_\alpha\}$, so it suffices to describe
$Z_{\alpha}$ for one root in each $\la\tau\ra$-orbit, that is, one root of each length. We shall
use the roots $\gamma_1=\epps_1$, $\gamma_2=(1+\sqrt 2)\epps_1$ and $\gamma_3=\frac 1{\sqrt 2}((1+\sqrt 2)\epps_1+\epps_2)$.
We have
\begin{gather*}
Z_{\gamma_1}=\{z_{\gamma_1}(r,s)= \widetilde x_{\frac{e_1-e_2}{\sqrt 2}}(r)\widetilde x_{ e_2}(r)\widetilde x_{ e_1}(r^*r+s)  \widetilde x_{\frac{e_1+e_2}{\sqrt 2}}(s):\ r,s \in R\},\\
Z_{\gamma_2}=\{z_{\gamma_2}(r)=z_{\gamma_1}(0,r)= \widetilde x_{ e_1}(r)\widetilde x_{\frac{e_1+e_2}{\sqrt 2}}(r):\ r\in R\},\\
Z_{\gamma_3}=\{z_{\gamma_3}(r)= \widetilde x_{\frac{e_1+e_4}{\sqrt 2}}(r)\widetilde x_{\frac 1 {  2}(e_1+e_2+e_3-e_4)}(r):\ r\in R\}.\end{gather*}
If $\gamma\in\Psi$ is a long or double  root, the corresponding root
subgroup $Z_{\gamma} $ is isomorphic to $(R,+)$, and if
$\gamma\in\Psi$ is a short root, then $Z_{\gamma}$ is nilpotent of class 2.

To simplify the notation we denote $\tau^i(\gamma_k)$ by $\gamma_{k,i}$ and  $z_{ \gamma_{k,i}}(r)$ by $z_{k,i}(r)$.
(Recall that $\tau$ acts on the root system $\Psi$ as a counterclockwise rotation by $\frac \pi 4$.)

We list the relevant commutation relations between the elements of the root subgroups.
All the commutation relations may be found in the Tits paper~\cite{Ti} (note that our notation is slightly different).
\begin{align*}
&(E1)& &[z_{1,i}(r,s),z_{1,i+1}(t,u)]=z_{3,i}(rt) &\\
&(E2)&
&[z_{1,i}(r,s),z_{3,i+1}(t)]=z_{1,i+1}(0, rt)=z_{2,i+1}(rt)&
\\
&(E3)&
&[z_{1,0}(r,s),z_{3,3}(t)]\equiv z_{1,3}(tr,0)
z_{ 1,2}( t^*s,0)z_{1,1}(tr^*r+ts,0) \mod {\prod_{i=1}^{3} Z_{\gamma_{3,i-1}} Z_{ \gamma_{2,i}}}&
\\
&(E4)&
&[z_{3, i+4}(r),z_{2,i+2}(s)]=z_{2,i+3}(rs)\mbox{ and } [z_{3,i}(r),z_{2,i+3}(s)]=z_{2,i+2}(rs)&
\\
&(E5)&
&[z_{3,i}(r),z_{2,i+2}(s)]=[z_{3,i+4}(r),z_{2,i+3}(s)]=1&\\
&(E6)&
&[z_{2,i}(r),z_{2,i+3}(s)]\equiv z_{3,i+1}(rs) \mod{  Z_{\gamma_{2,i+1}}Z_{\gamma_{2,i+2}}}&
\end{align*}
\begin{Proposition}
\label{prop:2F4}
Let $R$ be a finitely generated ring. Then the group $\St_{^2F_4}(R)$ has property $(T)$.
\end{Proposition}
\begin{proof} As usual, we need to check two things:
\begin{itemize}
\item[(a)] The $\Psi$-grading of $\St_{^2F_4}(R)$ is strong at $(\gamma,B)$ for each Borel subset $B$
and each root $\gamma$ in the core of $B$;
\item[(b)] The pair $(\St_{^2F_4}(R),Z_{\gamma})$ has relative $(T)$  for each $\gamma\in \Psi$.
\end{itemize}
As usual, it suffices to check (a) for the Borel set with boundary $\{\gamma_{1,0},\gamma_{2,0},\gamma_{3,3}\}$.
Relations (E1) imply condition (a) for the long root subgroups (that is, the root subgroups
$Z_{\gamma_{3,i}}$ with $0\leq i\leq 2$).  From relations (E2) and (E3) we obtain condition (a) for the short root subgroups.
This also implies condition (a) for the double root subgroups since they are contained in the
short root subgroups.

Now let us prove (b). Relations (E4) and (E5) imply that the pair $$(\la Z_{\gamma_{3,i}},Z_{\gamma_{3,i+4}},Z_{\gamma_{2,i+2}},Z_{\gamma_{2,i+3}}\ra,   Z_{\gamma_{2,i+2}}Z_{\gamma_{2,i+3}} )$$ is a quotient of  $(\St_2(R) \ltimes R^2,R^2)$. This yields relative property $(T)$ for the double root subgroups.
It follows from relations (E6) that any long root subgroup lies in a bounded product of double root subgroups,
so relative $(T)$ also holds for the long root subgroups.

It remains to prove relative property $(T)$ for the short root subgroups. By symmetry, it suffices
to treat the subgroup $Z_{\gamma_{1}}$. Fix $s\in R$, and consider
relation (E3) with $r=0$ and $t$ arbitrary. It implies that the set
$$P_0=\{z_{ 1,2}( t^*s,0)z_{1,1}(ts,0): t\in R \}$$ lies in a bounded product
of the double and long root subgroups. The same holds for
each of the sets $P_1=\{z_{ 1,2}( t^*,0)z_{1,1}(t,0): t\in R \}$
(setting $s=1$ in $P_0$),
$\,P_2=\{z_{ 1,2}( t^*s^*,0)z_{1,1}(ts,0): t\in R \}$ (replacing $t$ by $ts$ in $P_1$) and
$P_3=\{z_{ 1,2}( t^*(s+s^*),0) z_{1,1}(t(s+s^*),0): t\in R \}$
(replacing $s$ by $s+s^*$ in $P_0$).
Considering the product $P_0 P_2 P_3$ and using relations (E1) with $i=1$, we conclude that
the set $P=\{z_{1,1}(t(s+s^*),0): t\in R \}$ lies in a bounded product
of the double and long root subgroups (for any fixed $s$).

Let $I$ be the ideal of $R$
generated by $\{s+s^*:\ s\in R\}$. Since $R$ is finitely generated, it is Noetherian, and so $I$ is generated by a finite subset of $\{s+s^*:\ s\in R\}$. Hence,  by what we just proved the set $\{z_{1,1}(r,0): \ r\in I\}$ also lies in a bounded product
of the double and long root subgroups. On the other hand, every element of the quotient ring
$R/I$ is idempotent (since $r^2+r=((r^*)^*+r^*)+(r^*+r)$), whence $I$ has finite index in $R$.
Thus, the set $\{z_{1,1}(r,0): \ r\in R\}$ lies in a bounded product
of the double and long root subgroups and fixed elements of $G$.
Since the short root subgroup $Z_{\gamma_{1}}$ is a product of $\{z_{1,1}(r,0): \ r\in R\}$
and $Z_{\gamma_2}$, we have proved relative property $(T)$ for $Z_{\gamma_{1}}$.
\end{proof}

\subsection{More groups graded by root systems}
\label{sec:moregroups}

The families of groups graded by root systems which were described in this section
can be generalized in several different ways, and for many of those generalizations
one can prove property $(T)$ by similar methods. However, in the
absence of a specific application, it is not clear which results of this kind are useful
and which are not, so we will not try to achieve the utmost generality in this subsection.
Instead we will describe three new families of groups with property $(T)$ which will
be needed in \S~\ref{sec:mothergroup} where we will prove that the class of all finite simple groups
of Lie type of rank at least $2$ admits a mother group with property $(T)$.

All these families are slight variations of Steinberg groups described earlier in this
section and will be defined by the same procedure as, for instance, the groups
$\St_{C_n}^{\omega}(R,*,J)$ were obtained from $\St_{C_n}^{\omega}(R,*)$, that is,
by decreasing (some of) the root subgroups, so that the decreased subgroups still
form a grading (by the same root system) and then taking the graded cover of
the subgroup generated by these decreased root subgroups. The main difference
is that in the previously considered examples, how much a root subgroup $X_{\gamma}$
was decreased depended only on the root length of $\gamma$, while in the examples
in this subsection the procedure will not be ``symmetric''.


It will be convenient to use the following notation and terminology. Let $G$ be one
of the types of Steinberg groups considered in this section associated to a ring
$R$ endowed with some set of operations, that is, $G=\St_{A_n}(R)$,
$G=\St_{C_n}^{\omega}(R,*)$, etc., and let $\{X_{\gamma}\}_{\gamma\in\Phi}$ be its standard
grading. By definition, each root subgroup $X_{\gamma}$ is  isomorphic
to certain group $R_{\gamma}^{full}$ (which in many cases is defined as a subgroup
of $(R,+)$), and in the course of our construction we have chosen a specific
isomorphism between $X_{\gamma}$ and $R_{\gamma}^{full}$.
Below we recall the explicit description of $R_{\gamma}^{full}$ in the cases
which will be used in this section:
\begin{itemize}
\item[(i)] if $G=\St_{A_n}(R)$, then $R_{\gamma}^{full}=(R,+)$
for all $\gamma$;
\item[(ii)] if $G=\St_{C_n}^{\omega}(R,*)$, then $R_{\gamma}^{full}=(R,+)$
if $\gamma$ is short and $R_{\gamma}^{full}=(\Sym_{\omega}(R),+)$ if $\gamma$ is long
\item[(iii)] if $G=\St_{BC_n}(R,*)$, then $R_{\gamma}^{full}=(R,+)$ if $\gamma$ is long,
$R_{\gamma}^{full}=(\Asym(R),+)$ if $\gamma$ is double, and $R_{\gamma}^{full}=P(R,*)$
if $\gamma$ is short (where $P(R,*)$ is defined in Example~2),
\item[(iv)] if $G=\St_{BC_n}^{1}(R,*,\sigma)$, then $R_{\gamma}^{full}=(R^{\sigma},+)$ if $\gamma$ is long,
$R_{\gamma}^{full}=(\Asym(R^{\sigma}),+)$ if $\gamma$ is double, and $R_{\gamma}^{full}=Q(R,*,\sigma)$
if $\gamma$ is short (where $Q(R,*,\sigma)$ is defined in Example~3)
\end{itemize}

\begin{Definition}In the above setting, choose a subgroup $R_{\gamma}$ of $R_{\gamma}^{full}$
for each $\gamma\in \Phi$, and let $X_{\gamma}(R_{\gamma})$ be the image of $R_{\gamma}$ under the
chosen isomorphism $R_{\gamma}^{full}\to X_{\gamma}$. We will say that $\{R_{\gamma}\}$
is a \emph{root content for $G$} if $\{X_{\gamma}(R_{\gamma})\}_{\gamma\in\Phi}$ is a $\Phi$-grading.
\index{root content}
\end{Definition}

\begin{Example}\rm Let $n\geq 2$ be an integer, $R$ a ring, and let $G=\St_{A_n}(R)$. Choose a subgroup $R_{\gamma}$ of $(R,+)$ for each $\gamma\in A_n$.
Then $\{R_{\gamma}\}_{\gamma\in A_n}$ is a root content for $G$ if and only if
$R_{\alpha}R_{\beta}\subseteq R_{\alpha+\beta}$ whenever $\alpha+\beta$ is a root.
\end{Example}

If $\{R_{\gamma}\}$ is a root content for $G$, we can consider the subgroup $\la \cup X_{\gamma}(R_{\gamma})\ra$
generated by decreased root subgroups $X_{\gamma}(R_{\gamma})$ and take its graded cover. This graded
cover will be denoted by $G(\{R_{\gamma}\})$.

\begin{Example}\rm Let $G=\St_{C_n}^{\omega}(R,*)$, let $J$ be a form parameter of $(R,*,\omega)$,
and define  $\{R_{\gamma}\}_{\gamma\in C_n}$ by $R_{\gamma}=R$ if $\gamma$ is short and
$R_{\gamma}=J$ if $\gamma$ is long. Then $\{R_{\gamma}\}_{\gamma\in C_n}$ is a root content
for $G$, and the associated group $G(\{R_{\gamma}\})$ is the group $\St_{C_n}^{\omega}(R,*,J)$ as defined
in Example~2.
\end{Example}

We are now ready to describe the three families we will be interested in. In all statements below, the following convention
will be used: given additive subgroups $A$ and $B$ of a ring $R$,
by $AB$ we denote the additive subgroup of $R$ generated by the set $\{ab: a\in A,b\in B\}$.

\begin{Proposition}
\label{Trootcontent_typeAn}
Let $R$ be a ring, let $n\geq 4$ be an integer and $G=\St_{A_{n-1}}(R)$. Let $M$ be a left ideal of $R$ and $N$ a right ideal of $R$.
Define $\{R_{\alpha}\}_{\alpha\in A_{n-1}}$ by
$$R_{e_i-e_j}=\left\{
\begin{array}{ll}
R &\mbox{ if } 1\leq i\neq j\leq n-1\\
M &\mbox{ if } j= n\\
N &\mbox{ if } i= n.\\
\end{array}
\right.
$$
Then $\{R_{\alpha}\}_{\alpha\in A_{n-1}}$ is a root content for $G$, and the associated group
$G(\{R_{\alpha}\})$ has property $(T)$ if the following conditions hold:
\begin{itemize}
\item[(i)] $R$ is finitely generated as a ring
\item[(ii)] $M$ is finitely generated as a left ideal
\item[(iii)] $N$ is finitely generated as a right ideal
\item[(iv)] $MN=R$.
\end{itemize}
The group $G(\{R_{\alpha}\})$ will be denoted by $\St_{A_{n-1}}(R;M,N)$.
\end{Proposition}
\begin{proof} It is straightforward to check that $\{R_{\alpha}\}_{\alpha\in A_{n-1}}$
is a root content for $G$. For the rest of the proof we will denote the group
$G(\{R_{\alpha}\})$ by $H$, and we let $\{Y_{\alpha}=X_{\alpha}(R_{\alpha})\}_{\alpha\in A_{n-1}}$
be the canonical $A_{n-1}$-grading of $H$.

In the proof of property $(T)$ for $H$, the relative property $(T)$ part is virtually identical to the case of $\St_{A_{n-1}}(R)$ and uses conditions
(i)-(iii) above, but for completeness we provide the argument. The subgroup generated by $\{ Y_{e_i-e_j} : i,j<n \}$ is isomorphic to a quotient of $\St_{A_{n-2}}(R)$ and therefore relative property $(T)$ holds for
$(H,Y_{e_i-e_j})$ with $i,j<n$.

There is an epimorphism from $\St_{A_{n-2}}(R)\ltimes M^{n-1}$ onto the subgroup generated by 
$\{ Y_{e_i-e_j} : i<n \}$ which maps $M^{n-1}$ onto $\la\{ Y_{e_i-e_n} : i<n \} \ra$.
Since the pair $(\St_{A_{n-2}}(R)\ltimes M^{n-1}, M^{n-1})$ has relative property~$(T)$ by
Theorem~\ref{thm:relT_St_p}, it follows that $(H,Y_{e_i-e_n})$ has relative $(T)$ for all $i<n$. Similarly, $(H,Y_{e_n-e_i})$ has relative property $(T)$.

Finally, we need to verify that the grading $\{Y_{\alpha}\}$ is strong. Here things behave differently from the case $\St_{A_{n-1}}(R)$
and condition (iv) must be used, as we now demonstrate. Since Borel subgroups are no longer isomorphic to each other, we cannot restrict
ourselves to checking that the grading is strong at a pair $(\gamma,B)$ where
$B$ is the standard Borel, as we did in all the previous examples. Let us introduce the
following notations: given a subset $S$ of $A_{n-1}$, let
$$H_S=\la Y_{\alpha}: \alpha\in  S\ra.$$
Thus, we need to check that for any Borel subset $B$ and any root $\gamma\in C(B)$,
the core of $B$, we have
\begin{equation}
\label{strong_condition}
Y_{\gamma}\subseteq H_{B\setminus \R \gamma}.
\end{equation}
So, let us take any root $\gamma=e_{i}-e_j$ and any Borel $B$ such that $\gamma\in C(B)$. Then there exists $k\neq i,j$ such that $e_i-e_k$ and $e_k-e_j$ both lie in $B$.
Then $H_{B\setminus \R \gamma}$ contains $[x_{e_i-e_k}(r),x_{e_k-e_j}(s)]=
x_{e_i-e_j}(rs)$ for every $r\in R_{e_i-e_k}$ and $s\in R_{e_k-e_j}$, and so
we just have to check that the additive subgroup generated by all products of this
form coincides with $R_{e_i-e_j}$. If $k\neq n$, this is clearly true
(in fact, in this case it suffices to take just products, not their sums), and if $k=n$, this holds precisely because of condition (iv).
\end{proof}

\begin{Proposition}
\label{Trootcontent_typeCn}
Let $R$ be a ring, $*$ an involution of $R$,
$\omega$ an element of $Z(R)$ satisfying $\omega^*=\omega^{-1}$,
let $n\geq 4$ and $G=\St_{C_n}^{\omega}(R,*)$. Let $J$ be a form parameter of $(R,*,\omega)$, let $M$ be a left ideal of $R$, and let $M^*=\{m^*: m\in M\}$.
Let $J_M$ be the additive subgroup of $J$ generated by the set
$$\{m^* x m: x\in J, m\in  M\}
\cup \{m^* n-\omega n^* m : m,n\in M\}.$$
Define $\{R_{\alpha}\}_{\alpha\in C_n}$ by
\begin{align*}
&R_{\pm e_i\pm e_j}=R \mbox{ for } 1\leq i\neq j\leq n-1&\\
&R_{\pm e_i\pm e_n}=M \mbox{ for } 1\leq i\leq n-1&\\
&R_{\pm 2e_i}=J \mbox{ for } 1\leq i\leq n-1&\\
&R_{\pm 2e_n}=J_M &
\end{align*}
Then $\{R_{\alpha}\}_{\alpha\in C_n}$ is a root content for $G$, and the associated
group $G(\{R_{\alpha}\})$ has property $(T)$ if the following conditions hold:
\begin{itemize}
\item[(i)] $R$ is finitely generated as a ring
\item[(ii)] $J$ is finitely generated as a form parameter of $(R,*,\omega)$
\item[(iii)] $M$ is finitely generated as a left ideal of $R$
\item[(iv)] $M M^*=R$
\item[(v)] There exists $d\in\dbN$ such that every element of $M^*M$ can be written as a sum $\sum_{i=1}^d m_i^* n_i$ with $m_i,n_i\in M$.
\end{itemize}
The group $G(\{R_{\alpha}\})$ will be denoted by $\St_{C_n}^{\omega}(R,*,J;M)$.
\end{Proposition}
\begin{Remark} Condition (v) automatically holds if $M$ in principal which
will be the case in all applications of Proposition~\ref{Trootcontent_typeCn}
in \S~\ref{sec:mothergroup}.
\end{Remark}
\begin{proof} We will use the same general notations as in
Proposition~\ref{Trootcontent_typeAn}, that is $H=G(\{R_{\alpha}\})$ and
$\{Y_{\alpha}\}$ the canonical $C_n$-grading of $H$. As in Example~1, we will denote the elements of root subgroups of $G$ by $z_{\alpha}(r)$ (not $x_{\alpha}(r)$); on the other hand we denote the roots of $C_n$ by $\pm e_i\pm e_j$ and $\pm 2e_i$
(not $\pm \eps_i\pm \eps_j$ and $\pm 2\eps_i$ used in Example~1).

First, we establish relative property $(T)$; as in the case of type $A_n$, the argument is rather similar to Proposition~\ref{prop:ex1}.

If $\alpha$ is a short root, the root subgroup $Y_{\alpha}$ lies inside a quotient
of the Steinberg group $\St_{A_{n-1}}(R,M,M^*)$ which has property $(T)$ by
Proposition~\ref{Trootcontent_typeAn}. If $\alpha=\pm 2e_i$ with $i<n$,
then $Y_{\alpha}$ lies inside a quotient of the Steinberg group $G=\St_{C_{n-1}}^{\omega}(R,*,J)$ and thus has property $(T)$ by Proposition~\ref{prop:ex1}. Thus, the pair $(H,Y_{\alpha})$ has relative $(T)$ in
all these cases.

It remains to consider $\alpha=\pm 2e_n$. We will treat the case $\alpha=2e_n$;
the case $\alpha=-2e_n$ is analogous.

Let $T$ be a finite set which generates $J$ as a form parameter of $(R,*,\zetta)$.
It is easy to see that any element $r\in J_M$ can be written as
$$r=\sum_{t\in T} m_t^* t m_t + (u-\zetta u^*) \eqno(***)$$ for
some $m_t\in M$ and $u\in M^* M$. By condition (v), we can write
$u=\sum_{i=1}^d n_i^* p_i$ with $n_i,p_i\in M$ and $d$ independent of $r$.

Relations (E2) and (E3) yield the following
identity:
$$z_{2e_n}(r)=
\prod_{t\in T} [z_{2e_1}(t), z_{e_n-e_1}(m_t)]z_{e_n+e_1}(\sum_{t\in T}m_t t)
\prod_{i=1}^d [z_{e_n-e_1}(n_i),z_{e_1+e_n}(p_i)]
$$ It
follows that
$$
Y_{2e_n}\subseteq \prod_{t\in T}
(Y_{e_n-e_1}^{z_{2e_1}(t)} Y_{e_n-e_1} Y_{e_1+e_n})
(Y_{e_n-e_1} Y_{e_1+e_n})^{2d}.
$$
The set $\{z_{2e_1}( t):\ t\in T\}$ of conjugating elements is
finite, so $Y_{2e_n}$ lies in a bounded product of subgroups
for which relative $(T)$ has already been established.
\vskip .2cm

Finally, we check that the grading is strong, that is, \eqref{strong_condition} holds for
each root $\gamma$ and Borel subset $B$ containing $\gamma$ in its core. If $\gamma$ is a short root, this is checked precisely as in Proposition~\ref{Trootcontent_typeAn}, so we only need to consider long
roots. We shall treat the case when $\gamma$ is positive, that is, $\gamma=2e_{i}$
for some $i$; the case of negative $\gamma$ is analogous.

First consider the case $\gamma=2e_{n}$. Since $\gamma\in C(B)$,
$\gamma$ is representable as a sum of two short roots in $B$, so
there exists $1\leq k\leq n-1$ s.t. $e_{n}-e_{k}$ and
$e_{n}+e_{k}$ both lie in $B$. In addition, either $2 e_{k}$
or $-2 e_{k}$ lies in $B$. Both cases are analogous, so we shall
assume that $2 e_{k}\in B$.

By relations (E7) in Example~1 with $i=n$ and $j=k$, the group $H_{B\setminus \R\gamma}$  contains all elements of the form $z_{\gamma}(m^* n-\omega n^* m)$ with $m,n\in M$, and by relations (E9) with $i=n$ and $j=k$, $H_{B\setminus \R\gamma}$ contains all elements of the form $z_{\gamma}(m^* x m)$ with $m\in M$ and $x\in J$.
Thus, in view of (***), \eqref{strong_condition} holds for $(\gamma, B)$.

\vskip .2cm Now consider the case $\gamma=2e_{k}$ where $k<n$.
Again $B$ contains $e_{k}-e_{l}$ and $e_{k}+e_{l}$ for some $l\neq k$
and without loss of generality we can assume that $2 e_{l}\in B$. If $l\neq n$,
\eqref{strong_condition} obviously holds, so we will only consider the case $l=n$.
First, by relations (E2), $H_{B\setminus \R\gamma}$ contains all elements of the form
$z_{\gamma}(mn^*-\omega nm^*)$ with $m,n\in M$ which, by condition (iv), account for
all elements of the form $z_{\gamma}(r-\omega r^*)$ with $r\in R$. By relations (E3)
with $i=k$ and $j=n$, $H_{B\setminus \R\gamma}$ contains all elements of
the form $z_{\gamma}(mxm^*)$ with $x\in J_M$, $m\in M$, so in particular, all elements
of the form $z_{\gamma}(mn^* x nm^*)$ with $x\in J$, $m,n\in M$.

Now fix $x\in J$, and choose $m_i,n_i\in M$ such that $\sum_i m_i n_i^*=1$. Then
$$x=(\sum_i m_i n_i^*)x(\sum_i m_i n_i^*)^*=\sum_i m_i n_i^* x n_i m_i^*+y-\omega y^*$$
where $y=\sum_{i<j}m_i n_i^* x n_j m_j^*$. Thus, $z_{\gamma}(x)\in H_{B\setminus \R\gamma}$,
so \eqref{strong_condition} holds for $(\gamma, B)$.
\end{proof}

Before describing our last family, we introduce the following notations
which  generalize analogous notations used in Example~3.
Let $R$ be a ring, $*$ an involution of $R$
and $\sigma$ an automorphism of $R$ of order $\leq 2$ which commutes with $\sigma$.
For any subset $S$ of $R$ we put
$$S^{\sigma}=\{s\in S: \sigma(s)=s\}=S\cap R^{\sigma}.$$
Given additive subgroups  $S,I$ and $J$ of $(R,+)$, define
$$Q(S,*,\sigma,I,J)=\{(r,t): r\in I, t\in J \mbox{ and } t-r\sigma(r)^*\in S^{\sigma}\}$$
\begin{Proposition}
\label{Trootcontent_typeBCn}
Let $R$ be a ring, $*$ an involution on $R$
and $\sigma$ an automorphism of order $\leq 2$ which commutes with $*$.
Let $n\geq 4$ and $G=\St_{BC_n}^1(R,*,\sigma)$.
Let $I\subseteq R$ be a left $R^{\sigma}$-submodule, $J$ a form
parameter of $(R,*,1)$, $M$ be a left ideal of $R$ and
$M^*=\{m^*: m\in M\}$.
  Define $\{R_{\alpha}\}_{\alpha\in BC_n}$ by
\begin{align*}
&R_{\pm e_i\pm e_j}=R^{\sigma} \mbox{ for } 1\leq i\neq j\leq n-1&\\
&R_{\pm e_i\pm e_n}=M^{\sigma} \mbox{ for } 1\leq i\leq n-1&\\
&R_{\pm e_i}=Q(R,*,\sigma,I,J) \mbox{ for } 1\leq i\leq n-1&\\
&R_{\pm 2e_i}=J^{\sigma} \mbox{ for } 1\leq i\leq n-1&\\
&R_{\pm e_n}=Q(M^*M,*,\sigma,{(M^*)}^{\sigma} I,J) &\\
&R_{\pm 2e_n}=(J\cap M^* M)^{\sigma} &
\end{align*}

Then $\{R_{\alpha}\}_{\alpha\in BC_n}$ is a root content for $G$, and the associated
group $G(\{R_{\alpha}\})$ has property $(T)$ if the following conditions hold:
\begin{itemize}
\item[(i)] $R^{\sigma}$ is finitely generated as a ring
\item[(ii)] There exists $a\in R$ such that $a+\sigma(a)^*=1$
\item[(iii)] $I$ is finitely generated as an $R^{\sigma}$-module
\item[(iv)] $J^{\sigma}$ is finitely generated as a form parameter
of $(R^{\sigma},*,1)$
\item[(v)] $M^{\sigma}$ is finitely generated as a left ideal of $R^{\sigma}$
\item[(vi)] $M^{\sigma}{(M^*)}^{\sigma}=R^{\sigma}$
\item[(vii)] There exists $d\in\dbN$ such that every element of ${(M^*)}^{\sigma} M^{\sigma}$ can be written as a sum $\sum_{i=1}^d m_i^* n_i$ with $m_i,n_i\in M^{\sigma}$.
\item[(viii)] $(J\cap M^* M)^{\sigma}$ is equal to the additive
subgroup generated by the set
$$\{m^* x m: x\in J^{\sigma}, m\in  M^{\sigma}\} \cup \{m^* n-n^* m : m,n\in M^{\sigma}\}.$$

\end{itemize}
The group $G(\{R_{\alpha}\})$ will be denoted by $\St_{BC_n}^{1}(R,*,\sigma,I,J;M)$.
\end{Proposition}

Condition (viii) above may be rather difficult to check in general, but it always holds if $M$ is principal and generated by a $\sigma$-invariant idempotent:

\begin{Observation}
\label{obs:idempotent}
In the setting of Proposition~\ref{Trootcontent_typeBCn}, assume that
there exists $z\in M^{\sigma}$ such that $M=Rz$ and $z^2=z$. Then condition (viii) holds.
\end{Observation}
\begin{proof}Denote the additive subgroup generated by the set
$\{m^* x m: x\in J^{\sigma}, m\in  M^{\sigma}\} \cup \{m^* n-n^* m : m,n\in M^{\sigma}\}$
by $J_{\sigma,M}$. It is clear that $J_{\sigma,M}\subseteq (J\cap M^* M)^{\sigma}$.

To prove the reverse inclusion, take any $x\in M^* M$. Then, by assumption on $M$,
we have $x=z^*yz$ for some $y\in R$. Since $z^2=z$, we have $x=(z^2)^* y z^2=z^* xz$.
Thus, if also assume that $x\in J^{\sigma}$, then $x=z^* xz\in J_{\sigma,M}$ by definition.
\end{proof}

\begin{proof}[Proof of  Proposition~\ref{Trootcontent_typeBCn}]
 Similarly to Propositions~\ref{Trootcontent_typeAn} and \ref{Trootcontent_typeCn},
let $H=G(\{R_{\alpha}\})$ and $\{Y_{\alpha}\}$ the canonical $BC_n$-grading of $H$.

Thanks to condition (viii), the subgroup of $H$ generated by long and double root subgroups is isomorphic to a quotient of $\St_{C_n}^1(R^{\sigma},*,J^{\sigma})$, which has property $(T)$ by Proposition~\ref{Trootcontent_typeCn} thanks to conditions (i),(iv),(v),(vi) and (vii). Thus, relative $(T)$ holds for $(H,Y_{\alpha})$ whenever $\alpha$ is a long
or double root. If $\alpha$ is a short root, relative $(T)$ for $(H,Y_{\alpha})$
is verified exactly as in Proposition~\ref{propT_ex2}(a), using condition (iii). 

Now we check that the grading $\{Y_{\alpha}\}$ is strong at each pair
$(\gamma,B)$ with $\gamma\in C(B)$.
If $\gamma=\pm e_i\pm e_j$ is a long root, the proof is analogous to Proposition~\ref{Trootcontent_typeAn}. If $\gamma=\pm 2e_i$ is a double root, one can argue as in
Proposition~\ref{Trootcontent_typeCn}.

Finally, consider the case when $\gamma$ is a short root. As in the proof of
Proposition~\ref{Trootcontent_typeCn}, without loss of generality we can assume
that $\gamma=e_{i}$ and that $B$ contains $e_i-e_k$ and $e_k$ for some $k\neq i$.

For a short root $\alpha$ denote by $I_{\alpha}$ the projection of $R_{\alpha}$ onto the first
component. By condition (ii) and Lemma~\ref{lemma:BCn1_niceelement}, $I_{\pm e_j}=I$ for $j<n$ and $I_{\pm e_n}={(M^*)}^{\sigma} I$ (more precisely, for $\alpha=\pm e_n$ we need a suitable generalization of Lemma~\ref{lemma:BCn1_niceelement} whose proof is analogous). Since we already know that the grading is strong at $(\delta,B)$ for each double root $\delta$, by relations (E3) and analogous relations dealing with negative root subgroups in Example~3,
it suffices to check that
\begin{equation}
\label{eq:BC1_strong}
R_{e_i-e_k}I_{e_k}=I_{e_i} \mbox{ if } i<k \quad\mbox{ and }\quad
R_{e_i-e_k}^* I_{e_k}=I_{e_i} \mbox{ if } i>k.
\end{equation}
If $i$ and $k$ are both different from $n$, then $R_{e_i-e_k}=R^{\sigma}$ and $I_{e_i}=I_{e_k}=I$,
so \eqref{eq:BC1_strong} holds. If $i=n$, then $R_{e_i-e_k}^*={(M^*)}^{\sigma}$, $I_{e_k}=I$ and $I_{e_i}={(M^*)}^{\sigma} I$, so again \eqref{eq:BC1_strong} is clear. Finally, if $k=n$, by condition (vi) we have $R_{e_i-e_k}I_{e_k}=M^{\sigma}{(M^*)}^{\sigma} I=I= I_{e_i}$. This completes the proof.
\end{proof}


\section{Application: Mother group with property $(T)$}
\label{sec:mothergroup}

Let $\Gamma$ be a finite graph and let $\eps>0$ be a real number. We say that
$\Gamma$ is an \emph{$\eps$-expander} if for every subset $A$ consisting of at most half of vertices of $\Gamma$ we have $|\partial A|\ge \eps |A|$. Here $\partial A$ is the edge boundary of $A$, that is, the set of edges of $\Gamma$
which join  a vertex in $A$ with a vertex outside of $A$. Expander graphs play an important role in computer science and
  combinatorics, and many efforts have been devoted to their constructions (see, e.g.,~\cite{HLW}).
Particular attention has been paid
to the case of Cayley graphs. Recall that given a group $G$ and a symmetric generating set $S$, one defines $\Cay(G,S)$
to be the graph with vertex set $G$, in which two vertices $x$ and $y$ are connected by an edge if and only if
$y=xs$ for some $s\in S$. Note that if $|S| = k$ then the Cayley graph is $k$-regular. We will say that an infinite family $ \mathcal{F}$ of groups is a \emph{family of expanders} if there exists $k\in \mathbb{N}$  and $\eps>0$  such that every group $G\in \mathcal F$ has a symmetric
generating set $S$ with $|S|=k$ such that $\Cay(G,S)$ is an $\eps$-expander.
As a consequence of  several works (see~\cite{Ka2, KLN, BGT}) the following remarkable result was recently established.

\begin{Theorem} 
\label{expander} 
Any family of (non-abelian) finite simple groups is a  family of expanders.
\end{Theorem}

\begin{Definition}
Let $\mathcal  F$ be a family of groups and $G$ a group. We say that $G$ is a \emph{mother group} \index{mother group}
for $\mathcal  F$ if every group in $\mathcal  F$ is a quotient of $G$.
\end{Definition}

As discussed in the introduction, one of the conceptually simplest ways to prove  that a family $\mathcal  F$  of finite
groups is a family of expanders is to find a mother group for $\mathcal F$  with property $(T)$ or $(\tau)$. Recall that a
group $G$ has \emph{property $(\tau)$} \index{property $(\tau)$}
if there exists $\mu>0$ and a finite subset $S$ of $G$ such that $\kappa(G,S,V)\geq\mu $
for every non-trivial irreducible unitary representation $V$ of $G$ which factors through a finite quotient of $G$.

In view of Theorem~\ref{expander}, it is natural to ask which families of non-abelian finite simple groups admit a mother
group with $(T)$ or $(\tau)$.

\begin{Conjecture}$\empty$
\label{mother}
\begin{itemize}
\item[(a)] 
The family of all non-abelian finite simple groups has a mother group with property $(\tau)$.
\item[(b)] 
A family $\mathcal F$ of non-abelian finite simple groups has a mother group with property $(T)$ if and only if $\mathcal F$ contains only finitely many finite simple groups of
Lie type of rank $1$.
\end{itemize}
\end{Conjecture}

The main result of this section partially confirms part (b) of Conjecture~\ref{mother}:
\begin{Theorem}
\label{thm:mothergroup}
The family of all finite simple groups of Lie type and rank $\geq 2$ has a mother group with property $(T)$.
\end{Theorem}

Before starting the proof of Theorem~\ref{thm:mothergroup}, let us mention other known facts related to Conjecture~\ref{mother}. It is a folklore result that a group with property $(T)$ cannot map onto $\PSL_2(\F_q)$
for infinitely many $q$. 
Since we are unaware of the proof of this fact  in the literature, we include it at the end of this section. However some infinite families  of finite simple groups of  Lie  type and  rank 1 have a mother group with property $(\tau)$.  
For instance, $\SL_2(\Z[1/2])$ has property $(\tau)$ (see, e.g.~\cite[p.~60]{LZ})
and  clearly maps onto  $\PSL_2(\F_p)$ when $p$ is odd.
\vskip .2cm

With regard to Conjecture~\ref{mother}(a), we will also prove that alternating groups satisfy the conjecture:
\begin{Theorem}
\label{mother_alt_tau}
The family of all alternating groups $\Alt(n)$ has a mother group with property $(\tau)$.
\end{Theorem}
Theorem~\ref{mother_alt_tau} will be established in~\S~\ref{subsec:alt}.

\subsection{Some general reductions}
Recall that finite simple groups of Lie type can be realized as (possibly twisted) Chevalley groups over finite fields.
To simplify the exposition below we shall only discuss groups of rank at least two.

Given a reduced irreducible classical root system $\Phi$ of rank $\geq 2$, the (untwisted) finite simple group
of type $\Phi$ over the field of order $q$ will be denoted by $\Phi(q)$.
Twisted finite simple groups will be denoted by symbols of the form ${}^l \Phi(q)$ -- by definition
${}^l \Phi(q)$ is the subgroup of elements of $\Phi(q^l r)$ fixed by certain automorphism
of order $l$, where $r=1$ in all cases except $^l\Phi= {}^2F_4$, in which case $r=2$.
If $^l\Phi\neq {}^2F_4$, the parameter $q$ can be any prime power, and if $^l\Phi= {}^2F_4$,
we can take $q=2^k$ for any $k\geq 0$. The groups ${}^l \Phi(q)$ are simple with the exception
of ${}^2 B_2(2)\cong {}^2 C_2(2),{}^2 G_2(2)$ and ${}^2F_4(1)$, and in those three cases ${}^l \Phi(q)$
contains a simple subgroup of index $2$ (which is equal to the commutator subgroup of ${}^l \Phi(q)$).
All these groups are naturally graded by root systems and some of them are classical groups.

It will be convenient to use the following terminology: given two groups of Lie type
$^l X_m(q)$ and $^{l'} {X'}_{m'}(q')$ (where $X$ and $X'$ stand for $A$, $B$, $C$, $D$, $E$, $F$ or $G$,
and a non-twisted group $\Phi(q)$ is temporarily denoted by $^1 \Phi(q)$), we will say that

\begin{itemize}
\item[(i)] 
$^l X_m(q)$ and $^{l'} {X'}_{m'}(q')$ \emph{have the same Lie type} if \index{Lie type}
$X=X'$, $l=l'$ and $m=m'$;
\item[(ii)] 
$^l X_m(q)$ and $^{l'} {X'}_{m'}(q')$ \emph{lie in the same Lie family} \index{Lie family}
if $X=X'$, $l=l'$ and in addition $m$ and $m'$ have the same parity if
$X=A$ and $l=2$.
\end{itemize}
In Table~\ref{table:simplegroups} below we recollect all this information, with groups sorted
by their Lie family (we only list groups of Lie rank $\geq 2$). The first column contains the notation for the
group, in the second we put  its interpretation as a classical group (if such exists) and in the third we describe the graded cover of this
group as defined in~\S~\ref{sec:steinberg} and~\S~\ref{sec:twisted}. All these correspondences can be found in~\cite{Ca} (see Theorems 11.1.2, 11.3.2,  14.4.1, 14.5.1 and 14.5.2). The notation for classical groups is taken from~\cite{KL}. \\

\begin{table}
\label{table:simplegroups}
\begin{tabular}{|c|c|c|} 
\hline  & &\\[-0.5ex]
{\bf Lie type} & {\bf Classical group}  & {\bf Graded cover} 
\\[1.5ex] \hline & & \\[-0.7ex]
$A_{n}(q)$ & 
    $\PSL_{n+1}(\F_q)$ & 
        $\St_{A_{n}}(\F_q) $
\\[1.5ex] \hline & & \\[-0.7ex]
\begin{tabular}{c}$B_n(q)$ \\ ($q$ is odd) \end{tabular} & 
    $\mathrm{P}\Omega_{2n+1}(\F_q)$ & 
        \begin{tabular}{c} $\St_{B_n}(\F_q)=\St_{BC_n}(\F_q, \mathrm{id})$  \end{tabular}
\\[2.5ex] \hline & & \\[-0.7ex]
\begin{tabular}{c} $ C_n(q)$  \\ ($q\ne 2$ if $n= 2$) \end{tabular} & 
    $\mathrm{PSp}_{2n}(\F_q)$ & 
        \begin{tabular}{c} $\St_{C_n}(\F_q) = \St_{C_n}^{-1}(\F_q,\mathrm{id})$ \end{tabular}
\\[2.5ex] \hline & & \\[-0.7ex]
$D_n(q)$ &
    $\mathrm{P}\Omega_{2n}^+(\F_q)$ &   
        \begin{tabular}{c} $\St_{D_n}(\F_q) = \St_{C_n}^{1}(\F_q, \mathrm{id},\{0\})$ \end{tabular}
\\[2.5ex] \hline & & \\[-0.7ex]
\begin{tabular}{c} $\Phi(q)$ \\ $\Phi=E_n$  or $F_4$ \end{tabular} &
    &   
        $\St_{\Phi}(\F_q)$
\\[2.5ex] \hline & & \\[-0.7ex]
\begin{tabular}{c} $G_2(q)$ \\ $q\ne 2$  \end{tabular} &
    &   
        $\St_{G_2}(\F_q)$
\\[2.5ex] \hline & & \\[-0.7ex]
${}^2A_{2n-1}(q)$ & 
    $\mathrm{PSU}_{2n}(\F_q)$ &   
        \begin{tabular}{c} $\St_{C_n}^1(\F_{q^2},\upbar)$, \\[.5ex] $\bar x = x^q$ \end{tabular}
\\[2.5ex] \hline & & \\[-0.7ex]
${}^2A_{2n}(q)$ & 
    $\mathrm{PSU}_{2n+1}(\F_q)$ &  
        \begin{tabular}{c}  $\St_{BC_n}(\F_{q^2},\upbar)$, \\[.5ex] $\bar x = x^q$  \end{tabular}
\\[2.5ex] \hline & & \\[-0.7ex]
${}^2D_{n}(q)$ & 
    $P\Omega_{2n}^-(\F_q)$ &
        \begin{tabular}{c} $\St^1_{BC_{n-1}}(\F_{q^2},\mathrm{id},\upbar, \F_{q^2},\{0\})$ \\[.5ex] $=\St_{B_{n-1}}(\F_{q^2},\upbar),$ \\[.5ex] $\bar x = x^q$ \end{tabular}
\\[2.5ex] \hline & & \\[-0.7ex]
${}^3D_4(q)$ &
    &  
        \begin{tabular}{c} $\St_{G_2}(\F_{q^3},\theta)$, \\[.5ex] $\theta(x) = x^q$ \end{tabular}
\\[2.5ex] \hline & & \\[-0.7ex]
${}^2E_6(q)$ &
    &
        \begin{tabular}{c} $\St_{F_4}(\F_{q^2},\upbar)$, \\[.5ex] $\bar x=x^q$ \end{tabular}
\\[2.5ex] \hline & & \\[-0.7ex]
\begin{tabular}{c} $^2F_4(2^k)$ \\[0.5ex] $(k\ge 1)$ \end{tabular} &
    &
        \begin{tabular}{c}$\St_{^2F_4}(\F_{2^{2k+1}},*)$,\\[.5ex] $x^*=x^{2^{k+1}}$\end{tabular}
\\ [1.5ex] \hline
\end{tabular}
\vskip 3mm
\caption{Simple groups of Lie type of rank $\geq 2$}
\end{table}
\vskip .3cm

Observe that if $\mathcal F_1,\ldots, \mathcal F_k$ are families of finite groups
and $G_i$ is a mother group for $\mathcal F_i$, then $\prod_{i=1}^k G_i$ is a mother group for $\cup_{i=1}^k \mathcal F_i$.
Thus, to prove Theorem~\ref{thm:mothergroup}, it suffices to find a mother group
with property $(T)$ for all finite simple groups of rank $\geq 2$ within a given
Lie family.
For the same reason we can exclude any finite set of groups from consideration,
so we do not have to worry about the commutator subgroups of $C_2(2)$, $G_2(2)$ and
$^2 F_4(1)$ not included in the above table. Further, it will be convenient to split all groups in a given Lie
family into two subfamilies -- those of rank $\geq c$ and those of rank $<c$ (but $\geq 2$),
where $c$ is chosen separately for each Lie family. We will refer to the corresponding two
cases as unbounded rank case and bounded rank case, and the construction of a mother group
with $(T)$ in these two cases will be rather different. In the bounded rank case we will
treat all groups of a given Lie type separately (there are finitely many of such Lie types),
and the argument will work for any value of $c$, while in the unbounded rank case certain minimum
value of $c$ is required (we will often choose not the smallest possible $c$
to avoid unnecessary technicalities). Note that we need to consider unbounded rank case
only for groups in the families $A,B,C$, $D$, $^2A$ or $^2D$.

The existence of covering epimorphisms from Steinberg groups onto finite simple groups
claimed in the above table is a direct consequence
of the construction of finite simple groups of Lie type as (possibly twisted) Chevalley
groups, as defined in~\cite{Ca} and the definition of (possibly twisted) Steinberg groups
given in this paper. The case of the groups of type ${}^2F_4$ is somewhat exceptional --  in the definition of finite simple groups of type ${}^2F_4$ given in~\cite{Ca} the twisting involution on
$F_4(2^{k})$ is constructed as a composition of a graph automorphism and a field automorphism, while in our definition the twisting involution $q$ on Steinberg groups of type $F_4$ is defined directly. The fact that the involution on $F_4(2^{k})$ corresponding to $q$ coincides with
the one defined in~\cite{Ca} is easy to check by a direct computation; alternatively, the reader
may consult Tits' paper~\cite{Ti}, where both definitions are discussed.


We now start discussing the proof of Theorem~\ref{thm:mothergroup}.

\subsection{Bounded rank case}
Fix a reduced irreducible classical root system $\Phi$. In order to establish
Theorem~\ref{thm:mothergroup} in the bounded rank case we need to prove the following:
\begin{itemize}
\item[(1)] (untwisted case) If $rk(\Phi)\geq 2$, the family of all finite simple
groups of the form $\Phi(q)$ (where $q$ is an arbitrary prime power) admits a mother group with $(T)$.
\item[(2)] (twisted case) If $l=2$ or $3$ is such that the Lie type $^l \Phi$ is defined
and has rank $\geq 2$, the family of all finite simple
groups of the form $^l \Phi(q)$ admits a mother group with $(T)$.
\end{itemize}
We start with the untwisted case where the argument is very simple. According to our table,
it is enough to find a mother group with $(T)$ for the family $\St_{\Phi}(\F_q)$.
Since a finite field is generated by one element (as a ring), it is a quotient of $\Z[t]$ and thus $\St_{\Phi}(\F_q)$ is a quotient of $\St_{\Phi}(\Z[t])$, which has property $(T)$ since
$rk(\Phi)\geq 2$.

Now we turn to twisted groups.
Although $\Z[t]$ can still be used as the ``covering ring'' in many cases, to simplify
the arguments, we will use slightly larger rings. In cases~1-4 below we let
$R=\Z[t_1,t_2]$, the ring of polynomials in two (commuting) variables,
and let $*:R\to R$ be the involution (also an automorphism of order $2$ since $R$ is commutative)
which permutes $t_1$ and $t_2$.

\begin{Lemma}
\label{lem:symfun} 
Let $S=\Sym(R,*)$, the set of elements of $R$ fixed by $*$. The following hold:
\begin{itemize}
\item[(1)] $S=\Z[t_1+t_2,t_1 t_2]$
\item[(2)] $R$ is a finitely generated left module over $S$.
\item[(3)] $S$ is generated by $1$ as a form parameter of $(R,-1)$.
\item[(4)] $\Asym(R,*)=\Asym^{\min}(R,*)$
\end{itemize}
\end{Lemma}
\begin{proof} 
(1) is, of course, a standard result about symmetric functions,
and (2) holds since $R$ is finitely generated and integral over $S$.

We now prove (3). Let $J$ be the form parameter of $(R,-1)$ generated by
$1$. It suffices to show that $(t_1+t_2)^n (t_1 t_2)^m\in J$ for any $n,m>0$.
If $n=0$ this is clear since $(t_1 t_2)^m=t_1^m {(t_1^m)}^*$, and
if $n>0$, this follows from equality $(t_1+t_2)^n (t_1 t_2)^m=r+r^*$ where $r=t_1(t_1+t_2)^{n-1}(t_1 t_2)^m$.

Finally, to prove (4), note that any $r\in \Asym(R,*)$ cannot contain monomials of the form
$t_1^i t_2^i$ with nonzero coefficients, and for any $i\neq j$ the coefficients
of $t_1^i t_2^j$ and $t_1^j t_2^i$ must be opposite, so $r=s-s^*$ for some $s\in R$.
\end{proof}

Now we begin case-by-case proof.

\emph{Case 1: ${}^l\Phi= {}^2A_{2n-1} (n\ge 2)$.} As before, let
$\upbar:\F_{q^2}\to \F_{q^2}$ be the automorphism of $\F_{q^2}$ of order $2$.
As we see from the table, ${}^2A_{2n-1}(q)$ is a quotient of
$\St^1_{C_n}(\F_{q^2}, \upbar)$, and by Observation~\ref{symisoasym},  $\St^1_{C_n}(\F_{q^2}, \upbar)$
is isomorphic to  $\St^{-1}_{C_n}(\F_{q^2}, \upbar)$. Let $\beta$ be a generator of $\F_q$, and choose
$\alpha\in \F_{q^2}\setminus \F_q$ such that $\alpha+\bar \alpha=\beta$. Such $\alpha$ exists
since there are $q$ elements of $\F_{q^2}$ whose trace in $\F_q$ is equal to $\beta$ and not all
these elements lie in $\F_q$. Note that  $\bar\alpha=\alpha^q$, so the subfield generated by
$\alpha$ properly contains $\F_q$ and thus must equal $\F_{q^2}$.

Now define a homomorphism  $\pi:R\to \F_{q^2}$ by setting $\pi(t_1)=\alpha$ and $\pi(t_2)=\bar\alpha$.
By construction, $\pi$ is involution preserving, that is, $\pi(r^*)=\overline{\pi(r)}$ for all $r\in R$,
$\pi$ is surjective by the choice of $\alpha$ and $\pi(\Sym(R,*))=\Sym(\F_{q^2},\upbar)=\F_q$ by the choice
of $\beta$. Therefore, from the definition of Steinberg groups of type $C_n$, it is clear that
$\pi$ induces an epimorphism from $\St^{-1}_{C_n}(R, *)$ to $\St^{-1}_{C_n}(\F_{q^2}, \upbar)$.
The group $\St^{-1}_{C_n}(R, *)$ has property $(T)$ by Proposition~\ref{prop:ex1} and Lemma~\ref{lem:symfun}(2)(3),
with part (2) only needed for $n=2$. Note that the only reason we had to use isomorphism between
$\St^1_{C_n}(\F_{q^2}, \upbar)$ and $\St^{-1}_{C_n}(\F_{q^2}, \upbar)$ is to cover the case $n=2$
since we do not know any sufficient condition for property $(T)$ for Steinberg groups of type $\St^1_{C_2}$.
\vskip .2cm

\emph{Case 2: ${}^l\Phi= {}^2A_{2n} (n\ge 2)$.}
For $n\geq 3$ we can again use the ring $R=\Z[t_1,t_2]$ with the same involution $*$.
Let $\alpha$ be a generator of $\F_{q^2}$, and define an epimorphism $\pi:R\to \F_{q^2}$
by setting $\pi(t_1)=\alpha$ and $\pi(t_2)=\bar\alpha$.

Note that $\Asym(\F_{q^2},\upbar)$ is precisely the set of elements of $\F_{q^2}$ whose
$\F_q$-trace is equal to zero, and therefore $\Asym(\F_{q^2},\upbar)=\Asym^{\min}(\F_{q^2},\upbar)$
by Hilbert's Theorem 90. Therefore, 
$$
\pi(\Asym(R,*))=\Asym(\F_{q^2},\upbar).
$$ 
Indeed,
the left-hand side is clearly contained in the right-hand side, but on the other hand
$\pi(\Asym(R,*))\supseteq  \pi(\Asym^{\min}(R,*))=\Asym^{\min}(\F_{q^2},\upbar)= \Asym(\F_{q^2},\upbar)$.
This implies that $\pi$ induces an epimorphism from $\St^{1}_{BC_n}(R,*)$ to
$\St^{1}_{BC_n}(\F_{q^2},\upbar)$, and the group $\St^{1}_{BC_n}(R,*)$ has property $(T)$
by Proposition~\ref{propT_ex2} (note that the hypothesis about finite generation as a left ideal
holds automatically since $R$ is Noetherian).

If $n=2$, Proposition~\ref{propT_ex2} is not applicable in this setting since $R$ does not
have invertible antisymmetric elements. To fix this problem, we consider the larger ring
$R'=R[s,1/s]$, and extend the involution $*$ to $R'$ by setting $s^*=-s$.
Extend the map $\pi:R\to\F_{q^2}$ to an epimorphism $\pi':R'\to\F_{q^2}$ by sending
$s$ to any nonzero element with zero $\F_q$-trace. By the same argument as above,
$\pi'$ induces an epimorphism form $\St^{1}_{BC_2}(R',*)$ to $\St^{1}_{BC_2}(\F_{q^2},\upbar)$,
and it remains to show that the pair $(R',*)$ satisfies the hypotheses of Proposition~\ref{propT_ex2}
for $n=2$.

By construction, $s$ is an invertible antisymmetric element of $R'$. A straightforward
computation shows that $\Asym(R',*)$ is generated by $s$ as a form parameter of $(R',*,-1)$.
Finally, $R'$ is clearly a finitely generated module over $\Z[t_1+t_2,t_1t_2,s^2,1/s^2]$, which
is a finitely generated subring of $\Sym(R',*)$ (in fact, it is easy to show that this subring
coincides with $\Sym(R',*)$).

\vskip .2cm

\emph{Case 3: ${}^l\Phi= {}^2D_n(n\ge 4)$.} The proof in case is completely analogous to Case~1.
This time we use $\St_{B_{n-1}}(R,*,R)$ as a mother group. It has property $(T)$ by
Proposition~\ref{propT_Bn}.
\vskip .2cm

\emph{Case 4: ${}^l\Phi= {}^2E_6$.} The proof in this case is also analogous to Case 1.
\vskip .2cm

\emph{Case 5: ${}^l\Phi= {}^3D_4$.} Here we need a slight modification of the argument in Case~1.
Let $R=\Z[t_1,t_2,t_3]$, and let $\sigma$ be the automorphism of order $3$ which cyclically
permutes $t_1,t_2$ and $t_3$. As in Lemma~\ref{lem:symfun},
$R^{\sigma}=\Z[t_1+t_2+t_3,t_1t_2+t_1t_3+t_2t_3, t_1 t_2 t_3]$ and $R$ is a finitely
generated module over $R^{\sigma}$, so the group $\St_{G_2}(R,\sigma)$ has property $(T)$
by Proposition~\ref{prop:3D4}.

Now let $\beta$ be a generator of $\F_q$, and choose any
$\alpha\in \F_{q^3}\setminus \F_q$ with
$\alpha+\theta(\alpha)+\theta^2(\alpha)=\beta$. As in Case 1, $\alpha$ generates $\F_{q^3}$,
and the map $\pi:R\to \F_{q^3}$ given by $\pi(t_i)=\theta^{i-1}(\alpha)$ for $i=1,2,3$
induces an epimorphism from $\St_{G_2}(R,\sigma)$  to $\St_{G_2}(\F_{q^3},\theta)$.
\vskip .2cm

\emph{Case 6: ${}^l\Phi= {}^2F_4$.} 
Let $R=\F_2[t_1,t_2]$ and let $p:R\to R$ be the homomorphism
which sends $t_1$ to $t_2$ and $t_2$ to $t_1^2$, and define $*:\F_{2^{2k+1}}\to \F_{2^{2k+1}}$
by $x^*=x^{2^{k+1}}$. Let $\alpha$ be a generator of $\F_{2^{2k+1}}$, and define $\pi:R\to \F_{2^{2k+1}}$ by $\pi(t_1)=\alpha$ and $\pi(t_2)=\alpha^{2^{k+1}}$. Then $\pi(p(r))=r^*$ for any $r\in R$, so $G=\St_{^2F_4}(R,p)$ surjects onto 
$\St_{^2F_4}(\F_{2^{2k+1}},*)$ which, in turn,
surjects onto $ ^2F_4(2^{k})$ according to our table. Since $G$ has property $(T)$ by
Proposition~\ref{prop:2F4}, the proof is complete.

\subsection{Unbounded rank case: overview}
\label{subsec:overview}

Now we produce a mother group with property (T) for finite simple groups of Lie type of
sufficiently high rank, considering separately the families $A_n$, $C_n$, $D_n$, $^2A_n$ ($n$ odd),
$B_n$, $^2A_n$ ($n$ even) and  $^2D_n$ (in this order).

Our general procedure is as follows. Let ${\mathbf \Phi}=\{\Phi_n\}$ be a Lie family that we are
considering. The groups $\Phi_n(q)$ have a description as classical groups. We will define
a  $\Psi$-grading for each $\Phi_n(q)$ with $n$ sufficiently large, where $\Psi$ is a root
system depending only  on ${\mathbf \Phi}$ (and not on $n$). This grading can be obtained by
coarsening the canonical $\Phi_n$-grading of $\Phi_n(q)$ using a suitable reduction $\eta:\Phi_n\to\Psi$,
but this fact will not be essential for the proof. Then we will construct a $\Psi$-graded group
$G_{\mathbf \Phi}$ which maps onto all $\Phi_n(q)$ with $n$ sufficiently large and use of one of
the criteria from~\S~\ref{sec:twisted} to prove that $G_{\mathbf \Phi}$ has property $(T)$.
In order to simplify arguments  (in particular, to show that $G_{\mathbf \Phi}$ has property $(T)$),
we will sometimes make $\Psi$  larger than it could be.
\vskip .2cm

We start with a brief outline of the proof in the case ${\mathbf \Phi}=\{A_{n-1}\}$, which should also give the reader some idea on how other cases will be handled.
Let us first treat a special case when for a fixed integer $d\geq 3$, we consider only $n$ which are multiples of $d$.
If $n=dk$, then as we
already observed in the introduction, considering $n\times n$ matrices as $d\times d$ block
matrices with each block being a $k\times k$ matrix, we obtain a natural isomorphism
$\SL_n(\F_q)=\EL_n(\F_q)\cong \EL_d(\Mat_k(\F_q))$.
It is well known (see Lemma~\ref{lemma:matrixgeneration1} below) that $\Mat_k(\F_q)$ can be generated as a ring by two matrices,
so it is a quotient of $\Z\la x,y \ra$, the free associative ring in two variables.
Therefore, $\EL_d(\Mat_k(\F_q))$ is a quotient
of $\EL_d(\Z\la x,y \ra)$, and the latter group has property $(T)$.

In general, as long as $n\geq d$, we can consider an $n\times n$ matrix as a $d\times d$ block matrix,
but we can no longer guarantee that blocks will have the same size. We will use block decompositions
with $d=4$ in which the first three diagonal blocks have the same size $k$ and the last diagonal block has size $l\leq k$
(so that $n=3k+l$). The isomorphism $\EL_{dk}(\F_q)\cong \EL_d(\Mat_k(\F_q))$ has natural analogue in this setting, which yields the corresponding $A_{d-1}$-grading on $\EL_n(\F_q)$. In fact, this grading is the precisely the coarsened $A_{d-1}$-grading of $\EL_n(\F_q)$ corresponding to a suitable reduction $A_{n-1}\to A_{d-1}$.

The obtained $A_{d-1}$-grading shows that the group $\SL_n(\F_q)=\EL_n(\F_q)$ is a quotient of the Steinberg group $\St_{A_3}(R,M,N)$
described in Proposition~\ref{Trootcontent_typeAn} where $R=\Mat_{k}(\F_q)$, $M=\Mat_{k\times l}(\F_q)$
and $N=\Mat_{l\times k}(\F_q)$. If $\Rhat$ is a finitely generated ring, $\pi:\Rhat\to R$ an epimorphism
and we choose a left ideal $\Mhat$ of $\Rhat$ and a right ideal $\Nhat$ of $\Rhat$
such that $\pi(\Nhat)=N$ and $\pi(\Mhat)=M$, then it is clear from the definitions that $\St_{A_3}(R,M,N)$
becomes a quotient of $\St_{A_3}(\Rhat,\Mhat,\Nhat)$. However, to prove that the latter group has property $(T)$,
we need to know that $\Mhat$ and $\Nhat$ are finitely generated and that $\Mhat\Nhat=\Rhat$. We do not
know how to achieve these conditions if we simply take $\Rhat=\Z\la x,y \ra$ (most likely it is impossible),
so instead we will make the ring $\Rhat$ a little larger by adding suitable generators and relations.

Define $R_{main}$ to be the associative ring on four generators $x,y,z,w$
subject to one relation 
\begin{equation}
\label{eq:***}
z^2+wz^2w=1,
\end{equation}
that is,
$$
R_{main}=\Z\la x,y,z,w\ra/(z^2+wz^2w-1).
$$
In the case when $\Phi$ is of type $A$, we will take $\Rhat=R_{main}$, and the epimorphism $\pi:\Rhat\to \Mat_k(\F_q)$ will
be constructed so that we can take $\Mhat=\Rhat z$ and $\Nhat=z\Rhat $. Thus, $\Mhat\Nhat=\Rhat z^2\Rhat$, the two-sided
ideal generated by $z^2$, which is equal to $\Rhat$ by the relation~\eqref{eq:***} we imposed. The precise form of this relation
is chosen so that we can use $R_{main}$ as a model for the ``covering ring'' $\Rhat$ in other cases ($\Phi$ is not of type $A$) when an involution with suitable properties will need to be defined on $\Rhat$.
\vskip .2cm

\vskip .2cm
{\bf Some notations.}
Before proceeding, we introduce some general notations and state two results on generation of matrix rings
that will be repeatedly used below without further mention.

If $S$ is a ring (possibly non-commutative), we denote by $S\la t\ra$ the (ring-theoretic) free product
of $S$ with $\Z[t]$ and by $S[t]$ the largest quotient of $S\la t\ra$ in which $t$ is central,
that is, the ring of polynomials over $S$ in one variable $t$. For instance, $\Z\la x\ra=\Z[x]$  and $(\Z\la x\ra)\la y\ra=\Z\la x,y \ra$, the free associative ring in two variables.

\vskip .12cm
If $R$ is a ring, $r$ an element of $R$ and $i,j$ are positive integers, by $(r)_{i,j}$ we will denote
the matrix whose $(i,j)$-entry is equal to $r$ and all other entries are equal to $0$ -- the size of the matrix
is not specified in the notation, but will always be clear from the context. Using this notation, we put
$$
E_{i,j}=(1)_{i,j} \mbox{ and } \Id_k=\sum_{i=1}^k E_{i,i}.
$$
The matrix $\Id_k$ is, of course, the identity element of $\Mat_k(R)$, but may also be considered as an element
of $\Mat_{i\times j}(R)$ for any $i,j\geq k$.
\vskip .12cm

If $(k_1,\ldots, k_d)$ is a sequence of positive integers, by a block matrix of type $(k_1,\ldots, k_d)$,
we will mean a $d\times d$ block-diagonal matrix whose $(i,j)$-block is a $k_i\times k_j$ matrix. Thus,
any $n\times n$ matrix can be considered as a block matrix of type $(k_1,\ldots, k_d)$ whenever $\sum k_i=n$. For $1\leq i\leq n$ denote by ${\rm block}(i)$ the block into which index $i$ falls under this
decomposition, that is, ${\rm block}(i)=j$ if $\sum_{t<j}k_t<i\leq \sum_{t\leq j}k_t$.
\vskip .1cm

We will frequently use the notation $(r)_{i,j}$ introduced above in this setting of block matrices --
for instance, if we consider $10\times 10$ matrices as block matrices of type $(5,3,2)$, then for any $5\times 3$ matrix $A$,
the block matrix 
$
\begin{pmatrix} 0_{5\times 5}& A & 0_{5\times 2}\\
0_{3\times 5}& 0_{3\times 3} & 0_{3\times 2}\\ 0_{2\times 5}& 0_{2\times 3} & 0_{2\times 2}
\end{pmatrix}
$
will be denoted by $(A)_{1,2}$.
\vskip .12cm
We will need the following result on generation in matrix rings, which
will be proved at the end of~\S~\ref{unbounded}.

\begin{Lemma}
\label{lemma:matrixgeneration1} Let $F$ be a finite field and $k$ an integer. If $k\ge 2$, then $\Mat_{k}(F)$
can be generated by two symmetric matrices.
If $|F|=q^2$, $\sigma$ is the automorphism of $F$
of order $2$ and $k\ge 3$, then $\Mat_k(F)$ can be generated by two hermitian (with respect to $\sigma$)
matrices.
\end{Lemma}

\subsection{Unbounded rank case: proof}
\label{unbounded}
We now begin the formal case-by-case proof.
In each of the seven cases considered below we will define integers
$k$ and $l$ satisfying $l\leq k\leq 2l$ and let
$$Z=\Id_l\mbox{ and } W=\sum_{j=1}^{k-l}(E_{j,j+l}+E_{j+l,j}).$$
By direct computation we have
\begin{equation}
\label{lemma:matrixgeneration2}
\Id_k=Z^2+WZ^2W
\end{equation}
This equation reveals where the relation \eqref{eq:***} in the definition of the ring $R_{main}$
comes from.
\vskip .3cm

{\bf Case 1: $\mathbf \Phi=A$. }
There exists a group $G_A$ with property $(T)$
which maps onto  $\SL_n(\F_q)$ for $n\ge 18$.

Since $n\geq 18$, it is easy to see that we can write $n=3k+l$ where $l\leq k\leq 2l$.
Considering $n\times n$ matrices as $4\times 4$ block matrices of type $(k,k,k,l)$,
we obtain a natural $A_3$-grading $\{X_{\gamma}\}$ of $\SL_n(\F_q)$ described below.
Recall that $\Id_k$ denotes the unit $k\times k$ matrix.
The corresponding positive root subgroups are

{
\small
$$
X_{e_1-e_2}= \!\!
\left(\!\! 
\begin{array} {cccc} \Id_k & \Mat_k(\F_q) & 0& 0\\ 0 & \Id_k & 0 & 0 \\ 0 & 0 & \Id_k & 0\\0 & 0 & 0 & \Id_l \end{array}
\!\!\right),\
X_{e_1-e_4}= \!\!
\left(\!\!
\begin{array} {cccc} \Id_k & 0 & 0 &\Mat_{k\times l}(\F_q)  \\ 0 & \Id_k & 0 & 0 \\ 0 & 0 & \Id_k & 0\\0 & 0 & 0 & \Id_l \end{array}
\!\!\right),
$$
$$
X_{e_1-e_3}= \!\!
\left(\!\!
\begin{array} {cccc} \Id_k & 0& \Mat_k(\F_q)  & 0\\ 0 & \Id_k & 0 & 0 \\ 0 & 0 & \Id_k & 0\\0 & 0 & 0 & \Id_l \end{array}
\!\!\right),\
X_{e_2-e_4}= \!\!
\left(\!\!
\begin{array} {cccc} \Id_k &0 & 0& 0\\ 0 & \Id_k & 0 & \Mat_{k\times l}(\F_q) \\ 0 & 0 & \Id_k & 0\\0 & 0 & 0 & \Id_l \end{array}
\!\!\right),
$$
$$
X_{e_2-e_3}= \!\!
\left(\!\!
\begin{array} {cccc} \Id_k &0 & 0& 0\\ 0 & \Id_k &  \Mat_k(\F_q) & 0 \\ 0 & 0 & \Id_k & 0\\0 & 0 & 0 & \Id_l \end{array}
\!\!\right),\
X_{e_3-e_4}= \!\!
\left(\!\!
\begin{array} {cccc} \Id_k &0 & 0& 0\\ 0 & \Id_k &  0& 0 \\ 0 & 0 & \Id_k & \Mat_{k\times l}(\F_q) \\0 & 0 & 0 & \Id_l \end{array}
\!\!\right).
$$
}
If $\gamma\in A_3$ is a negative root, we put $X_\gamma=(X_{-\gamma})^{\transp}$, where
$\transp$ denotes transposition.

It is easy to check that this grading is the coarsened grading corresponding to the reduction  $\eta:A_{n-1}\to A_3$ given by $\eta(e_i)=e_{{\rm block}(i)}$.

\vskip .2cm
If we set $k_1=k_2=k_3=k$ and $k_4=l$, then using our shortcut notations, we can rewrite
the definition of the above root subgroups as follows:
$$
X_{e_i-e_j}=\{\Id_n+(A)_{i,j}\ : A\in \Mat_{k_i\times k_j}(\F_q)\}.
$$
\vskip .2cm
Let $R=\Mat_k(\F_q)$. By Lemma~\ref{lemma:matrixgeneration1}, $R$ is generated by two matrices, say $X$ and $Y$, and observe that $\Mat_{k\times l}(\F_q)=RZ$ and $\Mat_{k\times l}(\F_q)=ZR$ where as we recall $Z=\Id_l\in R$.

Now let $\Rhat=R_{main}$, and define
$G_A$ to be the subgroup of $\EL_4(\Rhat)$ generated by the subgroups $\{\Xhat_{\gamma}\}_{\gamma\in A_3}$ described below:
{
\small
$$
\Xhat_{e_1-e_2}= 
\!\!\left(\!\!
\begin{array} {cccc} 1 & \Rhat & 0& 0\\ 0 & 1 & 0 & 0 \\ 0 & 0 &1 & 0\\0 & 0 & 0 & 1 \end{array}
\!\!\right)\!\!,\
\Xhat_{e_1-e_3} = 
\!\!\left(\!\! 
\begin{array} {cccc} 1 & 0&\Rhat & 0\\ 0 & 1 & 0 & 0 \\ 0 & 0 & 1 & 0\\0 & 0 & 0 & 1\end{array}
\!\!\right)\!\!,\
\Xhat_{e_1-e_4} = 
\!\!\left(\!\!
\begin{array} {cccc} 1 & 0 & 0 &\Rhat z  \\ 0 &1 & 0 & 0 \\ 0 & 0 &1 & 0\\0 & 0 & 0 &1 \end{array}
\!\!\right)\!\!,
$$
$$
\Xhat_{e_2-e_3} = 
\!\!\left(\!\!
\begin{array} {cccc} 1 &0 & 0& 0\\ 0 & 1 & \Rhat& 0 \\ 0 & 0 & 1 & 0\\0 & 0 & 0 & 1 \end{array}
\!\!\right)\!\!,\
\Xhat_{e_2-e_4} = 
\!\!\left(\!\!
\begin{array} {cccc} 1&0 & 0& 0\\ 0 & 1 & 0 & \Rhat z\\ 0 & 0 & 1& 0\\0 & 0 & 0 &1\end{array}
\!\!\right)\!\!,\
\Xhat_{e_3-e_4} = 
\!\!\left(\!\!
\begin{array} {cccc} 1 &0 & 0& 0\\ 0 & 1 &  0& 0 \\ 0 & 0 & 1 &\Rhat z \\0 & 0 & 0 & 1 \end{array}
\!\!\right)\!\!,
$$
$$
\Xhat_{e_2-e_1}= 
\!\!\left(\!\!
\begin{array} {cccc} 1 & 0 & 0& 0\\ \Rhat & 1 & 0 & 0 \\ 0 & 0 &1 & 0\\0 & 0 & 0 & 1 \end{array}
\!\!\right)\!\!,\
\Xhat_{e_3-e_1} = 
\!\!\left(\!\!
\begin{array} {cccc} 1 & 0& 0 & 0\\ 0 & 1 & 0 & 0 \\ \Rhat & 0 & 1 & 0\\0 & 0 & 0 & 1\end{array}
\!\!\right)\!\!,\
\Xhat_{e_4-e_1} = 
\!\!\left(\!\!
\begin{array} {cccc} 1 & 0 & 0 & 0 \\ 0 &1 & 0 & 0 \\ 0 & 0 &1 & 0\\ \Rhat z& 0 & 0 &1 \end{array}
\!\!\right)\!\!,
$$
$$
\Xhat_{e_3-e_2} = 
\!\!\left(\!\!
\begin{array} {cccc} 1 &0 & 0& 0\\ 0 & 1 & 0& 0 \\ 0 & \Rhat & 1 & 0\\0 & 0 & 0 & 1 \end{array}
\!\!\right)\!\!,\
\Xhat_{e_4-e_2} = 
\!\!\left(\!\!
\begin{array} {cccc} 1&0 & 0& 0\\ 0 & 1 & 0 & 0\\ 0 & 0 & 1& 0\\0 & \Rhat z & 0 &1\end{array}
\!\!\right)\!\!,\
\Xhat_{e_4-e_3} = 
\!\!\left(\!\!
\begin{array} {cccc} 1 &0 & 0& 0\\ 0 & 1 &  0& 0 \\ 0 & 0 & 1 & 0\\0 & 0 & \Rhat z & 1 \end{array}
\!\!\right)\!\!.
$$
}
Recall that $W\in \Mat_k(\Z)$ satisfies $Z^2+WZ^2W=1_R$.
Let $\pi:\Rhat\to R$ be the (unique) epimorphism such that $\pi(x)=X$, $\pi(y)=Y$, $\pi(z)=Z$
and $\pi(w)=W$.
Note that $\pi$ induces an epimorphism from $G_A$ onto $\SL_n(\F_q)$. On the other hand,
it is clear from the above definition that $G_A$ is a quotient of the Steinberg group
$\St_{A_3}(\Rhat;\Rhat z, z\Rhat)$ defined in Proposition~\ref{Trootcontent_typeAn}.
The product $\Rhat z \cdot z\Rhat=\Rhat z^2\Rhat$ is equal to $\Rhat$ thanks to the relation
$z^2+wz^2w=1$. Hence by Proposition~\ref{Trootcontent_typeAn}, the group
$\St_{A_3}(\Rhat;\Rhat z,z\Rhat)$ has property $(T)$, and so does its quotient $G_A$.

\vskip .2cm

{\bf Case 2: $\mathbf \Phi=C$.} There exists a group $G_C$ with property $(T)$ which maps onto $\mathrm{Sp}_{2n}(\F_q)$ for $n\ge 18$.

\

Write $n=3k+l$ where $l\leq k\leq 2l$. Let 
{
\small
$$
J=\left ( \begin{array}{cccccccc} 
0 & 0 & 0 & 0 & 0 & 0 & 0 & \Id_k\\
0 & 0 & 0 & 0 & 0 & 0 & \Id_k&0\\  
0 & 0 & 0 & 0 & 0 &\Id_k& 0 & 0\\  
0 & 0 & 0 & 0 &\Id_l & 0 & 0 & 0\\  
0 & 0 & 0 &-\Id_l & 0 & 0 & 0 &0\\ 
0 & 0 &-\Id_k& 0 & 0 & 0 & 0 &0\\ 
0 & - \Id_k & 0 & 0 & 0 & 0 & 0 &0\\ 
-\Id_k & 0 & 0 & 0 & 0 & 0 & 0 &0
\end{array}\right )\in \Mat_{2n}(\F_q).
$$
We use the following realization of $Sp_{2n}(\F_q)$:
$$
\mathrm{Sp}_{2n}(\F_q)=\{A\in \Mat_{2n}(\F_q): \ A^{\transp}JA=J.\}
$$
Considering $2n\times 2n$ matrices as $8\times 8$ block matrices of type $(k,k,k,l,l,k,k,k)$,
we obtain a natural $C_4$-grading $\{X_{\gamma}\}_{\gamma\in C_4}$ of $\mathrm{Sp}_{2n}(\F_q)$.

Below $1\leq i<j\leq 4$, $\bar i=9-i$, $\bar j=9-j$, and we set $k_1=k_2=k_3=k$ and $k_4=l$.
The positive root subgroups $X_{\gamma}$ are defined as follows:
\begin{align*}
&X_{e_i-e_j}=\{\Id_{2n}+(A)_{i,j}-(A^{\transp})_{\bar j,\bar i}:\ A\in \Mat_{k_i\times k_j}(\F_q)\}&\\
&X_{e_i+e_j}=\{\Id_{2n}+(A)_{i,\bar j}+(A^{\transp})_{j,\bar i}:\ A\in \Mat_{k_i\times k_j}(\F_q)\}&\\
&X_{2e_i}=\{\Id_{2n}+(A)_{i,\bar i}:\ A=A^{\transp}\in \Mat_{k_i\times k_i}(\F_q)\}&
\end{align*}
The negative root subgroups can be obtained by the formulas 
$X_{-\gamma}=(X_{\gamma})^{\transp}$.\
\vskip .1cm

This grading is the coarsened grading corresponding to the reduction
$\eta:C_n\to C_4$ given by $\eta(e_i)=e_{\block(i)}$ for $1\leq i\leq n$.
\vskip .2cm

As in case 1, let $R=\Mat_k(\F_q)$, $Z=\Id_l\in R$, recall that
$\Mat_{k\times l}(\F_q)=RZ$ and $\Mat_{l\times k}(\F_q)=ZR$, and note that
$\Mat_{l}(\F_q)=ZRZ$. Also observe that
the set  
$$
I=\{A=A^{\transp}\in \Mat_{k}(\F_q)\}
$$ 
of symmetric matrices in $R$ is a form parameter of
$(R,\transp,-1)$ and that 
$$
\{A=A^{\transp}\in \Mat_{l}(\F_q)\}=ZIZ.
$$
It is also easy to see that $I$ is generated by $U=E_{11}$ (as a form parameter).

Now let $\Rhat=R_{main}\la u \ra=\Z\la x,y,z,w,u\ra/(z^2+wz^2w-1)$.
Let $*$ be the involution of $\Rhat$ that fixes $x,y,z,w$ and $u$
(such involution certainly exists on the free associative ring $\Z\la x,y,z,w,u\ra$,
and since that involution preserves the element
$z^2+wz^2w-1$, it induces an involution on $\Rhat$ with required properties).

Let $\Ihat$ be the form parameter of $(\Rhat,*,-1)$  generated by $u$. Define the subsets
$\{\Rhat_{\gamma}\}_{\gamma\in C_4}$ of $\Rhat$ by
\begin{align*}
\Rhat_{e_i-e_j}=&
\left\{
\begin{array}{ll}
\Rhat,&\mbox{ if } 1\leq i\neq j\leq 3\\
\Rhat z,&\mbox{ if } j=4\\
z\Rhat,&\mbox{ if } i=4\\
\end{array}
\right.
\\
\Rhat_{e_i+e_j}=&
\left\{
\begin{array}{ll}
\Rhat,&\mbox{ if } 1\leq i< j\leq 3\\
\Rhat z,&\mbox{ if } 1\leq i\leq 3, j=4
\end{array}
\right.
\\
\Rhat_{-(e_i+e_j)}=&
\left\{
\begin{array}{ll}
\Rhat,&\mbox{ if } 1\leq i<j\leq 3\\
z\Rhat,&\mbox{ if } 1\leq i\leq 3, j=4
\end{array}
\right.
\\
\Rhat_{2e_i}=&
\left\{
\begin{array}{ll}
\Ihat,&\mbox{ if } 1\leq i\leq 3\\
z\Ihat z,&\mbox{ if } i=4,
\end{array}
\right.
\end{align*}
and define $G_C$ to be the subgroup of $\EL_8(\Rhat)$ generated by the following subgroups $\{\Xhat_{\gamma}\}_{\gamma\in C_4}$:
\begin{align*}
\Xhat_{e_i-e_j}=&\{\Id_{8}+(r)_{i,j}-(r^*)_{\bar j,\bar i}:\ r\in \Rhat_{e_i-e_j}\}\\
\Xhat_{e_i+e_j}=&\{\Id_{8}+(r)_{i,\bar j}+(r^*)_{j,\bar i}:\ r\in \Rhat_{e_i+e_j}\}\\
\Xhat_{2e_i}=&\{\Id_8+(r)_{i,\bar i}:\ r\in\Rhat_{2e_i}\}
\end{align*}

Choose two symmetric matrices $X$ and $Y$ which generate $R$, and let $\pi:\Rhat\to R$ be the  epimorphism given by $\pi(x)=X, \pi(y)=Y, \pi(z)=Z$, $\pi(u)=U$ and $\pi(w)=W$.
By construction, $\pi$ is involution-preserving:
$\pi(r^*)=(\pi(r))^{\transp}$ for any $r\in \Rhat$, and from the above description
it is clear that $\pi$ induces an epimorphism from $G_C$ to $Sp_{2n}(\F_q)$.

On the other hand, by construction, $G_C$ is a quotient of the group
$\St_{C_4}^{-1}(\Rhat, *,\Ihat; \Rhat z)$ described in Proposition~\ref{Trootcontent_typeCn}.
Since $\Rhat z (\Rhat z)^*=\Rhat z^2 \Rhat=\Rhat$ and $\Ihat$ is finitely generated by
construction, we conclude that $G_C$ has property $(T)$.

{\bf Case 3: $\mathbf \Phi=D$.} There exists a group $G_D$ with property $(T)$ which maps onto $\Omega^+_{2n}(\F_q)$ for $n\ge 18$.

Again write $n=3k+l$, where $l\leq k\leq 2l$. We shall consider
$n\times n$ matrices as $8\times 8$ block matrices of type $\overrightarrow{k}=(k_1,k_2,k_3,k_4,k_5,k_6,k_7,k_8)=(k,k,k,l,l,k,k,k)$
and use the same notational convention as in case 2.

Let 
$$
J=\small 
\left( 
\begin{array}{cccccccc} 
0 & 0 & 0 & 0 & 0 & 0 & 0 & \Id_k \\
0 & 0 & 0 & 0 & 0 & 0 & \Id_k & 0 \\  
0 & 0 & 0 & 0 & 0 & \Id_k & 0 & 0 \\ 
0 & 0 & 0 & 0 & \Id_l & 0 & 0 & 0 \\  
0 & 0 & 0 & 0 & 0 & 0 & 0 & 0 \\ 
0 & 0 & 0 & 0 & 0 & 0 & 0 & 0 \\ 
0 & 0 & 0 & 0 & 0 & 0 & 0 & 0 \\ 
0 & 0 & 0 & 0 & 0 & 0 & 0 & 0
\end{array}
\right)=
\sum_{i=1}^4 (\Id_{k_i})_{i,\bar i}\in \Mat_{2n}(\F_q).
$$
We realize $O^+_{2n}(\F_q)$ as the group of  matrices $M\in \GL_{2n}(\F_q)$ that preserve the quadratic form 
$q(u,u)=u^{\transp}Ju$, where $u\in \Mat_{2n\times 1}(\F_q)\cong   \F_q^{2n}$:
$$
O^+_{2n}(\F_q)=\{M\in \GL_{2n}(\F_q)\ : q(Mu,Mu)=q(u,u) \mbox{ for all }u\in\F_q^{2n}\}
$$
Note that $O^+_{2n}(\F_q)$ is a subgroup of the group 
$$
\{M\in GL_{2n}(\F_q)\ : M^{\transp}(J^{\transp}+J)M=J^{\transp}+J\},
$$
and the two groups coincide if $q$ is odd.

For a positive integer $m$, define $\Asym_{0}(m,q)$ to be the set of antisymmetric matrices in $\Mat_m(\F_q)$
with diagonal entries equal to zero.
The group $\Omega^+_{2n}(\F_q)$ has the following $C_4$-grading $\{X_{\gamma}\}_{\gamma\in C_4}$:
\begin{align*}
X_{e_i-e_j}=&\{\Id_{2n}+(A)_{i,j}-(A^{\transp})_{\bar j,\bar i}:\ A\in \Mat_{k_i\times k_j}(\F_q)\}&\\
X_{e_i+e_j}=&\{\Id_{2n}+(A)_{i,\bar j}-(A^{\transp})_{j,\bar i}:\ A\in \Mat_{k_i\times k_j}(\F_q)\}&\\
X_{2e_i}=&\{\Id_{2n}+(A)_{i,\bar i}:\ A\in \Asym_0(k_i,q)\}&
\end{align*}
The negative root subgroups are given by $X_{-\gamma}=(X_{\gamma})^{\transp}$.
\vskip .1cm
This grading is the coarsened grading corresponding to the reduction
$\eta:D_n\to C_4$ given by $\eta(e_i)=e_{\block(i)}$ for $1\leq i\leq n$.

\vskip .2cm
Let $R=\Mat_k(\F_q)$. As in the previous cases, let $Z=\Id_l \in R$
and recall that $\Mat_{k\times l}(\F_q)=RZ$ and $\Mat_{l\times k}(\F_q)=ZR$. Let
$I=\Asym_0(k,q),$ and observe that
$$
I=\Asym^{\min}(\Mat_{k}(\F_q), \transp)\quad\mbox{ and }\quad ZIZ= \Asym^{\min}(ZRZ, \transp).
$$

Now let $\Rhat=R_{main}$, let $*$ be the involution of $\Rhat$ that fixes $x,y,z$ and $w$, and let $\Ihat=\Asym^{\min}(\Rhat,*)$.
Define subsets $\{R_{\gamma}\}_{\gamma\in C_4}$
precisely as in Case 2 (but with the new meaning of $\Rhat,\Ihat$ and $*$),
and define the group $G_D$ in terms of $\{R_{\gamma}\}$ as $G_C$ was defined in Case~2.

Choose symmetric matrices $X$ and $Y$ which generate $R$ as a ring, and let $\pi:\Rhat\to R$ be the epimorphism given by
$\pi(x)=X, \pi(y)=Y$, $\pi(z)=Z$ and $\pi(w)=W$.
It is clear that $\pi$ is involution preserving and $\pi(\Ihat)=I$.
Thus, $\pi$ induces an epimorphism $G_D\to \Omega^+_{2n}(\F_q)$.
On the other hand, $G_D$ is a quotient of the group $\St_{C_4}^1(\Rhat,*,\Ihat,\Rhat z)$,
which has property $(T)$ by Proposition~\ref{Trootcontent_typeCn}.

\vskip .2cm

{\bf Case 4: $\mathbf \Phi={}^2A_{odd}$.} The group $G_D$ (constructed in case 3) maps onto $\mathrm{SU}_{2n}(\F_q)$ for $n\ge 18$.

Again write $n=3k+l$, where $l\leq k\leq 2l$, and let
$$
J=\small\left(
\begin{array}{cccccccc} 
0 & 0 & 0 & 0 & 0 & 0 & 0 & \Id_k \\
0 & 0 & 0 & 0 & 0 & 0 & \Id_k & 0 \\  
0 & 0 & 0 & 0 & 0 & \Id_k & 0 & 0 \\  
0 & 0 & 0 & 0 & \Id_l & 0 & 0 & 0 \\  
0 & 0 & 0 & \Id_l & 0 & 0 & 0 & 0 \\ 
0 & 0 & \Id_k & 0 & 0 & 0 & 0 & 0 \\ 
0 & \Id_k & 0 & 0 & 0 & 0 & 0 & 0 \\ 
\Id_k & 0 & 0 & 0 & 0 & 0 & 0 & 0 
\end{array}\right )\in \Mat_{2n}(\F_{q^2}).
$$
Let $x\mapsto \overline x$ be the automorphism of order 2 of $\F_{q^2}$, and
given $A\in \Mat_{2n}(\F_{q^2})$, we let $\overline A$ be the matrix obtained from $A$ by applying this automorphism
to each entry. We realize the group $\mathrm{SU}_{2n}(\F_q)$ as follows:
$$
\mathrm{SU}_{2n}(\F_q)=\{ A\in \SL_{2n}(\F_{q^2})\ : \overline{A}^{\transp} J A=J\}.
$$
The group $\mathrm{SU}_{2n}(\F_q)$ admits the following $C_4$-grading $\{X_{\gamma}\}$. As in the previous case,
below $1\leq i<j\leq 4$, $\bar i=9-i$, $\bar j=9-j$, $k_1=k_2=k_3=k$ and $k_4=l$. We put
\begin{align*}
X_{e_i-e_j} =& \{\Id_{2n}+(A)_{i,j}-(\overline{A}^{\transp})_{\bar j,\bar i}:\ A\in \Mat_{k_i\times k_j}(\F_{q^2})\}&\\
X_{e_i+e_j} =& \{\Id_{2n}+(A)_{i,\bar j}-(\overline{A}^{\transp})_{j,\bar i}:\ A\in \Mat_{k_i\times k_j}(\F_{q^2})\}&\\
X_{2e_i} =& \{\Id_{2n}+(A)_{i,\bar i}:\ A=-\overline{A}^{\transp}\in \Mat_{k_i\times k_i}(\F_{q^2})\}&
\end{align*}
The negative root subgroups can be obtained by the formulas $X_{-\gamma}=(X_{\gamma})^{\transp}$.
\vskip .1cm

This grading is the coarsened grading corresponding to the reduction
$\eta:C_n\to C_4$ given by $\eta(e_i)=e_{\block(i)}$ for $1\leq i\leq n$.
\vskip .2cm

Let $R=\Mat_{k}(\F_{q^2})$, $Z=\Id_l\in R$, let $\tau:R\to R$
be the involution given by $\tau(A)=\overline{A}^{\transp}$, and let
$$
I=\{A\in R:\ A=-\tau(A)\}.
$$ 
By Hilbert's theorem 90,
any $\alpha\in \F_{q^2}$ satisfying $\overline\alpha=-\alpha$ is equal to $\overline\beta-\beta$ for some
$\beta\in \F_{q^2}$, which implies that $I=\Asym^{\min}(R,\tau).$

Let the ring $\Rhat$, the involution $*$, the form parameter $\Ihat$ and the group $G_D$ be defined
as in Case~3. Since $k\geq 3$ be assumption, $R$ can be generated by two hermitian (that is, $\tau$-invariant)
matrices. Hence there exists an involution preserving epimorphism $\pi:\Rhat\to R$. By construction, $\pi(\Ihat)=I$, whence
as in the previous case $\mathrm{SU}_{2n}(\F_q)$ is a quotient of $G_D$.
\vskip .3cm

{\bf Case 5: \it $\mathbf \Phi=B$, $q$ is odd,} and {\bf  Case 6: \it $\mathbf \Phi={}^2A_{even}$.}  
There exists a group $G_B$ with property $(T)$ which maps onto $\Omega_{2n+1}(\F_q)$ for $n\ge 36$ and $q$ odd and onto 
$\mathrm{SU}_{2n+1}(\F_q)$ for $n\ge 36$.

We treat these two cases simultaneously because the arguments are almost identical. The following
notations will have different meanings in cases 5 and 6. Let $q$ be a prime power, which we assume to be odd
in case 5. In case 5 we let $F=\F_q$ and $x\mapsto \xbar$ be the identity map on $F$, and in case 6 we let $F=\F_{q^2}$
and $x\mapsto \overline x$ the automorphism of $F$ of order $2$. In both cases, given a matrix $A$ with entries in $F$,
we denote by $\overline A$ the matrix obtained from $A$ by applying the map $x\mapsto \overline x$ to each entry
and we put $A^{\tau}=(\overline A)^{\transp}$. Finally, we let $G(n,q)=\Omega_{2n+1}(\F_q)$ in case 5 and
$G(n,q)=\mathrm{SU}_{2n+1}(\F_q)$ in case 6.

\

Since $n\geq 36$, it is easy to see that $n=3k+\frac{l-1}{2}$ where $l\leq k\leq 2l$.
We shall consider $(2n+1)\times (2n+1)$ matrices as $7\times 7$ block matrices
of type $(k_1,k_2,k_3,k_4,k_5,k_6,k_7)=(k,k,k,l,k,k,k)$. For $1\leq i\leq 7$ we put
$\bar i=8-i$.

Let 
$$
J=\small \left( 
\begin{array}{ccccccc}  
0 & 0 & 0 & 0 & 0 & 0 & \Id_k \\
0 & 0 & 0 & 0 & 0 & \Id_k & 0 \\
0 & 0 & 0 & 0 & \Id_k & 0 & 0 \\   
0 & 0 & 0 & \Id_l & 0 & 0 & 0 \\    
0 & 0 & \Id_k & 0 & 0 & 0 & 0 \\ 
0 & \Id_k & 0 & 0 & 0 & 0 & 0 \\ 
\Id_k & 0 & 0 & 0 & 0 & 0 & 0  
\end{array}\right)=
\sum_{i=1}^7 (\Id_{k_i})_{i,\bar i}
\in \Mat_{2n+1}(\F_q).
$$
Then (in both cases 5 and 6) the group $G(n,q)$ has the following realization:
$$
G(n,q)=\{A\in SL_{2n+1}(\F_q)\ : A^{\tau}JA=J\}.
$$
Define a $BC_3$-grading $\{X_{\gamma}\}_{\gamma\in BC_3}$
of the group $G(n,q)$ as follows. Let $R=\Mat_k(F)$, $Z=\Id_{l}\in R$, and recall that $RZ=\Mat_{k\times l}(F)$.
Define the set $P(R,\tau,RZ)$ as in Example~2 of \S~8. The positive root subgroups are given by
\begin{align*}
X_{e_i-e_j} =& \{\Id_{2n+1}+(A)_{i,j}-(A^{\tau})_{\bar j,\bar i}:\ A\in R\}&\\
X_{e_i+e_j} =& \{\Id_{2n+1}+(A)_{i,\bar j}-(A^{\tau})_{j,\bar i}:\ A\in R\}&\\
X_{e_i} =& \{\Id_{2n+1}+(A)_{i,4}-(A^{\tau})_{4,\bar i}+(B)_{i,\bar i}:\ (A,B)\in P(R,\tau,RZ) \}&\\
X_{2e_i} =& \{\Id_{2n+1}+(B)_{i,\bar i}:\ B=-B^{\tau}\in R\},&
\end{align*}
where $1\leq i<j\leq 3$ and $\bar x=8-x$. The negative root subgroups are given by $X_{-\gamma}=X_{\gamma}^{\transp}$.
\vskip .1cm

This grading is the coarsened grading corresponding to the reduction
$\eta:B_n\to BC_3$ in Case~5 and $\eta:B_n\to BC_3$ in Case~6
given by $\eta(e_i)=e_{\block(i)}$ for $1\leq i\leq 3k$ and
$\eta(e_i)=0$ for $3k+1\leq i\leq 3k+\frac{l-1}{2}=n$.

\vskip .2cm

Now let $\Rhat=R_{main}[a]$, and let $*$ be the involution of $\Rhat$ which fixes the canonical
generators of $R_{main}$ and sends $a$ to $1-a$.
Let $G_B$  be the subgroup of $\EL_8(\Rhat)$ generated by the subgroups $\{\Xhat_{\gamma}\}_{\gamma\in BC_3}$
which are defined as the corresponding subgroups $\{X_{\gamma}\}_{\gamma\in BC_3}$ with
$R$, $\tau$ and $P(R,\tau,RZ)$ replaced by $\Rhat$, $*$ and $P(\Rhat,*,\Rhat z)$, respectively.

It is clear from the definition that the group $G_B$ is a quotient of the Steinberg group
$\St_{BC_3}(\Rhat,*,\Rhat z)$ and thus has property $(T)$ by Proposition~\ref{propT_ex2new}.

Thus it remains to prove that $G(n,q)$ is a quotient of $G_B$, for which it suffices to find an involution preserving epimorphism
$\pi:\Rhat\to R$ such that
\begin{itemize}
\item[(i)] $\pi(\Rhat z)=RZ$
\item[(ii)] $\pi(\Asym(\Rhat,*))=\Asym(R,\tau)$
\item[(iii)] $\pi(P(\Rhat,*,\Rhat z))=P(R,\tau,RZ)$ where  $\pi((a,b))=(\pi(a),\pi(b))$.
\end{itemize}

Choose $\tau$-invariant matrices $X$ and $Y$ which generate $R$, and choose $\alpha\in F$
such that $\alpha+\overline \alpha=1$ and $\alpha\not\in\Fq$ in Case 6. In case 5 we simply set
$\alpha=1/2$, and in case 6 such $\alpha$ exists since the trace map $\F_{q^2}\to\F_q$ is surjective. Now define the
epimorphism $\pi:\Rhat\to R$ by setting $\pi(x)=X$, $\pi(y)=Y$, $\pi(z)=Z$, $\pi(w)=W$
and $\pi(a)=\alpha$.

By construction, $\pi$ is involution preserving and satisfies (i).
It is also clear that $\pi(\Asym^{\min}(\Rhat,*))=\Asym^{\min}(R,\tau)$. On the other hand,
$\Asym^{\min}(\Rhat,*)=\Asym(\Rhat,*)$ and $\Asym^{\min}(R,\tau)=\Asym(R,\tau)$ by Lemma~\ref{lemma:BCn_centralelement}(1),
which implies (ii). Finally, (iii) follows from (i),(ii) and Lemma~\ref{lemma:BCn_centralelement}(2).

\vskip .3cm

{\bf Case 7: $\mathbf \Phi={}^2D_n$.}  There exists a group $G_{^2D}$ with property $(T)$ which maps onto $\Omega^-_{2n}(\F_q)$ for $n\ge 19$.

Since $n\geq 19$, we can write $n=3k+l+1$, where $l\leq k\leq 2l$.  
We will define $\mathrm{O}^-_{2n}(\F_q)$ as a subgroup of $\mathrm{O}^+_{2n}(\F_{q^2})$. 
Let $M=\left (\begin{array}{cc} 0 & 1\\ 0 & 0\end{array}\right )$ and $x\mapsto \overline x$ the automorphism of $\F_{q^2}$ of order 2. 
Let
$$
S=\small\left(
\begin{array}{ccccccccc}  
0 & 0 & 0 & 0 & 0 & 0 & 0 & 0 & \Id_k \\
0 & 0 &  0 & 0 & 0 & 0& 0 & \Id_k & 0 \\   
0 & 0 & 0 & 0 & 0 & 0 & \Id_k & 0 & 0 \\ 
0 & 0 & 0 & 0 & 0 & \Id_l & 0 & 0 & 0 \\   
0 & 0 & 0 & 0 & M & 0 & 0 & 0 & 0 \\    
0 & 0 & 0 & 0 & 0 & 0 & 0 & 0 & 0 \\ 
0 & 0 & 0 & 0 & 0 & 0 & 0 & 0 & 0 \\ 
0 & 0 & 0 & 0 & 0 & 0 & 0 & 0 & 0 \\  
0 & 0 & 0 & 0 & 0 & 0 & 0 & 0 & 0  
\end{array}\right )\in \Mat_{2n}(\F_q)
$$ 
and
$$
T=\small\left( 
\begin{array}{ccccccccc}  
\Id_k & 0 & 0 & 0 & 0 & 0 & 0 & 0 & 0 \\
0 & \Id_k & 0 & 0 & 0 & 0 & 0 & 0 & 0 \\   
0 & 0 & \Id_k & 0 & 0 & 0 & 0 & 0 & 0 \\ 
0 & 0 & 0 & \Id_l & 0 & 0 & 0 & 0 & 0 \\   
0 & 0 & 0 & 0 & M+M^{\transp} & 0 & 0 & 0 & 0 \\
0 & 0 & 0 & 0 & 0 & \Id_k & 0 & 0 & 0 \\ 
0 & 0 & 0 & 0 & 0 & 0 & \Id_k & 0 & 0 \\ 
0 & 0 & 0 & 0 & 0 & 0 & 0 & \Id_k & 0 \\  
0 & 0 & 0 & 0 & 0 & 0 & 0 & 0  & \Id_k  
\end{array}\right )\in \Mat_{2n}(\F_q).
$$
Recall that $\mathrm{O}^+_{2n}(\F_{q^2})$ is realized as the group of  matrices $A\in \GL_{2n}(\F_{q^2})$ that fix the quadratic form $q(u,u)=u^{\transp}Su$, where $u\in \Mat_{2n\times 1}(\F_{q^2})\cong \F_{q^2}^{2n}$.   The  group $\mathrm{O}_{2n}^-(\F_q)$ can be defined as
follows:
$$
O_{2n}^-(\F_q)=\{B\in O_{2n}^+(\F_{q^2}): \ BT=T\overline{B}\}.
$$

We shall consider $2n\times 2n$ matrices as $9\times 9$ block matrices of type
$$
(k_1,k_2,k_3,k_4,k_5,k_6,k_7,k_8,k_9)=(k,k,k,l,2,l,k,k,k).
$$
Recall the following notation introduced in Example~3 of~\S~\ref{sec:twisted}
just before Proposition~\ref{Trootcontent_typeBCn}:
If $R$ is a ring with $1$, $*$ an involution on $R$ and $\sigma$ an automorphism
of $R$ of order $\leq 2$ commuting with $*$, then for any additive subgroups $S,I,J\subseteq R$
we put
$$
Q(S,*,\sigma, I,J)=\{(r,t): r\in I, t\in J \mbox{ and }t-r\sigma(r^*)\in S^{\sigma}\}
$$
Now define
$$
\Omega(R,*,\sigma, I,J)=
\!\!\left\{ 
(a,b,c): \
a=\!\!\left(v,\sigma(v)\right)\!,\,
c=\!\!\left(\!\!\begin{array}{c} -\sigma(v^*)\\-v^*\end{array}\!\!\!\right)\!\!,
(v,b+v\sigma(v^*))\in Q(R,*,\sigma,I,J)
\right\}
$$
As in Case 3, for $m\in\dbN$ we denote by
$\Asym_0(m,\F_q)$ the set of antisymmetric matrices in $\Mat_m(\F_q)$ with zeroes on the diagonal,
and we put
$$
\Omega_m=
\left\{ 
(A,B,C): \ 
\begin{array}{c} A=(V,\overline{V}),
C=\left(\begin{array}{c} -\overline V^{\transp}\\-V^{\transp}\end{array}\right ), V\in \Mat_{m\times 1}(\F_{q^2}),\\ 
 B\in \Mat_m(\F_q), B+V\overline{V}^{\transp}\in  \Asym_0(m,\F_{q^2})
\end{array}
\right\}.
$$
The group $\Omega_{2n}^-(\F_q)$ has the following $BC_4$-grading $\{X_{\gamma}\}$.
The positive root subgroups are described by
\begin{align*}
X_{e_i-e_j} =& \{I+(A)_{i,j}-(A^{\transp})_{\bar j,\bar i}:\ A\in \Mat_{k_i\times k_j}(\F_{q})\}&\\
X_{e_i+e_j} =& \{I+(A)_{i,\bar j}-(A^{\transp})_{j,\bar i}:\ A\in \Mat_{k_i\times k_j}(\F_{q})\}&\\
X_{e_i} =& \{I+(A)_{i,5}+(B)_{i,\bar i}+(C)_{5,\bar i}:\ (A,B,C)\in \Omega_{k_i})\}&\\
X_{2e_i} =& \{I+(B)_{i,\bar i}:\ B\in \Asym_0(k_i,q)\}&
\end{align*}
where $1\leq i<j\leq 4$, $\bar x=10-x$. The negative root subgroups are given by $X_{-\gamma}=X_{\gamma}^{\transp}$.
\vskip .1cm

This grading is the coarsened grading corresponding to the reduction
$\eta:BC_n\to BC_4$ given by $\eta(e_i)=e_{\block(i)}$ for $1\leq i\leq n-1$
and $\eta(e_n)=0$.

\vskip .2cm

Now let $R=\Mat_k(\F_{q^2})$. Note that the conjugation map $\conj:R\to R$ given by $\conj(A)=\overline A$
is an automorphism of order $2$ which commutes with the involution $\transp:R\to R$. Let $U=E_{11}\in R$
and $Z=\Id_l\in R$. Observe that
\begin{align*}
\Mat_{k\times 1}(\F_{q^2}) =& RU,& \Mat_{k\times l}(\F_{q^2}) =& RZ, \\
\Mat_{l\times k}(\F_{q^2}) =& ZR,& \Mat_{l\times l}(\F_{q^2}) =& ZRZ,  \\
\Mat_{k\times k}(\F_{q}) =& R^{\conj},& \Mat_{k\times l}(\F_{q}) =& R^{\conj} Z, \\
\Mat_{l\times k}(\F_{q}) =& ZR^{\conj},& \Mat_{l\times l}(\F_{q}) =& ZR^{\conj} Z  \\
\Asym_0(k,q^2)=& \Asym^{\min}(R,\transp)& & \\
\Asym_0(k,q) =& (\Asym^{\min}(R,\transp))^{\conj}& &\\
\Asym_0(l,q^2) =& \Asym^{\min}(R,\transp)\cap ZRZ& &\\
\Asym_0(l,q) =&(\Asym^{\min}(R,\transp)\cap ZRZ)^{\conj}& &\\
\Omega_{k} =& \Omega(R,\transp,\conj,RU,\Asym^{\min}(R,\transp))& &\\
\Omega_{l} =& \Omega(ZRZ,\transp,\conj,ZRU,\Asym^{\min}(R,\transp))& &
\end{align*}
Therefore, the above definition of the positive root subgroups can be rewritten as follows:
\begin{align*}
X_{e_i-e_j} =& \{I+(A)_{i,j}-(A^{\transp})_{\bar j,\bar i}:\ A\in R^{\conj}\} \mbox{ for }1\leq i<j\leq 3 \\
X_{e_i-e_4} =& \{I+(A)_{i,4}-(A^{\transp})_{6,\bar i}:\ A\in R^{\conj}Z\} \mbox{ for }1\leq i\leq 3 \\
X_{e_i+e_j} =& \{I+(A)_{i,\bar j}-(A^{\transp})_{j,\bar i}:\ A\in R^{\conj}\} \mbox{ for }1\leq i<j\leq 3\\
X_{e_i+e_4} =& \{I+(A)_{i,6}-(A^{\transp})_{4,\bar i}:\ A\in R^{\conj}Z\} \mbox{ for }1\leq i\leq 3\\
X_{e_i} =& \{I+(A)_{i,5}+(B)_{i,\bar i}+(C)_{5,\bar i}:\ (A,B,C)\in \Omega(R,\transp,\conj, RU,\Asym^{\min}(R,\transp))\}\\
&\mbox{ for }1\leq i\leq 3&\\
X_{2e_i} =& \{I+(B)_{i,\bar i}:\ B\in (\Asym^{\min}(R,\transp))^{\conj}\} \mbox{ for }1\leq i\leq 3 \\
X_{e_4} =& \{I+(A)_{4,5}+(B)_{4,6}+(C)_{5,6}:\ (A,B,C)\in \Omega(ZRZ,\transp,\conj,ZRU,\Asym^{\min}(R,\transp))\}\\
X_{2e_4} =& \{I+(B)_{4,6}:\ B\in (\Asym^{\min}(R,\transp)\cap ZRZ)^{\conj}\}
\end{align*}

Now, we will define the covering group $G_{^2D}$. Let $\Rhat_0$ be the quotient of the ring $R_{main}\la u \ra$
by the ideal generated by $zu-uz$,$zu-u$ and $z^2-z$, and let $\Rhat=\Rhat_0 [a]$.
Let $*$ be the involution of $\Rhat$ that fixes all 6 variables $x,y,z,w,u$ and $a$,
and let $\sigma$ be the automorphism of $\Rhat$ of order $2$ which fixes $x,y,z,w,u$
and sends $a$ to $1-a$. It is easy to show that 
$$
{\Rhat^{\sigma}}=\Rhat_0[a(1-a)].
$$

Define $G_{^2D}$ to be the subgroup of $\EL_{10}(\Rhat)$ generated by the subgroups $\{\Xhat_{\gamma}\}_{\gamma\in BC_4}$
which are defined in the same way as $\{X_{\gamma}\}$ (in their second description) with $R$, $U$, $Z$, $\conj$ and $\transp$
replaced by $\Rhat$, $u$, $z$, $\sigma$ and $*$, respectively. It is straightforward to check that $G_{^2D}$ is a quotient
of the Steinberg group $\St_{BC_4}^1(\Rhat,*,\sigma, I,J; M)$ described in Proposition~\ref{Trootcontent_typeBCn}
where $I=\Rhat u$, $M=\Rhat z$ and $J=\Asym^{\min}(\Rhat,*)$. 
Let us check that conditions (i)-(viii) of  Proposition~\ref{Trootcontent_typeBCn} hold. Condition (ii) holds by construction, and (i) and (iii) are clear from the above description of $\Rhat^{\sigma}$.
Condition (viii) holds by Observation~\ref{obs:idempotent} thanks to the relation $z^2=z$. The latter also implies
that 
$M^{\sigma}=M^{\sigma} z\subseteq \Rhat^{\sigma}z\subseteq (\Rhat z)^{\sigma}=
M^{\sigma}$, whence $M^{\sigma}=\Rhat^{\sigma}z$, 
so (v) is satisfied.
Condition (vi) holds as in all previous cases, and (vii) is automatic since $M^{\sigma}$ is a principal ideal of $\Rhat^{\sigma}$, as we just verified.

It remains to check (iv), for which it will suffice to prove that $(\Asym^{\min}(\Rhat,*))^{\sigma}=\Asym^{\min}(\Rhat^{\sigma},*)$.
The inclusion $\Asym^{\min}(\Rhat^{\sigma},*)\subseteq (\Asym^{\min}(\Rhat,*))^{\sigma}$ is clear. For the reverse inclusion
we will use the following fact which follows easily from the above description of $\Rhat^{\sigma}$:

\begin{itemize}
\item[(*)] Every $r\in\Rhat$ can be uniquely written as $r=r_1+ar_2$ with $r_1,r_2\in\Rhat^{\sigma}$.
\end{itemize}
Now take any $r\in (\Asym^{\min}(\Rhat,*))^{\sigma}$. Thus, $r\in \Rhat^{\sigma}$ and $r=t^*-t$ for some $t\in \Rhat$.
Write $t=c+ad$ with $c,d\in \Rhat^{\sigma}$. Then we have $r=c^*-c+a(d^*-d)$. Since $c^*-c, d^*-d$ and $r$ all lie
in $\Rhat^{\sigma}$, by (*) we must have $d^*=d$, whence $r=c^*-c\in \Asym^{\min}(\Rhat^{\sigma},*)$.

Thus, we verified conditions (i)-(viii), so $G_{^2D}$ has property $(T)$.
\vskip .2cm

It remains to show that $\mathrm{O}_{2n}^-(\F_q)$ is a quotient of $G_{^2 D}$  for which it suffices to construct
an epimorphism $\pi:\Rhat\to R$ satisfying the following conditions:
\begin{itemize}
\item[(1)] $\pi(r^*)=\pi(r)^{\transp} \mbox{ for all } r\in \Rhat$
\item[(2)] $\pi(\sigma(r))=\overline{\pi(r)} \mbox{ for all } r\in \Rhat$
\item[(3)] $\pi(\Rhat^{\sigma})=R^{\conj}$
\item[(4)] $\pi(z)=Z$ and $\pi(u)=U$
\item[(5)] $\pi((\Asym^{\min}(\Rhat,*))^{\sigma})=(\Asym^{\min}(R,\transp))^{\conj}$
\item[(6)] $\pi((\Asym^{\min}(\Rhat,*)\cap z\Rhat z)^{\sigma})=(\Asym^{\min}(R,\transp)\cap ZRZ)^{\conj}$
\end{itemize}

Choose symmetric matrices $X,Y\in \Mat_{k}(\F_q)$ which generate
$\Mat_{k}(\F_q)$ as  a ring, and choose an element $\alpha\in \F_{q^2}\setminus \F_q$
such that $\alpha+\overline\alpha=1$. Let $\pi:\Rhat\to R$ be the homomorphism given by
$\pi(x)=X$, $\pi(y)=Y$, $\pi(z)=Z$, $\pi(u)=U$, $\pi(w)=W$ and $\pi(a)=\alpha$.
Since $R=\Mat_{k}(\F_q)+\alpha\Mat_{k}(\F_q)$, $\pi$ is surjective, and conditions (1)-(4) trivially hold.
Condition (5) follows from the equality $(\Asym^{\min}(\Rhat,*))^{\sigma}=\Asym^{\min}(\Rhat^{\sigma},*)$
and the analogous equality $(\Asym^{\min}(R,\transp))^{\conj}=\Asym^{\min}(R^{\conj},\transp)$, which is checked
by direct calculation, with both sides equal to $\Asym_0(k,q)$. Finally, to prove (6) first note
that by (1)-(4), we have the inclusions
\begin{multline*}
\Asym^{\min}(ZR^{\conj}Z,\transp)=\pi((\Asym^{\min}(z\Rhat^{\sigma} z),*))\\
\subseteq \pi((\Asym^{\min}(\Rhat,*)\cap z\Rhat z)^{\sigma})\subseteq (\Asym^{\min}(R,\transp)\cap ZRZ)^{\conj},
\end{multline*}
so it remains to check that $(\Asym^{\min}(R,\transp)\cap ZRZ)^{\conj}=\Asym^{\min}(ZR^{\conj}Z,\transp)$.
Again by direct verification both sets are  equal to $\Asym_0(l,q)$. This completes the proof.

\vskip .2cm
\begin{proof}[Proof of Lemma~\ref{lemma:matrixgeneration1}] 
Let  $a$ be a generator of $F^{\times}$, the multiplicative group of $F$. 
An easy computation shows that the matrices $a E_{11}$ and $\sum_{i=1}^{k-1} (E_{i,i+1}+E_{i+1,i})$ generate $\Mat_{k}(F)$.

Now assume that $F=\F_{q^2}$. We will show that $A=aE_{1,2}+\sigma(a) E_{2,1}$  and
$B=  \sum_{i=1}^{k-1} E_{i+1,i}+E_{i,i+1}$ generate $\Mat_{k}(F)$.
Let $R$ be the subring generated by $A$ and $B$. Since $a$ generates $\F_{q^2}^{\times}$,
$a\sigma(a)$ generates $\F_q^{\times}$ and so the subring generated by the matrix $A^2$ contains   
$\alpha(E_{1,1}+E_{2,2})$  for every $\alpha\in \F_q$. Thus
$$
E_{1,1}=2(E_{1,1}+E_{2,2})-(E_{1,1}+E_{2,2})B^2(E_{1,1}+E_{2,2})\in R \textrm{\ and \ } E_{2,2}\in R.
$$ 
Hence for any $\alpha \in \F_q$,
$$
\alpha E_{1,2}=\alpha(E_{1,1}+E_{2,2})E_{1,1}B E_{2,2}\in R \textrm{\ and \ }\alpha a E_{1,2}=\alpha(E_{1,1}+E_{2,2})A   E_{2,2} \in R.
$$
Hence $\gamma E_{1,2}\in R$   for every $\gamma \in \F_{q^2}$. Similarly  $\gamma E_{2,1}\in R$   for every $\gamma \in \F_{q^2}$.  Therefore   $\gamma E_{1,1}$ and $\gamma E_{2,2}\in R$   for every $\gamma \in \F_{q^2}$.

Now, by induction on $k$ we show that $E_{i,k}, E_{k,i}\in R$ for any $i\le k\le n$. 
This clearly will finish the proof. The base of induction $k\le 2$ is already established 
and the inductive step follows from the following equalities.
$$
E_{k-1,k}=E_{k-1,k-1}B-E_{k-1,k-2},\ E_{k,k-1}=BE_{k-1,k-1}-E_{k-2,k-1}.
$$
\end{proof}

\subsection{Groups of type $A_1$.} 
We finish this section by proving that any infinite family of groups of the form $\PSL_2(q)$ cannot have a mother group with property $(T)$:


\begin{Theorem} Let $G$ be group which maps onto $\PSL_2(q)$ for infinitely many $q$.
Then  $G$ does not have  property (T).
\end{Theorem}

\begin{proof} 
By way of contradiction, assume that $G$ has property $(T)$. 
Write $G$ as $F/K$, where $F$ is a finitely generated free group and let $\widetilde G=F/K^2[K,F]$.
Then by Theorem~\ref{relKazhdan2} $\widetilde G$ also has property $(T)$. 
Since $G$ maps on infinitely many $\PSL_2(q)$ and $\SL_2(q)$ has no subgroups of index $2$, it is easy to see
that $\widetilde G$ maps on infinitely  many $\SL_2(q)$. 
Let $\widetilde G=\langle X|R\rangle$ be  a presentation of $\widetilde G$ with $|X|=n$ finite. 
Put  
$X_i=\left (\begin{array}{cc} x_{11}^i & x_{12}^i\\ x_{21}^i &x_{22}^i\end{array}\right )$
($i=1,\ldots, n$) and 
let $B=\Z[x_{kl}^i\ : k,l=1,2;\, l=1,\ldots, n]$ and $A=B/I$ where $I$ is the ideal of $B$ generated by
$$
x_{11}^ix_{22}^i -x_{12}^ix_{21}^i-1 \ (i=1,\ldots, n)
$$ 
and the entries of the following matrices 
$$
r(X_1,\ldots, X_n)- \left (\begin{array}{cc} 1 & 0\\ 0 &1\end{array}\right )
$$ 
for each $r\in R$.

Let $A_i=X_i (\!\!\mod I) \in \SL_2(A)$, and let $H$ be the group generated by $\{A_i\}_{i=1}^n $.
It is easy to check that $H$ is a quotient on $\widetilde G$, but on the other hand $H$ maps
onto any quotient of $\widetilde G$ of the form $\SL_2(R)$, with $R$ commutative. In particular,
$H$ maps onto infinitely many $\SL_2(q)$.

Let $N$ be the
nilradical of $A$. We claim that the image of $H$ in $\SL_2(A/N)$ is still infinite. Indeed,
since $A$ is Noetherian, $N$ is nilpotent, so the kernel of the map $\SL_2(A)\to \SL_2(A/N)$
is also nilpotent. Thus, if the image of $H$ in $\SL_2(A/N)$ was finite, $H$ would be virtually
nilpotent and could not map onto  infinitely many $\SL_2(q)$.

Again since $A$ is Noetherian, it has finitely many minimal prime ideals (whose intersection is $N$), 
so there exists a minimal prime ideal $P$ of $A$ such that the image of $ H$ in $\SL_2(A/P)$ is infinite. Now, we can finish the proof using, for instance~\cite[Theorem 4]{GHW}, which asserts that an infinite subgroup of $\SL_2(D)$, with $D$ a commutative domain, has the Haagerup property and so cannot have property~$(T)$. The obtained contradiction finishes the proof.
\end{proof}

\subsection{Alternating Groups}
\label{subsec:alt}

In this subsection we will show that there exists a mother group satisfying property $(\tau)$ for the family of alternating groups.
The existence of a such group is established by combining ideas from~\cite{Ka2} and~\cite{KN2}.

We start the proof with the following generalization of~\cite[Lemma 4.2]{Ka}.
\begin{Proposition}
\label{specialring} 
There exists a finitely generated dense subring $R$  in 
$$
\prod_{k=3}^{\infty} \Mat_k(\F_2)^{\times 2^{33k}}
$$ 
containing $\displaystyle \bigoplus_{k=3}^\infty \Mat_k(\F_2)^{\times 2^{33k}}$.
\end{Proposition}
\begin{proof}  
The proof is a combination of the ideas from the proof of~\cite[Lemma 2.1]{Ka2} and~\cite[Lemma 4.2]{Ka}. 
It is clear that it is enough to construct such a subring inside $\prod_{k=33}^{\infty} \Mat_k(\F_2)^{\times 2^{33k}}$.
Denote by $R_{k,i}$ ($1\le i\le 2^{33k})$ the $i^{\rm th}$ copy of the ring $\Mat_k(\F_2)$.
Let $R$ be a subring of $\prod_{k\ge 33, 1\le i\le 2^{33k}}R_{k,i}$   generated by the following 5 elements 
$\mathbf a =(a_{k,i})$, $\mathbf{ \bar a}=((a_{k,i})^{-1})$, $\mathbf b=(b_{k,i})$, $\mathbf c=(c_{k,i})$ and $\mathbf {x}=(x_{k,i})$, where
$$
a_{k,i}=E_{1,2}+E_{2,3}+\ldots +E_{k,1},\ b_{k,i}=E_{1,2},\ c_{k,i}=E_{2,1}
$$ 
and
$\{x_{k,i}:\ 1\le i\le 2^{33k}\}$ are different elements of $\Mat_k(\F_2)$ (this is possible because $k\ge 33)$. 
Now as in the proof of~\cite[Lemma 4.2]{Ka}, $R$ is dense in 
$\displaystyle \prod_{k\ge 33, 1\le i\le 2^{33k}}R_{k,i}$ and contains $\displaystyle \bigoplus_{k\ge 33, 1\le i\le 2^{33k}}R_{k,i}$.
\end{proof}

\begin{Corollary} 
\label{manysl}
There exists a finitely generated dense subgroup $G_0$ of 
$$
\prod_{k\geq 3} \SL_{3k}(\F_2)^{\times 2^{33k}}, 
$$ 
which contains $\displaystyle \bigoplus_{k\geq 3} \SL_{3k}(\F_2)^{\times 2^{33k}}$ and  has property $(\tau)$ with respect to the family of open subgroups.
\end{Corollary}
\begin{proof}
Let $R$ be as in Proposition~\ref{specialring}. Then from~\cite{EJ} we know that $G_0=\EL_3(R)$ has property $(T)$.  This group is clearly dense   in $$
\EL_3\left(\prod_{k=3}^{\infty} \M_k(\F_2)^{\times 2^{33k}}\right) \cong \prod_{k\geq 3} \SL_{3k}(\F_2)^{\times 2^{33k}} 
$$ 
and contains $\displaystyle \bigoplus_{k\geq 3} \SL_{3k}(\F_2)^{\times 2^{33k}}$.
\end{proof}
\begin{Proposition}
\label{thm:prod_alt_with_rel_tau}
There exists a finitely generated dense subgroup $G$ of $\prod_{n\geq 5} \Alt(n)$ which has property $(\tau)$ with respect to the family of open subgroups.
\end{Proposition}
\begin{proof}

Theorem 2 from~\cite{Ka2} states:
\begin{Theorem}
\label{thm:alt_as expanders}
For every $n \geq 5$ there exists a generating set $X_n$ of the alternating group $\Alt(n)$ such that:
(a) $|X_n| = K $ and (b) $\kappa(\Alt(n); X_n) \geq \epsilon_0 >0$, where $K$ and $\epsilon_0 > 0$ are some explicit constants.
\end{Theorem}

The proposition  cannot be derived from Theorem~\ref{thm:alt_as expanders}, but it follows relatively easily from its proof. 
Here we will only outline the main points.

The proof of Theorem~\ref{thm:alt_as expanders}   goes as follows. Let $n \geq 10^6$. We choose  $k_n$ such that 
$$
(2^{3k_n}-1)^6\le n<(2^{3(k_n+1)}-1)^6
$$ 
and we  put $l_n=(2^{3k_n}-1)^5$.  Then it is shown in~\cite{Ka2} that there exists an embedding
$$
\phi_n: G_n=\SL_{3k_n}(\F_2)^{\times l_n} \to \Alt(n).
$$
and the  elements $\{g_{1,n},\ldots,g_{C,n}\}$ of $\Alt(n)$ such that  the  Kazhdan constant 
\begin{equation}
\label{eq:alteps}
\kappa(\Alt(n), B_n)>\epsilon_1
\end{equation}
for some $\epsilon_1>0$, where 
$B_n = \bigcup_i  (\phi_n(G_n))^{g_{i,n}}$. Note that in particular $B_n$ generates $\Alt(n)$.
It is essential that the number of conjugates $C$ and $\epsilon_1$ are independent of $n$.

Let $G_0$ be as in Corollary~\ref{manysl}. Note that the number of different $n$ with the same $k_n$ is at most $2^{18k_n}$. 
Since $2^{18k_n}\cdot l_n\le 2^{ 33k_n}$, we can construct a homomorphism 
$\phi: G_0\to \prod_{n\geq 10^6} \Alt(n)$ such that $\phi(G_0)\cap \Alt(n)=\phi_n(G_n)$.

Let $S_0$ be a finite generating set of $G_0$. Denote by $g_i\in \prod_{n\geq 10^6} \Alt(n)$ 
($i=1,\ldots, C$) the element whose projection into $\Alt(n)$ is equal to $g_{i,n}$. Let $S=\cup_{i=1}^{C} \phi(S_0)^{g_i}$ and let $G=\langle S\rangle$. We want to show that $G$ satisfies the conclusion of Proposition~\ref{thm:prod_alt_with_rel_tau} in the product $\prod_{n\geq 10^6} \Alt(n)$. This will clearly imply the proposition.

Let $B=\cup_{i=1}^{C} \phi(G_0)^{g_i}$. It is clear that 
$\kappa_r(G,B;S)\geq \kappa_r(G_0,G_0,S_0)$ and 
$\kappa_r(G_0,G_0,S_0)\geq\frac{1}{2}\kappa(G_0,S_0)$ by Observation~\ref{Kazhrat}(ii).
Since $G_0$ has $(T)$, we deduce that $\kappa_r(G,B;S)>0$.

Next observe that since $\phi(G_0)$ contains $\phi_n(G_n)$ and $\bigcup_i  (\phi_n(G_n))^{g_{i,n}}$ generates $\Alt(n)$,
the group $G=\langle S\rangle$ contains $\displaystyle \bigoplus_{n\geq 10^6} \Alt(n)$. In particular, $G$ is dense in $\displaystyle \prod_{n\geq 10^6} \Alt(n)$.

Let $l\ge 10^6$. Denote by $U_l$ the subgroup $ \prod_{n\geq l} \Alt(n)$. In order to finish the proof of the proposition we have to show that Kazhdan constants
$$
\kappa (GU_l/U_l, SU_l/U_l)
$$ 
are uniformly bounded from zero. This will  follow if we show that the quantities $\kappa (G, S, V)$  are uniformly bounded from zero as $V$ ranges over non-trivial irreducible  unitary representations of 
$GU_l/U_l$ for different $l$. Fix such a representation $V$ -- it is isomorphic to a tensor product  
$\displaystyle \bigotimes_{i=10^6}^{l-1} V_i$ where each $V_i$ is an irreducible representation of $\Alt(i) $ (which we can also view as a representation of $G$). Since $V$ is non-trivial, $V_i$ is non-trivial for some $i$. Since $V$ is isomorphic to a direct sum of several copies of $V_i$
as a representation of $\Alt(i)$, it follows that $V$ does not contain nonzero 
$\Alt(i)$-invariant vectors. 
Take any $0\neq v\in V$. By construction $B\supseteq B_i$, so there exists $g\in B$ such that $\|gv-v\|\ge \epsilon_1\|v\|$, where $\epsilon_1$ is defined by \eqref{eq:alteps}.  
Thus there exists $s\in S$ such that $\|sv-v\|\ge \epsilon_1\epsilon_2\|v\| $
where $\epsilon_2=\kappa_r(G,B;S)$. Since $\epsilon_2>0$ as shown above, we are done.
\end{proof}

Most likely, the group $G$ does not have property $(\tau)$, because $G$ might have finite quotients which are not visible via the embedding of $G$ into the product of alternating groups. This complication can be bypassed using the ideas from~\cite{KN2}. First we introduce the following important definition.
\begin{Definition}
Let $\mathcal S=\prod_{n=1}^\infty S_n$ be a Cartesian product of finite groups. A finitely generated subgroup $G$ of $\mathcal S$ is \emph{a frame for $\mathcal S$} if the following hold:
\begin{enumerate}
\item[(a)] $G$ contains $\bigoplus_{n=1}^\infty S_n$.
\item[(b)] The natural surjection $\widehat G\to \mathcal S$ is an isomorphism
\end{enumerate}
\end{Definition}
The following property of frame subgroups was shown in~\cite{KN2}
\begin{Lemma}[{\cite[Lemma 8]{KN2}}]
\label{frame}
Let $A_n$, $B_n\le C_n$ ($n\in N$) be finite groups with $C_n=\langle A_n, B_n\rangle$. 
Suppose that $A$ (resp. $B$) is a frame for $\displaystyle \mathcal A=\prod_{n=1}^\infty A_n$ (resp. $\displaystyle \mathcal B=\prod_{n=1}^\infty B_n$). Each of $A$ and $B$ can be considered as a subgroup of $\displaystyle \mathcal C=\prod_{n=1}^\infty C_n$ in the natural way. Then the group $\langle A, B\rangle $ is a frame for
$\mathcal C$. 
\end{Lemma}
The following result is a consequence of Proposition 13 of~\cite{KN2}.

\begin{Proposition}
\label{thm:alt_profinite}
There exists a finitely generated frame subgroup $H$ for
$\displaystyle \prod_{n\geq 5} \Alt(n)$. 
\end{Proposition}

We will need a slight improvement of this proposition. Its proof will use the following lemma.

\begin{Lemma}
\label{conjugates} 
Let $X$ be a set with $k$ elements which generates a dense subgroup of  $\displaystyle \prod_{n\geq 5} \Alt(n)$. 
Then every $\mathbf b=(b_n)\in \prod_{n\geq 5} \Alt(n)$ is a product of  $16k^2$ conjugates (in $ \prod_{n\geq 5} \Sym(n)$) of elements from $X$.
\end{Lemma}
\begin{proof} 
Let  $X=\{\mathbf x_i=(x_{i,n}):\ i=1,\ldots, k\}$. In order to prove the lemma it is enough   to show that   the following  equations have  solutions in $\{z_{i,j,n}: 1\le i\le k, 1\le j\le16k, n\ge 5\}$:
$$
x_{1,n}^{z_{1,1,n}}\cdot \ldots \cdot x_{1,n}^{z_{1,16k,n}}\cdot \ldots \cdot x_{k,n}^{z_{k,1,n}}\cdot \ldots \cdot x_{k,n}^{z_{k,16k,n}}=b_n.
$$
Since $\{x_{i,n}:\ 1\le i\le |X|\}$ is a generating set of $\Alt(n)$ at least one of the elements $x_{i,n}$  moves  at least  $\frac n{k}$ points. Thus, the last claim follows from

\begin{Claim} 
Let $n\ge 5$ and let
$h\in\Alt(n)$ be a permutation that moves  at least  $\frac{n}{k}$ points. 
Then  every $g\in \Alt(n)$ is a product of exactly $16k$ conjugates of  $h$.
\end{Claim}
\begin{proof} 
Let $l=\max\{5,|{\rm supp}~ h|\}$. Without loss of generality we may assume that the support of $h$  is contained in $\{1,\ldots, l\}$.    Let $\mathcal K$ be the conjugacy class of $h$. In~\cite[Theorem 3.05]{Br} it is shown that   $\Alt(l)\subseteq \mathcal K^4$. 
Hence, since $l\ge \frac{n}{k}$, the set $\mathcal K^{4k}$ contains an element without fixed points.   
Applying~\cite[Theorem 3.05]{Br} again, we obtain that $\Alt(n)=\mathcal K^{16k}$.
\end{proof}
\renewcommand{\qedsymbol}{}
\vskip -.8cm
\end{proof}
\begin{Corollary}
\label{lm:alt_profinite_subset}
For any finite set $\displaystyle B  \subset \prod_{n\geq 5} \Alt(n)$ there exist a finitely generated frame subgroup $T$ for
$\displaystyle \prod_{n\geq 5} \Alt(n)$ containing $B$.
\end{Corollary}
\begin{proof} 
Let $X$ be the generating set of the group $H$ from Proposition~\ref{thm:alt_profinite}.
By Lemma~\ref{conjugates} one can find finitely many elements $g_t \in \prod_{n\ge 5}\Sym(n)$ such that every element of the set $B$ can be expressed as a   product of at most $C$ conjugates of elements from  $X$ by the elements $g_t$ (where $C$ is an absolute constant). 
Let $T$ denote the subgroup of $\prod \Alt(n)$ generated by the subgroups $\{H^{g_t}\}$.  
Applying (possibly several times) Lemma~\ref{frame}, we obtain that $T$ is also a frame.
\end{proof}

Using this corollary, we can finally prove Theorem~\ref{mother_alt_tau} restated below:

\begin{Theorem}
\label{thm:mother_with_tau_for_Alt}
There exists a group $\Gamma$ with property $(\tau)$ which surjects onto all alternating groups.
\end{Theorem}
\begin{proof}
It suffices to apply Corollary~\ref{lm:alt_profinite_subset} to the generating set of the group $G$ from Proposition~\ref{thm:prod_alt_with_rel_tau}. 
The resulting group $\Gamma$ contains $G$ and therefore has property $(\tau)$ with respect to the family of open subgroups. However, since $\Gamma$ is a frame,  every finite index subgroup of $\Gamma$ is open, that is, $\Gamma$ has property $(\tau)$ with respect to all finite index subgroups.
\end{proof}


\section{Estimating relative Kazhdan constants}
\label{sec:Kazhdanconstants}

The goal of this section is to prove Theorems~\ref{relKazhdan} and
\ref{nilpotentcodistance} from \S~2.
For convenience we shall use the following terminology
and notations:
\begin{itemize}
\item A unitary representation $V$ of a group $G$ will be referred to as a \emph{$G$-space}.
\item If $U$ is a subspace of $V$, by $P_U$ we denote the operator of orthogonal
projection onto $U$. For any nonzero $v\in V$ we set $P_v=P_{\dbC v}$.
\end{itemize}

\subsection{Hilbert-Schmidt scalar product}

Consider the space $HS(\HH)$ of
Hilbert-Schmidt  operators on $\HH$, i.e., linear operators
$A:\HH\to \HH$ such that
$\sum_i \|A(e_i)\|^2$ is finite where $\{e_i\}$ is an
orthonormal basis  of $\HH$. The space $HS(\HH)$ is endowed with the
\emph{Hilbert-Schmidt scalar product} given by
$$
\langle A,B\rangle=\sum_i \langle A(e_i),B(e_i)\rangle.
$$
By a standard argument this definition does not depend on the choice of
$\{e_i\}$.
The associated norm on $HS(\HH)$ will be called the \emph{Hilbert-Schmidt norm}.

If $\HH$ is a unitary representation of  a group $G$ then $HS(\HH)$ is also
a unitary representation of $G$ -- the action of an element $g\in G$
on an operator $A\in HS(\HH)$ is defined by
$(gA)(v)=gA(g^{-1}v).$
If the element $g$ acts by a scalar on $\HH$, for instance if
$g$ is in the center of $G$ and $\HH$ is an  irreducible representation,
then $g$ acts trivially on $HS(\HH)$.

For any unit vector $v\in \HH$ the projection $P_v: \HH \to \HH$
is an element in $HS(\HH)$ of norm $1$. The map $v \to P_v$ does not
preserve the scalar product. However, we have the following explicit formula for $\langle P_u,P_v\rangle$.

\begin{Lemma}
\label{pw}
If $u$ and $v$ are unit vectors in a Hilbert space $V$, then
$$\langle P_u,P_v\rangle=|\langle u,
v\rangle |^2\quad \mbox{ and therefore }\quad \|P_u-P_v\|\leq \sqrt{2}\|u-v\|$$
\end{Lemma}
\begin{proof}
Choose any orthonormal basis $\{e_i\}$ such that $e_1=u$.
Then
$$
\langle P_u,P_v\rangle=\sum_i\langle P_u(e_i),P_v(e_i)\rangle=\langle u, \langle u,v\rangle v\rangle=  |\langle u,
v\rangle |^2.
$$
Therefore,
\begin{multline*}
\|P_u-P_v\|^2=2(1-|\langle u,v\rangle|^2)\leq 2(1+|\langle u,v\rangle|)(1-{\mathrm Re}\langle u,v\rangle)\\
=
(1+|\langle u,v\rangle|)\|u-v\|^2\le 2\|u-v\|^2.
\qedhere
\end{multline*}
\end{proof}

One can define a non-linear, norm preserving, map $\iota: \HH \to HS(\HH)$ by
$$
\iota(v) = \| v \| P_{v}
$$
The following lemma imposes a restriction
on the change of codistance between vectors under the map $\iota$.
\begin{Lemma}
\label{HS-codist}
Let $v_1,\ldots, v_k$ be vectors in $\HH$. Then
$$
2\codist(v_1,v_2,\dots,v_k) -1 \leq
\codist\left(\iota(v_1),\iota(v_2),\dots, \iota(v_k)\right)
$$
where $\codist(u_1,\dots,u_k)$ denotes the ratio
$\displaystyle \frac{\|\sum u_i\|^2}{k \sum \| u_i\|^2}$.
\end{Lemma}
\begin{proof}
The inequality $\cos^2 \phi \geq 2\cos \phi -1$ implies that
$$
\la \iota(v), \iota(w) \ra = \|v\|\|w\|  \left | \left \la \frac{v}{\|v\|}, \frac{w}{\|w\|}\right \ra \right  |^2\ge 2 \la v, w \ra - \| v\|  \| w\|
$$
Therefore
$$
\left\| \sum \iota(v_i) \right\|^2 =
\sum_{i,j} \la \iota(v_i),\iota(v_j) \ra \geq
$$
$$
\sum_{i,j} (2\la v_i,v_j \ra - \|v_i\| \|v_j\|) =
2 \left\| \sum v_i \right\|^2 - \left(\sum \|v_i\|\right)^2.
$$
Since $\|\iota(v_i)\|=\|v_i\|$, we get
\begin{multline*}
\codist\left(\iota(v_1),\iota(v_2),\dots, \iota(v_k)\right)=
\frac{\|\sum \iota(v_i)\|^2}{k \sum \| v_i\|^2}\\
\geq 2\frac{\| \sum v_i \|^2}{k \sum \| v_i\|^2}-
\frac{(\sum \|v_i\|)^2}{k \sum \| v_i\|^2}\geq
2\codist(v_1,\dots,v_k)-1.
\end{multline*}
which translates into the stated inequality between codistances.
\end{proof}

Let $\{(U_i,\langle\ ,\ \rangle_i)\}_{i\in I}$ be a family of Hilbert spaces.  Recall that the Hilbert direct sum of $U_i$'s denoted by
$\oplus_{i\in I} U_i$ is the Hilbert space consisting of all families $(u_i)_{i\in I}$ with $u_i\in U_i$ such that $\sum_i \langle u_i,
u_i\rangle_i<\infty$ with inner product $$\langle (u_i),(w_i)\rangle=\sum_i\langle u_i, w_i\rangle_i.$$

Let $V$ be a unitary representation of $G$ and let $N$ be a subgroup of $G$. Denote by $(\hat N)_f$ the set of equivalence classes of
irreducible finite dimensional representations of $N$. Let $\pi \in (\hat N)_f$. Denote by $V(\pi)$ the $N$-subspace of $V$ g
spanned by all irreducible $N$-subspaces of $V$ isomorphic to $\pi$. By Zorn' s Lemma, $V(\pi)$ is isomorphic to a Hilbert
direct sum of $N$-spaces isomorphic to $\pi$. We may also
decompose $V$ as a Hilbert direct sum $V=V_\infty\oplus(\oplus_{\pi\in (\hat
N)_f} V(\pi))$, where   $V_\infty$ is the
orthogonal complement of $\oplus_{\pi\in (\hat N)_f} V(\pi)$ in
$V$.

\begin{Lemma} \label{pv} Let $v\in V$ be a unit vector and
$v=v_\infty+\sum_{\pi\in (\hat N)_f} v_\pi$ the decomposition
of $v$ such that $v_\infty\in V_\infty$ and $v_\pi\in V(\pi)$.
Then
$$\|P_{HS(V)^N}(P_v)\|^2\le \sum_{\pi\in (\hat N)_f}
\frac{\|v_\pi\|^4}{\dim\pi}$$
(where the norm on the left-hand side is the Hilbert-Schmidt norm).
Moreover, if $V$ is an irreducible $N$-space, then $\|P_{HS(V)^N}(P_v)\|^2=\frac{1}{\dim V}$.
\end{Lemma}
\begin{proof}
Let $T\in HS(V)^N$. Then $T$ preserves the decomposition
$V=V_\infty\oplus(\oplus_{\pi\in (\hat N)_f} V(\pi))$. Moreover, by
Proposition A.1.12 of \cite{BHV} $T$ maps $V_\infty$ to zero. Hence
we have a decomposition
$$HS(V)^N=\oplus_{\pi\in (\hat N)_f} HS(V)^{N,\pi}$$
where $HS(V)^{N,\pi}$ is the subspace of operators from $HS(V)^N$ which
map the orthogonal complement of $V(\pi)$ to zero. Thus, we may write
$T=\sum_{\pi\in (\hat N)_f} T_\pi$ where $T_\pi\in HS(V)^{N,\pi}$,
and $T_{\pi_1}$ and $T_{\pi_2}$ are orthogonal for
non-isomorphic $\pi_1$ and $\pi_2$.
Note also that
\begin{equation}
\label{eq:HS1}
\|P_{HS(V)^{N}}(P_v)\|^2=\sum_{\pi\in (\hat N)_f}\|P_{HS(V)^{N,\pi}}(P_v)\|^2.
\end{equation}

Now fix $\pi\in (\hat N)_f$, and decompose $V(\pi)$ as a Hilbert direct sum
$\oplus_{i\in I} U_{i}$ of (pairwise orthogonal) $N$-spaces $\{U_i\}$
each of which is isomorphic to $\pi$. Note that
$$HS(V)^{N,\pi}=\oplus_{i,j\in I} HS(V)^{N,\pi}_{i,j}$$
where $HS(V)^{N,\pi}_{i,j}$ is the subspace of operators from
$HS(V)^{N,\pi}$ which map $U_{i}$ onto $U_{j}$
and map the orthogonal complement of $U_{i}$ to zero.
A standard application of Schur's lemma shows that each subspace
$HS(V)^{N,\pi}_{i,j}$ is one-dimensional. Thus, if for each $i,j\in I$
we choose an element $T_{i,j}\in  HS(V)^{N,\pi}_{i,j}$
with $\|T_{i,j}\|=1$, then $\{T_{i,j}\}$ form an orthonormal basis of $HS(V)^{N,\pi}$.
Therefore,
\begin{equation}
\label{eq:HS2}
\|P_{HS(V)^{N,\pi}}(P_v)\|^2=\sum_{i,j} |\langle P_v,T_{i,j}\rangle|^2
\end{equation}

Decompose $v_{\pi}$ as $v_{\pi}=\sum_{i\in I}u_i$, where $u_i\in U_i$.
Since there exists an orthonormal basis of $V$ containing $v$,
we have
\begin{equation}
\label{eq:HS3}
\langle P_v, T_{i,j}\rangle=\langle v, T_{i,j}(v)\rangle=
\langle u_j, T_{i,j}(u_i)\rangle.
\end{equation}
Next note that
$T_{i,j}^*T_{i,j}$ is an element of $HS(V)^{N,\pi}_{i,i}$
and thus by an earlier remark must act as multiplication by
some scalar $\lam_i$ on $U_i$. Moreover, $\lam_i=\frac 1 {\dim\pi}$
because if $f_1,\ldots, f_k$ is an orthonormal basis for $U_i$, then
$$\lam_i\dim\pi= \sum_{l=1}^k\la T_{i,j}^*T_{i,j}f_l, f_l\ra =\la T_{i,j}f_l, T_{i,j} f_l\ra=\|T_{i,j}\|^2=1.$$
Hence
$\|T_{i,j} u_i\|^2=|\la T_{i,j}^*T_{i,j} u_i,u_i \ra|=\frac{\|u_i\|^2}{\dim\pi}$,
whence $|\langle P_v, T_{i,j}\rangle|^2\leq \frac{\|u_i\|^2\|u_j\|^2}{\dim\pi}$
by~\eqref{eq:HS3}, and~\eqref{eq:HS2} yields
$$
\|P_{HS(V)^{N,\pi}}(P_v)\|^2\leq  \sum_{i,j\in I}
\frac{\|u_i\|^2\|u_j\|^2}{\dim\pi}=\frac{\|v_{\pi}\|^4}{\dim\pi}.
$$
Combining this result with~\eqref{eq:HS1}, we deduce the first assertion of the lemma.

Now we prove the second assertion. Assume that $V$ is an irreducible $N$-space.
As above, if $V$ is infinite dimensional, then $HS(V)^N=0$, so $\|P_{HS(V)^N}(P_v)\|=0$.
If $V$ is finite-dimensional, then $HS(V)^N=HS(V)^G$ is one-dimensional
consisting of scalar operators. The operator of multiplication by $\lam$
has Hilbert-Schmidt norm $|\lam|\sqrt{\dim \pi}$, so $T_{1,1}$, being an
element of Hilbert-Schmidt norm $1$, must act as multiplication by some $\lam$
with $|\lam|=\frac{1}{\sqrt{\dim \pi}}$.
Therefore,
$$
\|P_{HS(V)^N}(P_{v})\|^2=|\la P_v, T_{1,1}\ra|^2=
|\la v, T_{1,1}v\ra|^2=\frac{1}{\dim \pi}.
\qedhere
$$
\end{proof}

\subsection{Relative property $(T)$ for group extensions}

We start with a simple result which reduces verification of relative property $(T)$
to the case of irreducible representations.

\begin{Lemma} \label{relirreducible}
Let $G$ be a countable group, $N$ a normal
subgroup of $G$ and $S$ a finite subset of $G$. Assume that there
exists a set of positive numbers $\{\eps_s:\ s\in S\}$ such that
for any irreducible $G$-space $U$ without nonzero $N$-invariant vectors
and any $0\ne u\in  U$ there exists
$s\in S$ with $\|su-u\|\geq \epsilon_s \|u\|$. Then $$\kappa
(G,N;S)\ge \frac{1}{\sqrt{\sum_{s\in S} \frac 1{\eps_s^2}}}.$$
\end{Lemma}
\begin{Remark} Since $N$ is normal in $G$, for any irreducible $G$-space $U$,
either $U$ has no nonzero $N$-invariant vectors or $N$ acts trivially on $U$.
\end{Remark}
\begin{proof} Let $V$ be a $G$-space without
nonzero $N$-invariant vectors. We need to show that
for any $0\ne v\in V$ there exists $g\in S$ such that
$\|gv-v\|\ge \frac{\|v\|}{\sqrt{\sum_{s\in S} \frac {1}{\eps_s^2}}}$.

By the remark following the definition of a relative Kazhdan constant,
we can assume that $V$ is a cyclic $G$-space. Since $G$ is countable,
this implies that $V$ is separable (that is, the ambient Hilbert space is separable).
In this case, by \cite[Theorem~F.5.3]{BHV},
$V$ is (unitarily equivalent to) the direct integral
$\int_Z^\oplus V(z)d\mu(z)$ of a measurable field of
irreducible $G$-spaces $V(z)$ over
a measure space $(Z,\mu)$, where $Z$ is a standard Borel space
and $\mu(Z)<\infty$. We refer the reader to \cite[\S~F.5]{BHV} for the background
on direct integrals.

Now take any $0\neq v\in V$, and write
it as $v=\int_Z v(z)$ with $v(z)\in V(z)$ for all $z$.
For every $s\in S$ we put
$$Z_s=\{z\in Z: \|sv(z)-v(z)\|\ge \epsilon_s \|v(z)\|\},$$
and let $$Z_0=\{z\in Z : N \mbox{ acts trivially on } V(z)\}.$$
By assumption, $(\cup_{s\in S }Z_s)\cup Z_0=Z$; moreover, $\mu(Z_0)=0$
since otherwise $V$ would have a nonzero $N$-invariant vector.
Hence
$$\sum _{s\in S}\int_{Z_s}\|v(z)\|^2d\mu(z)\ge \int_Z
\|v(z)\|^2d\mu(z)=\|v\|^2,$$ and therefore there exists $g\in S$ such
that
$$
\int_{Z_g}\|v(z)\|^2d\mu(z)\ge
\frac{\|v\|^2}{\eps_g^2\sum_{s\in S} \frac 1{\eps_s^2}}.
$$
Thus,
$$
\|gv-v\|^2 \ge \int_{Z_g} \|gv(z)-v(z)\|^2d\mu(z)>
\int_{Z_g} \epsilon_g^2\|v(z)\|^2d\mu(z)\ge\frac{
\|v\|^2}{\sum_{s\in S} \frac 1{\eps_s^2}}.
\qedhere
$$
\end{proof}

 We are now ready to prove Theorem~\ref{relKazhdan} whose
statement (in fact, an extended version of it) is recalled below.

\begin{Theorem}
\label{relKazhdan2}
Let $G$ be a group, $N$ a normal subgroup of
$G$ and   $Z\subseteq Z(G)\cap N$. Put $H=Z\cap [N,G]$. Let
$A$, $B$ and $C$ be subsets of $G$ satisfying the following conditions
\begin{enumerate}
\item $A$ and $N$ generate  $G$,
\item $\kappa (G/Z,N/Z;B)\ge
\epsilon$,
\item $\kappa (G/H,Z/H;C)\ge \delta$.
\end{enumerate}
Then the following hold:
\begin{itemize}
\item[(a)]
$$\kappa  (G,H;A\cup B)\ge    \frac{12\eps}{ 5\sqrt{72   \eps^2|A|+25|B|}}.$$
\item[(b)]
$$\kappa(G,N;A\cup B\cup C)\ge \frac 1{\sqrt 3}\min\{  \frac{12\eps}{ 5\sqrt{72   \eps^2|A|+25|B|}}, \delta \}.$$
\end{itemize}
\end{Theorem}
\begin{proof}

(a) Using Lemma~\ref{relirreducible} we are reduced to proving the following
claim:

\begin{ClaimN} Let $V$ be a non-trivial irreducible
$G$-space without nonzero $H$-invariant vectors.
Then there is no unit vector $v\in V$ such that
$$ \|sv-v\| \le \frac{\sqrt 2}{5} \textrm{\ for any $s\in A$ and \ }    \|sv-v\| \le \frac {12\eps
}{25} \textrm{\ for any $s\in B$}.$$
\end{ClaimN}
Let us assume the contrary, and let $v\in V$ be a unit vector satisfying the above conditions.
First we shall show that
\begin{equation}
\label{337625}
\|P_{HS(V)^N}(P_v)\|^2\le \frac{337}{625}
\end{equation}

\emph{Case 1:} $V$ has an $N$-eigenvector. In this case $V$ is spanned by
$N$-eigenvectors, and thus we may write $v=\sum_i v_i$, where
$v_i$ are $N$-eigenvectors corresponding to distinct characters.

Assume that $\|v_j\|>\frac 45$ for some $j$.  Since $N$ is normal in $G$, any $g\in G$ sends the vector $v_j$ to some eigenvector for $N$. Consider the subgroup
$$K=\{g\in G: g^{-1}ngv_j=nv_j \textrm{\ for any\ }n\in N\}$$
consisting of elements fixing the character corresponding to $v_j$. Note that   $v_j$ is $[K,N]$-invariant.
Since $V$ has no nonzero $H$-invariant vectors and $H\subseteq [G,N]$,
$K$ is a proper subgroup of $G$. Thus, since $N\subseteq K$, there should exist $s\in A$ which is
not in $K$. In particular $\langle sv_j,v_j\rangle =0$ as $sv_j$
and $v_j$ are both $N$-eigenvectors corresponding to distinct
characters.
Hence
\begin{equation}
\label{eq:10.relKazhdan2}
\|sv-v\|^2 \geq
 \left |\frac{\langle sv -v,v_j\rangle}{\|v_j\| }\right |^2+\left |\frac{\langle sv -v,sv_j\rangle}{\|sv_j\| }\right |^2=
 \left |\frac{\langle sv -v,v_j\rangle}{\|v_j\| }\right |^2+\left |\frac{\langle s^{-1}v -v,v_j\rangle}{\|v_j\| }\right |^2
\end{equation}
Since $\la sv_j,v_j\ra=\la v-v_j,v_j\ra=0$, we have
$\la sv -v,v_j\ra=\la s(v-v_j),v_j\ra-\|v_j\|^2$. Since $\|v_j\|\geq \frac{4}{5}$ and hence $\|v-v_j\|\leq\frac{3}{5}$,
we get
$$\left |\frac{\langle sv -v,v_j\rangle}{\|v_j\| }\right |\geq \|v_j\|-
\left |\frac{\la s(v-v_j),v_j\ra}{\|v_j\|} \right |\geq \frac{4}{5}-\frac{3}{5}=\frac{1}{5}.$$
Thus, $\left |\frac{\langle sv -v,v_j\rangle}{\|v_j\| }\right |^2\geq \frac{1}{25}$ and similarly
$\left |\frac{\langle s^{-1}v -v,v_j\rangle}{\|v_j\| }\right |^2\geq \frac{1}{25}$, so \eqref{eq:10.relKazhdan2}
yields $\|sv-v\|^2>\frac{2}{25}$, which contradicts our assumptions on $v$.
\vskip .15cm

Hence $\|v_i\|\le \frac
45$ for all $i$, and Lemma \ref{pv} easily implies that
$\|P_{HS(V)^N}(P_v)\|^2\le (4/5)^4+(3/5)^4=\frac{337}{625}$
(where the equality is achieved if after reindexing $\|v_1\|=4/5$,
$\|v_2\|=3/5$ and $v_i=0$ for $i\ne 1,2$).

\emph{Case 2:} $V$ has no $N$-eigenvectors.
Then we get directly from Lemma \ref{pv} that $\|P_{HS(V)^N}(P_v)\|^2\le
\frac{1}{2}<  \frac{337}{625}.$
\vskip .2cm

Thus, we have established \eqref{337625} in both cases.
Let $Q=P_{(HS(V)^N)^{\perp}}(P_v)$. Then $\|Q\|\ge \sqrt{1-\frac{337}{625}}=\frac{12\sqrt 2}{25}$,
so Lemma \ref{pw} yields
\begin{equation}
\label{1225}
\|sQ-Q\|=
\|sP_v-P_v\|=\|P_{sv}-P_v\|\le {\sqrt 2}\|sv-v\|\leq \frac{12\sqrt{2}\eps}{25}\le \eps \|Q\|
\end{equation}  for every $s\in B$.

Since $V$ is an irreducible $G$-space, the elements of $Z$ act as
scalars on $V$, so $Z$ acts trivially on $HS(V)$. Thus,
$(HS(V)^N)^{\perp}$ is a $G/Z$-space without nonzero $N/Z$-invariant vectors,
so \eqref{1225} violates the assumption $\kappa (G/Z,N/Z;B)\geq \epsilon$.
This contradiction proves the claim and hence also part (a).
\vskip .1cm

(b)  Let $V$ be a $G$-space without non-trivial $N$-invariant
vectors and $0\neq v \in V$. Let $U$ be the orthogonal complement of $V^Z$ in $V$
and $W$ the orthogonal complement of $U^H$ in $U$. Then
$V=V^Z\oplus U^H\oplus W$, so the projection of $v$ onto at least one of the three subspaces
$V^Z$, $U^H$ and $W$ has norm at least $\frac{\|v\|}{\sqrt{3}}$.

\emph{Case 1:} $\|P_{V^Z}(v)\|\ge \frac{\|v\|}{\sqrt 3}$.
Since $V^Z$ is a $G/Z$-space without nonzero $N/Z$-invariant vectors,
by condition (2) there exists $s\in B$ such that
$\|s P_{V^Z}(v)- P_{V^Z}(v)\|\geq \epsilon \|P_{V^Z}(v)\|$. Therefore,
$$\|sv-v\|\ge \|s P_{V^Z}(v)- P_{V^Z}(v)\|>\frac {\epsilon\|v\|}{\sqrt
3}>\frac{12\eps}{ 5\sqrt{72   \eps^2|A|+25|B|}} \frac {\|v\|}{\sqrt 3}.$$

\emph{Case 2:} $\|P_{U^H}(v)\|\ge \frac{\|v\|}{\sqrt 3}$.
Similarly, since $U^H$ is a $G/H$-space without nonzero $Z/H$-invariant vectors,
by condition (3) there exists $s\in C$ such that
$$\|sv-v\|\ge \|sP_{U^H}(v)-P_{U^H}(v)\|> \frac{\delta \|v\|}{\sqrt 3}.$$

\emph{Case 3:} $\|P_{W}(v)\|\ge \frac{\|v\|}{\sqrt 3}$. In this case,
since $W$ is a $G$-space without $H$-invariant vectors,
we can apply
part (a) to deduce that there exists $s\in A\cup B$ such that
$$\|sv-v\|\ge \|s P_{W}(v)- P_{W}(v)\|> \frac{12\eps}{ 5\sqrt{72   \eps^2|A|+25|B|}} \frac {\|v\|}{\sqrt 3}.$$
\end{proof}
\begin{Remark} Theorem~\ref{relKazhdan2} generalizes a similar result due to Serre
in the case $G=N$ (see, e.g., \cite[Theorem 1.7.11]{BHV} or
\cite[Theorem 1.8]{Ha}). The case of a pair of subgroups $(G,N)$
is also considered in \cite[Lemma 1.1]{NPS}.
\end{Remark}

\subsection{Codistance bounds in nilpotent groups} Let $G$ be a nilpotent group
generated by $k$ subgroups $X_1,\ldots, X_k$. In this subsection we
prove Theorem~\ref{nilpotentcodistance} which gives a bound for
the codistance $\codist(\{ X_i\})$. The case when $k=2$ and $G$ is of nilpotency class 2 was considered in \S~4 of \cite{EJ}. Here we strengthen and generalize
those results.

We will use the following auxiliary result.
\begin{Lemma}
\label{directint}
Let $(Z,\mu)$ be a measure space and $z\to V(z)$ a
measurable field of Hilbert spaces over $Z$. Let $A(z)$ and $B(z)$
be subspaces of $V(z)$. Put $A=\int_Z^\oplus A(z)$ and
$B=\int_Z^\oplus B(z)$. Then for any measurable subset $Z_1$ of $Z$ such that $\mu(Z\setminus Z_1)=0$, $$\orth(A,B)\le \sup_{z\in Z_1}\orth(A(z),B(z)).$$
\end{Lemma}

\begin{proof} Let $a=a(z)\in A(z)$ and $b=b(z)\in B(z)$ be two
 vectors. Then
$$\begin{array}{lll}
|\la a,b\ra| &=&\int_Z |\la a(z),b(z)\ra|\,d\mu(z)= \int_{Z_1} |\la a(z),b(z)\ra|\,d\mu(z)\\ &&\\ &\le & \int_{Z_1}
\orth(A(z),B(z))\|a(z)\|\|b(z)\|d\mu(z)\\ &&\\
&\le& \sup_{z\in Z_1}
\orth(A(z),B(z))\|a\|\|b\|.\end{array}$$
\end{proof}
\begin{Corollary} \label{codistirr} Let $G$ be a countable group generated by subgroups $X_1,\ldots, X_k$.
Then $\codist(X_1,\ldots,X_k)$ is equal to the supremum of the quantities $\codist (V^{X_1},\ldots, V^{X_k})$,
where $V$ runs over all  non-trivial irreducible unitary representations
of $G$.\end{Corollary}
\begin{proof}  Let    $V$ be a  unitary representation
of $G$ without $G$-invariant vectors. By the same argument as in Lemma~\ref{relirreducible},
$V\cong \int_ZV(z)d\mu(z)$ for some measurable field of irreducible $G$-spaces  $V(z)$ over a measure space
$(Z,\mu)$. Put  $$Z_0=\{z\in Z:\ V(z) \textrm{\ is a trivial $G$-space}\}.$$
Since $V$ has no nonzero $G$-invariant vectors, $\mu(Z_0)=0$, whence by Lemma \ref{directint},
$$\begin{array}{lll} \codist(V^{X_1},\ldots, V^{X_k})&=& (\orth(V^{X_1}\times\cdots\times V^{X_k},\diag V))^2\\ &&\\&\le& \sup_{z\in Z\setminus Z_0}(\orth (V(z)^{X_1}\times \cdots\times V(z)^{X_k},\diag V(z)))^2\\ &&\\ &=& \sup_{z\in Z\setminus Z_0} \codist (V(z)^{X_1},\ldots,V(z)^{X_k}).\end{array}$$
\end{proof}

We are now ready to prove the main result of this subsection.
\begin{Theorem}
\label{codist_step}
Let $G$ be a countable
group
generated by  subgroups $X_1,\ldots, X_k$. Let $H$ be a subgroup of $Z(G)$,
and let $m$ be the minimal dimension of an irreducible representation of $G$
which is not trivial on $H$. Denote by $\bar X_i$ the image of $X_i$ in $G/H$,
and let $\eps=1- \codist(\bar X_1,\ldots, \bar X_k)$.  Then
$\codist(X_1,\ldots, X_k)\leq 1-\frac{(m-1)\eps}{2m}$.
\end{Theorem}
\begin{proof}
By Corollary \ref{codistirr} we only have to
consider non-trivial irreducible $G$-spaces. Let $V$ be  a non-trivial
irreducible $G$-space, and let $n=\dim V\in \dbN\cup \{\infty\}$.
If $H$ acts trivially on $V$, there is nothing to prove since $\eps>{\frac{(m-1)\eps}{2m}}$.
Thus, we can assume that $H$ acts non-trivially,
so $n\geq m$.

Now take any vectors $v_i\in V^{X_i}$ ($i=1,\ldots, k$).  It is sufficient to show
that
\begin{equation}
\label{eq:10.0}
\codist(v_1,\dots,v_k)\leq 1 - \frac{(n-1)\eps}{2n}.
\end{equation}

Recall that for $v\in \HH$ we put $\iota(v)=\|v\| P_v\in HS(\HH)$ and that
$\|\iota(v)\|=\|v\|$. Lemma~\ref{pv} implies that
\begin{equation}
\|P_{HS(\HH)^G}(\iota(v_i))\|^2 = \frac{1}{n} \|v_i\|^2= \frac{1}{n} \|\iota(v_i)\|^2
\label{eq10.1}
\end{equation}
and so
$$
\|P_{HS(\HH)^G}(\displaystyle \sum_{i=1}^k  \iota(v_i))\|^2  \leq
k \displaystyle \sum_{i=1}^k  \|P_{HS(\HH)^G}(\iota(v_i))\|^2 =
\frac{k}{n} \displaystyle \sum_{i=1}^k \| \iota(v_i) \|^2.
$$
On the other hand, since  $Z(G)$ acts trivially on $HS(\HH)$,  the action factors
through $G/H$, which means that $(HS(\HH)^G)^\perp$ is a $G/H$-space
without invariant vectors. Since $v_i\in V^{X_i}$, we have $\iota(v_i)\in HS(\HH)^{\bar X_i}$
and $P_{(HS(\HH)^G)^\perp}(\iota(v_i))\in {((HS(\HH)^G)^\perp)}^{\bar X_i}$. Hence
by the definition of codistance
$$ \begin{array}{lll}
\|P_{(HS(\HH)^G)^\perp}(\displaystyle \sum_{i=1}^k  \iota(v_i))\|^2 & \leq &
k  \codist(\bar X_1,\dots,\bar X_k) \displaystyle \sum_{i=1}^k  \|P_{(HS(\HH)^G)^\perp}(\iota(v_i))\|^2\\ & &\\ &=&
\frac{k(n-1)}{n}  \codist(\bar X_1,\dots,\bar X_k)\displaystyle \sum_{i=1}^k  \|\iota(v_i)\|^2.\end{array},$$
where the last equality holds by \eqref{eq10.1}.

Combining these inequalities, we conclude that $\left\| \displaystyle \sum_{i=1}^k \iota(v_i) \right\|^2$
is bounded above by

$$
k \left( \frac{1}{n}  + \codist(\bar X_1,\dots,\bar X_k) \frac{n-1}{n} \right)
\displaystyle \sum_{i=1}^k  \|\iota(v_i)\|^2
$$
Therefore
$$
\codist(\iota(v_1),\dots,\iota(v_k)) \leq
\frac{1}{n}  + \codist(\bar X_1,\dots,\bar X_k) \frac{n-1}{n},
$$
which, combined with Lemma~\ref{HS-codist}, gives the following inequality
equivalent to \eqref{eq:10.0}:
$$
2\codist(v_1,\dots,v_k) - 1 \leq
\frac{1}{n}  + \codist(\bar X_1,\dots,\bar X_k) \frac{n-1}{n}=
\frac{1}{n}+(1-\eps)\frac{n-1}{n}=1-\frac{(n-1)\eps}{n}.
$$
\end{proof}

 We are now ready to prove Theorem~\ref{nilpotentcodistance}
(whose statement is recalled below).
\begin{Theorem}
\label{nilpotentcodistance2}
Let $G$ be a countable  nilpotent group of class $c$ generated by
subgroups $X_1,\dots, X_k$.
Then
$$
\codist(X_1,\dots, X_k)\le 1-\frac{1}{4^{c-1}k}.
$$
\end{Theorem}
\begin{proof}
We prove the theorem by induction on the nilpotency class $c$. The induction step follows from Theorem \ref{codist_step}.
To establish the base case $c=1$, in which case $G$ is abelian, we use
a separate induction on $k$. The case $k=2$ holds by \cite[Lemma~3.4]{EJ}.
We now do the induction step on $k$. Let $V$ be a $G$-space without $G$-invariant vectors,
and let $v_i\in V^{X_i}$. Then using the induction hypothesis in the fourth line, we obtain
$$
\begin{array}{lll}
\| \displaystyle \sum_{i=1}^{k} v_i\|^2 & = & \| \displaystyle \sum_{i=1}^{k-1}  P_{(V^{X_k})^\perp}(v_i)\|^2+ \| \displaystyle \sum_{i=1}^{k}  P_{V^{X_k}}(v_i)\|^2\\ &&\\
&\le &(k-1) \displaystyle \sum_{i=1}^{k-1}  \|P_{(V^{X_k})^\perp}(v_i)\|^2+
\|\displaystyle (k-2)\frac{\sum\limits_{i=1}^{k-1}  P_{V^{X_k}}(v_i)}{k-2}+v_k\|^2\\ &&\\
&\le &(k-1) \displaystyle \sum_{i=1}^{k-1}  \|P_{(V^{X_k})^\perp}(v_i)\|^2+\frac{k-1}{k-2}\|\displaystyle \sum_{i=1}^{k-1}  P_{V^{X_k}}(v_i)\|^2+(k-1)\|v_k\|^2\\ &&\\
&\le & (k-1) \displaystyle \sum_{i=1}^{k-1}  \|P_{V^{(X_k})^\perp}(v_i)\|^2+(k-1)\displaystyle \sum_{i=1}^{k-1}  \|P_{V^{X_k}}(v_i)\|^2+(k-1)\|v_k\|^2\\ &&\\
&=&(k-1)\displaystyle \sum_{i=1}^{k}  \|v_i\|^2.
\end{array}
$$
\end{proof}


\appendix
\section{}
\label{subsec:relative}

This section contains proofs of Theorem~\ref{KasShalom2} and~\ref{relativeAn}.
Both results are virtually identical to~\cite[Cor 1.8]{Ka} and~\cite[Cor 1.10]{Ka},
respectively (the only difference is that $\EL_n$ is replaced by $\St_n$, which does
not affect the argument), but we have chosen to provide proofs of both results to
make this paper essentially self-contained.

Let $R$ be a finitely generated associative ring (with $1$). Recall that we defined the
Steinberg group $\St_2(R)$ to be the free product $R \star R$. Given $r\in R$, the element of $\St_2(R)$
corresponding to $r$ from the first (resp. second) copy of $R$ will be denoted by $x_{12}(r)$
(resp. $x_{21}(r)$).

Let $\{ r_1,\dots r_k \}$ be a finite generating set of $R$, and let $M = \la m_1, \dots, m_k \ra$ be a
finitely generated left $R$-module.
Consider the action of the group $\St_2(R) = R \star R$ on $M^2$ where $x_{12}(r)$ acts by the matrix {\small $\begin{pmatrix} 1 & r \\ 0 & 1 \end{pmatrix}$} and $x_{21}(r)$ acts by the matrix {\small $\begin{pmatrix} 1 & 0 \\ r & 1 \end{pmatrix}$}.
Define $F_1\subset \St_2(R)$ and $F_2\subset M^2$ by
$$F_1=\{x_{12}(\pm r_i), x_{12}(\pm 1), x_{21}(\pm r_i), x_{21}(\pm 1)\}\quad \mbox{ and }\quad F_2= \{(\pm m_i,0), (0,\pm m_i)\}.$$
Note that $|F_1|\leq 4(k+1)$ and $|F_2|\leq 4d$.

\begin{Theorem}
\label{thm:relT}
Let $\la F_1\ra$ be the subgroup of $\St_2(R)$ generated by $F_1$,
let $\Gamma=\la F_1\ra\ltimes M^2$, and let $F = F_1 \cup F_2$.
Then the pair $(\Gamma, M^2)$ has relative property $(T)$, and moreover
$$\kappa_r(\Gamma, M^2; F) > \frac{1}{2K(k,d)},$$
where $K(k,d) = 12(\sqrt{k} + \sqrt{d}+1)$. \footnote{This bound can be improved slightly.}
\end{Theorem}
\begin{Remark} Theorem~\ref{thm:relT} remains true for the semi-direct product
$M^2\rtimes \la F_1 \ra$ where $M$ is a right $R$-module 
and elements of $F_1$ act on $M^2$ by right multiplication. This is because
such semi-direct product $M^2\rtimes \la F_1 \ra$ is isomorphic to
$\la (F_1)^{\rm op}\ra\ltimes (M^2)^{\rm op}$ where $(M^2)^{\rm op}=M^2$
considered as a left module over the opposite ring $R^{\rm op}$ and
$(F_1)^{\rm op}$ is the image of $F_1^{\rm op}$ in the Steinberg group $\St_2(R^{\rm op})$
under the canonical isomorphism $\St_2(R)\to \St_2(R^{\rm op})$.
\end{Remark}

\begin{proof}
Let $V$ be a  unitary representation  of $\Gamma$, and let $v\in V$ be $(F,\eps)$-invariant
for some $\eps>0$. By definition of the Kazhdan ratio $\kappa_r(\Gamma, M^2; F)$ we need to prove
that $v$ is $(M^2, 2 K(k,d)\eps)$-invariant. Without loss of generality we will assume
that $v$ is a unit vector.

Let $\mathcal{P}$ be the 
projection valued measure on the dual $(M^2)^* = \Md$, coming
from the restriction of the representation $V$ to $M^2$.
Here $M^*$ stands for the Pontryagin dual of $(M,+)$, the additive group of $M$, that is,
$
M^* = \Hom \left((M,+), S^1\right),
$
where
$S^1$ is the unit circle in the complex plane.
The set $M^*$ is an abelian group with natural topology.
Let $\mu_v$ be the probability measure on $\Md$, defined by %
$\mu_v(B)=\la \mathcal{P}(B)v,v \ra$.

The group $\St_2(R)$ has a natural right action on $\Md$, which is dual to the
standard left action of $\St_2(R)$ on $M^2$ and is given by
$$\chi^g((m,n))=\chi((gm,gn)) \mbox{ for all }g\in G, \chi\in \Md\mbox{ and }(m,n)\in M^2.$$
The following lemma shows that the measure $\mu_v$ is almost $F_1$-invariant with respect to this
action.

\begin{Lemma}
\label{lm:invariance}
For every measurable set $B \subset \Md$ and
every $g \in F_1$ we have
$$
|\mu_v(B^g) - \mu_v(B)| \leq 2\epsilon\sqrt{\mu_v(B)}+ \epsilon^2
\quad \mbox{or, equivalently, } \quad
|\sqrt{\mu_v(B^g)} - \sqrt{\mu_v(B)}| \leq \epsilon.
$$
\end{Lemma}

\begin{proof}
Using properties of projection valued measures, it is easy to show that
$\mu_v(B^g)=\la\pi(g^{-1})\mathcal{P}(B)\pi(g)v,v\ra$. Therefore,$$
\begin{array}{r@{\,\,}l}
 |\mu_v(B^g) - \mu_v(B)| = & 
       | \la\pi(g^{-1})\mathcal{P}(B)\pi(g)v,v\ra - \la \mathcal{P}(B)v,v\ra| \\
 \leq & \displaystyle \rule[15pt]{0pt}{0pt}
       |\la\pi(g^{-1})\mathcal{P}(B)(\pi(g)v-v),v\ra | +
       |\la \mathcal{P}(B)v,(\pi(g)v-v\ra)|  \\
\leq & \displaystyle \rule[15pt]{0pt}{0pt}
       2|\la\pi(g)v-v,\mathcal{P}(B)v\ra| +
       \la \mathcal{P}(B)(\pi(g)v-v),\pi(g)v-v\ra   \\
 \leq & \displaystyle \rule[15pt]{0pt}{0pt}
       2\epsilon \sqrt{ \mu_v(B)} + \epsilon^2,
\end{array}
$$
where the final inequality follows from the facts
that $v$ is $(F,\epsilon)$ invariant
vector and $||\mathcal{P}(B)v||^2 = \mu_v(B)$.
\end{proof}

We will need the following consequence of the above lemma:
\begin{Lemma}
\label{lm:minvariance}
Let $A$ and $B$ be measurable sets in $\Md$. Suppose that $A$ decomposes
as a disjoint union of the sets $A_i$ for $1\leq i\leq s$ and there
exist elements $g_i \in F_1$ such that the sets $B_i=(A_i)^{g_i}$ are disjoint
subsets of $B$. Then
$$
\sqrt{\mu_v(A)} \leq \sqrt{\mu_v(B)} + \sqrt{s}\epsilon.
$$
\end{Lemma}
\begin{proof}
Applying lemma~\ref{lm:invariance} to the sets $A_i$ yields
$$
\begin{array}{r@{}c@{}l}
\mu_v(A) & {}={} & \displaystyle
\sum_{i=1}^s \mu_v(A_i) \leq \sum_{i=1}^s \left[ \mu_v(B_i) + 2\epsilon \sqrt{ \mu_v(B_i)} + \epsilon^2 \right] \leq \\
& \leq & \displaystyle
\sum\mu_v(B_i) + 2\epsilon\sqrt{s\sum \mu_v(B_i)} + s\epsilon^2 \leq \\
& \leq & \displaystyle
\mu_v(B) + 2\epsilon\sqrt{s\mu_v(B)} + s\epsilon^2 = \left(\sqrt{\mu_v(B)} + \sqrt{s}\epsilon\right)^2,
\end{array}
$$
where we have used that for nonnegative numbers $a_i$  the following inequality holds:
$$
\displaystyle \sqrt {\sum_{i=1}^s a_i} \leq \sum_{i=1}^s \sqrt{a_i} \leq \sqrt {s \sum_{i=1}^s a_i}.
$$\end{proof}

For an element $\chi \in \Md$ we will write $\chi=(\chi_1,\chi_2)$,
where  $\chi_1,\chi_2\in M^*$.

\begin{Lemma}
\label{inv}
For $1\leq s\leq d$ and $i=1,2$
let $$P_{i,s} = \{\chi \in \Md : \,\, \mathrm{Re\, }\chi_i(m_s) > 0 \}.$$
Then $\mu_v(P_{i,s}) \geq 1-\epsilon^2/2$.
(Recall that $m_1,\ldots, m_d$ are the given generators of $M$).
\end{Lemma}
\begin{proof}
For $1\leq i\leq s$ let $g_{1,s}=(m_s,0)\in M^2$
and $g_{2,s}=(0,m_s)\in M^2$, so that by construction $g_{i,s}\in F_2$.
By the definition of the measure $\mu_v$, we have
$$
|| g_{i,s} v-v||^2 =
\int_{\Md} | \chi_i(m_s) -1 |^2\, d\mu_v.
$$
By assumption $v$ is $(F_2,\eps)$-invariant, whence
$||g_{i,s}v-v|| \leq \eps$.
If we break the integral into two integrals over $P_{i,s}$ and its complement, we get
$$
\epsilon^2 \geq \int_{\Md\setminus P_{i,s}} \!\!\!\!\!| \chi_i(m_s) -1 |^2\, d\mu_v +
\int_{P_{i,s}} | \chi_i(m_s) -1 |^2\, d\mu_v
\geq \int_{\Md\setminus P_{i,s}} \!\!\!\!\! 2 \, d\mu_v = 2(1-\mu_v(P_{i,s})),
$$
which yields the desired inequality.
\end{proof}

The rest of the proof shows that any probability measure on $\Md$ satisfying the above lemmas
must be very close to the Dirac measure at the origin, which implies that the vector $v$ must
be almost $M^2$-invariant.

First we consider the special case $R=M=\Z$, so that we can identify  $M^*$ with $S^1$ and
$\Md$ with the two-dimensional torus $\mathbb{T}^2$. The action of $\St_2(\Z)$ on $\mathbb{T}^2$
factors through the action of $\SL_2(\Z)$. It is well known that the only $\SL_2(\Z)$-invariant measures
on the torus $\mathbb{T}^2$ are multiples of the Dirac measure at the origin. The following lemma is a quantitative
version of this fact.

\begin{Lemma}
\label{lm:R2}
Let $\mu$ be a finitely additive measure on $\mathbb{T}^2$ such that
\begin{enumerate}
\item %
$\mu(\{(x,y) : \mathrm{Re\,} x <0 \}) \leq \epsilon^2/2$ and %
$\mu(\{(x,y) : \mathrm{Re\,} y <0 \}) \leq \epsilon^2/2$, %
\item %
$|\mu(B^g) - \mu(B)| \leq 2\epsilon \sqrt{\mu(B)} + \epsilon^2$ %
for any Borel set $B\subseteq \mathbb T^2$ and any elementary matrix $g\in \SL_2(\Z)$
with $\pm 1$ off the diagonal.
\end{enumerate}
Then the measure $\mu$ satisfies
$$
\mu(\mathbb T^2\setminus \{(0,0)\}) \leq (2+\sqrt{10})^2\epsilon^2.
$$
\end{Lemma}
\begin{proof}
Define the Borel subsets $A_i$ and $A_i'$ of $\mathbb{T}^2$ using
the picture below -- here we identify the torus $\mathbb{T}^2$ with the square  $(-1/2,1/2] \times (-1/2,1/2]$:

\setlength{\unitlength}{3947sp}%
\begin{picture}(3000,2900)(200,-2350)
\thinlines
\put(1800,-2100){\framebox(2400,2400){}}
\put(1800,-1500){\line(1,0){2400}}
\put(1800, -300){\line(1,0){2400}}
\put(2400, -900){\line(1,0){1200}}
\put(2400,-2100){\line(0,1){2400}}
\put(3600,-2100){\line(0,1){2400}}
\put(3000,-1500){\line(0,1){1200}}
\put(2400,-1500){\line(1, 1){1200}}
\put(2400, -300){\line(1,-1){1200}}
\put(1800,-1500){\line(1, 1){600}}
\put(1800, -300){\line(1,-1){600}}
\put(3600, -900){\line(1, 1){600}}
\put(3600, -900){\line(1,-1){600}}
\put(2400,-2100){\line(1, 1){600}}
\put(2400,  300){\line(1,-1){600}}
\put(3000,-1500){\line(1,-1){600}}
\put(3000, -300){\line(1,1){600}}
\put(3100,-520){$A_1$}
\put(3300,-770){$A_2$}
\put(3300,-1120){$A_3$}
\put(3100,-1370){$A_4$}
\put(2700,-1370){$A_1$}
\put(2500,-1120){$A_2$}
\put(2500,-770){$A_3$}
\put(2700,-520){$A_4$}
\put(3700,-520){$A_1'$}
\put(3300,-170){$A_2'$}
\put(3300,-1720){$A_3'$}
\put(3700,-1370){$A_4'$}
\put(2100,-1370){$A_1'$}
\put(2500,-1720){$A_2'$}
\put(2500,-170){$A_3'$}
\put(2100,-520){$A_4'$}
\put(1400,-2300){$(-1/2,-1/2)$}
\put(1400, 400){$(-1/2,1/2)$}
\put(4000,-2300){$(1/2,-1/2)$}
\put(4000,400){$(1/2,1/2)$}
\end{picture}

\noindent
Each set $A_i$ or $A_i'$ consists of the interiors of two
triangles and part of their boundary (not including the vertices).
The sets $A_i$ do not contain the side which is part of the small
square, they also do not contain their clockwise boundary
but contain the counter-clockwise one. Each set $A_i'$ includes
only the part of its boundary which lies on the small square.

It can be seen, from the picture, that the elementary matrices
$e_{ij}(\pm 1)\in \SL_2(\Z)$,
act on the sets $A_i$ as follows:
$$
\begin{array}{ll}
(A_3 \cup A_4')^{e_{21}(1)}  = A_3 \cup A_4 \quad &
(A_3'\cup A_4 )^{e_{12}(1)}  = A_3 \cup A_4 \\
(A_1'\cup A_2 )^{e_{21}(-1)}  = A_1 \cup A_2  \quad &
(A_1 \cup A_2')^{e_{12}(-1)}  = A_1 \cup A_2.
\end{array}
$$
In view of hypothesis (2) of the lemma, these equalities
yield the following inequalities:
$$
\begin{array}{l}
\mu(A_1) + \mu(A_2) \leq
     \mu(A_1') + \mu(A_2) + \epsilon^2 +
     2\epsilon\sqrt{\mu(A_1') + \mu(A_2 )} \\
\mu(A_1) + \mu(A_2) \leq
     \mu(A_1) + \mu(A_2') + \epsilon^2 +
     2\epsilon\sqrt{\mu(A_1 ) + \mu(A_2')} \\
\mu(A_3) + \mu(A_4) \leq
     \mu(A_3') + \mu(A_4) + \epsilon^2 +
     2\epsilon\sqrt{\mu(A_3') + \mu(A_4 )} \\
\mu(A_3) + \mu(A_4) \leq
     \mu(A_3) + \mu(A_4') + \epsilon^2 +
     2\epsilon\sqrt{\mu(A_3 ) + \mu(A_4')}. \\
\end{array}
$$
Adding these inequalities and noticing that
$$
\begin{array}{l}
\mu(A_1') + \mu(A_4') \leq
     \mu(\{ |x| \geq 1/4\}) \leq
     \epsilon^2/2 \,\,\, \mbox{ and}\\
\mu(A_2') + \mu(A_3') \leq
     \mu(\{ |y| \geq 1/4\}) \leq
     \epsilon^2/2
\end{array}
$$
we obtain
$$
\sum \mu(A_i) \leq
    4\epsilon^2 + \sum \mu(A_i') +
        2\epsilon \sqrt{4\left(\sum \mu(A_i) +
    \sum \mu(A_i')\right)} \leq
$$
$$
    \leq 5\epsilon^2 + 4\epsilon \sqrt{\sum \mu(A_i) + \epsilon^2}.
$$
This yields $\sum\mu(A_i) \leq (13 +4 \sqrt{10}) \epsilon^2$,
and therefore
$$
\mu(\mathbb{T}^2\setminus \{(0,0)\}) \leq \sum \mu(A_i) + \mu(\{ |x| \geq
1/4\}) + \mu(\{ |y| \geq 1/4\}) \leq
$$
$$
\leq (14 + 4 \sqrt{10}) \epsilon^2 = (2+\sqrt{10})^2\epsilon^2,
$$
which completes the proof.
\end{proof}

We are now ready to prove the analogue of Lemma~\ref{lm:R2}
dealing with arbitrary ring $R$ and $R$-module $M$.

\begin{Lemma}
\label{lm:R2k}
Let $\mu$ be a finitely additive measure on $\Md$ such that
\begin{enumerate}%
\item %
$\mu(\{\chi: \mathrm{Re\, } \chi_i(m_s) < 0 \}) \leq \epsilon^2/2$ for all $i=1,2$ and $s=1,\dots, d$.

\item %
$|\mu(B^g) - \mu(B)| \leq 2\epsilon \sqrt{\mu(B)} + \epsilon^2$ %
for any Borel set $B\subseteq \Md$ and any $g\in F_1 \subset \St_2(R)$.
\end{enumerate}
Then we have
$$
\mu( \Md \setminus \{(0,0)\}) \leq K(k,d)^2 \epsilon^2, 
$$
where the constant $K(k,d)$ is defined in Theorem~\ref{thm:relT}.
\end{Lemma}
\begin{proof}
Define the following increasing filtration (as abelian group) of the module $M$ :
Let $M^{(0)} =  \mathrm{span}_{\Z}\{m_1,\dots, m_d\}$ and
$M^{(i+1)} = M^{(i)} + r_1  M^{(i)} + \cdots  + r_k M^{(i)}$, that is, $M^{(i)}$ is the span of all elements in $M$ which can be obtained from the module generators using at most $i$ multiplications by the ring generators.
One has $M = \bigcup_i M^{(i)}$, since $R$ and $M$ are generated by $\{r_i\}$ and $\{m_j\}$ respectively.

This filtration induces a decreasing filtration of the dual
$$
M^*_{(i)} = \left\{ \chi \in M^* \mid \chi(m) = 1,\, \forall m \in M^{(i)} \right\}
$$
such that $\{0 \} = \cap M^*_{(i)}$. This filtration yields a valuation $\nu: M^* \to \mathbb{N} \cup \{\infty\}$
by setting $\nu(\chi) = i$ if $\chi \in M^*_{(i)} \setminus M^*_{(i+1)}$ for some $i$ and $\nu(0)=\infty$.

Let us define the following subsets of $\Md\setminus\{0,0\}$:
$$
\begin{array}{r@{}c@{}l}
A & {} = {} & \left\{ (\chi_1,\chi_2) \mid \nu(\chi_1) > \nu(\chi_2) > 0 \right\} \\
B &    =    & \left\{ (\chi_1,\chi_2) \mid \nu(\chi_1) = \nu(\chi_2) > 0 \right\} \\
C &    =    & \left\{ (\chi_1,\chi_2) \mid \nu(\chi_2) > \nu(\chi_1) > 0 \right\} \\
D &    =    & \left\{ (\chi_1,\chi_2) \mid \nu(\chi_1)=0 \mbox{ or } \nu(\chi_2) = 0 \right\}. \\
\end{array}
$$

Consider the action of $x_{12}(r),x_{21}(r) \in F_1\subset\St_2(R)$ on the element
$(\chi_1,\chi_2)\in \Md$. By the definition of this action we have
$$
\left((\chi_1,\chi_2)^{x_{12}(r)} \right)\left(\begin{array}{@{}c@{}}m_1 \\m_2\end{array}\right) =
(\chi_1,\chi_2)\left(x_{12}(r) \left(\begin{array}{@{}c@{}}m_1 \\m_2\end{array}\right)\right) =
(\chi_1,\chi_2)\left(\begin{array}{@{}c@{}}m_1  + r\cdot m_2\\m_2\end{array}\right)
$$
and
$$
\left((\chi_1,\chi_2)^{x_{21}(r)} \right)\left(\begin{array}{@{}c@{}}m_1 \\m_2\end{array}\right) =
(\chi_1,\chi_2)\left(x_{21}(r)\left(\begin{array}{@{}c@{}}m_1 \\m_2\end{array}\right)\right) =
(\chi_1,\chi_2)\left(\begin{array}{@{}c@{}}m_1 \\m_2+r\cdot m_1\end{array}\right)
$$
This shows that
$$
A^{x_{21}(1)} \subseteq B  \quad \mbox{ and } \quad
C^{x_{12}(1)} \subseteq B,
$$
and by hypothesis (2) of Lemma~\ref{lm:R2k} we get
\begin{equation}
\label{AC_bound}
\sqrt{\mu(A)} \leq \sqrt{\mu(B)} + \epsilon
\quad \mbox{and} \quad
\sqrt{\mu(C)} \leq \sqrt{\mu(B)} + \epsilon.
\end{equation}

\begin{Claim} The union $A\cup B$ can be decomposed as a disjoint union of
$k$ sets $A_i$ such that the sets $\widetilde{A_i}= \left(A_i\right)^{x_{21}(r_i)}$ are
disjoint and lie in $C\cup D$. Similarly, $C\cup B$ can be written as $\bigcup C_i$ such that
$\widetilde{C_i}= \left(C_i\right)^{x_{12}(r_i)}$ are disjoint subsets of $A\cup D$.
\end{Claim}
\begin{proof}
Let $(\chi_1,\chi_2)\in A\cup B$.
Define $\chi_{2,j}\in M^*$ by
$$
\chi_{2,j}(f) = \chi_2 (r_j f).
$$
Using the definition of the subsets $M^{(i)}$ one can see that there exists some $j$ between $1$ and $k$ such that
$\nu(\chi_{2,j}) =\nu(\chi_2) - 1$.
Let $s(\chi_2)$ be the smallest index $j$ such that $\chi_{2,j}$ satisfies this condition.
Define the sets $A_i$ by
$$
A_i = \left\{ (\chi_1,\chi_2)\in A\cup B : s(\chi_2) = i \right\}.
$$
It is easy to check that $A \cup B$ is a disjoint union of $A_i$ and that
$\widetilde{A_i}= \left(A_i\right)^{x_{21}(r_i)}$
are disjoint subsets of $C\cup D$.

The second part of the claim is proved by the same argument applied
to the first component $\chi_1$ instead of the second one $\chi_2$.
\end{proof}

We are ready to complete the proof of Lemma~\ref{lm:R2k}.
Applying Lemma~\ref{lm:minvariance} to the sets $A\cup B$ and $C\cup B$ yields
$$
\begin{array}{r@{}c@{}l}
\sqrt{\mu(A\cup B)} &{}\leq{}& \sqrt{\mu(C \cup D)} + \sqrt{k}\epsilon\\
\sqrt{\mu(C\cup B)} &{}\leq{}& \sqrt{\mu(A \cup D)} + \sqrt{k}\epsilon.
\end{array}
$$

Squaring and adding the the above inequalities and taking \eqref{AC_bound} into account,
we get
$$
\begin{array}{r@{}c@{}l}
2\mu(B) & {} \leq {} & \displaystyle \rule[15pt]{0pt}{0pt}
2\mu(D) + 2 \epsilon\sqrt{k(\mu(C) + \mu(D))} + 2 \epsilon\sqrt{k(\mu(A) + \mu(D))} + 2k\epsilon^2  \leq \\
&\leq & \displaystyle \rule[15pt]{0pt}{0pt}
2\mu(D) + 4 \epsilon\sqrt{k\mu(D)} + 2 \epsilon\sqrt{k\mu(A)} +2 \epsilon\sqrt{k\mu(C)} + 2k\epsilon^2 \leq \\
&\leq & \displaystyle \rule[15pt]{0pt}{0pt}
2\mu(D) + 4 \epsilon\sqrt{k\mu(D)} + 4 \epsilon\sqrt{k\mu(B)} + (2k+4\sqrt{k})\epsilon^2 \\
\end{array}
$$
The last inequality can be rewritten as
\begin{equation}
\label{B_bound}
(\sqrt{\mu(B)} - \sqrt{k}\epsilon)^2 \leq (\sqrt{\mu(D)}+\sqrt{k}\epsilon)^2 + (k+2\sqrt{k})\epsilon^2.
\end{equation}
Note that \eqref{AC_bound} and \eqref{B_bound} yield an upper bound for $\mu(A),\mu(B)$ and $\mu(C)$
in terms of $\mu(D)$, so to finish the argument we just need to find a suitable bound for $\mu(D)$.

\vskip .15cm
Fix $s\in\{1,\ldots, d\}$, and let $M_s$ be the subgroup of $M^2$ generated by $(m_s,0)$ and $(0,m_s)$.
Restricting the functionals $\chi$ from $M$ to $M_s$,  we obtain a map $\Md \to (M_s^*)^{2}$.
Let $\pi_s: \Md \to \mathbb{T}^2$ be the composition of this map with the natural embedding
$(M_s^*)^{2}\to \mathbb{T}^2$ (which is an isomorphism if $m_s$ has infinite additive order).
Let $\mu_s  = \mu \circ \pi_s^{-1}$ be the measure on $\mathbb{T}^2$ obtained from $\mu$
by pullback via $\pi_s$.

Since $M_s$ is invariant under the natural action of $\SL_2(\Z)$ on $M^2$,
the kernel of the map $\pi_s$ is invariant under the dual (right) action of $\SL_2(\Z)$ on $\Md$, and therefore
we have a well-defined right action of $\SL_2(\Z)$ on $\dbT^2$ given by
$$f^g=\pi_s((\pi_s^{-1}(f))^{g})\mbox{ for all }f\in \dbT^2 \mbox{ and }g\in SL_2(\Z).$$
It is easy to see that the measure $\mu_s$ satisfies the hypotheses of Lemma~\ref{lm:R2} with respect to this action,
so $\mu_s(\mathbb{T}^2\setminus(0,0) ) \leq (2+\sqrt{10})^2\epsilon^2$.

Finally, observe that the set $D$ is the same as $\bigcup_s \pi_s^{-1} (\mathbb{T}^2\setminus(0,0) )$, and therefore
\begin{equation}
\label{D_bound}
\mu(D) \leq d  (2+\sqrt{10})^2\epsilon^2.
\end{equation}
The desired inequality $
\mu(\Md \setminus(0,0))\leq K(k,d)^2 \epsilon^2$ now follows from \eqref{AC_bound},\eqref{B_bound} and \eqref{D_bound} by a straightforward computation.
\end{proof}

Theorem~\ref{thm:relT} now follows easily.
For any $g\in M^2$ we have
$$
||g v - v ||^2 = \int_{\Md} |\chi(g) - 1|^2 \, d \mu_v \leq
\int_{\Md\setminus (0,0)} 4 \, d \mu_v   = 4 \mu_v(\Md \setminus(0,0))\leq 4 K(k,d)^2 \epsilon^2,
$$
where the last inequality holds by Lemma~\ref{lm:R2k}. Hence $v$
is $(F, 2K(k,d) \epsilon)$-invariant, as desired.
\vskip .1cm

Alternatively, we can finish the argument as follows.
Lemma~\ref{lm:R2k} implies that
for $\eps<\frac{1}{K(k,d)}$ we have $\mu_v(\Md \setminus(0,0))<1$. Hence
the projection operator $\mathcal P(\{(0,0\})$ is nonzero, that is,
$V$ has a nonzero $M^2$-invariant vector. Thus, by definition,
$\kappa(\Gamma,M^2; F)\geq \frac{1}{K(k,d)}$, whence
 $\kappa_r(\Gamma,M^2; F)\geq \frac{1}{2 K(k,d)}$ by
Observation~\ref{Kazhrat}(ii).
\end{proof}

Using a similar method we can prove the higher-dimensional analog of Theorem~\ref{thm:relT}:

\begin{Theorem}
\label{thm:relT_St_p}
Let $R$ and $M$ be as in Theorem~\ref{thm:relT}, with generating set $\{r_1,\ldots, r_k\}$ and $\{m_1,\ldots, m_d\}$,
respectively.
Let $\Gamma=\St_p(R)\ltimes M^p$ for some $p \geq 3$, let $F_1$ be the generating set of $\St_p(R)$ consisting of
the elements of root subgroups of the form $x_{ij}(\pm r_l)$ and $x_{ij}(\pm 1)$, and let $F_2$ be the set of
standard generators of $M^p$ and their inverses (so that $|F_1|\leq 2p(p-1)(k+1)$ and $|F_2|\leq 2pd$),
and let $F=F_1\cup F_2$.
Then the pair $(\Gamma,M^p)$ has relative property $(T)$ and
$$\kappa_r(\Gamma, M^p; F) > \frac{1}{2 K(k,d,p)},$$
where $K(k,d,p) = 20(\sqrt{k} + \sqrt{d}+\sqrt{p})$.
\end{Theorem}
\begin{Remark}
In the above theorem one can replace the group $\St_p(R)$ by the free product of $p(p-1)$ copies of the additive group of $R$.
Also note that $St_p(R)=\la F_1\ra$ since $p\geq 3$.
\end{Remark}
\begin{proof}
The proof is similar to that of Theorem~\ref{thm:relT}.
We start with a unitary representation $V$ of $\Gamma=\St_p(R)\ltimes M^p$ and an $(F,\eps)$-invariant
unit vector $v\in V$. Let $\mu_v$ be the measure on $\left.M^*\right.^p$ coming from the restriction of the
representation to the abelian group $M^p$.
As in the case $p=2$, the measure $\mu_v$
is almost invariant under the action of the generators of $\St_p(R)$.

We will write an element  $\chi \in \left.M^*\right.^{p}$ as $(\chi_1,\dots,\chi_p)$,
where $\chi_i \in M^*$.
Define the Borel subsets $B_i,C_i$ of $\left.M^*_k\right.^p$ for $2\leq i\leq p$
 by
$$
\begin{array}{r@{}c@{}l}
B_i &{}= {}&\{ \chi\in \left.M^*\right.^{p} : \chi_j =0 \mbox{ for } j\leq i\} \mbox{ and}\\
C_i &= &\{ \chi\in \left.M^*\right.^{p} : \chi_1 = \chi_i \not =0, \chi_j = 0 \mbox{ for } 1<j<i\}.
\end{array}
$$
Using the restriction $\left.M^*\right.^p \to\left.M^*\right.^2$ coming form the
inclusion $M^2 \subset M^p$ and Lemma~\ref{lm:R2k}, it is easy to see that
$$
\mu_v(\{\chi : \chi_1\not= 0 \mbox{ or } \chi_2\not= 0 \}) \leq K(k,d) \epsilon^2, \mbox{ that is,}
$$
\begin{equation}
\label{eq:B2p1}
\mu_v(\left.M^*\right.^{p}\setminus B_2)\leq K(k,d)^2 \epsilon^2.
\end{equation}

On the other hand, the elementary matrix $x_{1i}(1) \in \St_p(R)$ sends $B_{i-1} \setminus
B_i$ into $C_i$ for any $i\geq 3$.
Now notice that the sets $C_i$ are
disjoint for $i=2,\dots, p$ and their union lies in the set %
\begin{equation}
\label{eq:setC}
C:=\{ \chi: \chi_1 \not =0, \chi_2 = 0 \} \subseteq \left.M^*\right.^{p}\setminus  B_2.
\end{equation}
Applying Lemma~\ref{lm:minvariance} to the action of $x_{1i}(1)$ on the set $B_{i-1}\setminus B_i$ for $3\leq i\leq p$,
after an easy computation
we get
\begin{equation}
\label{eq:B2p2}
\sqrt{\mu_v(B_2 \setminus B_p )} \leq \sqrt{ \mu_v(C)} + \sqrt{(p-2)}\,\epsilon.
\end{equation}
Combining \eqref{eq:B2p1}, \eqref{eq:setC} and \eqref{eq:B2p2} and observing that $B_p=\{(0,\cdots,0)\}$,
we conclude that
$$
\begin{array}{r@{\,\,}l}
\mu_v(\left.M^*\right.^p\setminus & \{(0,\dots,0)\}) =
    \mu_v(\left.M^*\right.^p \setminus B_2) + \mu_v(B_2 \setminus B_p) \leq \\
\leq & \displaystyle \rule[15pt]{0pt}{0pt}
 K(k,d)^2\epsilon^2 + \left(K(k,d)\epsilon + \sqrt{(p-2)}\,\epsilon\right)^2
\leq  K(k,d,p)^2 \epsilon^2.
\end{array}
$$
The assertion of Theorem~\ref{thm:relT_St_p} follows from this inequality
by the same argument as in Theorem~\ref{thm:relT}.
\end{proof}

The above argument can be generalized even further. Let $N$ be an $R$-bimodule generated (as a bi-module) by $d$ elements.
Let $\Gamma$ denote the semidirect product $(\St_p(R) \times \St_q(R)) \ltimes N^{pq}$, $p,q\geq 2$, where
we identify $N^{pq}$ with the abelian group $\Mat_{p\times q} (N)$ of $p\times q$ matrices over $N$ and let
$\St_p(R)$ (resp. $\St_q(R)$) act by left multiplication (resp. right multiplication). Somewhat informally,
we can think of $\Gamma$ as the ``block upper-triangular'' group
$$
\left(
\begin{array}{cc}
\St_p(R) & \Mat_{p\times q} (N) \\
0 & \St_q(R)
\end{array}
\right).
$$
This group has a natural generating set $F$ consisting of the elementary matrices with the generators of $R$ and $N$
off the diagonal.

\begin{Theorem}
\label{thm:relT_Stp_Stq}
In the above setting, the pair $(\Gamma, N^{pq})$ has relative property $(T)$ and
$$\kappa_r(\Gamma, N^{pq}; F) > \frac{1}{2 K(k,d,p,q)},$$
where $K(k,d,p) = 50(\sqrt{k} + \sqrt{d}+\sqrt{p} + \sqrt{q})$.
\end{Theorem}


\printindex

\end{document}